\documentclass[a4paper]{article}

\usepackage{latexsym,amsthm,amscd}
\usepackage[all]{xy}
\usepackage{rotating}
\usepackage[utf8]{inputenc}
\usepackage[T1]{fontenc}
\usepackage[english]{babel}
\usepackage{amsmath}
\usepackage{graphicx}
\usepackage{amssymb}
\usepackage{amsfonts}
\usepackage[svgnames]{xcolor}
\usepackage{geometry}
\usepackage{fancyhdr}
\usepackage{lastpage} 
\usepackage{array}
\usepackage{stmaryrd}
\usepackage{hyperref}
\usepackage{tikz}
\usepackage{bm}
\usepackage[french]{minitoc}
\usepackage{silence}

\usetikzlibrary{patterns}
\usepgflibrary{decorations.pathmorphing}
\usetikzlibrary{decorations.pathmorphing}
\usetikzlibrary{shadows}
\usetikzlibrary{positioning}
\usetikzlibrary{arrows,automata}
\usetikzlibrary{arrows.meta}

\newtheorem*{rep@theorem}{\rep@title}
\newcommand{\newreptheorem}[2]{%
\newenvironment{rep#1}[1]{%
 \def\rep@title{#2 \ref{##1}}%
 \begin{rep@theorem}}%
 {\end{rep@theorem}}}

\newtheorem{theorem}{Theorem}
\newtheorem{proposition}[theorem]{Proposition}
\newtheorem{definition}[theorem]{Definition}

\newtheorem{corollary}[theorem]{Corollary}
\newtheorem{lemma}[theorem]{Lemma}
\newtheorem{example}[theorem]{Example}

\newtheorem{question}{Question}
\newtheorem{remark}{Remark}

\newreptheorem{theorem}{Theorem}
\newreptheorem{lemma}{Lemma}

\title{Quantified block gluing, aperiodicity and 
entropy of multidimensional subshifts of finite type}

\author{Silvère Gangloff, Mathieu Sablik}

\def\N{\mathbb N}

\def\Z{\mathbb Z}

\newcommand{\Lang}{\mathcal{L}}

\newcommand{\id}{\operatorname{Id} }

\newcommand{\A}{\mathcal{A}}

\renewcommand{\vec}[1]{\textbf{#1}}
\newcommand{\U}{\mathbb{U}}

\newcommand{\define}[1]{\textbf{#1}}

\newcommand{\blacksq}{
\begin{tikzpicture}
\fill[Black] (0,0) rectangle (0.3,0.3) ;
\draw (0,0)--(0.3,0) ;
\draw (0,0)--(0,0.3) ;
\draw (0,0.3)--(0.3,0.3) ; 
\draw (0.3,0)--(0.3,0.3) ;
\end{tikzpicture}}

\newcommand{\greensq}{
\begin{tikzpicture}
\fill[YellowGreen] (0,0) rectangle (0.3,0.3) ;
\draw (0,0)--(0.3,0) ;
\draw (0,0)--(0,0.3) ;
\draw (0,0.3)--(0.3,0.3) ; 
\draw (0.3,0)--(0.3,0.3) ;
\end{tikzpicture}}

\newcommand{\redsq}{
\begin{tikzpicture}
\fill[Red] (0,0) rectangle (0.3,0.3) ;
\draw (0,0)--(0.3,0) ;
\draw (0,0)--(0,0.3) ;
\draw (0,0.3)--(0.3,0.3) ; 
\draw (0.3,0)--(0.3,0.3) ;
\end{tikzpicture}}

\newcommand{\purplesq}{
\begin{tikzpicture}
\fill[Purple] (0,0) rectangle (0.3,0.3) ;
\draw (0,0)--(0.3,0) ;
\draw (0,0)--(0,0.3) ;
\draw (0,0.3)--(0.3,0.3) ; 
\draw (0.3,0)--(0.3,0.3) ;
\end{tikzpicture}}

\newcommand{\orangesq}{
\begin{tikzpicture}
\fill[Orange] (0,0) rectangle (0.3,0.3) ;
\draw (0,0)--(0.3,0) ;
\draw (0,0)--(0,0.3) ;
\draw (0,0.3)--(0.3,0.3) ; 
\draw (0.3,0)--(0.3,0.3) ;
\end{tikzpicture}}

\newcommand{\blanksq}{
\begin{tikzpicture}
\draw (0,0)--(0.3,0) ;
\draw (0,0)--(0,0.3) ;
\draw (0,0.3)--(0.3,0.3) ; 
\draw (0.3,0)--(0.3,0.3) ;
\end{tikzpicture}}

\newcommand{\robionegauche}[2]{
\draw (#1,#2) rectangle (#1+2,#2+2);
\draw [-latex] (#1+2,#2+1) -- (#1+0,#2+1);
\draw [-latex] (#1+1,#2+0) -- (#1+1,#2+1); 
\draw [-latex] (#1+1,#2+2) -- (#1+1,#2+1);}
\newcommand{\robionedroite}[2]{
\draw (#1,#2) rectangle (#1+2,#2+2);
\draw [-latex] (#1+0,#2+1) -- (#1+2,#2+1);
\draw [-latex] (#1+1,#2+0) -- (#1+1,#2+1); 
\draw [-latex] (#1+1,#2+2) -- (#1+1,#2+1);}
\newcommand{\robionebas}[2]{
\draw (#1,#2) rectangle (#1+2,#2+2);
\draw [-latex] (#1+1,#2+0) -- (#1+1,#2+2);
\draw [-latex] (#1+0,#2+1) -- (#1+1,#2+1); 
\draw [-latex] (#1+2,#2+1) -- (#1+1,#2+1);}
\newcommand{\robionehaut}[2]{
\draw (#1,#2) rectangle (#1+2,#2+2);
\draw [-latex] (#1+1,#2+2) -- (#1+1,#2+0);
\draw [-latex] (#1+0,#2+1) -- (#1+1,#2+1); 
\draw [-latex] (#1+2,#2+1) -- (#1+1,#2+1);}

\newcommand{\robitwobas}[2]{
\draw (#1,#2) rectangle (#1+2,#2+2);
\draw [-latex] (#1+1,#2+2) -- (#1+1,#2+0) ;
\draw [-latex] (#1+0,#2+1) -- (#1+1,#2+1) ; 
\draw [-latex] (#1+0,#2+0.5) -- (#1+1,#2+0.5) ; 
\draw [-latex] (#1+2,#2+1) -- (#1+1,#2+1) ;
\draw [-latex] (#1+2,#2+0.5) -- (#1+1,#2+0.5) ;}
\newcommand{\robitwohaut}[2]{
\draw (#1,#2) rectangle (#1+2,#2+2);
\draw [-latex] (#1+1,#2+0) -- (#1+1,#2+2) ;
\draw [-latex] (#1+0,#2+1) -- (#1+1,#2+1) ; 
\draw [-latex] (#1+0,#2+1.5) -- (#1+1,#2+1.5) ; 
\draw [-latex] (#1+2,#2+1) -- (#1+1,#2+1) ;
\draw [-latex] (#1+2,#2+1.5) -- (#1+1,#2+1.5) ;}
\newcommand{\robitwodroite}[2]{
\draw (#1,#2) rectangle (#1+2,#2+2);
\draw [-latex] (#1+0,#2+1) -- (#1+2,#2+1) ;
\draw [-latex] (#1+1,#2+0) -- (#1+1,#2+1) ; 
\draw [-latex] (#1+1.5,#2+0) -- (#1+1.5,#2+1) ; 
\draw [-latex] (#1+1,#2+2) -- (#1+1,#2+1) ;
\draw [-latex] (#1+1.5,#2+2) -- (#1+1.5,#2+1) ;}
\newcommand{\robitwogauche}[2]{
\draw (#1,#2) rectangle (#1+2,#2+2);
\draw [-latex] (#1+2,#2+1) -- (#1+0,#2+1) ;
\draw [-latex] (#1+1,#2+0) -- (#1+1,#2+1) ; 
\draw [-latex] (#1+0.5,#2+0) -- (#1+0.5,#2+1) ; 
\draw [-latex] (#1+1,#2+2) -- (#1+1,#2+1) ;
\draw [-latex] (#1+0.5,#2+2) -- (#1+0.5,#2+1) ;}

\newcommand{\robithreebas}[2]{
\draw (#1,#2) rectangle (#1+2,#2+2) ;
\draw [-latex] (#1+1,#2+2) -- (#1+1,#2+0) ;
\draw [-latex] (#1+0.5,#2+2) -- (#1+0.5,#2+0) ; 
\draw [-latex] (#1+0,#2+1) -- (#1+0.5,#2+1) ; 
\draw [-latex] (#1+2,#2+1) -- (#1+1,#2+1) ;}
\newcommand{\robithreehaut}[2]{
\draw (#1,#2) rectangle (#1+2,#2+2) ;
\draw [-latex] (#1+1,#2+0) -- (#1+1,#2+2) ;
\draw [-latex] (#1+0.5,#2+0) -- (#1+0.5,#2+2) ; 
\draw [-latex] (#1+0,#2+1) -- (#1+0.5,#2+1) ; 
\draw [-latex] (#1+2,#2+1) -- (#1+1,#2+1) ;}
\newcommand{\robithreegauche}[2]{
\draw (#1,#2) rectangle (#1+2,#2+2) ;
\draw [-latex] (#1+2,#2+1) -- (#1+0,#2+1) ;
\draw [-latex] (#1+2,#2+0.5) -- (#1+0,#2+0.5) ; 
\draw [-latex] (#1+1,#2+0) -- (#1+1,#2+0.5) ; 
\draw [-latex] (#1+1,#2+2) -- (#1+1,#2+1) ;}
\newcommand{\robithreedroite}[2]{
\draw (#1,#2) rectangle (#1+2,#2+2) ;
\draw [-latex] (#1+0,#2+1) -- (#1+2,#2+1) ;
\draw [-latex] (#1+0,#2+0.5) -- (#1+2,#2+0.5) ; 
\draw [-latex] (#1+1,#2+0) -- (#1+1,#2+0.5) ; 
\draw [-latex] (#1+1,#2+2) -- (#1+1,#2+1) ;}

\newcommand{\robifourbas}[2]{
\draw (#1,#2) rectangle (#1+2,#2+2) ;
\draw [-latex] (#1+1,#2+2) -- (#1+1,#2+0) ;
\draw [-latex] (#1+1.5,#2+2) -- (#1+1.5,#2+0) ; 
\draw [-latex] (#1+0,#2+1) -- (#1+1,#2+1) ; 
\draw [-latex] (#1+2,#2+1) -- (#1+1.5,#2+1) ;}
\newcommand{\robifourgauche}[2]{
\draw (#1,#2) rectangle (#1+2,#2+2) ;
\draw [-latex] (#1+2,#2+1) -- (#1+0,#2+1) ;
\draw [-latex] (#1+2,#2+1.5) -- (#1+0,#2+1.5) ; 
\draw [-latex] (#1+1,#2+0) -- (#1+1,#2+1) ; 
\draw [-latex] (#1+1,#2+2) -- (#1+1,#2+1.5) ;}
\newcommand{\robifourhaut}[2]{
\draw (#1,#2) rectangle (#1+2,#2+2) ;
\draw [-latex] (#1+1,#2+0) -- (#1+1,#2+2) ;
\draw [-latex] (#1+1.5,#2+0) -- (#1+1.5,#2+2) ; 
\draw [-latex] (#1+0,#2+1) -- (#1+1,#2+1) ; 
\draw [-latex] (#1+2,#2+1) -- (#1+1.5,#2+1) ;}
\newcommand{\robifourdroite}[2]{
\draw (#1,#2) rectangle (#1+2,#2+2) ;
\draw [-latex] (#1+0,#2+1) -- (#1+2,#2+1) ;
\draw [-latex] (#1+0,#2+1.5) -- (#1+2,#2+1.5) ; 
\draw [-latex] (#1+1,#2+0) -- (#1+1,#2+1) ; 
\draw [-latex] (#1+1,#2+2) -- (#1+1,#2+1.5) ;}

\newcommand{\robifivebas}[2]{
\draw (#1,#2) rectangle (#1+2,#2+2) ;
\draw [-latex] (#1+1,#2+2) -- (#1+1,#2+0) ; 
\draw [-latex] (#1+0,#2+1) -- (#1+1,#2+1) ; 
\draw [-latex] (#1+2,#2+1) -- (#1+1,#2+1) ;
\draw [-latex] (#1+0,#2+0.5) -- (#1+1,#2+0.5) ; 
\draw [-latex] (#1+2,#2+0.5) -- (#1+1,#2+0.5) ;}
\newcommand{\robifivegauche}[2]{
\draw (#1,#2) rectangle (#1+2,#2+2) ;
\draw [-latex] (#1+2,#2+1) -- (#1+0,#2+1) ; 
\draw [-latex] (#1+1,#2+0) -- (#1+1,#2+1) ; 
\draw [-latex] (#1+1,#2+2) -- (#1+1,#2+1) ;
\draw [-latex] (#1+0.5,#2+0) -- (#1+0.5,#2+1) ; 
\draw [-latex] (#1+0.5,#2+2) -- (#1+0.5,#2+1) ;}
\newcommand{\robifivehaut}[2]{
\draw (#1,#2) rectangle (#1+2,#2+2) ;
\draw [-latex] (#1+1,#2+0) -- (#1+1,#2+2) ; 
\draw [-latex] (#1+0,#2+1) -- (#1+1,#2+1) ; 
\draw [-latex] (#1+2,#2+1) -- (#1+1,#2+1) ;
\draw [-latex] (#1+0,#2+1.5) -- (#1+1,#2+1.5) ; 
\draw [-latex] (#1+2,#2+1.5) -- (#1+1,#2+1.5) ;}

\newcommand{\robisixbas}[2]{
\draw (#1,#2) rectangle (#1+2,#2+2) ;
\draw [-latex] (#1+1,#2+2) -- (#1+1,#2+0) ;
\draw [-latex] (#1+0.5,#2+2) -- (#1+0.5,#2+0) ; 
\draw [-latex] (#1+0,#2+1) -- (#1+0.5,#2+1) ; 
\draw [-latex] (#1+2,#2+1) -- (#1+1,#2+1) ;
\draw [-latex] (#1+0,#2+0.5) -- (#1+0.5,#2+0.5) ; 
\draw [-latex] (#1+2,#2+0.5) -- (#1+1,#2+0.5) ;}
\newcommand{\robisixgauche}[2]{
\draw (#1,#2) rectangle (#1+2,#2+2) ;
\draw [-latex] (#1+2,#2+1) -- (#1+0,#2+1) ;
\draw [-latex] (#1+2,#2+0.5) -- (#1+0,#2+0.5) ; 
\draw [-latex] (#1+1,#2+0) -- (#1+1,#2+0.5) ; 
\draw [-latex] (#1+1,#2+2) -- (#1+1,#2+1) ;
\draw [-latex] (#1+0.5,#2+0) -- (#1+0.5,#2+0.5) ; 
\draw [-latex] (#1+0.5,#2+2) -- (#1+0.5,#2+1) ;}
\newcommand{\robisixhaut}[2]{
\draw (#1,#2) rectangle (#1+2,#2+2) ;
\draw [-latex] (#1+1,#2+0) -- (#1+1,#2+2) ;
\draw [-latex] (#1+0.5,#2+0) -- (#1+0.5,#2+2) ; 
\draw [-latex] (#1+0,#2+1) -- (#1+0.5,#2+1) ; 
\draw [-latex] (#1+2,#2+1) -- (#1+1,#2+1) ;
\draw [-latex] (#1+0,#2+1.5) -- (#1+0.5,#2+1.5) ; 
\draw [-latex] (#1+2,#2+1.5) -- (#1+1,#2+1.5) ;}
\newcommand{\robisixdroite}[2]{
\draw (#1,#2) rectangle (#1+2,#2+2) ;
\draw [-latex] (#1+0,#2+1) -- (#1+2,#2+1) ;
\draw [-latex] (#1+0,#2+0.5) -- (#1+2,#2+0.5) ; 
\draw [-latex] (#1+1,#2+0) -- (#1+1,#2+0.5) ; 
\draw [-latex] (#1+1,#2+2) -- (#1+1,#2+1) ;
\draw [-latex] (#1+1.5,#2+0) -- (#1+1.5,#2+0.5) ; 
\draw [-latex] (#1+1.5,#2+2) -- (#1+1.5,#2+1) ;}

\newcommand{\robisevenbas}[2]{
\draw (#1,#2) rectangle (#1+2,#2+2) ;
\draw [-latex] (#1+1,#2+2) -- (#1+1,#2+0) ;
\draw [-latex] (#1+1.5,#2+2) -- (#1+1.5,#2+0) ; 
\draw [-latex] (#1+0,#2+1) -- (#1+1,#2+1) ; 
\draw [-latex] (#1+2,#2+1) -- (#1+1.5,#2+1) ;
\draw [-latex] (#1+0,#2+0.5) -- (#1+1,#2+0.5) ; 
\draw [-latex] (#1+2,#2+0.5) -- (#1+1.5,#2+0.5) ;}
\newcommand{\robisevengauche}[2]{
\draw (#1,#2) rectangle (#1+2,#2+2) ;
\draw [-latex] (#1+2,#2+1) -- (#1+0,#2+1) ;
\draw [-latex] (#1+2,#2+1.5) -- (#1+0,#2+1.5) ; 
\draw [-latex] (#1+1,#2+0) -- (#1+1,#2+1) ; 
\draw [-latex] (#1+1,#2+2) -- (#1+1,#2+1.5) ;
\draw [-latex] (#1+0.5,#2+0) -- (#1+0.5,#2+1) ; 
\draw [-latex] (#1+0.5,#2+2) -- (#1+0.5,#2+1.5) ;}
\newcommand{\robisevenhaut}[2]{
\draw (#1,#2) rectangle (#1+2,#2+2) ;
\draw [-latex] (#1+1,#2+0) -- (#1+1,#2+2) ;
\draw [-latex] (#1+1.5,#2+0) -- (#1+1.5,#2+2) ; 
\draw [-latex] (#1+0,#2+1) -- (#1+1,#2+1) ; 
\draw [-latex] (#1+2,#2+1) -- (#1+1.5,#2+1) ;
\draw [-latex] (#1+0,#2+1.5) -- (#1+1,#2+1.5) ; 
\draw [-latex] (#1+2,#2+1.5) -- (#1+1.5,#2+1.5) ;}
\newcommand{\robisevendroite}[2]{
\draw (#1,#2) rectangle (#1+2,#2+2) ;
\draw [-latex] (#1+0,#2+1) -- (#1+2,#2+1) ;
\draw [-latex] (#1+0,#2+1.5) -- (#1+2,#2+1.5) ; 
\draw [-latex] (#1+1,#2+0) -- (#1+1,#2+1) ; 
\draw [-latex] (#1+1,#2+2) -- (#1+1,#2+1.5) ;
\draw [-latex] (#1+1.5,#2+0) -- (#1+1.5,#2+1) ; 
\draw [-latex] (#1+1.5,#2+2) -- (#1+1.5,#2+1.5) ;}

\newcommand{\robibluebasgauche}[2]{
\fill[blue!40] (#1+0.5,#2+0.5) rectangle (#1+1,#2+2) ;
\fill[blue!40] (#1+0.5,#2+0.5) rectangle (#1+2,#2+1) ;
\node[scale=0.75] at (#1+1.5,#2+1.5) {0};
\draw (#1,#2) rectangle (#1+2,#2+2) ;
\draw [-latex] (#1+0.5,#2+0.5) -- (#1+0.5,#2+2) ; 
\draw [-latex] (#1+0.5,#2+0.5) -- (#1+2,#2+0.5) ;
\draw [-latex] (#1+1,#2+1) -- (#1+1,#2+2) ; 
\draw [-latex] (#1+1,#2+1) -- (#1+2,#2+1) ; 
\draw [-latex] (#1+0.5,#2+1) -- (#1+0,#2+1) ; 
\draw [-latex] (#1+1,#2+0.5) -- (#1+1,#2+0) ;}
\newcommand{\robibluebasdroite}[2]{
\fill[blue!40] (#1+1.5,#2+0.5) rectangle (#1+1,#2+2) ;
\fill[blue!40] (#1+1.5,#2+0.5) rectangle (#1+0,#2+1) ;
\draw (#1,#2) rectangle (#1+2,#2+2) ;
\node[scale=0.75] at (#1+0.5,#2+1.5) {0};
\draw [-latex] (#1+1.5,#2+0.5) -- (#1+1.5,#2+2) ; 
\draw [-latex] (#1+1.5,#2+0.5) -- (#1+0,#2+0.5) ;
\draw [-latex] (#1+1,#2+1) -- (#1+1,#2+2) ; 
\draw [-latex] (#1+1,#2+1) -- (#1+0,#2+1) ; 
\draw [-latex] (#1+1.5,#2+1) -- (#1+2,#2+1) ; 
\draw [-latex] (#1+1,#2+0.5) -- (#1+1,#2+0) ;}
\newcommand{\robibluehautgauche}[2]{
\fill[blue!40] (#1+2,#2+1) rectangle (#1+0.5,#2+1.5) ;
\fill[blue!40] (#1+1,#2+0) rectangle (#1+0.5,#2+1.5) ;
\draw (#1,#2) rectangle (#1+2,#2+2) ;
\node[scale=0.75] at (#1+1.5,#2+0.5) {0};
\draw [-latex] (#1+0.5,#2+1.5) -- (#1+0.5,#2+0) ; 
\draw [-latex] (#1+0.5,#2+1.5) -- (#1+2,#2+1.5) ;
\draw [-latex] (#1+1,#2+1) -- (#1+1,#2+0) ; 
\draw [-latex] (#1+1,#2+1) -- (#1+2,#2+1) ; 
\draw [-latex] (#1+0.5,#2+1) -- (#1+0,#2+1) ; 
\draw [-latex] (#1+1,#2+1.5) -- (#1+1,#2+2) ;}
\newcommand{\robibluehautdroite}[2]{
\fill[blue!40] (#1+0,#2+1) rectangle (#1+1.5,#2+1.5) ;
\fill[blue!40] (#1+1,#2+0) rectangle (#1+1.5,#2+1.5) ;
\draw (#1,#2) rectangle (#1+2,#2+2) ;
\node[scale=0.75] at (#1+0.5,#2+0.5) {0};
\draw [-latex] (#1+1.5,#2+1.5) -- (#1+1.5,#2+0) ; 
\draw [-latex] (#1+1.5,#2+1.5) -- (#1+0,#2+1.5) ;
\draw [-latex] (#1+1,#2+1) -- (#1+1,#2+0) ; 
\draw [-latex] (#1+1,#2+1) -- (#1+0,#2+1) ; 
\draw [-latex] (#1+1.5,#2+1) -- (#1+2,#2+1) ; 
\draw [-latex] (#1+1,#2+1.5) -- (#1+1,#2+2) ;}

\newcommand{\robibluebastgauche}[2]{
\fill[blue!40] (#1+0.5,#2+0.5) rectangle (#1+1,#2+2) ;
\fill[blue!40] (#1+0.5,#2+0.5) rectangle (#1+2,#2+1) ;
\draw (#1,#2) rectangle (#1+2,#2+2) ;
\draw [-latex] (#1+0.5,#2+0.5) -- (#1+0.5,#2+2) ; 
\draw [-latex] (#1+0.5,#2+0.5) -- (#1+2,#2+0.5) ;
\draw [-latex] (#1+1,#2+1) -- (#1+1,#2+2) ; 
\draw [-latex] (#1+1,#2+1) -- (#1+2,#2+1) ; 
\draw [-latex] (#1+0.5,#2+1) -- (#1+0,#2+1) ; 
\draw [-latex] (#1+1,#2+0.5) -- (#1+1,#2+0) ;}
\newcommand{\robibluebastdroite}[2]{
\fill[blue!40] (#1+1.5,#2+0.5) rectangle (#1+1,#2+2) ;
\fill[blue!40] (#1+1.5,#2+0.5) rectangle (#1+0,#2+1) ;
\draw (#1,#2) rectangle (#1+2,#2+2) ;
\draw [-latex] (#1+1.5,#2+0.5) -- (#1+1.5,#2+2) ; 
\draw [-latex] (#1+1.5,#2+0.5) -- (#1+0,#2+0.5) ;
\draw [-latex] (#1+1,#2+1) -- (#1+1,#2+2) ; 
\draw [-latex] (#1+1,#2+1) -- (#1+0,#2+1) ; 
\draw [-latex] (#1+1.5,#2+1) -- (#1+2,#2+1) ; 
\draw [-latex] (#1+1,#2+0.5) -- (#1+1,#2+0) ;}
\newcommand{\robibluehauttgauche}[2]{
\fill[blue!40] (#1+2,#2+1) rectangle (#1+0.5,#2+1.5) ;
\fill[blue!40] (#1+1,#2+0) rectangle (#1+0.5,#2+1.5) ;
\draw (#1,#2) rectangle (#1+2,#2+2) ;
\draw [-latex] (#1+0.5,#2+1.5) -- (#1+0.5,#2+0) ; 
\draw [-latex] (#1+0.5,#2+1.5) -- (#1+2,#2+1.5) ;
\draw [-latex] (#1+1,#2+1) -- (#1+1,#2+0) ; 
\draw [-latex] (#1+1,#2+1) -- (#1+2,#2+1) ; 
\draw [-latex] (#1+0.5,#2+1) -- (#1+0,#2+1) ; 
\draw [-latex] (#1+1,#2+1.5) -- (#1+1,#2+2) ;}
\newcommand{\robibluehauttdroite}[2]{
\fill[blue!40] (#1+0,#2+1) rectangle (#1+1.5,#2+1.5) ;
\fill[blue!40] (#1+1,#2+0) rectangle (#1+1.5,#2+1.5) ;
\draw (#1,#2) rectangle (#1+2,#2+2) ;
\draw [-latex] (#1+1.5,#2+1.5) -- (#1+1.5,#2+0) ; 
\draw [-latex] (#1+1.5,#2+1.5) -- (#1+0,#2+1.5) ;
\draw [-latex] (#1+1,#2+1) -- (#1+1,#2+0) ; 
\draw [-latex] (#1+1,#2+1) -- (#1+0,#2+1) ; 
\draw [-latex] (#1+1.5,#2+1) -- (#1+2,#2+1) ; 
\draw [-latex] (#1+1,#2+1.5) -- (#1+1,#2+2) ;}

\newcommand{\robiredbasgauche}[2]{
\fill[red!40] (#1+0.5,#2+0.5) rectangle (#1+1,#2+2) ;
\fill[red!40] (#1+0.5,#2+0.5) rectangle (#1+2,#2+1) ;
\draw (#1,#2) rectangle (#1+2,#2+2) ;
\draw [-latex] (#1+0.5,#2+0.5) -- (#1+0.5,#2+2) ; 
\draw [-latex] (#1+0.5,#2+0.5) -- (#1+2,#2+0.5) ;
\draw [-latex] (#1+1,#2+1) -- (#1+1,#2+2) ; 
\draw [-latex] (#1+1,#2+1) -- (#1+2,#2+1) ; 
\draw [-latex] (#1+0.5,#2+1) -- (#1+0,#2+1) ; 
\draw [-latex] (#1+1,#2+0.5) -- (#1+1,#2+0) ;}
\newcommand{\robiredbasdroite}[2]{
\fill[red!40] (#1+1.5,#2+0.5) rectangle (#1+1,#2+2) ;
\fill[red!40] (#1+1.5,#2+0.5) rectangle (#1+0,#2+1) ;
\draw (#1,#2) rectangle (#1+2,#2+2) ;
\draw [-latex] (#1+1.5,#2+0.5) -- (#1+1.5,#2+2) ; 
\draw [-latex] (#1+1.5,#2+0.5) -- (#1+0,#2+0.5) ;
\draw [-latex] (#1+1,#2+1) -- (#1+1,#2+2) ; 
\draw [-latex] (#1+1,#2+1) -- (#1+0,#2+1) ; 
\draw [-latex] (#1+1.5,#2+1) -- (#1+2,#2+1) ; 
\draw [-latex] (#1+1,#2+0.5) -- (#1+1,#2+0) ;}
\newcommand{\robiredhautgauche}[2]{
\fill[red!40] (#1+2,#2+1) rectangle (#1+0.5,#2+1.5) ;
\fill[red!40] (#1+1,#2+0) rectangle (#1+0.5,#2+1.5) ;
\draw (#1,#2) rectangle (#1+2,#2+2) ;
\draw [-latex] (#1+0.5,#2+1.5) -- (#1+0.5,#2+0) ; 
\draw [-latex] (#1+0.5,#2+1.5) -- (#1+2,#2+1.5) ;
\draw [-latex] (#1+1,#2+1) -- (#1+1,#2+0) ; 
\draw [-latex] (#1+1,#2+1) -- (#1+2,#2+1) ; 
\draw [-latex] (#1+0.5,#2+1) -- (#1+0,#2+1) ; 
\draw [-latex] (#1+1,#2+1.5) -- (#1+1,#2+2) ;}
\newcommand{\robiredhautdroite}[2]{
\fill[red!40] (#1+0,#2+1) rectangle (#1+1.5,#2+1.5) ;
\fill[red!40] (#1+1,#2+0) rectangle (#1+1.5,#2+1.5) ;
\draw (#1,#2) rectangle (#1+2,#2+2) ;
\draw [-latex] (#1+1.5,#2+1.5) -- (#1+1.5,#2+0) ; 
\draw [-latex] (#1+1.5,#2+1.5) -- (#1+0,#2+1.5) ;
\draw [-latex] (#1+1,#2+1) -- (#1+1,#2+0) ; 
\draw [-latex] (#1+1,#2+1) -- (#1+0,#2+1) ; 
\draw [-latex] (#1+1.5,#2+1) -- (#1+2,#2+1) ; 
\draw [-latex] (#1+1,#2+1.5) -- (#1+1,#2+2) ;}

\newcommand{\supertiletwobasgauche}[2]{
\robibluebasgauche{#1+0}{#2+0};
\robibluebasgauche{#1+8}{#2+0};
\robibluebasgauche{#1+0}{#2+8};
\robibluebasgauche{#1+8}{#2+8};
\robiredbasgauche{#1+2}{#2+2};
\robiredbasgauche{#1+6}{#2+6};
\robibluehautgauche{#1+0}{#2+4};
\robibluehautgauche{#1+8}{#2+4};
\robibluehautgauche{#1+0}{#2+12};
\robibluehautgauche{#1+8}{#2+12};
\robiredhautgauche{#1+2}{#2+10};
\robibluebasdroite{#1+4}{#2+0};
\robibluebasdroite{#1+12}{#2+0};
\robibluebasdroite{#1+4}{#2+8};
\robibluebasdroite{#1+12}{#2+8};
\robiredbasdroite{#1+10}{#2+2};
\robibluehautdroite{#1+4}{#2+4};
\robibluehautdroite{#1+12}{#2+4};
\robibluehautdroite{#1+4}{#2+12};
\robibluehautdroite{#1+12}{#2+12};
\robiredhautdroite{#1+10}{#2+10};
\robitwobas{#1+2}{#2+0}
\robitwobas{#1+10}{#2+0}
\robitwobas{#1+6}{#2+2}
\robitwogauche{#1+0}{#2+2}
\robitwogauche{#1+2}{#2+6}
\robitwogauche{#1+0}{#2+10}
\robitwodroite{#1+12}{#2+2}
\robitwodroite{#1+12}{#2+10}
\robitwohaut{#1+2}{#2+12}
\robitwohaut{#1+10}{#2+12}
\robionehaut{#1+6}{#2+0}
\robionehaut{#1+6}{#2+4}
\robionegauche{#1+0}{#2+6}
\robionegauche{#1+4}{#2+6}
\robisixbas{#1+2}{#2+8}
\robisixdroite{#1+4}{#2+2}
\robisixhaut{#1+2}{#2+4}
\robisevendroite{#1+4}{#2+10}
\robithreehaut{#1+6}{#2+8}
\robisixhaut{#1+6}{#2+10}
\robithreehaut{#1+6}{#2+12}
\robisixgauche{#1+8}{#2+2}
\robisevengauche{#1+8}{#2+10}
\robisevenbas{#1+10}{#2+8}
\robisevenhaut{#1+10}{#2+4}
\robisixdroite{#1+10}{#2+6}
\robithreedroite{#1+8}{#2+6}
\robithreedroite{#1+12}{#2+6}}

\newcommand{\supertiletwobasdroite}[2]{
\robibluebasgauche{#1+0}{#2+0};
\robibluebasgauche{#1+8}{#2+0};
\robibluebasgauche{#1+0}{#2+8};
\robibluebasgauche{#1+8}{#2+8};
\robiredbasgauche{#1+2}{#2+2};
\robiredbasdroite{#1+6}{#2+6};
\robibluehautgauche{#1+0}{#2+4};
\robibluehautgauche{#1+8}{#2+4};
\robibluehautgauche{#1+0}{#2+12};
\robibluehautgauche{#1+8}{#2+12};
\robiredhautgauche{#1+2}{#2+10};
\robibluebasdroite{#1+4}{#2+0};
\robibluebasdroite{#1+12}{#2+0};
\robibluebasdroite{#1+4}{#2+8};
\robibluebasdroite{#1+12}{#2+8};
\robiredbasdroite{#1+10}{#2+2};
\robibluehautdroite{#1+4}{#2+4};
\robibluehautdroite{#1+12}{#2+4};
\robibluehautdroite{#1+4}{#2+12};
\robibluehautdroite{#1+12}{#2+12};
\robiredhautdroite{#1+10}{#2+10};
\robitwobas{#1+2}{#2+0}
\robitwobas{#1+10}{#2+0}
\robitwobas{#1+6}{#2+2}
\robitwogauche{#1+0}{#2+2}
\robitwodroite{#1+10}{#2+6}
\robitwogauche{#1+0}{#2+10}
\robitwodroite{#1+12}{#2+2}
\robitwodroite{#1+12}{#2+10}
\robitwohaut{#1+2}{#2+12}
\robitwohaut{#1+10}{#2+12}
\robionehaut{#1+6}{#2+0}
\robionehaut{#1+6}{#2+4}
\robionedroite{#1+12}{#2+6}
\robionedroite{#1+8}{#2+6}
\robisixbas{#1+2}{#2+8}
\robisixdroite{#1+4}{#2+2}
\robisixhaut{#1+2}{#2+4}
\robisevendroite{#1+4}{#2+10}
\robifourhaut{#1+6}{#2+8}
\robisevenhaut{#1+6}{#2+10}
\robifourhaut{#1+6}{#2+12}
\robisixgauche{#1+8}{#2+2}
\robisevengauche{#1+8}{#2+10}
\robisevenbas{#1+10}{#2+8}
\robisevenhaut{#1+10}{#2+4}
\robisixgauche{#1+2}{#2+6}
\robithreegauche{#1+4}{#2+6}
\robithreegauche{#1+0}{#2+6}}

\newcommand{\supertiletwohautdroite}[2]{
\robibluebasgauche{#1+0}{#2+0};
\robibluebasgauche{#1+8}{#2+0};
\robibluebasgauche{#1+0}{#2+8};
\robibluebasgauche{#1+8}{#2+8};
\robiredbasgauche{#1+2}{#2+2};
\robiredhautdroite{#1+6}{#2+6};
\robibluehautgauche{#1+0}{#2+4};
\robibluehautgauche{#1+8}{#2+4};
\robibluehautgauche{#1+0}{#2+12};
\robibluehautgauche{#1+8}{#2+12};
\robiredhautgauche{#1+2}{#2+10};
\robibluebasdroite{#1+4}{#2+0};
\robibluebasdroite{#1+12}{#2+0};
\robibluebasdroite{#1+4}{#2+8};
\robibluebasdroite{#1+12}{#2+8};
\robiredbasdroite{#1+10}{#2+2};
\robibluehautdroite{#1+4}{#2+4};
\robibluehautdroite{#1+12}{#2+4};
\robibluehautdroite{#1+4}{#2+12};
\robibluehautdroite{#1+12}{#2+12};
\robiredhautdroite{#1+10}{#2+10};
\robitwobas{#1+2}{#2+0}
\robitwobas{#1+10}{#2+0}
\robitwohaut{#1+6}{#2+10}
\robitwogauche{#1+0}{#2+2}
\robitwodroite{#1+10}{#2+6}
\robitwogauche{#1+0}{#2+10}
\robitwodroite{#1+12}{#2+2}
\robitwodroite{#1+12}{#2+10}
\robitwohaut{#1+2}{#2+12}
\robitwohaut{#1+10}{#2+12}
\robionebas{#1+6}{#2+12}
\robionebas{#1+6}{#2+8}
\robionedroite{#1+12}{#2+6}
\robionedroite{#1+8}{#2+6}
\robisixbas{#1+2}{#2+8}
\robisixdroite{#1+4}{#2+2}
\robisixhaut{#1+2}{#2+4}
\robisevendroite{#1+4}{#2+10}
\robifourbas{#1+6}{#2+4}
\robisevenbas{#1+6}{#2+2}
\robifourbas{#1+6}{#2+0}
\robisixgauche{#1+8}{#2+2}
\robisevengauche{#1+8}{#2+10}
\robisevenbas{#1+10}{#2+8}
\robisevenhaut{#1+10}{#2+4}
\robisevengauche{#1+2}{#2+6}
\robifourgauche{#1+4}{#2+6}
\robifourgauche{#1+0}{#2+6}}

\newcommand{\supertiletwohautgauche}[2]{
\robibluebasgauche{#1+0}{#2+0};
\robibluebasgauche{#1+8}{#2+0};
\robibluebasgauche{#1+0}{#2+8};
\robibluebasgauche{#1+8}{#2+8};
\robiredbasgauche{#1+2}{#2+2};
\robiredhautgauche{#1+6}{#2+6};
\robibluehautgauche{#1+0}{#2+4};
\robibluehautgauche{#1+8}{#2+4};
\robibluehautgauche{#1+0}{#2+12};
\robibluehautgauche{#1+8}{#2+12};
\robiredhautgauche{#1+2}{#2+10};
\robibluebasdroite{#1+4}{#2+0};
\robibluebasdroite{#1+12}{#2+0};
\robibluebasdroite{#1+4}{#2+8};
\robibluebasdroite{#1+12}{#2+8};
\robiredbasdroite{#1+10}{#2+2};
\robibluehautdroite{#1+4}{#2+4};
\robibluehautdroite{#1+12}{#2+4};
\robibluehautdroite{#1+4}{#2+12};
\robibluehautdroite{#1+12}{#2+12};
\robiredhautdroite{#1+10}{#2+10};
\robitwobas{#1+2}{#2+0}
\robitwobas{#1+10}{#2+0}
\robitwohaut{#1+6}{#2+10}
\robitwogauche{#1+0}{#2+2}
\robitwogauche{#1+2}{#2+6}
\robitwogauche{#1+0}{#2+10}
\robitwodroite{#1+12}{#2+2}
\robitwodroite{#1+12}{#2+10}
\robitwohaut{#1+2}{#2+12}
\robitwohaut{#1+10}{#2+12}
\robionebas{#1+6}{#2+12}
\robionebas{#1+6}{#2+8}
\robionegauche{#1+0}{#2+6}
\robionegauche{#1+4}{#2+6}
\robisixbas{#1+2}{#2+8}
\robisixdroite{#1+4}{#2+2}
\robisixhaut{#1+2}{#2+4}
\robisevendroite{#1+4}{#2+10}
\robithreebas{#1+6}{#2+4}
\robisixbas{#1+6}{#2+2}
\robithreebas{#1+6}{#2+0}
\robisixgauche{#1+8}{#2+2}
\robisevengauche{#1+8}{#2+10}
\robisevenbas{#1+10}{#2+8}
\robisevenhaut{#1+10}{#2+4}
\robisevendroite{#1+10}{#2+6}
\robifourdroite{#1+8}{#2+6}
\robifourdroite{#1+12}{#2+6}}

\newcommand{\supertilesynchrotwo}[2]{
\fill[gray!90] (#1-2,#2-2) rectangle (#1+32,#2+32);
\fill[gray!40] (#1+0,#2+0) rectangle (#1+30,#2+30);
\fill[red!40] (#1+0,#2+0) rectangle (#1+14,#2+14);
\fill[orange!40] (#1+16,#2+16) rectangle (#1+30,#2+30);
\fill[purple!40] (#1+0,#2+16) rectangle (#1+14,#2+30);
\fill[yellow!40] (#1+16,#2+0) rectangle (#1+30,#2+14); 
\fill[gray!90] (#1+2,#2+2) rectangle (#1+12,#2+12);
\fill[gray!90] (#1+2,#2+18) rectangle (#1+12,#2+28);
\fill[gray!90] (#1+18,#2+18) rectangle (#1+28,#2+28);
\fill[gray!90] (#1+18,#2+2) rectangle (#1+28,#2+12);
\fill[gray!40] (#1+4,#2+20) rectangle (#1+10,#2+26);
\fill[gray!40] (#1+20,#2+20) rectangle (#1+26,#2+26);
\fill[gray!40] (#1+20,#2+4) rectangle (#1+26,#2+10);
\fill[gray!40] (#1+4,#2+4) rectangle (#1+10,#2+10);
\fill[red!40] (#1+4,#2+4) rectangle (#1+6,#2+6);
\fill[yellow!40] (#1+8,#2+4) rectangle (#1+10,#2+6);
\fill[purple!40] (#1+4,#2+8) rectangle (#1+6,#2+10);
\fill[orange!40] (#1+8,#2+8) rectangle (#1+10,#2+10);
\fill[red!40] (#1+20,#2+4) rectangle (#1+22,#2+6);
\fill[yellow!40] (#1+24,#2+4) rectangle (#1+26,#2+6);
\fill[purple!40] (#1+20,#2+8) rectangle (#1+22,#2+10);
\fill[orange!40] (#1+24,#2+8) rectangle (#1+26,#2+10);
\fill[red!40] (#1+20,#2+20) rectangle (#1+22,#2+22);
\fill[yellow!40] (#1+24,#2+20) rectangle (#1+26,#2+22);
\fill[purple!40] (#1+20,#2+24) rectangle (#1+22,#2+26);
\fill[orange!40] (#1+24,#2+24) rectangle (#1+26,#2+26);
\fill[red!40] (#1+4,#2+20) rectangle (#1+6,#2+22);
\fill[yellow!40] (#1+8,#2+20) rectangle (#1+10,#2+22);
\fill[purple!40] (#1+4,#2+24) rectangle (#1+6,#2+26);
\fill[orange!40] (#1+8,#2+24) rectangle (#1+10,#2+26);}

\newcommand{\supertilesynchroone}[2]{
\fill[gray!90] (#1+0,#2+0) rectangle (#1+10,#2+10);
\fill[gray!40] (#1+2,#2+2) rectangle (#1+8,#2+8);
\fill[red!40] (#1+2,#2+2) rectangle (#1+4,#2+4);
\fill[yellow!40] (#1+6,#2+2) rectangle (#1+8,#2+4);
\fill[purple!40] (#1+2,#2+6) rectangle (#1+4,#2+8);
\fill[orange!40] (#1+6,#2+6) rectangle (#1+8,#2+8);
}

\newcommand{\supertilesynchrothree}{
\fill[gray!90] (-2,-2) rectangle (2*64,2*64);
\fill[gray!40] (0,0) rectangle (2*63,2*63);
\fill[red!40] (0,0) rectangle (2*31,2*31);
\fill[orange!40] (2*32,2*32) rectangle (2*63,2*63);
\fill[yellow!40] (2*32,2*31) rectangle (2*63,0);
\fill[purple!40] (2*31,2*32) rectangle (0,2*63);
\supertilesynchrotwo{16}{16}
\supertilesynchrotwo{16}{16+4*16}
\supertilesynchrotwo{16+4*16}{16+4*16}
\supertilesynchrotwo{16+4*16}{16}
\supertilesynchroone{66}{66}
\supertilesynchroone{-4*16+66}{66}
\supertilesynchroone{-3*16+66}{66}
\supertilesynchroone{-2*16+66}{66}
\supertilesynchroone{-1*16+66}{66}
\supertilesynchroone{16+66}{66}
\supertilesynchroone{2*16+66}{66}
\supertilesynchroone{3*16+66}{66}
\supertilesynchroone{-4*16+66}{16+66}
\supertilesynchroone{-3*16+66}{16+66}
\supertilesynchroone{-2*16+66}{16+66}
\supertilesynchroone{-1*16+66}{16+66}
\supertilesynchroone{16+66}{16+66}
\supertilesynchroone{66}{16+66}
\supertilesynchroone{2*16+66}{16+66}
\supertilesynchroone{3*16+66}{16+66}
\supertilesynchroone{-4*16+66}{2*16+66}
\supertilesynchroone{-3*16+66}{2*16+66}
\supertilesynchroone{-2*16+66}{2*16+66}
\supertilesynchroone{-1*16+66}{2*16+66}
\supertilesynchroone{66}{2*16+66}
\supertilesynchroone{16+66}{2*16+66}
\supertilesynchroone{2*16+66}{2*16+66}
\supertilesynchroone{3*16+66}{2*16+66}

\supertilesynchroone{-4*16+66}{3*16+66}
\supertilesynchroone{-3*16+66}{3*16+66}
\supertilesynchroone{-2*16+66}{3*16+66}
\supertilesynchroone{-1*16+66}{3*16+66}
\supertilesynchroone{66}{3*16+66}
\supertilesynchroone{16+66}{3*16+66}
\supertilesynchroone{2*16+66}{3*16+66}
\supertilesynchroone{3*16+66}{3*16+66}

\supertilesynchroone{-4*16+66}{-1*16+66}
\supertilesynchroone{-3*16+66}{-1*16+66}
\supertilesynchroone{-2*16+66}{-1*16+66}
\supertilesynchroone{-1*16+66}{-1*16+66}
\supertilesynchroone{66}{-1*16+66}
\supertilesynchroone{16+66}{-1*16+66}
\supertilesynchroone{2*16+66}{-1*16+66}
\supertilesynchroone{3*16+66}{-1*16+66}

\supertilesynchroone{-4*16+66}{-2*16+66}
\supertilesynchroone{-3*16+66}{-2*16+66}
\supertilesynchroone{-2*16+66}{-2*16+66}
\supertilesynchroone{-1*16+66}{-2*16+66}
\supertilesynchroone{66}{-2*16+66}
\supertilesynchroone{16+66}{-2*16+66}
\supertilesynchroone{2*16+66}{-2*16+66}
\supertilesynchroone{3*16+66}{-2*16+66}

\supertilesynchroone{-4*16+66}{-3*16+66}
\supertilesynchroone{-3*16+66}{-3*16+66}
\supertilesynchroone{-2*16+66}{-3*16+66}
\supertilesynchroone{-1*16+66}{-3*16+66}
\supertilesynchroone{66}{-3*16+66}
\supertilesynchroone{16+66}{-3*16+66}
\supertilesynchroone{2*16+66}{-3*16+66}
\supertilesynchroone{3*16+66}{-3*16+66}

\supertilesynchroone{-4*16+66}{-4*16+66}
\supertilesynchroone{-3*16+66}{-4*16+66}
\supertilesynchroone{-2*16+66}{-4*16+66}
\supertilesynchroone{-1*16+66}{-4*16+66}
\supertilesynchroone{66}{-4*16+66}
\supertilesynchroone{16+66}{-4*16+66}
\supertilesynchroone{2*16+66}{-4*16+66}
\supertilesynchroone{3*16+66}{-4*16+66}
}

\begin{document}

\maketitle

\begin{abstract} 
In this text, 
we introduce the notion of block gluing
with gap function. We study the interplay between these properties
and properties related to computability
for $\Z^2$-SFTs.
In particular, we prove that there exist linearly block 
gluing SFTs which are aperiodic and 
that all the non-negative $\Pi_1$-computable numbers can 
be realized as entropy of linearly block gluing 
$\Z^2$-subshifts of finite type. 
As linearly block gluing implies transitivity, 
this last point answers 
a question asked in~\cite{Hochman-Meyerovitch-2010} about 
the characterization of the entropies of 
transitive subshift of finite type.
\end{abstract}

\section{Introduction}

It appeared recently that 
it is possible to understand many dynamical properties 
of multidimensional subshifts of finite type 
through characterizations results, 
beginning with the entropy~\cite{Hochman-Meyerovitch-2010}. 
This interplay between dynamical properties and   
computability or decidability properties was 
shown by a lot of recent works: 
a characterization of the projective sub-action of a 
SFT~\cite{Hochman-2009,Aubrun-Sablik-2010,DRS}, 
measure of the computationnally simplest configurations 
with Medvedev degrees 
\cite{sim2014} and sets of Turing degrees~
\cite{JeandelV13}, characterization of the possible 
sets of periods in terms of 
complexity theory~\cite{Jeandel-Vanier-2014}, etc.

The importance of computability 
considerations in these models has been clearly established for 
decades now. A new direction is to see if dynamical properties 
can prevent embedding universal computing. 

One can observe for instance this phenomenon for some mixing-like properties.
To be more precise, the result proved by 
M. Hochman and T. Meyerovitch~\cite{Hochman-Meyerovitch-2010} 
states that the set of numbers that are entropy of a multidimensional SFT 
is exactly the set of non-negative $\Pi_1$-computable real numbers. 
When the SFT is strongly irreducible, the entropy 
becomes computable. 
Only some step results 
towards a characterization are known~\cite{Pavlov-Schraudner-2014}. 
We interpret this obstruction as a reduction of the computational 
power of the model under this restriction, which is manifested 
by the reduction of possible entropies.

In~\cite{Pavlov-Schraudner-2014} the authors studied SFT
which are block gluing. This means 
that there exists a 
constant $c$ such that for any couple of square 
blocks in the language of the subshift, 
the pattern 
obtained by gluing these two blocks in any 
way respecting 
a distance $c$ between them is also in the language. 

We propose in this text to study similar 
properties that consist
in imposing the gluing property when 
the two blocks are spaced by 
$f(n)$, where $n$ is the length 
of two square blocks having the same size, 
and $f:\mathbb{N}\to\mathbb{N}$. \bigskip   

We got interested in the influence of this property on 
the possibility of non-existence of periodic orbits 
(which is the first ingredient for embedding computations in SFT), 
and on the set of possible entropies, according to the gap
function $f$. We observed two regimes:
\begin{itemize}
\item \textbf{Sub-logarithmic:}
if $f\in o(\log(n))$ and $f \le id$, then the set of periodic orbits of any $f$-block gluing SFT is dense 
(Proposition~\ref{prop.6}) and its 
language is decidable. Moreover, 
the entropy is computable 
(Proposition~\ref{prop.computability.sublog}).

 \item \textbf{Linear:}
\begin{itemize}
 \item There exists some 
$f\in O(n)$ such that 
there exists an aperiodic $f$-block gluing SFT (Theorem~\ref{thm.aperiodicity}).
 \item There exists $f\in O(n)$ such that there  exists a $f$-block gluing SFT with non decidable language 
(Proposition~\ref{prop.langundec}).
 \item The possible
entropies of $f$-block gluing SFT 
with $f \in O(n)$ are the right recursively enumerable non-negative real numbers 
(Theorem~\ref{thm.entropy.block.gluing}).
\end{itemize}
\end{itemize}

The aim was to understand the limit between these 
two regimes.
We extended to the 
sub-logarithmic regime some 
results 
on block gluing subshifts of~\cite{Pavlov-Schraudner-2014}.
We adapted the realization result 
of~\cite{Hochman-Meyerovitch-2010} 
to block gluing with linear 
gap function. This adaptation 
involved more complex and structured 
constructions.
The mechanisms introduced exhibit 
an attractive analogy with very 
simple molecular biology objects.
This analogy is present in the words that 
we use in to describe 
the construction, that are useful to visualize and 
understand the construction.
In particular, the applied constraint 
has resulted in the 
centralization and fixation
of information (called DNA) 
in the centers (called nuclei) of the computing 
units (called cells). The computing 
machines  
communicate using error 
signals to have access to this information, 
which codes for its behavior.

This leads to the intuition that this type of results could be read in 
order to understand the principles of \textit{information 
processing systems}, as 
\textit{''mixing implies centralization of the information''}. \bigskip

The text is organized as follows: 
\begin{itemize} 
\item In Section~\ref{sec.definitions.generales} we recall 
some definitions related to the dynamics of subshifts 
of finite type. 
\item Section~\ref{sec.definition.robinson} we present some 
rigidified version of the Robinson subshift used 
in the main construction. We proved 
some properties of this subshift. 
\item In Section \ref{sec.definition.block.gluing}
we define the 
 notion of block gluing with gap function, 
and provide many examples.
\item Section \ref{section.periodic} is 
devoted to study of the existence and density 
of periodic points under block gluing 
constraint.
\item Section \ref{sec.Entropy} 
is devoted to the entropy.
\end{itemize}

\section{\label{sec.definitions.generales} Definitions}

In this section, we recall some definitions: subshifts, entropy dimension, 
$\Delta_2$-computable numbers. 

\subsection{Subshifts dynamical systems}

Let $\mathcal{A}$ be some finite set, called
\textbf{alphabet}. Let $d \ge 1$ be an integer.
The set 
$\mathcal{A}^{\Z^d}$ is a topological 
space with the product of the discrete 
topology on $\mathcal{A}$. Its 
elements are called \textbf{configurations}. 
We denote $(\vec{e}^1 , ... , \vec{e}^d)$
the canonical sequence of generators 
of $\Z^d$. Let us 
denote $\sigma$ the action 
of $\Z^d$ on this space defined 
by the following equality for all $\vec{u} \in 
\Z^d$ and $x$ element of the space:
\[\left(\sigma^{\vec{u}}.x\right)_{\vec{v}} 
= x_{\vec{v}+\vec{u}}.\]
A compact subset 
$X$ of this space is called a \textbf{subshift} 
when this subset is stable 
under the action of the shift. This means 
that for all $\vec{u} \in \Z^d$: 
\[\sigma^{\vec{u}}.X \subset X.\]

Consider some finite subset $\mathbb{U}$ 
of $\Z^d$. An element $p$ 
of $\mathcal{A}^{\mathbb{U}}$ 
is called a \textbf{pattern} on \textbf{support}
$\mathbb{U}$. This pattern \textbf{appears} in 
a configuration $x$ when there 
exists a translate $\mathbb{V}$ of $\mathbb{U}$ 
such that $x_{\mathbb{V}}=p$.
It appears in a subshift $X$ when it appears 
in a configuration of $X$. 
The set of patterns of $X$ that appear in 
it is called the \textbf{language} of $X$. 
The number of patterns on 
support $\llbracket 1,n \rrbracket ^d$ that 
appear in $X$ 
is denoted $N_n (X)$.

We say that a subshift $X$ is \textbf{minimal} when any pattern 
in its language appears in any of its configurations.

A subshift $X$ defined by forbidding patterns 
in some finite set $\mathcal{F}$ 
to appear in the configurations, 
formally: 
\[X= \bigcap_{\mathbb{U}
\subset \Z^2} \left\{ x \in \mathcal{A}^{\Z^2} : 
x_{\mathbb{U}} \notin \mathcal{F} \right\}\]
is called 
a subshift of \textbf{finite type} (SFT).

\subsection{Computability notions}

\begin{definition}
A real number $x$ is said to be \textbf{computable}
when there exists an algorithm which given 
as input some integer $n$ outputs some rational 
number $r_n$ such that 
\[|x-r_n|\le 2^{-n}.\]
\end{definition}

\begin{definition}
A real number $x$ is said to be 
$\Pi_1$-\textbf{computable}
when there exists an algorithm which given 
as input some integer $n$ outputs some rational 
number $r_n$ such that 
\[x= \inf_n r_n.\]
\end{definition}

\subsection{Entropy}

The \textbf{entropy}  
of $X$ is the number
\[h (X)
= \lim \frac{\log_2 
(N_n (X))}{n}.
\]

\begin{theorem}[\cite{Hochman-Meyerovitch-2010}]
The possible entropies of bidimensional 
subshifts of finite type are exactly 
the $\Pi_1$-computable numbers. 
\end{theorem}

\section{\label{sec.definition.robinson} Robinson subshift - a rigid version}

The Robinson subshift was constructed by R. Robinson~\cite{R71} in order to prove 
undecidability results. It has been used by in other constructions 
of subshifts of finite type as a structure layer~\cite{Pavlov-Schraudner-2014}.

In this section, we present a version of this subshift which is 
adapted to constructions under the 
dynamical constraints 
that we consider. In order to understand this 
section, it is preferable to read before the description 
of the Robinson subshift done in~\cite{R71}.
Some results are well known and we don't give a proof.
We refer instead to the initial article of R. Robinson. 

Let us denote $X_{adR}$ this subshift, which is constructed as the product of two layers.
We present the first layer in Section~\ref{sec.valued.robinson}, and then 
describe some hierarchical structures appearing in this layer in Section~\ref{sec.hierarchical.structures}.
In Section~\ref{sec.alignment.positioning}, we describe the second layer. This layer allows 
adding rigidity to 
the first layer, in order 
to enforce dynamical properties.

\subsection{\label{sec.valued.robinson} Robinson layer}

The first layer has the following {\textit{symbols}}, and their transformation by 
rotations by $\frac{\pi}{2}$, $\pi$ or $\frac{3\pi}{2}$: 

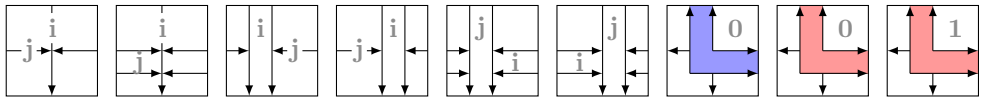
\[\begin{tikzpicture}
\draw (0,0) rectangle (1.2,1.2) ;
\draw (0.6,1.2) -- (0.6,1.1);
\draw [-latex] (0.6,0.7) -- (0.6,0) ;
\draw (0,0.6) -- (0.2,0.6);
\draw [-latex] (0.4,0.6) -- (0.6,0.6) ; 
\draw [-latex] (1.2,0.6) -- (0.6,0.6) ;
\node[scale=1,color=gray!90] at (0.6125,0.9) {\textbf{i}};
\node[scale=1,color=gray!90] at (0.3,0.6) {\textbf{j}};
\end{tikzpicture} \ \
\begin{tikzpicture}
\draw (0,0) rectangle (1.2,1.2) ;
\draw (0.6,1.2) -- (0.6,1.1);
\draw [-latex] (0.6,0.7) -- (0.6,0) ;
\draw [-latex] (0,0.6) -- (0.6,0.6) ; 
\draw [-latex] (0,0.3) -- (0.6,0.3) ; 
\draw [-latex] (1.2,0.6) -- (0.6,0.6) ;
\draw [-latex] (1.2,0.3) -- (0.6,0.3) ;
\node[scale=1,color=gray!90] at (0.6125,0.9) {\textbf{i}};
\node[scale=1,color=gray!90] at (0.3,0.45) {\textbf{j}};
\end{tikzpicture} \ \ \begin{tikzpicture}
\draw (0,0) rectangle (1.2,1.2) ;
\draw [-latex] (0.6,1.2) -- (0.6,0) ;
\draw [-latex] (0.3,1.2) -- (0.3,0) ; 
\draw [-latex] (0,0.6) -- (0.3,0.6) ; 
\draw [-latex] (0.8,0.6) -- (0.6,0.6) ;
\draw (1.2,0.6) -- (1,0.6);
\node[scale=1,color=gray!90] at (0.45,0.9) {\textbf{i}};
\node[scale=1,color=gray!90] at (0.9,0.6) {\textbf{j}};
\end{tikzpicture} \ \  \begin{tikzpicture}
\draw (0,0) rectangle (1.2,1.2) ;
\draw [-latex] (0.6,1.2) -- (0.6,0) ;
\draw [-latex] (0.9,1.2) -- (0.9,0) ; 
\draw (0,0.6) -- (0.2,0.6);
\draw [-latex] (0.4,0.6) -- (0.6,0.6) ; 
\draw [-latex] (1.2,0.6) -- (0.9,0.6) ;
\node[scale=1,color=gray!90] at (0.75,0.9) {\textbf{i}};
\node[scale=1,color=gray!90] at (0.3,0.6) {\textbf{j}};
\end{tikzpicture} \ \ 
\begin{tikzpicture}
\draw (0,0) rectangle (1.2,1.2) ;
\draw [-latex] (0.6,1.2) -- (0.6,0) ;
\draw [-latex] (0.3,1.2) -- (0.3,0) ; 
\draw [-latex] (0,0.6) -- (0.3,0.6) ; 
\draw [-latex] (1.2,0.6) -- (0.6,0.6) ;
\draw [-latex] (0,0.3) -- (0.3,0.3) ; 
\draw [-latex] (1.2,0.3) -- (0.6,0.3) ;
\node[scale=1,color=gray!90] at (0.9,0.45) {\textbf{i}};
\node[scale=1,color=gray!90] at (0.45,0.9) {\textbf{j}};
\end{tikzpicture} \ \
\begin{tikzpicture}
\draw (0,0) rectangle (1.2,1.2) ;
\draw [-latex] (0.6,1.2) -- (0.6,0) ;
\draw [-latex] (0.9,1.2) -- (0.9,0) ; 
\draw [-latex] (0,0.6) -- (0.6,0.6) ; 
\draw [-latex] (1.2,0.6) -- (0.9,0.6) ;
\draw [-latex] (0,0.3) -- (0.6,0.3) ; 
\draw [-latex] (1.2,0.3) -- (0.9,0.3) ;
\node[scale=1,color=gray!90] at (0.3,0.45) {\textbf{i}};
\node[scale=1,color=gray!90] at (0.75,0.9) {\textbf{j}};
\end{tikzpicture} \ \ \begin{tikzpicture}
\fill[blue!40] (0.3,0.3) rectangle (0.6,1.2) ;
\fill[blue!40] (0.3,0.3) rectangle (1.2,0.6) ;
\draw (0,0) rectangle (1.2,1.2) ;
\draw [-latex] (0.3,0.3) -- (0.3,1.2) ; 
\draw [-latex] (0.3,0.3) -- (1.2,0.3) ;
\draw [-latex] (0.6,0.6) -- (0.6,1.2) ; 
\draw [-latex] (0.6,0.6) -- (1.2,0.6) ; 
\draw [-latex] (0.3,0.6) -- (0,0.6) ; 
\draw [-latex] (0.6,0.3) -- (0.6,0) ;
\node[scale=1,color=gray!90] at (0.9,0.9) {\textbf{0}};
\end{tikzpicture} \ \ \begin{tikzpicture}
\fill[red!40] (0.3,0.3) rectangle (0.6,1.2) ;
\fill[red!40] (0.3,0.3) rectangle (1.2,0.6) ;
\draw (0,0) rectangle (1.2,1.2) ;
\draw [-latex] (0.3,0.3) -- (0.3,1.2) ; 
\draw [-latex] (0.3,0.3) -- (1.2,0.3) ;
\draw [-latex] (0.6,0.6) -- (0.6,1.2) ; 
\draw [-latex] (0.6,0.6) -- (1.2,0.6) ; 
\draw [-latex] (0.3,0.6) -- (0,0.6) ; 
\draw [-latex] (0.6,0.3) -- (0.6,0) ;
\node[scale=1,color=gray!90] at (0.9,0.9) {\textbf{0}};
\end{tikzpicture} \ \ \begin{tikzpicture}
\fill[red!40] (0.3,0.3) rectangle (0.6,1.2) ;
\fill[red!40] (0.3,0.3) rectangle (1.2,0.6) ;
\draw (0,0) rectangle (1.2,1.2) ;
\draw [-latex] (0.3,0.3) -- (0.3,1.2) ; 
\draw [-latex] (0.3,0.3) -- (1.2,0.3) ;
\draw [-latex] (0.6,0.6) -- (0.6,1.2) ; 
\draw [-latex] (0.6,0.6) -- (1.2,0.6) ; 
\draw [-latex] (0.3,0.6) -- (0,0.6) ; 
\draw [-latex] (0.6,0.3) -- (0.6,0) ;
\node[scale=1,color=gray!90] at (0.9,0.9) {\textbf{1}};
\end{tikzpicture}\]

The symbols $i$ and $j$ can have 
value $0,1$ and 
are attached respectively to vertical and 
horizontal arrows. In the text, we refer to this value as 
the value of the \textbf{$0,1$-counter}. In 
order to simplify the representations, these 
values will often be omitted on the figures. \bigskip

In the text we will often designate as \textbf{corners} the two last symbols.
The other ones are called \textbf{arrows symbols} and are 
specified by the number of arrows in the symbol. 
For instance a six arrows symbols are the images by rotation of the fifth and sixth symbols. \bigskip

The {\textit{rules}} 
are the following ones: \begin{enumerate} \item the outgoing arrows and incoming ones 
correspond for two adjacent symbols. For instance, 
the pattern 
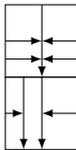
\[\begin{tikzpicture}[scale=0.8]
\draw (0,1.2) rectangle (1.2,2.4) ;
\draw [-latex] (0.6,2.4) -- (0.6,1.2) ;
\draw [-latex] (0,1.8) -- (0.6,1.8) ; 
\draw [-latex] (0,1.5) -- (0.6,1.5) ; 
\draw [-latex] (1.2,1.8) -- (0.6,1.8) ;
\draw [-latex] (1.2,1.5) -- (0.6,1.5) ;
\draw (0,0) rectangle (1.2,1.2) ;
\draw [-latex] (0.6,1.2) -- (0.6,0) ;
\draw [-latex] (0.3,1.2) -- (0.3,0) ; 
\draw [-latex] (0,0.6) -- (0.3,0.6) ; 
\draw [-latex] (1.2,0.6) -- (0.6,0.6) ;
\end{tikzpicture}\]
is forbidden, but the pattern 
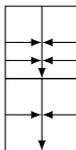
\[\begin{tikzpicture}[scale=0.8]
\draw (0,0) rectangle (1.2,1.2) ;
\draw [-latex] (0.6,1.2) -- (0.6,0) ;
\draw [-latex] (0,0.6) -- (0.6,0.6) ; 
\draw [-latex] (1.2,0.6) -- (0.6,0.6) ;
\draw (0,1.2) rectangle (1.2,2.4) ;
\draw [-latex] (0.6,2.4) -- (0.6,1.2) ;
\draw [-latex] (0,1.8) -- (0.6,1.8) ; 
\draw [-latex] (0,1.5) -- (0.6,1.5) ; 
\draw [-latex] (1.2,1.8) -- (0.6,1.8) ;
\draw [-latex] (1.2,1.5) -- (0.6,1.5) ;
\end{tikzpicture}\]
is allowed. 
\item in every $2 \times 2$ square there is a blue 
symbol and the presence of a blue symbol in position $\vec{u} \in \Z^2$ 
forces the presence of a blue symbol in the positions $\vec{u}+(0,2), \vec{u}-(0,2), 
\vec{u}+(2,0)$ and $\vec{u}-(2,0)$.
\item on a position having 
mark $(i,j)$, the first coordinate is transmitted 
to the horizontally adjacent positions and the second one is transmitted 
to the vertically adjacent positions.
\item on a six arrows symbol, like 
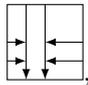
\[\begin{tikzpicture}[scale=0.83]
\draw (0,0) rectangle (1.2,1.2) ;
\draw [-latex] (0.6,1.2) -- (0.6,0) ;
\draw [-latex] (0.3,1.2) -- (0.3,0) ; 
\draw [-latex] (0,0.6) -- (0.3,0.6) ; 
\draw [-latex] (1.2,0.6) -- (0.6,0.6) ;
\draw [-latex] (0,0.3) -- (0.3,0.3) ; 
\draw [-latex] (1.2,0.3) -- (0.6,0.3) ;
\end{tikzpicture},\] 
or a five arrow symbol, 
like 
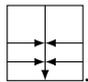
\[\begin{tikzpicture}[scale=0.83]
\draw (0,0) rectangle (1.2,1.2) ;
\draw [-latex] (0.6,1.2) -- (0.6,0) ;
\draw [-latex] (0,0.6) -- (0.6,0.6) ; 
\draw [-latex] (0,0.3) -- (0.6,0.3) ; 
\draw [-latex] (1.2,0.6) -- (0.6,0.6) ;
\draw [-latex] (1.2,0.3) -- (0.6,0.3) ;
\end{tikzpicture},\]
the marks $i$ and $j$ are different.
\end{enumerate} \bigskip

The Figure~\ref{figure.order2supertile} shows some pattern in the language of 
this layer. The subshift on this alphabet 
and generated by these rules is denoted $X_R$: this 
is the Robinson subshift.\bigskip

The main aspect of this subshift is 
the following property: 

\begin{definition}
A $\Z^d$-subshift $X$ is said aperiodic when 
for all configuration $x$ in the subshift, 
and $\vec{u} \in \Z^d \backslash (0,0)$, 
\[{\sigma}^{\vec{u}} (x) \neq x.\]
\end{definition}

\begin{theorem}[\cite{R71}]
The subshift $X_R$ is non-empty and aperiodic.
\end{theorem}

In the following, we state some properties of 
this subshift. The proofs of these properties 
can also be found in \cite{R71}.

\subsection{\label{sec.hierarchical.structures} 
Hierarchical structures}

In this section we describe some observable hierarchical 
structures in the elements of the Robinson subshift.

Let us recall that for 
all $d \ge 1$ and $k \ge 1$, 
we denote $\mathbb{U}^{(d)}_{k}$
the set $\llbracket 0,k-1 \rrbracket^d$.

\subsubsection{Finite supertiles}

\begin{figure}[h!]
\[\begin{tikzpicture}[scale=0.4]

\begin{scope}
\robibluebastgauche{0}{0};
\robibluebastgauche{8}{0};
\robibluebastgauche{0}{8};
\robibluebastgauche{8}{8};
\robiredbasgauche{2}{2};
\robiredbasgauche{6}{6};
\robibluehauttgauche{0}{4};
\robibluehauttgauche{8}{4};
\robibluehauttgauche{0}{12};
\robibluehauttgauche{8}{12};
\robiredhautgauche{2}{10};
\robibluebastdroite{4}{0};
\robibluebastdroite{12}{0};
\robibluebastdroite{4}{8};
\robibluebastdroite{12}{8};
\robiredbasdroite{10}{2};
\robibluehauttdroite{4}{4};
\robibluehauttdroite{12}{4};
\robibluehauttdroite{4}{12};
\robibluehauttdroite{12}{12};
\robiredhautdroite{10}{10};
\robitwobas{2}{0}
\robitwobas{10}{0}
\robitwobas{6}{2}
\robitwogauche{0}{2}
\robitwogauche{2}{6}
\robitwogauche{0}{10}
\robitwodroite{12}{2}
\robitwodroite{12}{10}
\robitwohaut{2}{12}
\robitwohaut{10}{12}
\robionehaut{6}{0}
\robionehaut{6}{4}
\robionegauche{0}{6}
\robionegauche{4}{6}
\robisixbas{2}{8}
\robisixdroite{4}{2}
\robisixhaut{2}{4}
\robisevendroite{4}{10}
\robithreehaut{6}{8}
\robisixhaut{6}{10}
\robithreehaut{6}{12}
\robisixgauche{8}{2}
\robisevengauche{8}{10}
\robisevenbas{10}{8}
\robisevenhaut{10}{4}
\robisixdroite{10}{6}
\robithreedroite{8}{6}
\robithreedroite{12}{6}
\draw (0,0) rectangle (14,14);
\foreach \x in {1,...,6} \draw (0,2*\x) -- (14,2*\x);
\foreach \x in {1,...,6} \draw (2*\x,0) -- (2*\x,14);
\end{scope}

\begin{scope}[xshift=17.5cm]
\fill[blue!40] (8.5,0.5) rectangle (13.5,5.5);
\fill[blue!40] (8.5,8.5) rectangle (13.5,13.5);
\fill[blue!40] (0.5,8.5) rectangle (5.5,13.5);
\fill[white] (9,1) rectangle (13,5);
\fill[white] (9,9) rectangle (13,13);
\fill[white] (1,9) rectangle (5,13);

\fill[blue!40] (0.5,0.5) rectangle (5.5,5.5);
\fill[white] (1,1) rectangle (5,5);

\fill[red!40] (2.5,2.5) rectangle (3,11.5);
\fill[red!40] (2.5,2.5) rectangle (11.5,3);
\fill[red!40] (2.5,11) rectangle (11.5,11.5);
\fill[red!40] (11,2.5) rectangle (11.5,11.5);

\fill[orange!40] (6.5,6.5) rectangle (7,14); 
\fill[orange!40] (6.5,6.5) rectangle (14,7);

\draw[step=2] (0,0) grid (14,14);
\end{scope}
\end{tikzpicture}\]
\caption{The south west order $2$ supertile denoted $St_{sw}(2)$ 
and petals intersecting it.}\label{figure.order2supertile}
\end{figure}
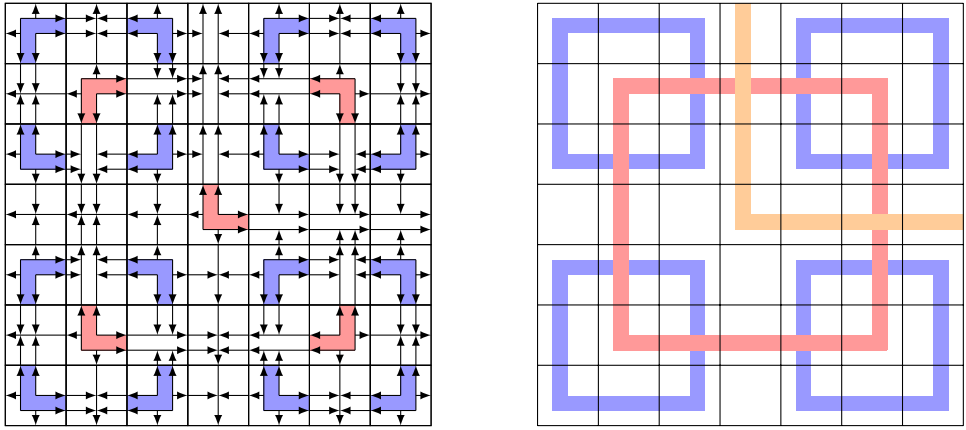 \bigskip

Let us define by induction the south west 
(resp. south east, north west, north east) 
\define{supertile of order $n\in\N$}. For $n=0$, one has
\[
St_{sw} (0)=\vbox to 14pt{\hbox{
\begin{tikzpicture}[scale=0.4]
 \robibluebastgauche{0}{0};
\end{tikzpicture}}},\quad
St_{se} (0)=\vbox to 14pt{\hbox{
\begin{tikzpicture}[scale=0.4]
 \robibluebastdroite{0}{0};
\end{tikzpicture}}},\quad
St_{nw} (0)=\vbox to 14pt{\hbox{
\begin{tikzpicture}[scale=0.4]
 \robibluehauttgauche{0}{0};
\end{tikzpicture}}},\quad
St_{ne} (0)=\vbox to 14pt{\hbox{
\begin{tikzpicture}[scale=0.4]
 \robibluehauttdroite{0}{0};
\end{tikzpicture}}}.
\] 

For $n\in\N$, the support of the supertile $St_{sw} (n+1)$ 
(resp. $St_{se} (n+1)$, $St_{nw} (n+1)$, $St_{ne} (n+1)$) is 
$\U^{(2)}_{2^{n+2}-1}$. 
On position $\vec{u}=(2^{n+1}-1,2^{n+1}-1)$ write
\[
St_{sw} (n+1) _{\vec{u}}=\vbox to 14pt{\hbox{
\begin{tikzpicture}[scale=0.4]
 \robiredbasgauche{0}{0};
\end{tikzpicture}}},\quad
St_{se} (n+1) _{\vec{u}}=\vbox to 14pt{\hbox{
\begin{tikzpicture}[scale=0.4]
 \robiredbasdroite{0}{0};
\end{tikzpicture}}},\quad
St_{nw} (n+1) _{\vec{u}}=\vbox to 14pt{\hbox{
\begin{tikzpicture}[scale=0.4]
 \robiredhautgauche{0}{0};
\end{tikzpicture}}},\quad
St_{ne} (n+1) _{\vec{u}}=\vbox to 14pt{\hbox{
\begin{tikzpicture}[scale=0.4]
 \robiredhautdroite{0}{0};
\end{tikzpicture}}}.
\]

Then complete the supertile such that the restriction 
to $\mathbb{U}^{(2)}_{2^{n+1}-1}$ (resp. 
$(2^{n+1},0)+\mathbb{U}^{(2)}_{2^{n+1}-1}$, 
 $(0,2^{n+1})+\mathbb{U}^{(2)}_{2^{n+1}-1}$,
$(2^{n+1},2^{n+1})+\mathbb{U}^{(2)}_{2^{n+1}-1}$)
is 
$St_{sw}(n)$ (resp. $St_{se} (n)$, $St_{nw} (n)$,
$St_{ne} (n)$).

Then complete the cross with 
the symbol 
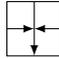
\[\begin{tikzpicture}[scale=0.6]
\draw (0,0) rectangle (1.2,1.2) ;
\draw [-latex] (0.6,1.2) -- (0.6,0) ;
\draw [-latex] (0,0.6) -- (0.6,0.6) ; 
\draw [-latex] (1.2,0.6) -- (0.6,0.6) ;
\end{tikzpicture}\]
or with the symbol
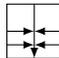
\[\begin{tikzpicture}[scale=0.6]
\draw (0,0) rectangle (1.2,1.2) ;
\draw [-latex] (0.6,1.2) -- (0.6,0) ;
\draw [-latex] (0,0.6) -- (0.6,0.6) ; 
\draw [-latex] (0,0.3) -- (0.6,0.3) ; 
\draw [-latex] (1.2,0.6) -- (0.6,0.6) ;
\draw [-latex] (1.2,0.3) -- (0.6,0.3) ;
\end{tikzpicture}\] in the south vertical arm with the first symbol when there is one incoming arrow, and the second when there are two. 
The other arms are completed in a similar way. For instance, Figure \ref{figure.order2supertile}
shows the south west supertile of order two.

\begin{proposition}[\cite{R71}]
For all configuration $x$, if an order $n$ supertile 
appears in this configuration, 
then there is an order $n+1$ supertile, having 
this order $n$ supertile as sub-pattern, which appears 
in the same configuration.
\end{proposition}

\subsubsection{Infinite supertiles}

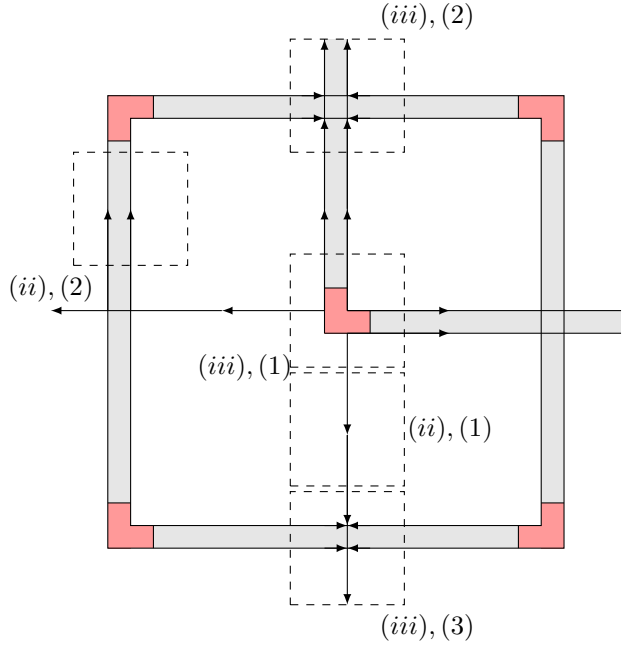
\begin{figure}[ht]
\[\begin{tikzpicture}[scale=0.3]
\fill[gray!20] (0,0) rectangle (20,20); 
\fill[white] (1,1) rectangle (19,19);
\fill[gray!20] (9.5,9.5) rectangle (10.5,22.5);
\fill[gray!20] (9.5,9.5) rectangle (22.5,10.5);
\fill[red!40] (9.5,9.5) rectangle (10.5,11.5);
\fill[red!40] (9.5,9.5) rectangle (11.5,10.5);
\fill[red!40] (0,0) rectangle (1,2);
\fill[red!40] (0,0) rectangle (2,1);
\draw (0,2) -- (1,2);
\draw (2,0) -- (2,1);

\fill[red!40] (0,20) rectangle (1,18);
\fill[red!40] (0,20) rectangle (2,19);
\draw (0,18) -- (1,18);
\draw (2,19) -- (2,20);

\fill[red!40] (20,20) rectangle (18,19);
\fill[red!40] (20,20) rectangle (19,18);
\draw (19,18) -- (20,18);
\draw (18,19) -- (18,20);

\fill[red!40] (20,0) rectangle (18,1);
\fill[red!40] (20,0) rectangle (19,2);
\draw (19,2) -- (20,2);
\draw (18,1) -- (18,0);

\draw[-latex] (9.5,10.5) -- (5,10.5);
\draw[-latex] (5,10.5) -- (-2.5,10.5);
\draw[-latex] (10.5,9.5) -- (10.5,5);
\draw[-latex] (10.5,5) -- (10.5,-2.5);
\draw[-latex] (10.5,10.5) -- (10.5,15);
\draw[-latex] (9.5,10.5) -- (9.5,15);
\draw[-latex] (10.5,10.5) -- (15,10.5);
\draw[-latex] (10.5,9.5) -- (15,9.5);
\draw[-latex] (9.5,15) -- (9.5,19);
\draw[-latex] (10.5,15) -- (10.5,19);
\draw[-latex] (9.5,20) -- (9.5,22.5);
\draw[-latex] (10.5,20) -- (10.5,22.5);
\draw[-latex] (8.5,19) -- (9.5,19);
\draw[-latex] (8.5,20) -- (9.5,20);
\draw[-latex] (11.5,19) -- (10.5,19);
\draw[-latex] (11.5,20) -- (10.5,20);
\draw[-latex] (10.5,5) -- (10.5,1);
\draw[-latex] (9.5,1) -- (10.5,1);
\draw[-latex] (9.5,0) -- (10.5,0);
\draw[-latex] (11.5,1) -- (10.5,1);
\draw[-latex] (11.5,0) -- (10.5,0);
\draw[-latex] (0,10.5) -- (0,15);
\draw[-latex] (1,10.5) -- (1,15);
\draw (9.5,11.5) -- (10.5,11.5);
\draw (11.5,9.5) -- (11.5,10.5);
\draw (0,0) rectangle (20,20);
\draw (1,1) rectangle (19,19);
\draw (9.5,9.5) -- (22.5,9.5);
\draw (9.5,9.5) -- (9.5,22.5);
\draw (10.5,10.5) -- (10.5,22.5);
\draw (10.5,10.5) -- (22.5,10.5);
\draw[dashed] (8,8) rectangle (13,13);
\draw[dashed] (8,17.5) rectangle (13,22.5);
\draw[dashed] (8,-2.5) rectangle (13,2.5);
\draw[dashed] (-1.5,12.5) rectangle (3.5,17.5);
\draw[dashed] (8,2.75) rectangle (13,7.75);
\node at (15,5.25) {$(ii),(1)$};
\node at (14,23.5) {$(iii),(2)$};
\node at (6,8) {$(iii),(1)$};
\node at (14,-3.5) {$(iii),(3)$};
\node at (-2.5,11.5) {$(ii),(2)$};
\end{tikzpicture}\]
\caption{\label{fig.infinite.supertiles} Correspondence 
between infinite supertiles and sub-patterns of order $n$ supertiles. The whole picture 
represents a schema of some finite order supertile.}
\end{figure}

Let $x$ be a configuration in the first layer and consider 
the equivalence relation $\sim_x$ on $\Z^2$ defined by 
$\vec{i} \sim_x \vec{j}$ if there is a finite
supertile in $x$ which contains $\vec{i}$ and $\vec{j}$.
An \define{infinite order} supertile is an infinite 
pattern over an equivalence 
class of this relation. Each configuration is amongst the 
following types (with types corresponding with 
types numbers on Figure~\ref{fig.infinite.supertiles}): 
\begin{itemize}
\item[(i)] A unique infinite order supertile which covers $\Z^2$.
\item[(ii)] Two infinite order supertiles separated by a line or a column with only three-arrows symbols
(1) or only four arrows symbols (2).  
In such a configuration, the order $n$
finite supertiles appearing in the two 
infinite supertiles 
are not necessary aligned, whereas this 
is the case in a type (i) or (iii) configuration.
\item[(iii)] Four infinite order supertiles, separated by a cross, whose center is superimposed with:
\begin{itemize}
\item a red symbol, and arms are filled with arrows symbols induced by the red one. (1)
\item a six arrows symbol, and arms are filled with double arrow symbols 
induced by this one. (2)
\item a five arrow symbol, and arms are filled with double arrow symbols and simple 
arrow symbols induced 
by this one. (3)
\end{itemize}
\end{itemize} 

Informally, the types of infinite 
supertiles correspond to configurations 
that are limits (for type $(ii)$ infinite supertiles 
this will be true after alignment 
[Section~\ref{sec.alignment.positioning}]) of 
a sequence of configurations centered on 
particular sub-patterns of 
finite supertiles of order $n$. 
This correspondence is illustrated on 
Figure~\ref{fig.infinite.supertiles}.
We notice this fact so that it helps 
to understand how 
patterns in configurations having multiple 
infinite supertiles are sub-patterns of 
finite supertiles. \bigskip

We say that a 
pattern $p$ on support $\mathbb{U}$ appears 
periodically in the horizontal (resp. vertical) direction 
in a configuration $x$ of a subshift $X$
when there exists some $T>0$ and 
$\vec{u}_0 \in \Z^2$ 
such that 
for all $k \in \Z$, 
\[x_{\vec{u}_0+\mathbb{U}+kT(1,0)}=p\]
(resp. $x_{\vec{u}_0+\mathbb{U}+kT(0,1)}=p$).
The number $T$ is called the period of this periodic appearance.

\begin{lemma}[\cite{R71}] 
\label{lem.repetition.supertiles}
For all $n$ and $m$ integers 
such that $n \ge m$, any order $m$ supertile 
appears periodically, horizontally and vertically, in 
any supertile of order $n \ge m$ with period $2^{m+2}$.
This is also true inside any infinite supertile. 
\end{lemma}

\subsubsection{Petals}

For a configuration 
$x$ of the Robinson subshift some finite subset of $\Z^2$ 
which has the following properties is called 
a \textbf{petal}.

\begin{itemize}
\item this set is minimal with respect to the inclusion, 
\item it contains some 
symbol with more than three arrows, 
\item  if a position is in the petal, 
the next position in the direction, or the opposite one, 
of the double arrows, is also in it,
\item and in the case of a six arrows 
symbol, the previous 
property is true only for one couple of arrows.
\end{itemize}

These sets are represented on 
the figures as squares joining 
four corners when these corners 
have the right orientations.
\bigskip

Petals containing blue symbols are called 
order $0$ petals. Each one intersect 
a unique greater order petal.
The other ones intersect 
four smaller petals and a greater 
one: if the intermediate petal is of order $n \ge 1$, 
then the four smaller are of order $n-1$ and the greatest 
one is of order $n+1$. Hence they 
form a hierarchy, and 
we refer to this in the text as 
the \textbf{petal hierarchy} (or hierarchy).

We usually call 
the petals valued with $1$ 
\textbf{support petals}, 
and the other ones 
are called 
\textbf{transmission petals}.

\begin{lemma}[\cite{R71}]
For all $n$, an order $n$ petal has size $2^{n+1}+1$.
\end{lemma}

We call order $n$ \textbf{two dimensional cell} the part 
of $\Z^2$ which is enclosed in an order 
$2n+1$ petal, for $n \ge 0$.
We also sometimes refer 
to the order $2n+1$ petals 
as the cells borders.

In particular, order $n \ge 0$ 
two-dimensional cells 
have size $4^{n+1}+1$ and 
repeat periodically with 
period $4^{n+2}$, vertically 
and horizontally, 
in every cell or supertile 
having greater order. 

See an illustration on 
Figure~\ref{figure.order2supertile}.

\subsection{ \label{sec.alignment.positioning} 
Alignment positioning
} 

If a configuration of the first layer has two infinite order 
supertiles, then the two sides of the column or line which separates them are 
non dependent. The two infinite order supertiles of this configuration 
can be shifted vertically (resp. horizontally) one from each 
other, while the configuration obtained stays 
an element of the subshift.
This is an obstacle to dynamical properties such as 
minimality or transitivity, 
since a pattern which crosses the separating line can not 
appear in the other configurations.
In this section, we describe additional layers 
that allow aligning all the supertiles 
having the same order and 
eliminate this phenomenon. \bigskip

Here is a description of the second layer: \bigskip

\textit{Symbols:} $nw, ne, sw, se$, 
and a blank symbol. \bigskip

The {\textit{rules}} are 
the following ones: 

\begin{itemize}
\item \textbf{Localization:} 
the symbols $nw$, $ne$ ,$sw$ and $se$
are superimposed only on three arrows and five arrows 
symbols in the Robinson layer.
\item \textbf{Induction 
of the orientation:} on a position with 
a three arrows symbol such 
that the long arrow 
originate in a corner
is superimposed a symbol 
corresponding to the orientation 
of the corner.
\item \textbf{Transmission rule:} 
on a three or five arrows symbol position, 
the symbol in this layer is 
transmitted to the position in the direction pointed 
by the long arrow when the Robinson symbol 
 is a three or five arrows symbol 
with long arrow pointing in the same direction. 
\item \textbf{Synchronization rule:} 
On the pattern 
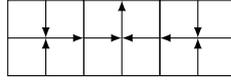
\[\begin{tikzpicture}[scale=0.5]
\robionedroite{0}{0}
\robionebas{2}{0}
\robionegauche{4}{0}
\end{tikzpicture}\]
or 
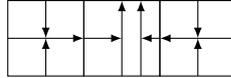
\[\begin{tikzpicture}[scale=0.5]
\robionedroite{0}{0}
\robifourhaut{2}{0}
\robionegauche{4}{0}
\end{tikzpicture}\]
in the Robinson layer, if 
the symbol on the left side is $ne$
(resp. $se$), 
then the symbol on the right side is $nw$
(resp. $sw$).
On the images by rotation of these patterns, 
we impose similar rules.
\item \textbf{Coherence rule:} 
the other couples of symbols are forbidden 
on these patterns.
\end{itemize}
 \bigskip

\textit{Global behavior:} the 
symbols 
$ne,nw,sw,se$ designate 
orientations: 
north east, north west, south 
west and south east. We will 
re-use this symbolisation 
in the following. The 
localization rule implies 
that these symbols 
are superimposed on and only on 
straight paths connecting 
the corners of adjacent order $n$ 
cells for some integer $n$.

The effect of transmission and synchronization rules 
is stated by the following lemma: 

\begin{lemma}
In any configuration $x$ of the subshift $X_{adR}$, 
any order $n$ supertile appears 
periodically in the whole configuration, 
with period $2^{n+2}$, horizontally and vertically.
\end{lemma}

\begin{proof}
\begin{itemize}
\item This property is true in an infinite supertile: 
this is the statement of Lemma~\ref{lem.repetition.supertiles}.
Hence the statement is true in a type $(i)$ configuration. 
This is also true in a type $(iii)$ configuration, 
since the infinite supertiles are aligned, 
and that the positions where the order $n$ supertiles appear 
are the same in any infinite supertile. This statement uses 
the property that 
an order $n$ supertile forces the 
presence of an order $n+1$ 
one. 
\item Consider a configuration of the subshift $X_{adR}$ 
which is of type $(ii)$. Let 
us assume that the separating line is vertical, the 
other case being similar. In order to simplify 
the exposition we assume that this column intersects 
$(0,0)$.

\begin{enumerate}

\item \textbf{Positions of the supertiles along 
the infinite line:}

From Lemma~\ref{lem.repetition.supertiles}, 
there exists a sequence of numbers $ 0 \le z_n < 2^{n+2}-1$ 
and $ 0 \le z'_n < 2^{n+2}-1$ 
such that for all $k \in \Z$,
the orientation symbol on positions 
$(-1,z_n + k.2^{n+2})$ (in 
the column on the left of the 
separating one) is $se$ and the 
orientation symbols on positions 
$(1,z'_n + k.2^{n+2})$ is $sw$. The symbol 
on positions $(-1,z_n + k.2^{n+2}+2^{n+1})$ 
is then $ne$ and is $nw$ on 
positions 
on positions $(1,z'_n + k.2^{n+2}+2^{n+1})$: 
this comes from the fact that an order $n$ petal 
has size $2^{n+1}+1$. 

Let us prove 
that for all $n$, $z_n = z'_n$. 
This means that the supertiles 
of order $n$ on the 
two sides of the separating line 
are aligned.

\item \textbf{Periodicity of these positions:}

Since for all $n$, there is a space of
$2^{n}$ columns 
between the rightmost or 
leftmost order $n$ supertile in 
a greater order supertile and the border 
of this supertile (by a recurrence argument), 
this means that the space between 
the rightmost order $n$ supertiles of the left 
infinite supertile and the leftmost order $n$ supertile 
of the right infinite supertile 
is $2^{n+1}+1$. Since two adjacent of these supertiles 
have opposite orientations, this implies 
that each supertile appears periodically 
in the horizontal 
direction (and hence both horizontal 
and vertical directions) 
with period $2^{n+2}$. \bigskip

\begin{figure}[ht]
\[\begin{tikzpicture}[scale=0.4]
\fill[gray!20] (0,0) rectangle (1,20);
\draw[dashed] (0,0) rectangle (1,20);  
\draw[-latex] (0.5,-1) -- (0.5,21);
\draw (0.25,1) -- (0.75,1);
\node at (7.5,1) {$z'_m + k' 2^{m+2}$};

\draw[-latex] (-4,1) -- (-4,17); 
\node at (-5,9) {$2^{n+1}$};

\robiredbasgauche{2}{1}
\robiredbasdroite{-3}{1}

\robiredhautgauche{2}{5}
\robiredbasgauche{2}{9}
\robiredhautgauche{2}{13}
\robiredbasgauche{2}{17}

\robiredhautdroite{-3}{17}
\end{tikzpicture} \]
\caption{\label{fig.decalage.toeplitz} Schema
of the proof. The separating line 
is colored gray.}
\end{figure}
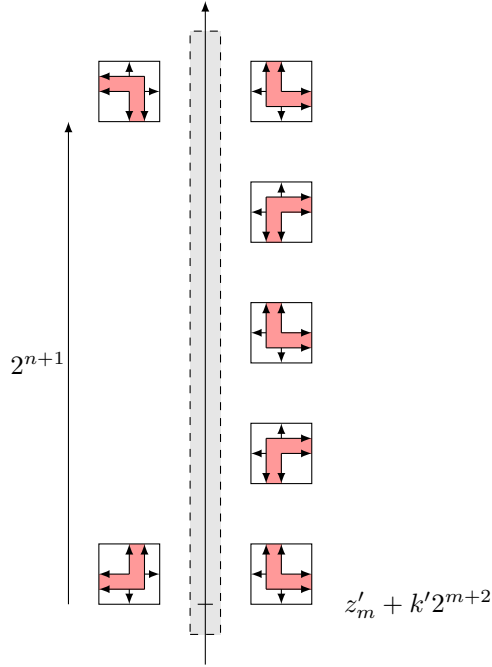

\item \textbf{The orientation symbols 
force alignment:}

Assume that there exists 
some $n$ such that $z_n \neq z'_n$.
Since $\{z_n + k.2^{n+2}, \ (n,k) \in \N \times \Z\} = \Z$,
this implies that there exist some $m \neq n$ and 
some $k,k'$ such that 
\[z_n + k.2^{n+2} = z'_m +k'.2^{m+2}.\] 
One can assume without loss of generality that 
$m<n$, exchanging $m$ and $n$ if necessary.
Then the position 
$(-1,z_n+k.2^{n+2}+2^{n+1})$ has
orientation symbol equal to $ne$.
As a consequence, the position 
$(1,z'_m +k'.2^{m+2} + 2^{n-m-1}.2^{m+2})$ 
has the same symbol. However, by definition,
this position has symbol $se$: there is 
a contradiction. This situation is illustrated on 
Figure~\ref{fig.decalage.toeplitz}. 

\end{enumerate}
\end{itemize}
\end{proof}

\subsection{Completing blocks}

Let $\chi : \N^{*} \rightarrow \N^{*}$ such 
that for all $n \ge 1$, 
\[\chi(n) = \Bigl\lceil \log_2 (n) \Bigr\rceil + 4.\]
Let us also denote $\chi'$ the function 
such that for all $n \ge 1$,
\[\chi'(n) =  \Bigl\lceil\frac{\Bigl\lceil \log_2 (n) 
\Bigr\rceil}{2} \Bigr\rceil + 2.\]

The following lemma will be extensively used 
in the following of this text, in order to prove 
dynamical properties of the constructed subshifts: 

\begin{lemma} \label{prop.complete.rob} 
For all $n \ge 1$, any $n$-block in the language of 
$X_{adR}$ is sub-pattern of some order $\chi(n)$ 
supertile, and is sub-pattern of some order $\chi'(n)$ 
order cell.
\end{lemma}

\begin{proof}

\begin{enumerate}
\item \textbf{Completing into 
an order $2^{\lceil \log_2 (n) \rceil +1}-1$ block:}

Consider some $n$-block $p$ that appears 
in some configuration $x$ 
of the SFT $X_{adR}$. We can complete it into 
a $2^{\lceil \log_2 (n) \rceil +1}-1$ block, 
since $2^{\lceil \log_2 (n) \rceil +1} - 1 
\ge 2n-1 \ge n$ for all $n\ge 1$.

\item \textbf{Intersection with four order 
$\lceil \log_2 (n) \rceil$ supertiles:}

From the periodic appearance property of the order 
$\lceil \log_2 (n) \rceil$ supertiles 
in each configuration, this last block 
intersects at most four
supertiles having this order. Let us complete $p$ 
into the block whose support is the union 
of the supports of the supertiles and the 
cross separating these. 

\item \textbf{Possible patterns after this completion
according to the center symbol:}

Since this pattern 
is determined by the symbol at the center of 
the cross and the orientations of the supertiles, 
the possibilities for this pattern 
are listed on Figure~\ref{fig.orientation.supertiles0}, 
Figure~\ref{fig.orientation.supertiles1} and 
Figure~\ref{fig.orientation.supertiles2}. Indeed, 
when the orientations of the supertiles are 
like on Figure~\ref{fig.orientation.supertiles0}, 
each of the supertiles forcing the 
presence of an order $\lfloor \log_2 (n) \rfloor+1$
supertile, the center is a red corner.
When the orientations of the supertiles are 
like on Figure~\ref{fig.orientation.supertiles1},
the center of the block can not be superimposed with 
a red corner since the two west supertiles 
force an order $\lfloor \log_2 (n) \rfloor+1$ 
supertile, as well as the two east supertiles. 
This forces a non-corner symbol on the position 
considered.

For type $4,5,9,10$ patterns, there are 
two possibilities: the values of the two arms 
of the central cross are equal or not. 
Hence the notation $4,4'$, where $4'$ 
designates the case where the two values are different.

One completes the alignment layer on $p$ 
according to the restriction of the configuration
$x$.

On these patterns, the value of symbols on the 
cross is opposed to the value of the 
symbols on the crosses of the four supertiles composing 
it.

\item \textbf{Localization of these patterns 
as part of a greater cell:}

The way to complete the obtained pattern 
is described as follows: 

\begin{enumerate}
\item When the pattern is the one 
on Figure~\ref{fig.orientation.supertiles0}, this is an order 
$\lfloor \log_2 (n) \rfloor+1$ supertile 
and the statement is proved. Indeed,  
any order $\lceil \log_2 (n) \rceil+1$ supertile
is 
a sub-pattern of any order 
$\lfloor \log_2 (n) \rfloor+4$ one.

\item One can see the patterns on
Figure~\ref{fig.orientation.supertiles1} and 
Figure~\ref{fig.orientation.supertiles2} 
in an order 
$\lfloor \log_2 (n) \rfloor+1$, 
$\lfloor \log_2 (n) \rfloor+2$, 
$\lfloor \log_2 (n) \rfloor+3$, or
$\lfloor \log_2 (n) \rfloor +4$ supertile,
depending on how was completed the initial pattern 
thus far (this correspondance is shown 
on Figure~\ref{fig.completing.supertile}), 
hence a sub-pattern of an order 
$\lfloor \log_2 (n) \rfloor+4$
supertile.
\end{enumerate}

The orientation of the greater order 
supertiles implied 
in this completion are chosen according 
to the symbols of the alignment layer. 
This layer is then completed.

\item This
implies that 
any $n$-block is the sub-pattern 
of an order $2 ( \Bigl \lceil 
\frac{1}{2} \lceil \log_2 (n) \rceil \Bigr \rceil +2)
+1$ supertile, which is included into 
an order $\Bigl \lceil 
\frac{1}{2} \lceil \log_2 (n) 
\rceil \Bigr \rceil +2$
cell.
\end{enumerate}
\end{proof}

\section{Notion of block gluing with gap function}\label{sec.definition.block.gluing}

In this section, we introduce the notion of block gluing with intensity function. 
Then we give some examples of subshifts of finite type which are 
block gluing for some particular intensity functions ($f$ is constant,
$f=\log$ and $f$ is linear).

In this 
text, a configuration $x\in\A^{\Z^2}$ is said 
to be (doubly) \define{periodic} if there exist
some $m,n>0$ such that 
\[\sigma^{(m,0)} (x) = \sigma^{(0,n)}(x)=x.\] 

In all the following, we say 
that a subset $\Lambda$ of $\Z^d$ 
is \textbf{connected} when for 
any two of its elements 
$\vec{l}^1 , \vec{l}^2$, 
there exists $n \ge 1$ and a function 
$$\varphi : \llbracket 0, n\rrbracket \rightarrow \Lambda$$
such that $\varphi (0) = \vec{l}^1$, $\varphi(1) = \vec{l}^2$, 
and for all $t \in \llbracket 0, n-1\rrbracket$, 
$$\varphi (t+1) - \varphi (t) \in \{\vec{e}^1 , ... , \vec{e}^d\}.$$

\subsection{Block gluing notions}

\subsubsection{Definitions}

In this section, $X$ is a 
subshift on the alphabet $\A$ and 
$f : \N \rightarrow \N$ is a non decreasing function. We denote 
$||.||_{\infty}$ the norm 
defined by 
\[||\vec{i}||_{\infty} = \max\{\vec{i}_1,\vec{i}_2\}\] 
for all $\vec{i}\in\Z^2$. 
We denote $d_{\infty}$ the associated distance 
function.

\begin{definition}\label{definition.GluingSet} Let $n\in\N$ be an integer. 
The \define{gluing set} in the subshift $X$ of some $n$-block $p$ relative 
to some other $n$-block $q$ is the set of $\vec{u} \in \Z^2$ such 
that there exists a configuration in $X$ where $q$ appears in position $
(0,0)$ and $p$ appears in position $\vec{u}$ (see Figure~\ref{Figure.GluingSet}). 
This set is denoted $\Delta_X (p,q)$. Formally
$$\Delta_X (p,q)=\left\{\vec{u}\in\Z^2:\textrm{ $\exists x \in X$ 
such that $x_{\llbracket 0,n-1 \rrbracket^2}=q$ and $x_{\vec{u}+\llbracket 0,n-1 \rrbracket^2}=p$}\right\}$$

When the intersection of the sets 
$\Delta_X (p,q)$ for $(p,q)$ couples 
of $n$-blocks is non empty, 
we denote this intersection $\Delta_X (n)$. 
This set is called the \define{gluing set of $n$-blocks} in $X$.

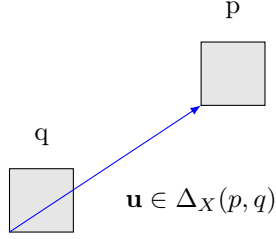
\begin{figure}[ht]
\begin{center}\begin{tikzpicture}[scale=0.7]
\fill[gray!20] (0,0) rectangle (1.2,1.2) ;
\fill[gray!20] (3.6,2.4) rectangle (4.8,3.6);
\draw (0,0) rectangle (1.2,1.2) ; 
\draw (3.6,2.4) rectangle (4.8,3.6);
\draw[->,>=latex,color=blue] (0,0) -- (3.6,2.4);
\node at (0.6,1.8) {q} ;
\node at (4.2,4.2) {p} ; 
\node at (3.6,0.6) {$ \vec{u} \in \Delta_X (p,q)$} ;
\end{tikzpicture}
\end{center}
\caption{ Illustration of Definition~\ref{definition.GluingSet}.}\label{Figure.GluingSet}
\end{figure} \bigskip
\end{definition}

\begin{definition}
 A subshift $X$ is said 
to be \define{$f$-block transitive} 
if for all $n\in\N$ one has
 $$\Delta_X (n)\cap\left\{\vec{u}\in\Z^2:||\vec{u}||_{\infty}\leq n+f(n)\right\}\ne\emptyset.$$
The function $f$ is called the \define{gap 
function}.
\end{definition}

\begin{remark}
The condition on the vectors $\vec{u}\in\Z^2$ 
is $||\vec{u}||_{\infty}\leq n+f(n)$ since we consider 
the distance between the positions 
where the patterns appear 
instead of the space between them.
\end{remark}

\begin{definition}\label{definition.NetGluing}
A subshift $X$ is said to be 
\define{$f$ net gluing} 
if there exists a function 
$\vec{u} : \mathcal{L} (X) ^2 \rightarrow \Z^2$
and a function $\widetilde{f} : 
\mathcal{L} (X) ^2 \rightarrow \N$ 
such that 
 for all $n\in\N$ and 
for all $n$-blocks $p$ and $q$, 
\[\vec{u}(p,q)+(n+\widetilde{f}(p,q)) (\Z^2 \backslash \{0\}) \subset \Delta_X (p,q)\]
and 
\[\max_{p,q \in \mathcal{L}_n (X)} \widetilde{f}(p,q) \leq f(n).\] 
\end{definition} 

\begin{remark}
This property is different from 
quasiperiodicity properties 
in the sense that 
the configuration where two patterns 
appears can depend on the relative 
position of the two patterns.
\end{remark}

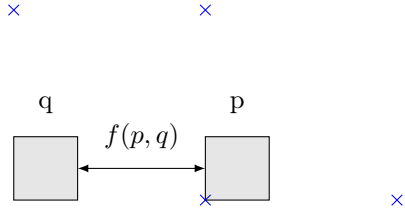
\begin{figure}[ht]
\begin{center}
\begin{tikzpicture}[scale=0.7]
\fill[gray!20] (0,0) rectangle (1.2,1.2) ;
\fill[gray!20] (3.6,0) rectangle (4.8,1.2);
\draw (0,0) rectangle (1.2,1.2) ; 
\node at (0.6,1.8) {q} ;
\draw[color=blue] (-0.1,3.5) -- (0.1,3.7);
\draw[color=blue] (-0.1,3.7) -- (0.1,3.5);
\draw[color=blue] (3.5,-0.1) -- (3.7,0.1);
\draw[color=blue] (3.7,-0.1) -- (3.5,0.1);
\draw[color=blue] (3.5,3.5) -- (3.7,3.7);
\draw[color=blue] (3.7,3.5) -- (3.5,3.7);
\draw[<->,>=latex] (1.2,0.6) -- (3.6,0.6);
\node at (2.4,1.2) {$f(p,q)$};
\draw (3.6,0) rectangle (4.8,1.2);
\node at (4.2,1.8) {p};
\draw[color=blue] (7.1,-0.1) -- (7.3,0.1);
\draw[color=blue] (7.3,-0.1) -- (7.1,0.1);
\end{tikzpicture}
\end{center}
\caption{Illustration of Definition~\ref{definition.NetGluing}. The crosses 
designate elements of the gluing set of $p$ 
relatively to $q$ in $X$.}\label{figure.NetGluing}
\end{figure}

\begin{definition} A subshift $X$ is 
\define{$f$-block gluing} when  
$$\left\{\vec{u} \in \Z^2, ||\vec{u}||_{\infty} \ge f(n)+n\right\} \subset \Delta_X(n).$$ 
\end{definition}

For any function $f$, one has
\[\textrm{$f$-block gluing}\Longrightarrow 
\textrm{$f$-net gluing}\Longrightarrow 
\textrm{$f$-block transitive}\]

A subshift is said \define{$O(f)$-block gluing} (resp. \define{$O(f)$-net gluing}, \define{$O(f)$-block transitive}) if it is $g$ 
(block gluing) (resp. $g$-net gluing, $g$-block transitive) for a function $g:\N\to\N$ such that there 
exists $C>0$ such that $g(n)\leq C\, f(n)$ for all $n\in\N$. 
A property verified on the class of $O(f)$ block gluing (resp. $g$ 
net gluing, $g$ block transitive) subshifts
is said to be 
\define{sharp} if the property is false
for all $h\in o(f)$ (this means that for all $\epsilon >0$ 
there exists $n_0$ such that $h(n)\leq 
\epsilon f(n)$ for all $n\geq n_0$).

A subshift is \define{linearly block gluing} 
(resp. \define{linearly net gluing}, \define{linearly transitive}) if it 
is $O(n)$ block gluing (resp. $O(n)$ 
net gluing, $O(n)$ block transitive).

\subsubsection{Equivalent definition}

The following proposition gives an equivalent definition for linear 
block gluing and net gluing subshifts using some exceptional values: 

\begin{proposition}\label{proposition.DefLinearPower}
A subshift $X\subset\A^{\Z^2}$ is 
linearly block gluing 
if and only if there exists 
a function $f\in O(n)$, $c\geq 2$ an integer 
and $m\in\N$ such that
$$\{\vec{u} \in \Z^2, ||\vec{u}||_{\infty} \geq f(c^l+m)+c^l+m\} \subset \Delta_X (c^l+m) \qquad\forall l \geq 0.$$

A similar assertion is true for net gluing.
\end{proposition}
\begin{proof}
Clearly a linear-block gluing subshift verifies this property. Reciprocally, 
let $p$ and $q$ be two $n$-blocks, and consider 
$l(n)=\left \lceil{\log_c (n-m)}\right \rceil$, where for 
all real number $x$, $\left \lceil{x}\right \rceil$ designates the smallest integer 
greater than $x$. Consider $p'$ and $q'$ some $c^{l(n)}+m$-blocks
whose restrictions on $\llbracket 0,n-1 \rrbracket ^k$ are 
respectively $p$ and $q$. The set $\Delta_X (p',q')$ 
contains $\{\vec{u} \in \Z^2, ||\vec{u}||_{\infty} 
\ge f(c^{l(n)}+m)+c^{l(n)}+m\}$. As a consequence, $\Delta_X (p,q)$ contains 
$\{\vec{u} \in \Z^2, ||\vec{u}||_{\infty} \ge g(n)+n\}$, 
where $g(n)=f(c^{l(n)}+m)+c^{l(n)}-n+m$. Since $c^{l(n)} \le c*(n+|m|)$, 
the function $g$ is in $O(n)$, hence $X$ is $O(n)$-block gluing.
\end{proof}

The proofs of block gluing and net gluing properties 
often rely on this proposition. In our 
construction, we first complete 
patterns into patterns over cells, 
whose sizes are given by an exponential 
sequence as in the statement of the 
proposition. Then we prove 
the gluing property for two cells.

\subsubsection{Block gluing and morphisms}

The following proposition shows that a factor of a 
block gluing (resp. net gluing) subshift
is also block gluing (resp. net gluing) and gives 
a precise gap function.

\begin{proposition} Let $\varphi : X \rightarrow Y$
 be some onto $r$-block map 
between two $\Z^2$-subshifts,
and $f:\N\to\N$ a non decreasing function. If the subshift $X$ is $f$-block gluing (resp. $f$-
net gluing), then $Y$ is $g$-block gluing (resp. $g$-net gluing) where 
$g:n\longmapsto f(n+2r)+2r$.
\end{proposition}
\begin{proof}
Denote $\overline{\varphi}:\A_X^{\llbracket-r,r 
\rrbracket ^2}\to\A_Y$ the local 
rule of $\varphi$.

\begin{enumerate} 
\item \textbf{Relation between 
gluing sets:}

Let $p',q'$ be two $n$-blocks  in the language of $Y$. There exist $p$ and $q$ two $(n+2r)$-blocks in 
the language of $X$ such that $p'$ and $q'$ are respectively the image of $p$ and $q$ by $\overline{
\varphi}$. Let $\vec{u}\in\Delta_X (p,q)$. There exists $x \in X$ such that $x_{\llbracket 0,n+2r-1
\rrbracket ^2}=p$, and $x_{\vec{u}+\llbracket 0,n+2r-1\rrbracket^2}=q$. 
Applying $\varphi$ to $\sigma^{r(1,1)}(x)$, 
we obtain some $y \in Y$ such that $y_{\llbracket 0,n-1\rrbracket ^k}=p'$, and 
$y_{\vec{u}+ \llbracket 0,n-1\rrbracket ^k}=q'$. We deduce that
$$\Delta_X(p,q)\subset \Delta_Y(p',q') \quad\textrm{ so }\quad\Delta_X(n+2r)\subset\Delta_Y(n).$$

\item \textbf{Consequence of the gluing property:}

Thus if $X$ is $f$-block gluing then $Y$ is $g$-block gluing where $g:n\longmapsto f(n+2r)+2r$.

If $X$ is $f$-net gluing, then the gluing set of two $(2n+r)$-blocks $p,q$ contains 
\[\vec{u}(p,q)+(n+2r+\widetilde{f}(p,q))(\Z^2 \backslash \{(0,0)\}),\]
such that $\widetilde{f}(p,q) \leq f(n+2r)$. Hence the gluing set of $p'$, image of $p$ by $\overline{\varphi}$, relative to $q'$, image of $q$ by $\overline{\varphi}$, in $Z$ contains this set. One deduces that $Y$ is $g$-net gluing where $g:n\longmapsto f(n+2r)+2r$.
\end{enumerate}

\end{proof}

We deduce that the classes of subshifts defined by 
these properties 
are invariant of conjugacy under some assumption 
on $f$. The set functions that verify 
this assumption includes all the possible gap 
functions we already know: 

\begin{corollary}
Let $f$ be some non decreasing function. If for 
all $r\in\N$,
there is a constant $C$ such that for all $n \ge 0$, 
$Cf(n) \geq \,f(n+2r)$ then the following classes of 
subshifts are invariant under conjugacy: $O(f)$-block transitive, $O(f)$-net gluing, $O(f)$-block gluing, 
sharp $O(f)$-net gluing and sharp $O(f)$-block gluing subshifts.

In particular it is verified when $f$ is constant or 
$n\mapsto n^k$ with $k>0$ or 
$n\mapsto e^n$ or $n\mapsto\log (n)$.
\end{corollary}

\subsection{Some examples}\label{sec.examples}

We say that two blocks $p,q$ 
having respective supports $\mathbb{U},\mathbb{V}$
are \define{spaced by distance $k$}
when  
\[\max_{\vec{u} \in \mathbb{U}} 
\min_{\vec{v} \in \mathbb{V}} ||\vec{u}-\vec{v}||_{\infty} \ge k.\] 
Since the subshifts that we consider in this text are bi-dimensional, 
this means that there are 
at least $k$ column or at least $k$ lines between the two blocks.

\subsubsection{First examples} 

We present here some examples of block gluing SFT. 

\begin{example} Consider the SFT $X_{\texttt{Chess}}$ defined by the following set of forbidden patterns:
\begin{center}
\begin{tikzpicture}[scale=0.25]
\begin{scope}[xshift=0cm]
\fill[black] (0,0) rectangle ++(1,2);
\draw[gray!50,line width=1pt] (0,0) grid++(1,2);
\end{scope}

\begin{scope}[xshift=4cm,yshift=0.5cm]
\fill[black] (0,0) rectangle ++(2,1);
\draw[gray!50,line width=1pt] (0,0) grid++(2,1);
\end{scope}

\begin{scope}[xshift=9cm]
\draw[gray!50,line width=1pt] (0,0) grid++(1,2);
\end{scope}

\begin{scope}[xshift=13cm,yshift=0.5cm]
\draw[gray!50,line width=1pt] (0,0) grid++(2,1);
\end{scope}
\end{tikzpicture} 

\end{center}

 This subshift has two configurations (see Figure~\ref{figure.echec} for an example) 
and both of them are periodic. 
It is $1$-net gluing, but not block gluing.
Indeed, the gluing set of the pattern $\blacksquare$ relatively to itself is
 $$\Delta_{X_{\texttt{Chess}}}(\blacksquare,\blacksquare)=2\Z^2\setminus\{(0,0)\} \cup \left(2\Z^2 + (1,1)\right)$$

\begin{figure}[ht]
\begin{center}
\begin{tikzpicture}[scale=0.5]
 \clip[draw,decorate,decoration={random steps, segment length=3pt, amplitude=1pt}] (-7.75,-3.75) rectangle (7.75,3.75);
\foreach \x in{-8,-7,...,8}{
\foreach \y in{-4,-3,...,4}{
\fill[black] (\x,\y) rectangle ++(0.5,0.5);
}}
\foreach \x in{-7.5,-6.5,...,8}{
\foreach \y in{-3.5,-2.5,...,4}{
\fill[black] (\x,\y) rectangle ++(0.5,0.5);
}}

\draw[gray!50,line width=1pt,step=0.5] (-8,-4) grid (8,4);

\end{tikzpicture}
\caption{An example of configuration of $X_{\texttt{Chess}}$.}\label{figure.echec}
\end{center}
\end{figure}
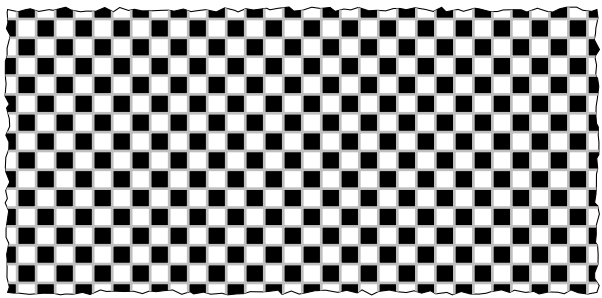
\end{example}  

\begin{example} Consider the SFT $X_{\texttt{Even}}$ defined by the following set of forbidden patterns:
\begin{center}
\begin{tikzpicture}[scale=0.25]
\begin{scope}[xshift=0cm]
\fill[black] (0,0) rectangle ++(1,2);
\draw[gray!50,line width=1pt] (0,0) grid++(1,2);
\end{scope}

\begin{scope}[xshift=4cm,yshift=0.5cm]
\fill[black] (0,0) rectangle ++(2,1);
\draw[gray!50,line width=1pt] (0,0) grid++(2,1);
\end{scope}
\end{tikzpicture} 
\end{center}

An example of configuration in this subshift is given in Figure~\ref{figure.even}. This subshift 
is $1$-block gluing since two blocks in its language can be glued with distance 1, 
filling the configuration with $\square$ symbols. 
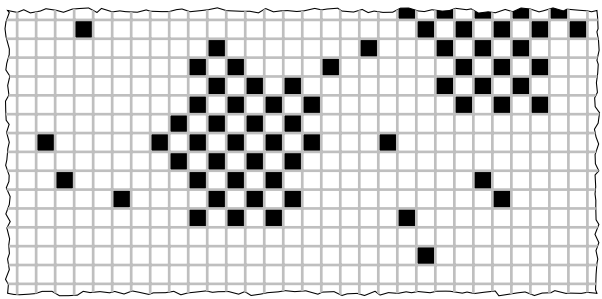
\begin{figure}[h!]
\begin{center}
\begin{tikzpicture}[scale=0.5]
 \clip[draw,decorate,decoration={random steps, segment length=3pt, amplitude=1pt}] (-7.75,-3.75) rectangle (7.75,3.75);

\begin{scope}
  \clip (-3,-2) -- ++(3,0)--++(0,2)--++(0.5,0)--++(0,2)--++(-2,0)--++(0,1)--++(-1.5,0)--++(0,-2)--++(-1,0)--++(0,-1.5)--++(1,0)--cycle;
\foreach \x in{-8,-7,...,8}{
\foreach \y in{-4,-3,...,4}{
\fill[black] (\x,\y) rectangle ++(0.5,0.5);
}}
\foreach \x in{-7.5,-6.5,...,8}{
\foreach \y in{-3.5,-2.5,...,4}{
\fill[black] (\x,\y) rectangle ++(0.5,0.5);
}}
\end{scope}

\begin{scope}
  \clip (3.5,1) -- ++(3,0)--++(0,2)--++(1,0)--++(0,1)--++(-5,0)--++(0,-1)--++(1,0)--++(0,-2.5)--cycle;
\foreach \x in{-8,-7,...,8}{
\foreach \y in{-4,-3,...,4}{
\fill[black] (\x,\y) rectangle ++(0.5,0.5);
}}
\foreach \x in{-7.5,-6.5,...,8}{
\foreach \y in{-3.5,-2.5,...,4}{
\fill[black] (\x,\y) rectangle ++(0.5,0.5);
}}
\end{scope}

\foreach \x/\y in{3/-3,2.5/-2,5/-1.5,4.5/-1,-5/-1.5,-6/3,-7/0,-6.5/-1,0.5/2,1.5/2.5,0.5/2,2/0}
{\fill[black] (\x,\y) rectangle ++(0.5,0.5);}

\draw[gray!50,line width=1pt,step=0.5] (-8,-4) grid (8,4);

\end{tikzpicture}
\caption{An example of configuration of $X_{\texttt{Even}}$.}\label{figure.even}
\end{center}
\end{figure}
\end{example}  

\begin{example} \label{ex.linear} 
Consider the SFT $X_{\texttt{Linear}}$ defined by the following set of forbidden patterns:
\begin{center}
\begin{tikzpicture}[scale=0.25]
\begin{scope}[xshift=0cm]
\fill[black] (0,0) rectangle ++(1,2);
\draw[gray!50,line width=1pt] (0,0) grid++(1,2);
\draw[gray!50,line width=1pt] (0,0) rectangle ++(-1,1);
\end{scope}

\begin{scope}[xshift=4cm]
\fill[black] (0,0) rectangle ++(1,2);
\draw[gray!50,line width=1pt] (0,0) grid++(2,1);
\draw[gray!50,line width=1pt] (1,1) rectangle ++(1,-1);
\end{scope}

\begin{scope}[xshift=9cm]
\fill[black] (0,0) rectangle ++(3,1);
\draw[gray!50,line width=1pt] (0,0) grid++(3,1);
\draw[gray!50,line width=1pt] (1,1) rectangle ++(1,1);
\end{scope}

\end{tikzpicture} 
\end{center}

The local rules imply that if a configuration contains the pattern $\square\blacksquare^n\square$ 
then it contains $\ast\square\blacksquare^{n-2}\square\ast$ just above, 
where $\ast\in\{\square,\blacksquare\}$. Thus a configuration of $X_{\texttt{Linear
}}$ 
can be seen as a layout 
of triangles of made of symbols~$\blacksquare$ on a background of $\square$ symbols (an example 
of configuration is given on 
Figure~\ref{figure.linear}). 

This subshift is sharp linearly block gluing. Indeed, consider two  $n$-blocks in its language 
separated horizontally 
or vertically by $2n$ cells. They contain pieces 
of triangles that we complete with the smallest triangle possible, 
the other symbols of the configuration being all $\square$ symbols. The worst case for gluing 
two $n$-blocks 
is when the blocks are filled with the symbol $\blacksquare$. In this case we can complete each of 
the two blocks 
by a triangle which base is constituted by $\blacksquare^{3n}$. Hence 
every couple of blocks can be glued horizontally and vertically with linear distance. 
To prove that $X_{\texttt{Linear}}$ 
is not $f$-block gluing with $f(n)\in 
o(n)$, we consider the rectangle 
\[\square\blacksquare^n\square\] 
that we would like 
to glue above itself. To do that we need to separate the two copies of this pattern
by about $\lceil\frac{n}{2}\rceil$ cells.
\begin{figure}[ht]
\begin{center}
\begin{tikzpicture}[scale=0.5]
 \clip[draw,decorate,decoration={random steps, segment length=3pt, amplitude=1pt}] (-7.75,-3.75) rectangle (7.75,3.75);

\foreach \i in{0,0.5,...,2}{
\fill[black] (0+\i,\i) rectangle (4-\i,\i+0.5);
}
\foreach \i in{0,0.5,...,2}{
\fill[black] (-10.5+\i,1.5+\i) rectangle (-\i-4.5,\i+2);
}
\foreach \i in{0,0.5,...,3}{
\fill[black] (-6+\i,-2+\i) rectangle (-\i,\i-1.5);
}
\foreach \i in{0,0.5,...,3}{
\fill[black] (2+\i,\i-5) rectangle (8.5-\i,\i-4.5);
}

\draw[gray!50,line width=1pt,step=0.5] (-8,-4) grid (8,4);

\end{tikzpicture}
\caption{An example of configuration of $X_{\texttt{Linear}}$.}\label{figure.linear}
\end{center}
\end{figure}
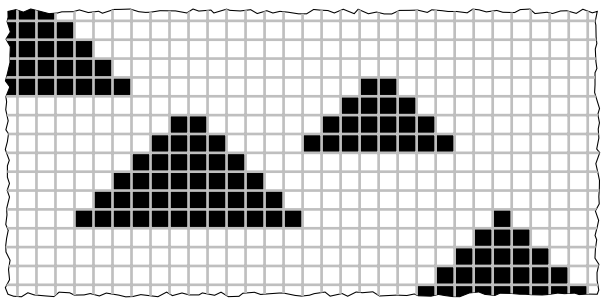
\end{example}

\subsubsection{Linearly net gluing subshifts given by substitutions}

\begin{figure}[ht]
\begin{center}
 \begin{tikzpicture}[scale=0.5]
 \clip[draw,decorate,decoration={random steps, segment length=3pt, amplitude=1pt}] (-7.75,-3.75) rectangle (7.75,3.75);
\foreach \x in{-8,-7,...,8}{
\foreach \y in{-4,-3,...,4}{
\fill[black] (\x,\y) rectangle ++(0.5,0.5);
}}
\foreach \x in{-7.5,-5.5,...,8}{
\foreach \y in{-3.5,-1.5,...,4}{
\fill[black] (\x,\y) rectangle ++(0.5,0.5);
}}
\foreach \x in{-4.5,-0.5,...,8}{
\foreach \y in{-0.5,3.5,...,4}{
\fill[black] (\x,\y) rectangle ++(0.5,0.5);
}}

\foreach \x in{(-2.5,1.5),(1.5,-2.5),(5.5,1.5)}{
\fill[black] \x rectangle ++(0.5,0.5);}
\draw[gray!50,line width=1pt,step=0.5] (-8,-4) grid (8,4);
\end{tikzpicture}
\end{center}
\caption{A part of a configuration of $X_s$.}\label{figure.ExSubstituiton}
\end{figure}
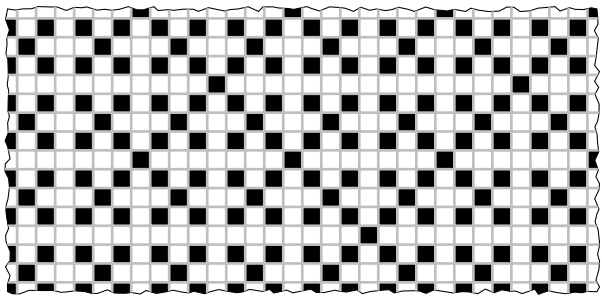

Let $\A$ be a finite alphabet. A \define{substitution} rule is a map $s : \A \rightarrow \A ^{
\mathbb{U}_m}$, for some $m \ge 1$. This function can be extended naturally on blocks in view to iterate it.
The subshift $X_s$ associated to this substitution is the set of configurations such 
that any pattern appearing in it appears
as a sub-pattern of some $s^n (a)$ with $n \geq 0$ and $a \in \A$.

Consider the following substitution $s$ defined by
\begin{center}
  \begin{tikzpicture}[scale=0.25]
  \draw[gray!50,line width=1pt] (0,0.5) rectangle++(1,1);
  \draw[|->,>=latex] (1.5,1)--(3.5,1);
  \draw[gray!50,line width=1pt] (4,0) grid (6,2);
  \fill[black] (4,0) rectangle++(1,1);
\begin{scope}[xshift=8cm]
   \draw[gray!50,line width=1pt,fill=black] (0,0.5) rectangle++(1,1);
  \draw[|->,>=latex] (1.5,1)--(3.5,1);
  \draw[gray!50,line width=1pt] (4,0) grid (6,2);
  \fill[black] (4,0) rectangle++(1,1);
  \fill[black] (5,1) rectangle++(1,1);
\end{scope}
  \end{tikzpicture}
\end{center}
where an exemple of configuration is given in Figure~\ref{figure.ExSubstituiton}. Since $\blacksquare$ appears on position $(0,0)$ in $s(\square)$ and $s(\blacksquare)$, we deduce that for any configuration $x$, there exists $\vec{i}_1\in\llbracket 0,1\rrbracket^2$ such that $x_{\vec{i}_1+2\Z^2}=\blacksquare$. By induction, for all $n\geq 1$,  there exists $\vec{i}_n \in\llbracket 0,2^n-1\rrbracket^2$ such that $x_{\vec{i}_n+2^n\Z^2}=s^n(\blacksquare)$. Since every pattern of  $X_s$ appears in $s^n(\blacksquare)$ for some $n\in\N$, we deduce that $X_s$ has the linear net-gluing property, using Proposition~\ref{proposition.DefLinearPower}.

This argument can be easily generalized for substitution $s$ 
for which there exists $i\in\N$, a subset $\mathcal{Z} \subset \llbracket 0,m^i-1
\rrbracket^2$ and an invertible map $\nu : \A \rightarrow \mathcal{Z}$ such that $a\in\A$ appears on the same position $\nu (a)$ 
in any pattern of the 
patterns $s^i (d)$ with $d \in \A$. 

\subsubsection{Intermediate intensities}

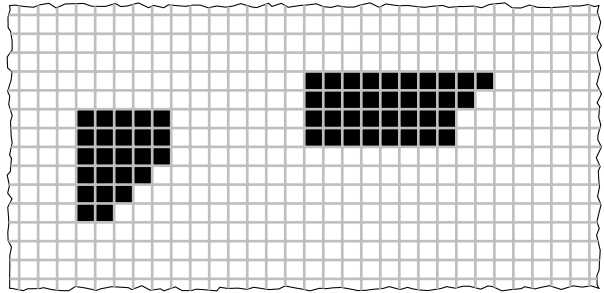
\begin{figure}[ht]
\begin{center}
\begin{tikzpicture}[scale=0.5]
 \clip[draw,decorate,decoration={random steps, segment length=3pt, amplitude=1pt}] (-7.75,-3.75) rectangle (7.75,3.75);

\fill[black] (0,0) rectangle ++(4,0.5);
\fill[black] (0,0.5) rectangle ++(4,0.5);
\fill[black] (0,1) rectangle ++(4.5,0.5);
\fill[black] (0,1.5) rectangle ++(5,0.5);

\fill[black] (-6,-2) rectangle ++(1,0.5);
\fill[black] (-6,-1.5) rectangle ++(1.5,0.5);
\fill[black] (-6,-1) rectangle ++(2,0.5);
\fill[black] (-6,-0.5) rectangle ++(2.5,0.5);
\fill[black] (-6,0) rectangle ++(2.5,0.5);
\fill[black] (-6,0.5) rectangle ++(2.5,0.5);

\draw[gray!50,line width=1pt,step=0.5] (-8,-4) grid (8,4);

\end{tikzpicture}
\caption{An example of configuration that respects the rules of the 
first layer of $X_{\texttt{Log}}$.}\label{figure.log}
\end{center}
\end{figure}

Here we present an example of block gluing 
SFT whose gap 
function is strictly between 
linear and constant classes.

Consider the SFT $X_{\texttt{Log}}$ having 
two layers, with 
the following characteristics: 

\noindent \textbf{\textit{Symbols:}} 

The first layer 
has symbols  $\begin{tikzpicture}[scale=0.3]
\fill[black] (0,0) rectangle (1,1);
\end{tikzpicture}$ and  $\begin{tikzpicture}[scale=0.3]
\draw (0,0) rectangle (1,1);
\end{tikzpicture}$, and the second one the symbols: 
\[\begin{tikzpicture}
\begin{scope}[xshift=0cm]
\draw (0,0) rectangle (1,1);
\node at (0.15,0.5) {0};
\node at (0.85,0.5) {0};
\node at (0.5,0.15) {0};
\node at (0.5,0.85) {0};
\end{scope}

\begin{scope}[xshift=2cm]
\draw (0,0) rectangle (1,1);
\node at (0.15,0.5) {0};
\node at (0.85,0.5) {0};
\node at (0.5,0.15) {1};
\node at (0.5,0.85) {1};
\end{scope}

\begin{scope}[xshift=4cm]
\draw (0,0) rectangle (1,1);
\node at (0.15,0.5) {1};
\node at (0.85,0.5) {0};
\node at (0.5,0.15) {0};
\node at (0.5,0.85) {1};
\end{scope}

\begin{scope}[xshift=6cm]
\draw (0,0) rectangle (1,1);
\node at (0.15,0.5) {1};
\node at (0.85,0.5) {1};
\node at (0.5,0.15) {1};
\node at (0.5,0.85) {0};
\end{scope}

\begin{scope}[xshift=8cm]
\draw (0,0) rectangle (1,1);
\end{scope}\end{tikzpicture}\]
The first four of these 
symbols are thought as coding for the adding machine. Each one contains 
four symbols: the west one is the initial state of the machine, the east one the forward state, 
the south one the input letter, and the last symbol is the output. \bigskip

\noindent \textbf{\textit{Local rules:}}

\begin{itemize}
\item \textbf{First layer:}

The following patterns 
are forbidden in the first layer: 

\begin{center}
\begin{tikzpicture}[scale=0.25]
\begin{scope}[xshift=0cm]
\fill[black] (0,0) rectangle ++(1,2);
\fill[black] (0,1) rectangle ++(-1,1);
\draw[gray!50,line width=1pt] (0,1) grid++(-1,1);
\draw[gray!50,line width=1pt] (0,0) grid++(1,2);
\draw[gray!50,line width=1pt] (0,0) rectangle ++(-1,1);
\end{scope}

\begin{scope}[xshift=4cm]
\fill[black] (0,0) rectangle ++(1,2);
\fill[black] (0,0) rectangle ++(-1,1);
\draw[gray!50,line width=1pt] (0,1) grid++(-1,1);
\draw[gray!50,line width=1pt] (0,0) grid++(1,2);
\draw[gray!50,line width=1pt] (0,0) rectangle ++(-1,1);
\end{scope}

\begin{scope}[xshift=9cm]
\fill[black] (-1,0) rectangle ++(1,2);
\fill[black] (-1,0) rectangle ++(2,1);
\draw[gray!50,line width=1pt] (0,1) grid++(-1,1);
\draw[gray!50,line width=1pt] (0,0) grid++(1,2);
\draw[gray!50,line width=1pt] (0,0) rectangle ++(-1,1);
\end{scope}

\begin{scope}[xshift=13cm]
\fill[black] (-1,0) rectangle ++(1,2);
\fill[black] (-1,1) rectangle ++(3,1);
\draw[gray!50,line width=1pt] (0,1) grid++(-1,1);
\draw[gray!50,line width=1pt] (0,0) grid++(2,2);
\draw[gray!50,line width=1pt] (0,0) rectangle ++(-1,1);
\end{scope}

\end{tikzpicture} 
\end{center}

These local rules imply that if a configuration contains 
the pattern $\square\blacksquare^n\square$ in the first layer, then it contains $
\square\blacksquare^{n}\square$, $\square \blacksquare^{n+1}$, or $\square ^{n+2}$ just above. 
Thus a configuration of the first layer of $X_{\texttt{Log
}}$ can be seen as triangular shapes of  symbols~$\blacksquare$ on a background of $\square$ symbols (an example of configuration is given in 
Figure~\ref{figure.log}). \bigskip

\item \textbf{Second layer:}

\begin{itemize}
\item the adding machine symbols are superimposed on black squares, 
the other ones on blank squares. 
\item for two adjacent machine symbols, the symbols on the sides have to match. 
\item on a pattern $\blacksquare \square$, 
on the machine symbol over the black square, the east symbol have 
to be $0$. 
\item on a pattern $\begin{tikzpicture}[scale=0.25]
 \fill[black] (-1,1) rectangle ++(2,1);
\fill[black] (-1,0) rectangle ++(1,1);
\draw[gray!50,line width=1pt] (0,1) grid++(-1,1);
\draw[gray!50,line width=1pt] (0,0) grid++(1,2);
\draw[gray!50,line width=1pt] (0,0) rectangle ++(-1,1);
\end{tikzpicture}$, the machine have a south symbol 
being $0$ on the north west black square.
\end{itemize}

\end{itemize}

This subshift is sharp $O(log)$-block gluing. Indeed, any two $n$-blocks in its language can 
be glued vertically with distance 1. 
For the horizontal gluing, 
the worst case for gluing two
$n$-blocks is when the two blocks are 
filled with black squares and the adding machine symbols on the leftmost 
column of the blocks are only $1$ (thus maximizing the number of lines 
where the rectangular shape into which we complete the block have to be greater in length
than the one just below). 
In this case, we can complete the block such that each line (from the bottom to the top) is extended 
from the one below with one $\blacksquare$ symbol on the 
right when the machine symbol have a 1 on its west side, and adding 
blank squares to obtain a rectangle. The number of columns added is smaller than the maximal number 
of bits added by the adding 
machine to a length $n$ string of $0,1$ symbols 
in $n$ steps, which is $O(log(n))$.
This means that 
two $n$-blocks can be glued horizontally with distance $O(log(n))$. To see that this property is sharp, 
consider the horizontal gluing of two $1 \times n$ rectangles of black squares, similarly 
as in the linear case.

\begin{remark}
The set of possible tight gap 
functions of block gluing SFT 
seems restricted.
We don't know for instance if, when $f$ is the square root function, 
there exist subshifts for 
which the $f$ block gluing property 
is tight. 
\end{remark}

\subsection{The rigid version 
of the Robinson subshift is linearly net gluing}

Let us denote $id : \N \rightarrow \N$ the function 
defined by $id(n)=n$ for all $n$.

\begin{proposition} \label{prop.rob.linear.net.gluing} 
The subshift $X_{adR}$  is 
$32\textrm{id}$-net gluing, 
hence linearly net gluing.
\end{proposition}
\begin{proof}
Let $p,q$ be two $n$-blocks in the language of $X_{adR}$ 
with $n\geq 1$. There exists $m\in\N
$ such that $2^{m+1}-1< n\leq 2^{m+2}-1$. Hence, 
there is a supertile $S$ of 
order $m+5$ (this comes from the 
completion result on this subshift) 
where $p$ appears. Consider a 
configuration $x\in X_{adR}$ in which 
the pattern $q$ appears in position $(0,0)$. 
The supertile $S$ appears periodically in $x$ with 
period $2^{m+6}=32.2^{m+1} \leq 32n$. Thus the 
gluing set of $p$ relatively to $q$ in $X_{adR}$ contains a set $\vec{u}+2^{m+6}(\Z^2 
\backslash \{(0,0)\})$ for some 
$\vec{u} \in \Z^2$. Thus $X_{adR}$ 
is $32\textrm{id}$ net gluing.
 \end{proof}


\section{Existence of periodic points for \texorpdfstring{$f$}{f}-block gluing SFT}\label{section.periodic} 

In this section we study the existence of a periodic point in 
$f$-block gluing SFT according to the gap 
function $f$.

\subsection{Under some threshold for the gap function, 
there exist periodic points}

In~\cite{Pavlov-Schraudner-2014}, the authors show that any constant 
block gluing SFT admits a periodic point. Using a 
similar argument, we obtain an upper bound on the 
gap functions forcing 
the existence of periodic points. 

\begin{proposition}\label{proposition.StongBlockGluingPeriodicity} Let $X\subset\A^{\Z^2}$ be some SFT having 
rank $r \ge 2$ which is $f$-block 
gluing for some function $f$. If this function verifies that 
there exists $n\in\N$ such that 
\[f(n)<\frac{\log_{|\A|} (n-r+2)}{r-1} -r+2,\] then $X$ admits a periodic point.
\end{proposition}

\begin{figure}[ht]
\[\begin{tikzpicture}
\fill[gray!20] (0,0) rectangle (8,0.5);
\fill[gray!20] (0,2.5) rectangle (8,3);
\draw (0,0) rectangle++ (8,0.5);
\draw (8,0.25) node[right] {$w$};
\draw (0,2.5) rectangle++ (8,0.5);
\draw[<->] (8.75,2.5) -- ++ (0,0.5);
\draw (8.75,2.75) node[right] {$r-1$};
\draw[<->] (8.75,0) -- ++ (0,0.5);
\draw (8.75,0.25) node[right] {$r-1$};
\draw (8,2.75) node[right] {$w$};
\draw[<->,>=latex] (0,3.25) -- (8,3.25);
\draw (4,3.25) node[above] {$n$};
\draw[<->,>=latex] (7.75,0.5) -- (7.75,2.5);
\draw (7.75,1.5) node[right] {$f(n)$};
\draw (0.5,0) rectangle ++(0.5,2.5);
\fill[pattern=north west lines, pattern color=blue]  (0.5,0) rectangle ++(0.5,2.5);
\draw (0.5,1.5) node[left]  {$x_{\{k\}\times\llbracket 0, f(n)+r-2 \rrbracket}$};
\draw (3.5,0) rectangle ++(0.5,2.5);
\fill[pattern=north west lines, pattern color=blue]  (3.5,0) rectangle ++(0.5,2.5);
\draw (4,1.5) node[right] {$x_{\{k+l\}\times\llbracket 0, f(n)+r-2 \rrbracket}$};
\draw[<->] (0.5,-0.25) -- ++(0.5,0);
\draw (0.5,-0.5) node[right] {$r-1$};
\draw[<->] (3.5,-0.25) -- ++(0.5,0);
\draw (3.5,-0.5) node[right] {$r-1$};

\end{tikzpicture}\] 
\caption{If $n-r+2>|\A|^{(r-1)(f(n)+r-2)}$, it is possible to find a pattern of $x$ for 
which the horizontal and vertical borders are similar.}\label{figure.PeriodicAndGluing} 
\end{figure}
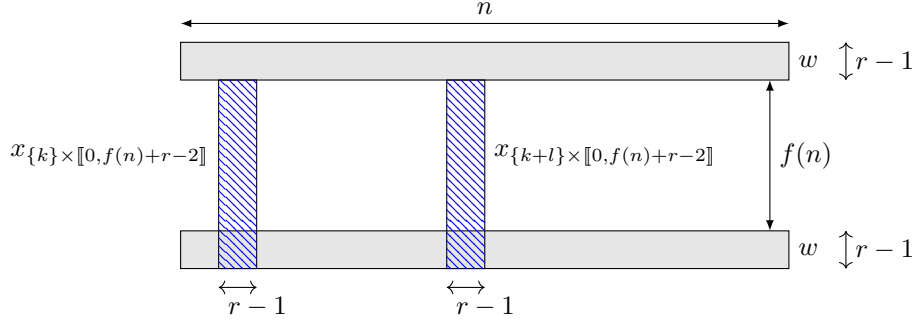

\begin{proof}
 Let $w$ be a $n \times (r-1)$ pattern in the language of $X$. 

\begin{enumerate}
\item \textbf{Gluing the pattern $w$ 
over itself:}

By the $f$-block gluing property, there exists $x \in X$ such that 
\[x_{\llbracket 0, n-1 \rrbracket \times \llbracket 0, r-2 \rrbracket} = w = x_{\llbracket 0, n-1 \rrbracket \times \llbracket f(n)+r-1,f(n)+2r-3\rrbracket}.\] 

\item \textbf{Taking $w$ long enough, 
two of the columns in the obtained pattern 
are equal:}

Consider the 
sub-patterns of $x_{\llbracket 0, n-1 \rrbracket \times \llbracket 0,f(n)+r-2 \rrbracket}$ over 
supports $\llbracket k,k+r-2 \rrbracket \times \llbracket 0,(f(n)+r-2) \rrbracket$. 
There are $n-(r-2)$ of them and the number of possibilities 
is $|\A|^{(r-1)(f(n)+r-2)}$. Since we have 
\[f(n)<\frac{\log_{|\A|}(n-r+2)}{r-1}-r+2,\] 
by the pigeon hole principle, there exists $k
\in\llbracket 0, n-r+1 \rrbracket$ and $l \ge 1$ 
such that 
 \[x_{\llbracket k,k+r-2 \rrbracket\times\llbracket 0, f(n)+r-2 \rrbracket}=
x_{\llbracket k+l,k+l+r-2 \rrbracket\times\llbracket 0, f(n)+r-2 \rrbracket}\] 
(see Figure~\ref{figure.PeriodicAndGluing}). 

\item \textbf{Construction 
of a periodic configuration:}

Consider the configuration defined by 
\[z_{(il, j(f(n
)+r-2))+ \llbracket 0,l-1 \rrbracket\times\llbracket 0,f(n)+r-2\rrbracket}=x_{\llbracket 
k,k+l-1\rrbracket\times\llbracket 0,f(n)+r-2\rrbracket}\]
 for all $(i,j)\in\Z^2$. It consists in 
covering $\Z^2$ with the pattern $x_{\llbracket 
k,k+l-1\rrbracket\times\llbracket 0,f(n)+r-2\rrbracket}$.
This configuration is periodic by definition. 
Moreover, it satisfies the local 
rules of $X$. We deduce that $z\in X$. Thus $X$ admits 
a periodic point.

\end{enumerate}
\end{proof}

\subsection{Under some smaller threshold, 
the set of periodic points is dense and 
the language is decidable}

Using a similar argument as in \cite{Pavlov-Schraudner-2014}, we obtain also an upper bound on 
the gap functions that force 
the density of periodic points, and 
as a consequence the decidability of 
the language. 

\begin{proposition} \label{prop.6} Let $X$ be 
some $f$-block gluing 
$\Z^2$-SFT, where $f$ is 
a function such that \[f(n) \in o(\log (n))\]
and $f \le id$.
Then $X$ has a dense set of periodic points.
\end{proposition}

\begin{figure}[ht]
\[\begin{tikzpicture}[scale=0.5]

\fill[gray!20] (0,0) rectangle (17,1);
\fill[gray!20] (0,8) rectangle (17,9);
\fill[gray!20] (2,3) rectangle (5,6);
\fill[gray!20] (5.5,3) rectangle (8.5,6);
\fill[gray!20] (12,3) rectangle (15,6);

\fill[pattern=north west lines, pattern color=blue] (0,0) rectangle (1,8);
\fill[pattern=north west lines, pattern color=blue] (10,0) rectangle (11,8);
\node at (8.5,10) {R};
\node at (8.5,-1) {R};
\draw (0,0) -- (0,9);
\draw (0,0) -- (17,0);
\draw (0,9) -- (17,9);
\node at (3.5,4.5) {P};
\draw (2,3) rectangle (5,6);
\node at (7,4.5) {P};
\draw (5.5,3) rectangle (8.5,6);
\node at (13.5,4.5) {P};
\draw (12,3) rectangle (15,6);
\draw (0,0) rectangle (1,9);
\draw (10,0) rectangle (11,9);
\draw (0,0) rectangle (17,1);
\draw (0,8) rectangle (17,9);
\end{tikzpicture}\]
\caption{An illustration of the proof of Proposition~\ref{prop.6}. 
The result of the procedure 
is the colored rectangle.}\label{pic.illustrationdensity}
\end{figure}
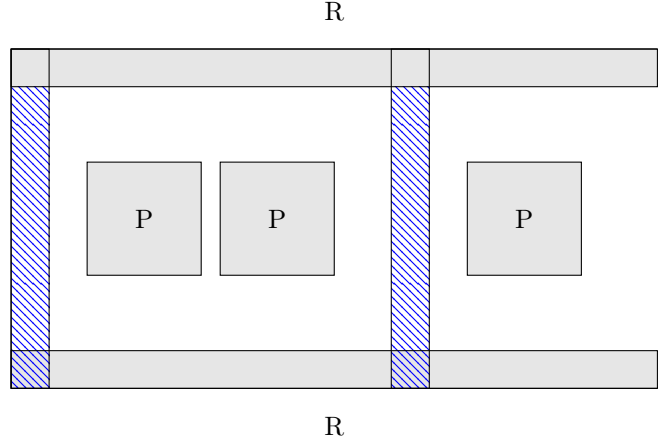

\begin{proof} 

\begin{enumerate}
\item \textbf{Gluing multiple times 
the same block $P$ horizontally:}

Consider $P$ some $n$-block in the language of $X$, and take $2^k$ copies of it. We group them 
by two and glue the couples horizontally, at distance $f(n)$. Then glue 
the obtained patterns 
after grouping them by two, at distance $f(2n+f(n))$, 
and repeat this operation until having 
one rectangular block $Q$. The size of this block is equal to $(2id+f)^{\circ k} (n)$.

\item \textbf{Glue the same rectangular pattern 
over and under the obtained one:}

Then consider some $(2id+f)^{\circ k} (n) \times (r-1)$ pattern $R$, 
where $r$ is the rank of the SFT $X$. 
Glue it on the top of $Q$ with 
$f(\text{max} ((2id+f)^{\circ k} (n), n , r))$ lines between the two rectangles. 
Then glue the rectangle $R$ under the 
obtained pattern
with \[f( \text{max} (f(\text{max} ((2id+f)^{\circ k} (n), n , r))+r+n,(2id+f)^{\circ k} (n)))\]
lines between them. For $k$ great enough (depending on n), 
these two last distances are equal to $f((2id+f)^{\circ k} (n))$, and 
$f(f((2id+f)^{\circ k} (n)) + n + r)$ respectively.
By the gluing property, the obtained pattern (see Figure~\ref{pic.illustrationdensity}) 
is in the language of $X$.

\item \textbf{Pigeon hole 
principle on the columns of the obtained 
pattern:}

Consider the 
$(r-1) \times (r+n+f((2id+f)^{\circ k} (n))+f(f((2id+f)^{\circ k} (n) + n + r))$ sub-patterns that appear on the bottom 
of the columns just on the right of each occurrence of the pattern $P$. There are $2^k$ of them, 
and there are at most
$(|\A|) ^{r+n+f((2id+f)^{\circ k} (n)+f(f((2id+f)^{\circ k} (n) + n + r))}$ different possibilities. 
From the fact that $f \le id$,
it follows that 
\[(|\A|) ^{r+n+f((2id+f)^{\circ k} (n)+f(f((2id+f)^{\circ k} (n) + n + r))} \le 
(|\A|) ^{2(r+n)+2f(3^{k} n)} \le 2^k\]
for $k$ great enough.

By the pigeon hole principle, two of these 
patterns are equal. Consider the rectangle 
between these two occurrences (including the second one).

\item \textbf{Construction 
of a periodic configuration 
containing $P$:}

This rectangle can be repeated on the whole plane to get a periodic configuration which is in $X$.

The set of periodic configurations obtained by this method is dense 
in $X$ (for every pattern 
in its language appear in such 
a configuration).
\end{enumerate}

\end{proof}

\begin{proposition} \label{proposition.decidable.language} A subshift $X$ verifying the conditions of 
Proposition~\ref{prop.6} with $f$ computable 
has decidable language. 
\end{proposition}

\begin{proof}
It is sufficient to prove that it is possible 
to decide if a block is in the language of $X$. 
Indeed, a pattern is in the language of $X$ 
if and only if it is a sub-pattern of a block
in the language of $X$ whose size is 
determined by the pattern. In order to decide 
if a pattern is in the language of $X$, we consider 
all the blocks of this size having this pattern 
as sub-pattern, and use the algorithm on 
block to decide if these blocks are in the 
language of $X$. The pattern is in 
the language of $X$ if and only if one
of these blocks is in it.

We proved, in the proof of Proposition~\ref{prop.6}, 
that any $n$-block is sub-pattern of a rectangular 
pattern whose size is bounded by a computable 
function of $n$ (since $f$ is computable), 
whose topmost $r$ lines are equal to the bottommost $r$ 
lines and the leftmost $r$ columns are equal to 
the rightmost $r$ columns. To decide if 
some $n$-block $P$ is in the language of $X$, 
we test for all the rectangular patterns 
that verify these conditions, if it has $P$ as 
sub-pattern. The block $P$ is in the language
if and only if this is the case for at least one of these 
patterns.
\end{proof}

\subsection{An example of linearly block gluing SFT with non decidable language}

The property obtained in Proposition~\ref{proposition.decidable.language}
is no longer true considering linearly block gluing subshifts of finite type.
In this section we provide an example of 
linearly block gluing subshift of finite type 
having non decidable language: 

\begin{proposition} \label{prop.langundec}
There exists some $O(n)$-block gluing $\Z^2$-SFT with 
non decidable language.
\end{proposition}

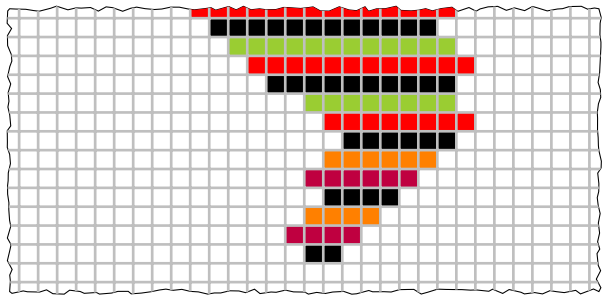
\begin{figure}[ht]
\begin{center}
\begin{tikzpicture}[scale=0.5]
 \clip[draw,decorate,decoration={random steps, segment length=3pt, amplitude=1pt}] (-7.75,-3.75) rectangle (7.75,3.75);

\fill[black] (0,-3) rectangle ++(1,0.5);
\fill[purple] (-0.5,-2.5) rectangle ++(2,0.5);
\fill[orange] (0,-2) rectangle ++(2,0.5);
\fill[black] (0.5,-1.5) rectangle ++(2,0.5);
\fill[purple] (0,-1) rectangle ++(3,0.5);
\fill[orange] (0.5,-0.5) rectangle ++(3,0.5);
\fill[black] (1,0) rectangle ++(3,0.5);
\fill[red] (0.5,0.5) rectangle ++(4,0.5);
\fill[YellowGreen] (0,1) rectangle ++(4,0.5);
\fill[black] (-1,1.5) rectangle ++(5,0.5);
\fill[red] (-1.5,2) rectangle ++(6,0.5);
\fill[YellowGreen] (-2,2.5) rectangle ++(6,0.5);
\fill[black] (-2.5,3) rectangle ++(6,0.5);
\fill[red] (-3,3.5) rectangle ++(7,0.5);

\draw[gray!50,line width=1pt,step=0.5] (-8,-4) grid (8,4);

\end{tikzpicture}
\caption{An example of configuration that respects the rules of the 
first layer of $X_{\texttt{Undec}}$.}\label{figure.undec}
\end{center}
\end{figure}

\noindent \textbf{Idea of the proof:} \textit{we construct 
a structure subshift which consists of infinite and constantly growing 
areas for the computations of machines. These areas 
can be distorted by shifting its adjacent lines one with respect to the 
other. This can be done in any of the two possible directions.
This allows the linear block gluing. Then we implement 
a universal Turing machine, and forbid it to stop.
This implies the non decidability.}

\begin{proof} Let $X_{undec}$ 
be a subshift, product of two layers. Here 
is a description of these layers: 

\begin{enumerate}
\item \textbf{Computation areas layer:} 

This layer has the following 
symbols: 
\[\redsq, \purplesq, \greensq, \orangesq, \blacksq, 
\blanksq\]
Its local rules are the following ones:
\begin{itemize}
\item two horizontally 
adjacent non blank symbols have the same color.
\item two vertically adjacent non blank symbols verify the following rules: 
\begin{enumerate}
\item if the bottom symbol is $\blacksq$, the top symbol is 
$\redsq$ or $\purplesq$.
\item if the bottom symbol is $\redsq$, the top symbol is 
$\greensq$.
\item if the bottom symbol is $\purplesq$, the top symbol 
is $\orangesq$.
\item if the bottom symbol is $\greensq$ or $\orangesq$, 
the top symbol is $\blacksq$.
\end{enumerate}
These rules allow to shift a row of 
the area from the one under, 
choosing the direction. Moreover the shifts 
happen each time by groups of two, 
so that the shift counterbalances the growth 
of the area.
\item the patterns 
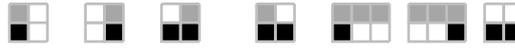
\[
\begin{tikzpicture}[scale=0.25]
\begin{scope}[xshift=-4cm]
\fill[gray!70] (-1,1) rectangle ++(1,1);
\fill[black] (-1,0) rectangle ++(1,1);
\draw[gray!50,line width=1pt] (0,1) grid++(-1,1);
\draw[gray!50,line width=1pt] (0,0) grid++(1,2);
\draw[gray!50,line width=1pt] (0,0) rectangle ++(-1,1);
\end{scope}

\begin{scope}[xshift=0cm]
\fill[gray!70] (0,1) rectangle ++(1,1);
\fill[black] (0,0) rectangle ++(1,1);
\draw[gray!50,line width=1pt] (0,1) grid++(-1,1);
\draw[gray!50,line width=1pt] (0,0) grid++(1,2);
\draw[gray!50,line width=1pt] (0,0) rectangle ++(-1,1);
\end{scope}

\begin{scope}[xshift=4cm]
\fill[black] (0,0) rectangle ++(1,2);
\fill[black] (0,0) rectangle ++(-1,1);
\fill[gray!70] (0,1) rectangle ++(1,1);
\draw[gray!50,line width=1pt] (0,1) grid++(-1,1);
\draw[gray!50,line width=1pt] (0,0) grid++(1,2);
\draw[gray!50,line width=1pt] (0,0) rectangle ++(-1,1);
\end{scope}

\begin{scope}[xshift=9cm]
\fill[black] (-1,0) rectangle ++(1,2);
\fill[black] (-1,0) rectangle ++(2,1);
\fill[gray!70] (-1,1) rectangle ++(1,1);
\draw[gray!50,line width=1pt] (0,1) grid++(-1,1);
\draw[gray!50,line width=1pt] (0,0) grid++(1,2);
\draw[gray!50,line width=1pt] (0,0) rectangle ++(-1,1);
\end{scope}

\begin{scope}[xshift=13cm]
\fill[black] (-1,0) rectangle ++(1,2);
\fill[gray!70] (-1,1) rectangle ++(3,1);
\draw[gray!50,line width=1pt] (0,1) grid++(-1,1);
\draw[gray!50,line width=1pt] (0,0) grid++(2,2);
\draw[gray!50,line width=1pt] (0,0) rectangle ++(-1,1);
\end{scope}

\begin{scope}[xshift=17cm]
\fill[black] (1,0) rectangle ++(1,2);
\fill[gray!70] (-1,1) rectangle ++(3,1);
\draw[gray!50,line width=1pt] (0,1) grid++(-1,1);
\draw[gray!50,line width=1pt] (0,0) grid++(2,2);
\draw[gray!50,line width=1pt] (0,0) rectangle ++(-1,1);
\end{scope}

\begin{scope}[xshift=21cm]
\fill[black] (-1,0) rectangle ++(2,1);
\draw[gray!50,line width=1pt] (0,1) grid++(-1,1);
\draw[gray!50,line width=1pt] (0,0) grid++(1,2);
\draw[gray!50,line width=1pt] (0,0) rectangle ++(-1,1);
\end{scope}

\end{tikzpicture} 
\] are forbidden, where 
the gray symbol stands for any non blank symbol.

\item the patterns \[
\begin{tikzpicture}[scale=0.25]
\begin{scope}[xshift=-4cm]
\fill[gray!70] (0,1) rectangle ++(1,1);
\fill[red] (0,0) rectangle ++(1,1);
\draw[gray!50,line width=1pt] (0,1) grid++(-1,1);
\draw[gray!50,line width=1pt] (0,0) grid++(1,2);
\draw[gray!50,line width=1pt] (0,0) rectangle ++(-1,1);
\end{scope}

\begin{scope}[xshift=0cm]
\fill[red] (0,0) rectangle ++(1,2);
\fill[red] (0,0) rectangle ++(-1,1);
\fill[gray!70] (0,1) rectangle ++(1,1);
\draw[gray!50,line width=1pt] (0,1) grid++(-1,1);
\draw[gray!50,line width=1pt] (0,0) grid++(1,2);
\draw[gray!50,line width=1pt] (0,0) rectangle ++(-1,1);
\end{scope}

\begin{scope}[xshift=4cm]
\fill[red] (1,0) rectangle ++(1,2);
\fill[gray!70] (-1,1) rectangle ++(3,1);
\draw[gray!50,line width=1pt] (0,1) grid++(-1,1);
\draw[gray!50,line width=1pt] (0,0) grid++(2,2);
\draw[gray!50,line width=1pt] (0,0) rectangle ++(-1,1);
\end{scope}

\begin{scope}[xshift=8cm]
\fill[gray!70] (-1,1) rectangle ++(2,1);
\fill[red] (-1,0) rectangle ++(1,1);
\draw[gray!50,line width=1pt] (-1,0) grid++(2,2);
\end{scope}

\begin{scope}[xshift=12cm]
\fill[gray!70] (-1,1) rectangle ++(1,1);
\fill[red] (-1,0) rectangle ++(1,1);
\draw[gray!50,line width=1pt] (-1,0) grid++(2,2);
\end{scope}

\begin{scope}[xshift=17cm]
\fill[gray!70] (-1,1) rectangle ++(1,1);
\fill[red] (-1,0) rectangle ++(3,1);
\draw[gray!50,line width=1pt] (-1,0) grid++(3,2);
\end{scope}

\begin{scope}[xshift=21cm]
\fill[red] (-1,0) rectangle ++(2,1);
\draw[gray!50,line width=1pt] (0,1) grid++(-1,1);
\draw[gray!50,line width=1pt] (0,0) grid++(1,2);
\draw[gray!50,line width=1pt] (0,0) rectangle ++(-1,1);
\end{scope}

\end{tikzpicture} 
\] are forbidden, where 
the gray symbol stands for any non blank symbol.
Similar rules are imposed, replacing the red symbol with a green one.

\item the patterns \[
\begin{tikzpicture}[scale=0.25]
\begin{scope}[xshift=-4cm]
\fill[gray!70] (-1,1) rectangle ++(1,1);
\fill[purple] (-1,0) rectangle ++(1,1);
\draw[gray!50,line width=1pt] (0,1) grid++(-1,1);
\draw[gray!50,line width=1pt] (0,0) grid++(1,2);
\draw[gray!50,line width=1pt] (0,0) rectangle ++(-1,1);
\end{scope}

\begin{scope}[xshift=0cm]
\fill[purple] (-1,0) rectangle ++(2,1);
\fill[gray!70] (-1,1) rectangle ++(1,1);
\draw[gray!50,line width=1pt] (0,1) grid++(-1,1);
\draw[gray!50,line width=1pt] (0,0) grid++(1,2);
\draw[gray!50,line width=1pt] (0,0) rectangle ++(-1,1);
\end{scope}

\begin{scope}[xshift=4cm]
\fill[purple] (-1,0) rectangle ++(1,1);
\fill[gray!70] (-1,1) rectangle ++(3,1);
\draw[gray!50,line width=1pt] (0,1) grid++(-1,1);
\draw[gray!50,line width=1pt] (0,0) grid++(2,2);
\draw[gray!50,line width=1pt] (0,0) rectangle ++(-1,1);
\end{scope}

\begin{scope}[xshift=8cm]
\fill[gray!70] (-1,1) rectangle ++(2,1);
\fill[purple] (0,0) rectangle ++(1,1);
\draw[gray!50,line width=1pt] (-1,0) grid++(2,2);
\end{scope}

\begin{scope}[xshift=12cm]
\fill[gray!70] (0,1) rectangle ++(1,1);
\fill[purple] (0,0) rectangle ++(1,1);
\draw[gray!50,line width=1pt] (-1,0) grid++(2,2);
\end{scope}

\begin{scope}[xshift=17cm]
\fill[gray!70] (1,1) rectangle ++(1,1);
\fill[purple] (-1,0) rectangle ++(3,1);
\draw[gray!50,line width=1pt] (-1,0) grid++(3,2);
\end{scope}

\begin{scope}[xshift=21cm]
\fill[purple] (-1,0) rectangle ++(2,1);
\draw[gray!50,line width=1pt] (0,1) grid++(-1,1);
\draw[gray!50,line width=1pt] (0,0) grid++(1,2);
\draw[gray!50,line width=1pt] (0,0) rectangle ++(-1,1);
\end{scope}

\end{tikzpicture} 
\] are forbidden, where 
the gray symbol stands for any non blank symbol.
Similar rules replacing the red symbol with an orange one. 
These rules allow to control the shape of the areas.
\end{itemize}

These rules imply that: 
\begin{itemize}
\item above $\blanksq \ \blacksq ^ n \blanksq$
there is $\redsq \ \redsq ^ n \redsq$, or $\purplesq \ \purplesq ^ n \purplesq$.
\item above $\blanksq \ \redsq ^ n \blanksq$ there is 
$\greensq ^ n \blanksq  \blanksq$.
\item above $\blanksq \ \purplesq ^ n \blanksq$ there is 
$\blanksq \blanksq \orangesq ^ n$.
\item above $\blanksq \ \greensq ^ n \blanksq$ there is 
$\blacksq ^ n \blanksq  \blanksq$.
\item above $\blanksq \ \orangesq ^ n \blanksq$ there is 
$\blanksq \blanksq \blacksq ^ n$.
\end{itemize} 

The computation areas consist of colored areas. They lie on a background of 
$\blanksq$ symbols.
These areas are distorted infinite triangles: in two adjacent rows, the intersection of the
area with the top row is larger than in the bottom row by one position 
on the right and one position on the left. Then this row is shifted or not, 
horizontally in one of the two directions, depending on the colors of the rows.
See Figure~\ref{figure.undec} for an illustration.

\item \textbf{Machines layer:}

The second layer consists in the 
implementation of a Turing machine 
over these areas. 

The symbols are the elements of 
$\mathcal{Q} \times \mathcal{A} \times \{\rightarrow,\leftarrow\}$, 
and the local rules are the following ones: 

\begin{enumerate}
\item the blank symbols are superimposed 
with a blank symbol, and the Turing machine symbols are superimposed over 
non blank symbols. 
\item Moreover, considering a $3 \times 2$ pattern whose projection 
on the first layer is fully colored and the bottom row is black,
the rules of the space-time diagram of the machine apply.
When the bottom row is not black, the symbols of this row are copied 
on the top row, on the shifted position according to the color of the bottom row.
There rules of the space time diagram are adapted when on the border 
of the area.
\item Any halting state is forbidden.
\end{enumerate}

These rules imply that for two adjacent rows of an area: 
\begin{itemize}
\item if the bottom row is colored black, then the 
top row is the image of the bottom row by the machine process.
\item in the other cases, the top row is just the image of the bottom 
row by the shift in the direction corresponding to the color of 
the bottom row.
\end{itemize} \bigskip

We use a universal Turing machine, which has 
the following behavior when the initial tape 
is written with $w \#^{\infty}$ for $w$ a word 
on $\{0,1\}$: 
it reads the word $w$ which codes 
for the number of a Turing 
machine, and then simulates this machine on empty tape. 

\item \textbf{Properties of this subshift:}

\begin{enumerate}
\item \textbf{The language of this subshift is not 
decidable:} if it was, 
we would be able to 
decide which of the words $w \#^{n}$ 
can be written on the bottom of an area. This is impossible 
since the halting problem is not decidable. 

\item 
\textbf{This subshift is sharp 
linearly block gluing:} 
the worst case for gluing two blocks is 
when these blocks are filled with colored symbols in the first layer. 
In order to glue them, we complete 
the projection of 
the two blocks in the first layer, 
as in example~\ref{ex.linear}, into the bottom of 
a computation area, surrounding it with blank symbols except on 
the top. Then we complete the trajectory of the machine head 
if there is one.
The two extended patterns can be glued horizontally 
without constraint on the distance, because lines 
can be shifted towards opposite directions. 
For vertical gluing, we extend one of this patterns shifting the area 
in one direction so that the columns above this patterns are blank.
The number of rows and columns depends linearly on the size of this pattern. 
Then we glue the second pattern on the top, shifting the area in the 
opposite direction and filling all the positions of $\Z^2$ left undefined 
by blank symbols.
\end{enumerate}
\end{enumerate}
\end{proof}

\subsection{Existence of aperiodic linearly 
block gluing subshifts}

In this section, we give a proof of the following theorem: 

\begin{theorem}
\label{thm.aperiodicity}
There exists a linearly block 
gluing aperiodic $\Z^2$-SFT.
\end{theorem}

\begin{figure}[h]
\[\begin{tikzpicture}[scale=0.4]
\begin{scope}
\foreach \x in {1,...,9} {
\draw[color=gray!60] 
(\x,0) -- (\x,10);}
\foreach \x in {1,...,9} {
\draw[line width =0.4mm] (0,\x) -- (10,\x);}
\end{scope}

\draw[-latex] (13,5) -- (17,5);

\begin{scope}[xshift=20cm]
\foreach \x in {1,...,9} {
\draw[color=gray!60] (\x,0) -- (\x,10);}
\draw[line width =0.4mm] (0,9) -- (8,9) 
-- (8,8) -- (10,8);
\draw[line width =0.4mm] (0,8) -- (7,8) 
-- (7,7) -- (10,7);
\draw[line width =0.4mm] (0,7) -- (3,7) 
-- (3,6) -- (10,6);
\draw[line width =0.4mm] (0,5) -- (10,5) ;
\draw[line width =0.4mm] (0,4) -- (10,4);
\draw[line width =0.4mm] (0,3) -- (5,3) 
-- (5,2) -- (10,2);
\draw[line width =0.4mm] (0,1) -- (8,1) 
-- (8,0) -- (10,0);
\end{scope}
\end{tikzpicture}\]
\caption{\label{fig.deformation} Illustration 
of the misshaping principle of the transformation 
for the proof of Theorem~\ref{thm.aperiodicity}.}
\end{figure}
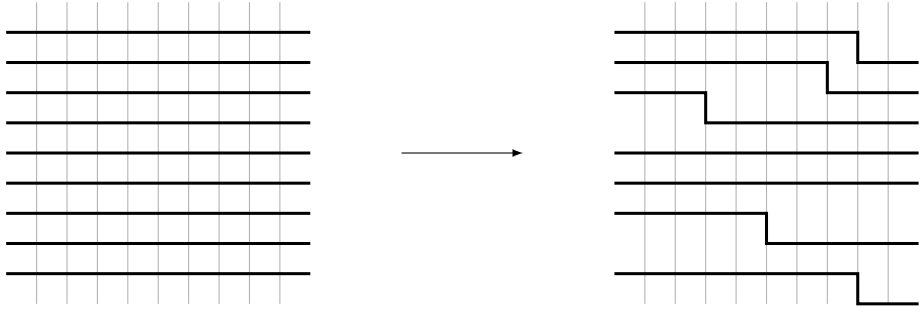

\noindent \textbf{Idea of the proof:} \textit{Let 
us recall the idea of the proof for 
Theorem~\ref{thm.aperiodicity},
presented in the introduction. This proof uses an operator
on subshifts which transforms linearly net gluing subshifts of 
finite type into 
linearly block gluing subshifts of finite type.
Moreover, it preserves the aperiodicity.
The principle 
of this transformation is to distort $\Z^2$,
as illustrated on Figure~\ref{fig.deformation}, multiple 
times and in different directions.}

\textit{We then apply this transformation on the 
Robinson subshift
which is known to be aperiodic and that we proved 
to be linearly net gluing (Proposition~\ref{prop.rob.linear.net.gluing}). }

\subsubsection{ \label{sec.pseudo.covering}
A subshift inducing pseudo-coverings by curves}

\paragraph{Definition}

Let us denote $\Delta$ the $\Z^2$ SFT on alphabet $\{\rightarrow,\downarrow\}$, 
defined by the following 
forbidden patterns: 
\[\begin{array}{c}
\downarrow\\
\downarrow
\end{array}, \begin{array}{cc}
\rightarrow & \downarrow \\
\rightarrow & \rightarrow
\end{array}.\] 

\paragraph{Pseudo-coverings by curves} 
Let us introduce some words in order to talk about the global behavior induced by these rules: 
\begin{itemize}
\item An (infinite) \textbf{curve} in $\Z^2$ is a set ${\mathcal{C}}=\varphi (\Z)$ for some application 
$\varphi : \Z \rightarrow \Z^2$ such that for all $k \in \Z$, 
$\varphi (k+1) =\varphi (k) + (1,0)$ or $\varphi (k+1) =\varphi (k) + (1,-1)$.
\item We say that a curve is {\bf{shifted downwards}} at position $\vec{j}\in \Z^2$ when there exists some $k\in \Z$ 
such that $\varphi (k)=\vec{j}$ and $\varphi(k+1)=\vec{j}+(1,-1)$.
\item A {\bf{pseudo-covering of $\Z^2$ by curves}} is
a sequence of curves $({\mathcal{C}}_k)_{k\in\Z}$ such that for every $\vec{j} \in \Z^2$, there exists some $k \in \Z$ such that
$\{\vec{j}, \vec{j}+(0,1)\} \bigcap {\mathcal{C}}_k \neq \emptyset$ (meaning that every element of $\Z^2$ is in 
a curve or the vector just above is), and for every $k \neq k'$, 
${\mathcal{C}}_k \bigcap {\mathcal{C}}_{k'} = \emptyset$ (the curves do not intersect). 
We say that two curves in this 
pseudo-covering are {\bf{contiguous}} when the area delimited by these two curves does not contain any third curve. The gap 
between two contiguous curves in some column is the distance between the intersection of these two curves with the column. 
This gap is 0 or 1 between two contiguous curves in a pseudo-covering.
\end{itemize}

\paragraph{\label{sec3.3.3} A configuration in $\Delta$ induces a pseudo-covering by curves} 

Let $\delta \in \Delta$. Let us consider the pseudo-covering of $\Z^2$ by curves 
$({\mathcal{C}}_k (\delta))_{k\in\Z}$, such that
${\mathcal{C}}_k (\delta) = \varphi_{\delta,k} (\Z)$ and 
where $\varphi_{\delta,k}$ is as follows. \bigskip

If $\delta_{(0,0)} = \rightarrow$, then $(0,0) = \varphi_{\delta,0} (0)$. Else 
$(0,1) = \varphi_{\delta,0} (0)$.
In addition, for $m$ 
the biinfinite sequence of integers such 
that $\delta_{(m_k,0)}$ is the $k$th $\rightarrow$ in the column 0, counting 
from the previous considered one, then $(m_k,0)=\varphi_{\delta,k} (0)$. \bigskip

For every $\vec{i} \in \Z^2$ such that $\vec{i} = \varphi_{\delta,k} (n)$ for some $k\in \Z , n\in \Z$: 

\begin{itemize}
\item if $\delta_{\vec{i}} = \rightarrow$, and $\delta_{\vec{i}+(1,0)} = \downarrow$, 
then $\varphi_{\delta,k} (n+1) = \vec{i} + (1,-1)$
(the curve is shifted downwards in this column).
\item else $\delta_{\vec{i}} = \rightarrow$ and 
$\delta_{\vec{i} + (1,0)} = \rightarrow$, then $\varphi_{\delta,k} (n+1) = \vec{i} + (1,0)$.
\end{itemize}

The first rule implies that all the curves of this pseudo-covering can not be shifted downwards multiple times 
in the same column. The second one implies that 
if a curve is shifted downwards at position $\vec{i} \in \Z^2$, then
there is no curve going through position $\vec{i} - (1,1)$. \bigskip

Figure~\ref{fig.curves} gives an illustration of the definition of a curve in a configuration of $\Delta$.
 
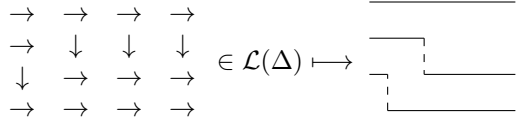
\begin{figure}[ht]

\[\begin{tikzpicture}[scale=0.8]
\draw (2.4,0.3) -- (0.3,0.3);
\draw[dashed] (0.3,0.3) -- (0.3,0.9);
\draw (0.3,0.9) -- (0,0.9);
\draw (0.9,0.9) -- (2.4,0.9);
\draw[dashed] (0.9,1.5) -- (0.9,0.9);
\draw (0.9,1.5) -- (0,1.5);
\draw (2.4,2.1) -- (0,2.1);

\draw(-0.1,1.1) node[left]{
$\begin{array}{cccc}
\rightarrow & \rightarrow & \rightarrow & \rightarrow \\
\rightarrow & \downarrow & \downarrow & \downarrow \\
\downarrow & \rightarrow & \rightarrow & \rightarrow \\
\rightarrow & \rightarrow & \rightarrow & \rightarrow 
\end{array}\in\Lang(\Delta)\longmapsto$
};

\end{tikzpicture}\]
\caption{\label{fig.curves} An example of an 
admissible pattern in $\Delta$ 
and the curves going through it.}
\end{figure}

\subsubsection{Distortion operators on subshifts of finite type}
Let $\A$ be some alphabet.
Denote $S_{\A}$ the set of SFT over $\A$. 
We introduce operators 
$d_{\A} : {\mathcal{S}}_{\A} \rightarrow {\mathcal{S}}_{\tilde{\A}}$, with
$\tilde{\A}= \left(\A \cup \{\begin{tikzpicture} \draw (0,0) rectangle (0.3,0.3);
\end{tikzpicture}\} \right) \times \{\rightarrow,\downarrow\}$.

\paragraph{Pseudo-projection}

Consider the subshift $\Delta_{\A} \subset \left( \A \cup \{\begin{tikzpicture} \draw (0,0) rectangle (0.3,0.3);
\end{tikzpicture}\}\right)^{\Z^2} \times \Delta$, 
where the forbidden patterns are the ones 
defining $\Delta$ and the patterns where a symbol in $\A$ is superimposed 
to a $\downarrow$ symbol or where $\begin{tikzpicture} \draw (0,0) rectangle (0.3,0.3);
\end{tikzpicture}$ is superimposed to a $\rightarrow$ symbol. \bigskip

Define a \textbf{pseudo-projection} ${\mathcal{P}} : 
\Delta_{\A} \rightarrow {\A}^{\Z^2}$, as follows: for $(y,\delta) \in \Delta_{\A}$, 
\[({\mathcal{P}}(y,\delta))_{i,j} = y_{\varphi_{\delta,j} (i)}.\] 
Notice that the function ${\mathcal{P}}$ is continuous but not shift invariant.

We denote $\pi_1$ the projection on the first layer 
($\pi_1 (y,\delta)=y$), and $\pi_2$ the projection on the second layer. \bigskip

\paragraph{Definition of the operators} 

Let $X$ be some SFT on the alphabet $\A$, and define $d_{\A} (X)= \mathcal{P} ^{-1} (X)$. 
Denoting $r$ the rank of the SFT $X$, $d_{\A} (X)$ can be defined 
by imposing that, considering the intersection of a set of $r$ contiguous curves with 
$r$ consecutive columns, the corresponding $r$-block is not a forbidden pattern in $X$. 
Because the gap between two contiguous curves is bounded, 
$d_{\A}(X)$ is defined by a finite set of forbidden patterns. Then this is a SFT. \bigskip

One can think to $d_{\mathcal{A}} (X)$ 
as having two layers. The first one has 
alphabet $\left(\A \cup \{\begin{tikzpicture} \draw (0,0) rectangle (0.3,0.3);
\end{tikzpicture}\} \right)$ and called 
the $X$ layer. The second one has alphabet 
$\{\rightarrow, \downarrow\}$ and is called the $\Delta$ layer.
For $X$ is the subshift $X_R$ or $X_{adR}$, 
Figure~\ref{fig.preprojrob} shows an example of pattern in the $X$ layer of the subshift 
$d_{\A} (X)$ whose pseudo-projection 
is the supertile south west order two supertile.

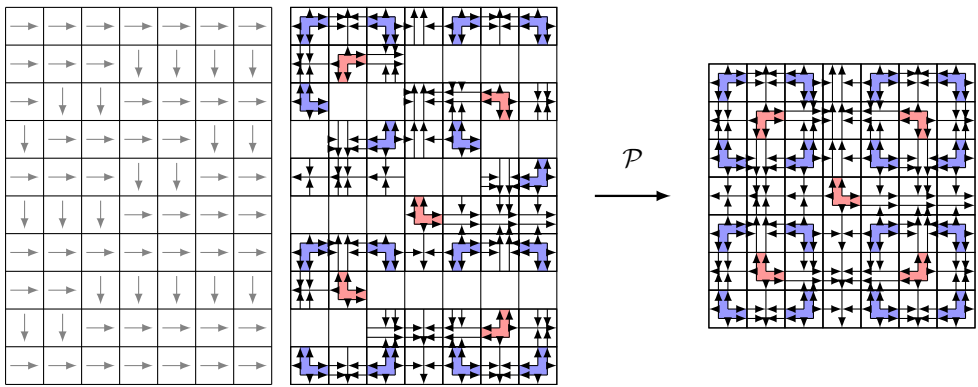
\begin{figure}[ht]
\[\begin{tikzpicture}[scale=0.25]
\robibluebastgauche{0}{-6};
\robibluebastgauche{8}{-6};
\robibluebastgauche{0}{8};
\robibluebastgauche{8}{6};
\robiredbasgauche{2}{-2};
\robiredbasgauche{6}{2};
\robibluehauttgauche{0}{0};
\robibluehauttgauche{8}{0};
\robibluehauttgauche{0}{12};
\robibluehauttgauche{8}{12};
\robiredhautgauche{2}{10};
\robibluebastdroite{4}{-6};
\robibluebastdroite{12}{-6};
\robibluebastdroite{4}{6};
\robibluebastdroite{12}{4};
\robiredbasdroite{10}{-4};
\robibluehauttdroite{4}{0};
\robibluehauttdroite{12}{0};
\robibluehauttdroite{4}{12};
\robibluehauttdroite{12}{12};
\robiredhautdroite{10}{8};
\robitwobas{2}{-6}
\robitwobas{10}{-6}
\robitwobas{6}{-4}
\robitwogauche{0}{-2}
\robitwogauche{2}{4}
\robitwogauche{0}{10}
\robitwodroite{12}{-4}
\robitwodroite{12}{8}
\robitwohaut{2}{12}
\robitwohaut{10}{12}
\robionehaut{6}{-6}
\robionehaut{6}{0}
\robionegauche{0}{4}
\robionegauche{4}{4}
\robisixbas{2}{6}
\robisixdroite{4}{-4}
\robisixhaut{2}{0}
\robisevendroite{4}{10}
\robithreehaut{6}{6}
\robisixhaut{6}{8}
\robithreehaut{6}{12}
\robisixgauche{8}{-4}
\robisevengauche{8}{8}
\robisevenbas{10}{4}
\robisevenhaut{10}{0}
\robisixdroite{10}{2}
\robithreedroite{8}{2}
\robithreedroite{12}{2}
\draw (0,-6) rectangle (14,14);
\foreach \x in {-2,-1,0,1,...,6} \draw (0,2*\x) -- (14,2*\x);
\foreach \x in {1,...,6} \draw (2*\x,-6) -- (2*\x,14);
\foreach \x in {0,...,7} \draw (-1-2*\x,-6) -- (-1-2*\x,14);
\foreach \x in {-3,...,7} \draw (-1,2*\x) -- (-15,2*\x);

\begin{scope}[xshift= 22cm,yshift=-3cm]
\robibluebastgauche{0}{0};
\robibluebastgauche{8}{0};
\robibluebastgauche{0}{8};
\robibluebastgauche{8}{8};
\robiredbasgauche{2}{2};
\robiredbasgauche{6}{6};
\robibluehauttgauche{0}{4};
\robibluehauttgauche{8}{4};
\robibluehauttgauche{0}{12};
\robibluehauttgauche{8}{12};
\robiredhautgauche{2}{10};
\robibluebastdroite{4}{0};
\robibluebastdroite{12}{0};
\robibluebastdroite{4}{8};
\robibluebastdroite{12}{8};
\robiredbasdroite{10}{2};
\robibluehauttdroite{4}{4};
\robibluehauttdroite{12}{4};
\robibluehauttdroite{4}{12};
\robibluehauttdroite{12}{12};
\robiredhautdroite{10}{10};
\robitwobas{2}{0}
\robitwobas{10}{0}
\robitwobas{6}{2}
\robitwogauche{0}{2}
\robitwogauche{2}{6}
\robitwogauche{0}{10}
\robitwodroite{12}{2}
\robitwodroite{12}{10}
\robitwohaut{2}{12}
\robitwohaut{10}{12}
\robionehaut{6}{0}
\robionehaut{6}{4}
\robionegauche{0}{6}
\robionegauche{4}{6}
\robisixbas{2}{8}
\robisixdroite{4}{2}
\robisixhaut{2}{4}
\robisevendroite{4}{10}
\robithreehaut{6}{8}
\robisixhaut{6}{10}
\robithreehaut{6}{12}
\robisixgauche{8}{2}
\robisevengauche{8}{10}
\robisevenbas{10}{8}
\robisevenhaut{10}{4}
\robisixdroite{10}{6}
\robithreedroite{8}{6}
\robithreedroite{12}{6}
\draw (0,0) rectangle (14,14);
\foreach \x in {1,...,6} \draw (0,2*\x) -- (14,2*\x);
\foreach \x in {1,...,6} \draw (2*\x,0) -- (2*\x,14);
\end{scope}

\draw[gray!97,-latex] (-2.75,13) -- (-1.25,13);
\draw[gray!97,-latex] (-4.75,13) -- (-3.25,13);
\draw[gray!97,-latex] (-6.75,13) -- (-5.25,13);
\draw[gray!97,-latex] (-8.75,13) -- (-7.25,13);
\draw[gray!97,-latex] (-10.75,13) -- (-9.25,13);
\draw[gray!97,-latex] (-12.75,13) -- (-11.25,13);
\draw[gray!97,-latex] (-14.75,13) -- (-13.25,13);

\draw[gray!97,-latex] (-10.75,11) -- (-9.25,11);
\draw[gray!97,-latex] (-12.75,11) -- (-11.25,11);
\draw[gray!97,-latex] (-14.75,11) -- (-13.25,11);
\draw[gray!97,-latex] (-2,11.75) -- (-2,10.25);
\draw[gray!97,-latex] (-4,11.75) -- (-4,10.25);
\draw[gray!97,-latex] (-6,11.75) -- (-6,10.25);
\draw[gray!97,-latex] (-8,11.75) -- (-8,10.25);

\draw[gray!97,-latex] (-14.75,9) -- (-13.25,9);
\draw[gray!97,-latex] (-2.75,9) -- (-1.25,9);
\draw[gray!97,-latex] (-4.75,9) -- (-3.25,9);
\draw[gray!97,-latex] (-6.75,9) -- (-5.25,9);
\draw[gray!97,-latex] (-8.75,9) -- (-7.25,9);
\draw[gray!97,-latex] (-10,9.75) -- (-10,8.25);
\draw[gray!97,-latex] (-12,9.75) -- (-12,8.25);

\draw[gray!97,-latex] (-6.75,7) -- (-5.25,7);
\draw[gray!97,-latex] (-8.75,7) -- (-7.25,7);
\draw[gray!97,-latex] (-10.75,7) -- (-9.25,7);
\draw[gray!97,-latex] (-12.75,7) -- (-11.25,7);
\draw[gray!97,-latex] (-14,7.75) -- (-14,6.25);
\draw[gray!97,-latex] (-2,7.75) -- (-2,6.25);
\draw[gray!97,-latex] (-4,7.75) -- (-4,6.25);

\draw[gray!97,-latex] (-2.75,5) -- (-1.25,5);
\draw[gray!97,-latex] (-4.75,5) -- (-3.25,5);
\draw[gray!97,-latex] (-10.75,5) -- (-9.25,5);
\draw[gray!97,-latex] (-12.75,5) -- (-11.25,5);
\draw[gray!97,-latex] (-14.75,5) -- (-13.25,5);
\draw[gray!97,-latex] (-6,5.75) -- (-6,4.25);
\draw[gray!97,-latex] (-8,5.75) -- (-8,4.25);

\draw[gray!97,-latex] (-2.75,3) -- (-1.25,3);
\draw[gray!97,-latex] (-4.75,3) -- (-3.25,3);
\draw[gray!97,-latex] (-6.75,3) -- (-5.25,3);
\draw[gray!97,-latex] (-8.75,3) -- (-7.25,3);
\draw[gray!97,-latex] (-14,3.75) -- (-14,2.25);
\draw[gray!97,-latex] (-12,3.75) -- (-12,2.25);
\draw[gray!97,-latex] (-10,3.75) -- (-10,2.25);

\draw[gray!97,-latex] (-2.75,1) -- (-1.25,1);
\draw[gray!97,-latex] (-4.75,1) -- (-3.25,1);
\draw[gray!97,-latex] (-6.75,1) -- (-5.25,1);
\draw[gray!97,-latex] (-8.75,1) -- (-7.25,1);
\draw[gray!97,-latex] (-10.75,1) -- (-9.25,1);
\draw[gray!97,-latex] (-12.75,1) -- (-11.25,1);
\draw[gray!97,-latex] (-14.75,1) -- (-13.25,1);

\draw[gray!97,-latex] (-12.75,-1) -- (-11.25,-1);
\draw[gray!97,-latex] (-14.75,-1) -- (-13.25,-1);
\draw[gray!97,-latex] (-10,-0.25) -- (-10,-1.75);
\draw[gray!97,-latex] (-8,-0.25) -- (-8,-1.75);
\draw[gray!97,-latex] (-6,-0.25) -- (-6,-1.75);
\draw[gray!97,-latex] (-4,-0.25) -- (-4,-1.75);
\draw[gray!97,-latex] (-2,-0.25) -- (-2,-1.75);

\draw[gray!97,-latex] (-2.75,-3) -- (-1.25,-3);
\draw[gray!97,-latex] (-4.75,-3) -- (-3.25,-3);
\draw[gray!97,-latex] (-6.75,-3) -- (-5.25,-3);
\draw[gray!97,-latex] (-8.75,-3) -- (-7.25,-3);
\draw[gray!97,-latex] (-10.75,-3) -- (-9.25,-3);
\draw[gray!97,-latex] (-14,-2.25) -- (-14,-3.75);
\draw[gray!97,-latex] (-12,-2.25) -- (-12,-3.75);

\draw[gray!97,-latex] (-2.75,-5) -- (-1.25,-5);
\draw[gray!97,-latex] (-4.75,-5) -- (-3.25,-5);
\draw[gray!97,-latex] (-6.75,-5) -- (-5.25,-5);
\draw[gray!97,-latex] (-8.75,-5) -- (-7.25,-5);
\draw[gray!97,-latex] (-10.75,-5) -- (-9.25,-5);
\draw[gray!97,-latex] (-12.75,-5) -- (-11.25,-5);
\draw[gray!97,-latex] (-14.75,-5) -- (-13.25,-5);

\draw[-latex,line width =0.3mm] (16,4) -- (20,4);
\node at (18,6) {$\mathcal{P}$};
\end{tikzpicture}\]

\caption{\label{fig.preprojrob} Example of a pattern 
which is sent to the south west two order supertile 
by pseudo-projection. }
\end{figure}

\paragraph{Properties of the operators 
\texorpdfstring{$d_{\A}$}{dA}}

We use the following properties of the operators $d_{\A}$ 
in order to prove 
Theorem~\ref{thm.aperiodicity}.

\begin{proposition} \label{prop.8} For an aperiodic 
SFT $X$ on the alphabet $\A$, 
$d_{\A}(X)$ is also aperiodic.
\end{proposition} 

\noindent \textbf{Idea:} \textit{the main 
argument of this proof is that 
if a configuration in $d_{\mathcal{A}}(X)$ 
is periodic, then the projection on 
the $\Delta$ layer is periodic. This means that
although there is a distortion of the configuration 
in $X$, the distortion is done in a periodic 
way. From this, we deduce that the 
pseudo-projection on $X$ of this configuration 
is periodic.}

\begin{figure}[ht] 
\begin{center}
\begin{tikzpicture} 
\draw (0,-0.3) -- (0,5.1) ;
\draw (0.6,-0.3) -- (0.6,5.1);
\foreach \x in {1,2,4,5,7,8} \draw (0,\x*0.6) -- (0.6,\x*0.6);
\foreach \x in {0,3,6} \draw[red!] (0,\x*0.6) -- (0.6,\x*0.6);
\draw (4.8,-0.3) -- (4.8,5.1) ;
\draw (5.4,-0.3) -- (5.4,5.1);
\foreach \x in {1,2,4,5,7,8} \draw (4.8,\x*0.6) -- (5.4,\x*0.6);
\foreach \x in {0,3,6} \draw[red!] (4.8,\x*0.6) -- (5.4,\x*0.6);
\node at (-1.5,1.8) {$\omega^0_i$};
\node at (-1.5,3.6) {$\omega^0_j$};

\node at (6.9,0) {$\omega^{c}_i$};
\node at (6.9,1.8) {$\omega^{c}_j$};
\draw[<->] (0,5.4) -- (4.8,5.4);
\node at (2.4,5.7) {c};
\draw[dashed] (0.6,3.6) -- (2.4,3.6);
\draw[dashed] (2.4,3.6) -- (2.4,3);
\draw[dashed] (2.4,3) -- (3,3);
\draw[dashed] (3,3) -- (3,2.4);
\draw[dashed] (3,2.4) -- (3.6,2.4);
\draw[dashed] (3.6,2.4) -- (3.6,1.8);
\draw[dashed,-latex] (3.6,1.8) -- (4.8,1.8);
\draw[dashed] (0.6,1.8) -- (1.2,1.8);
\draw[dashed] (1.2,1.8) -- (1.2,1.2);
\draw[dashed] (1.2,1.2) -- (1.8,1.2);
\draw[dashed] (1.8,1.2) -- (1.8,0.6);
\draw[dashed] (1.8,0.6) -- (2.4,0.6);
\draw[dashed] (2.4,0.6) -- (2.4,0);
\draw[dashed,-latex] (2.4,0) -- (4.8,0);
\end{tikzpicture}
\end{center}
\caption{\label{fig.aperiodicity} Schema of the proof of Proposition~\ref{prop.8}.}
\end{figure}
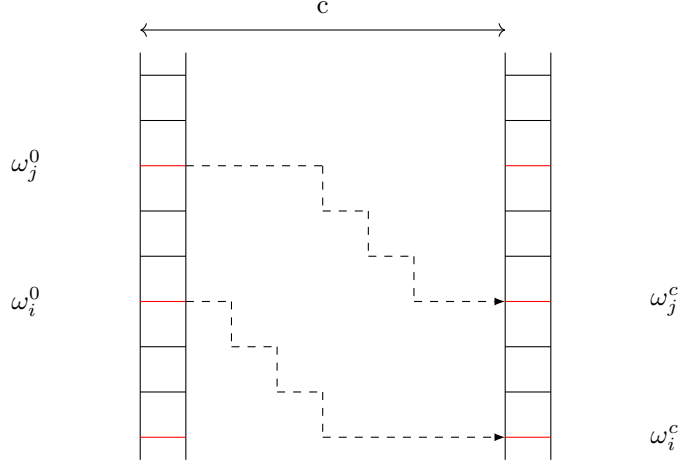

\begin{proof}

Assume that there exists a 
configuration $z \in d_{\A}(X)$ which 
is periodic: there exists $n>0$ such that 
for all $i,j$, $z_{i+n,j}=z_{i,j+n}=z_{i,j}$.
We will prove that the pseudo-projection of $z$ on $X$, 
$x = {\mathcal{P}}(z)$ is periodic. 

\begin{enumerate}

\item \textbf{Coding 
the positions of intersections 
of the curves with a column:}

To each column $k$ in $z$ we associate the bi-infinite word $\omega^k$ in $(\Z / n \Z)^{\Z}$ 
such that for all $i \in \Z$, $\omega^k_i$ is the 
element $\overline{m_i}$ 
of $\Z/n\Z$, class modulo $n$ of $m_i$ where $(0,m_i)$ is the 
intersection position of the $i$th curve of 
$\pi_2 (z)$ with the column k.

\item \textbf{Function relating 
the codings of two columns:}

Following a curve (see Figure~\ref{fig.aperiodicity}) from the column 0 to the column $n$, we get an application $\psi$
from the set of possible $\overline{m_i}$ into itself. 
This 
comes from the vertical periodicity of the projection of $z$ on the second layer. 
The word $\omega^n$ is obtained from $\omega^0$ applying $\psi$ to all 
the letters in $\omega^n$. 

\item \textbf{Coding
of the intersections and periodicity:}

Since $\psi$ is an invertible
function from a finite set into itself (indeed, we have an inverse map following the curve backwards), 
there exists some $c>0$ integer such that 
$\psi ^c = \id$. As a consequence, 
$\omega^{nc+j} = \omega^j$ for all integers $j$. 
That means that the column $cn+j$ is obtained by shifting 
$kn$ times downwards the column $j$, for some $k\ge 0$. Using the horizontal periodicity of $y$, we then have that \[(x_{j,l})_{l\in \Z} = (x_{cn+j,l+kn})_{l \in \Z}.\] 
Using the vertical periodicity, that $(x_{j,l})_{l\in \Z} = (x_{cn+j,l})_{l \in \Z}$, hence 
the configuration $x \in X$ is periodic, which can not be true.
\end{enumerate}

As a consequence, no configuration in $d_{\A}(X)$ can be periodic. Thus this subshift is aperiodic.
\end{proof}

The following proposition will be a 
useful tool in order to prove that 
the operators $d_{\A}$ transform linearly 
net gluing subshifts into block gluing ones.

\begin{proposition}[Completing blocks]

\label{prop.completing.block.aperiodicity}
There exists an algorithm ${\mathcal{T}}$ that, taking as input some locally admissible 
$n$-block $p$ of $\Delta$, outputs 
a rectangular pattern ${\mathcal{T}}(p)$ which has $p$ as a sub-pattern and such that:
\begin{itemize}
\item the number of curves in ${\mathcal{T}}(p)$ is equal to the number of its columns,
\item the dimensions of 
${\mathcal{T}}(p)$ are smaller than $5n$,
\item the top and bottom rows of ${\mathcal{T}}(p)$ have only $\rightarrow$ symbols- 
this means that all the curves crossing ${\mathcal{T}}(p)$ comes from 
its left side and go to the right side.
\end{itemize}
\end{proposition}

\begin{remark} The properties of the pattern 
${\mathcal{T}}(p)$ ensure that 
this is a globally admissible pattern. Hence every locally admissible pattern of 
the subshift $\Delta$ is globally admissible.
\end{remark}

\noindent \textbf{Idea:} \textit{the proof 
consists in extending the curves 
that cross a pattern from above and below.
Then we add curves on the top and bottom 
that are straighter and straighter.}

\begin{proof} If $p$ is a $1$-block, and $p$ is a 
single $\rightarrow$, then the result is direct. If $p$ is a single $\downarrow$, 
then it can be extended in 
\[\begin{array}{ccc} \rightarrow & \rightarrow & \rightarrow \\
 \rightarrow & \downarrow & \downarrow \\
\downarrow & \rightarrow & \rightarrow \\
\rightarrow & \rightarrow & \rightarrow \end{array}, \]
which verifies the previous assertion. \bigskip

If $p$ is a $n$-block with $n \ge 2$: \bigskip

\textbf{First step. Extending the curves that enter in the block upside/downside:} 

\begin{enumerate} \item If
in the top row of the block $p$ 
there is the pattern $\downarrow \rightarrow$ -
this means that there 
is an incoming curve 
(the position of the $\rightarrow$ is 
in this curve) - then we 
add symbols in the row just above.
We extend each of 
the incoming curves, considering 
the curves from left to right, in this row. 
For this purpose we add a $\downarrow$ over 
the $\rightarrow$ for each 
of the patterns $\downarrow \rightarrow$. 
Then we add $\rightarrow$ symbols 
on the left of this one 
until meeting another $\downarrow$ or 
the left side of the block.
If in the added row there are 
$\downarrow \rightarrow$ patterns, 
then return to the beginning 
of this step. Else, stop.

\item Do similar operations on the bottom of the block.\end{enumerate} \bigskip

Since the number of $\downarrow \rightarrow$
patterns in the top row is strictly 
decreasing, this series of operations 
stops at some point.

\begin{example} If we take $p$ the following $4$-block 
\[\begin{array}{ccccccc}
\rightarrow & \rightarrow & \downarrow & \rightarrow \\
\rightarrow & \downarrow & \rightarrow & \rightarrow \\
\downarrow & \rightarrow & \rightarrow & \rightarrow \\
\rightarrow & \rightarrow & \rightarrow & \rightarrow 
\end{array}\]
at this point, we obtain: 

\[\begin{array}{ccccccc}
\rightarrow & \rightarrow & \rightarrow & \downarrow  \\
\rightarrow & \rightarrow & \downarrow & \rightarrow \\
\rightarrow & \downarrow & \rightarrow & \rightarrow \\
\downarrow & \rightarrow & \rightarrow & \rightarrow \\
\rightarrow & \rightarrow & \rightarrow & \rightarrow
\end{array}\]
\end{example} \bigskip

\textbf{Second step. Completing the pattern on the top and bottom until the top row and 
bottom row are straight:} \bigskip

While the top curve of the pattern is not straight, apply the following procedure: 
\begin{enumerate}
\item On the top of the last column, add a $\downarrow$ and keep adding $\downarrow$ on the left until meeting 
on the left an already defined symbol (there 
can be such symbols, introduced in 
the first step) or 
the left extremity.
\item Add another curve above by the following procedure. Add a $\rightarrow$ on the top of the last column, and then 
add $\rightarrow$ symbols on the left until meeting an already defined symbol on the left. When that happens, 
add a $\downarrow$ above and then add $\rightarrow$'s on the left until reaching a defined symbol. Repeat this operation 
until reaching the first column.
\end{enumerate} \bigskip

Since the number of times that the top 
curve is shifted downwards decreases at each step, 
this series of operations stops. \bigskip

Do similar operations on the bottom. \bigskip

\begin{example}
At this point, we obtain: 

\[\begin{array}{ccccccc}
\rightarrow & \rightarrow & \rightarrow & \rightarrow \\
\rightarrow & \rightarrow & \rightarrow & \downarrow \\
\rightarrow & \rightarrow & \downarrow & \rightarrow  \\
\rightarrow & \downarrow & \rightarrow & \rightarrow \\
\downarrow & \rightarrow & \rightarrow & \rightarrow \\
\rightarrow & \rightarrow & \rightarrow & \rightarrow 
\end{array}\]
\end{example} \bigskip

\textbf{Third step. Equalization of the number of curves and the number of columns:} \bigskip

If the number of columns is smaller than the number of curves, then add a number of columns equal to the difference, 
by adding copies of the last column 
on its right side. If the number of curves is smaller, then add lines of $\rightarrow$ symbols 
on the top. \bigskip

\begin{example} After this last step we obtain: \[\begin{array}{cccccc}
\rightarrow & \rightarrow & \rightarrow & \rightarrow & \rightarrow \\
\rightarrow & \rightarrow & \rightarrow & \downarrow & \downarrow \\
\rightarrow & \rightarrow & \downarrow & \rightarrow & \rightarrow \\
\rightarrow & \downarrow & \rightarrow & \rightarrow & \rightarrow \\
\downarrow & \rightarrow & \rightarrow & \rightarrow & \rightarrow \\
\rightarrow & \rightarrow & \rightarrow & \rightarrow & \rightarrow 
\end{array}\]\end{example}

For $p$ some $n$-block, 
the dimensions of ${\mathcal{T}}(p)$ 
are smaller than the sum of:
\begin{enumerate}
\item the dimension of $p$ (equal to $n$)
\item  two times 
the number of entering curves by the top and outgoing by 
the bottom (one for completing the curves (first step)
\item 
and one for 
reducing the shifts (second step)).
\end{enumerate} 
Each one of these numbers is 
smaller than $n$. The third step does not 
make this bound greater, 
because in this pattern the number of curves is smaller than the number of lines. 
As a consequence, the dimensions of ${\mathcal{T}}(p)$ 
are smaller than $5n$. \end{proof}

\begin{proposition} \label{proposition.transfo.gluing.sets}
Let $X$ be some linearly 
net gluing subshift on 
alphabet $\A$. There exists some vector function 
$\vec{u}'$ and a function $g$ that:
\begin{enumerate}
\item take as arguments two $n$-blocks 
$p,q$ for some $n$, 
\item respectively 
associates to these blocks 
an element of $\Z^2$ and an element of $\N$
\end{enumerate}
These functions verify that 
for any couple of $n$-blocks 
$p,q$ in the language of $d_{\mathcal{A}}(X)$,
the gluing set of $p$ relative to $q$ 
contains infinite columns regularly 
displayed.
Moreover, the gluing set contains 
regularly displayed 
positions in a central column:

\[ 
g(p,q) (\Z\backslash \{0\})\vec{e}^2 + \vec{u}'(p,q) \subset \Delta_{d_{\A}(X)}(p,q),\]
\[ g(p,q) (\Z \backslash \{-3,...,3\}) \vec{e}^1 + \Z \vec{e}^2 +
\vec{u}'(p,q) \subset \Delta_{d_{\A}(X)}(p,q),\]
where function $g$ verifies that 
\[\max_{p,q} g(p,q) = O(n),\]
where the maximum is over the $n$-blocks.
\end{proposition}

\noindent \textbf{Idea:} \textit{this proof 
consists in analyzing how the operator 
acts on the gluing set of a $n$-block $p$ 
relative to another pattern $q$. The operator 
allows perturbations to be introduced on 
these sets.}

\begin{proof}

Let $X$ be some $f$ net 
gluing subshift on alphabet $\mathcal{A}$, 
where $f(n) = O(n)$.

\paragraph{Formulation of the linear net 
gluing of $X$:} \bigskip

This means that there exist 
two function $\vec{u} : \mathcal{L}_n (X) ^ 2 
\rightarrow \Z^2$ and $\tilde{f} : \mathcal{L}_n (X) ^ 2 
\rightarrow \N$ such that  
for all $n >0$ and every couple 
of $n$-blocks $r,s$ in the language of $X$, the gluing 
set of $r$ relative to $s$ in $X$ contains 
\[\vec{u} (r,s)
+(n+\tilde{f}(r,s))(\Z^2 -(0,0)),\] 
and for all $r,s$ $n$-blocks,
\[\tilde{f}(r,s) \le f(n).\] 

We consider in this proof that 
$\vec{u}=\vec{0}$, since this proof 
can be adapted to a general function $\vec{u}$ 
without difficulty.

\paragraph{Sufficient conditions to verify:} \bigskip

It is sufficient to prove the statement 
of the proposition for 
patterns whose projection on 
$\Delta$ are $\mathcal{T}(\pi_2(p))$ 
and $\mathcal{T}(\pi_2(q))$. Indeed,
the size of these patterns is bounded by 
a linear function of the size of the patterns $p,q
\in d_{\mathcal{A}}(X)$.

Let $p,q$ two $n$-blocks in the language 
of the subshift $d_{\A}(X)$. 
Without loss of generality, 
we can consider that the 
patterns ${\mathcal{T}}(\pi_2(q))$ 
and ${\mathcal{T}}(\pi_2(p))$ 
have the same number of curves
crossing them.
We denote 
$\tilde{p}$ and $\tilde{q}$ some 
admissible patterns whose projections 
of $\Delta$ are respectively 
 ${\mathcal{T}}(\pi_2(p))$ 
and ${\mathcal{T}}(\pi_2(q))$.
The pseudo-projections 
of these patterns on $X$ are $m$-block of $X$, 
where $m$ is the number of curves 
in ${\mathcal{T}}(\pi_2(p))$ and 
${\mathcal{T}}(\pi_2(q))$. 
This is due to the 
fact that the top and bottom 
rows of these patterns 
are straight. These pseudo-projections 
are denoted $\mathcal{P} (\tilde{p})$ and 
$\mathcal{P} (\tilde{q})$, according 
to previous notations.

We place the pattern 
$\mathcal{P}(\tilde{q})$
on position $(0,0)$.

\paragraph{The gluing sets of $d_{\A}(X)$ contain 
infinite columns periodically displayed:} \bigskip

Let us show that the 
gluing set of $\tilde{p}$ 
relative 
to $\tilde{q}$ 
contains infinite columns periodically 
displayed.

Let $k \ge 4$ and $l$ integers such 
that $l \neq 0$, and consider some vector
\[\vec{u} = (m+\tilde{f}(\mathcal{P} (\tilde{p}),
\mathcal{P} (\tilde{q})) (k,l).\]

We will prove the following: 

\begin{itemize}
\item When $l \ge 2$ and $t$ 
is any integer such that 
\[0 \le t \le \tilde{f}(\mathcal{P} (\tilde{p}),\mathcal{P} (\tilde{q}))+m,\]
the 
position 
$\vec{u} +t.\vec{e}^2$ 
is in the gluing set in $d_{\A} (X)$ 
of the pattern $\tilde{p}$ relative 
to $\tilde{q}$.
\item When $l \le -2$, 
this set contains the position 
$\vec{u}-t.\vec{e}^2$, for the 
same integers $t$.
\item When $l =1$ or $l=-1$,
then this set contains
$\vec{u}-t.\vec{e}^2$, for all 
$t$ such that
\[0 \le t \le 2(\tilde{f}(\mathcal{P} (\tilde{p}),\mathcal{P} (\tilde{q}))+m).\] 
\end{itemize}

As a consequence, this gluing set contains 
the whole infinite column that contains
\[(m+\tilde{f}(\mathcal{P} (\tilde{p}),
\mathcal{P} (\tilde{q})) (k,0)\]
for all $k \ge 8$. Indeed, it contains 
an infinity of segments which overlap only 
on their border. 
We have the same property for $k \le -4$, 
by reversing $p$ and $q$.

\begin{enumerate}

\item \textbf{When $|l| \ge 2$:}

Let $t$ be some integer such that 
\[0 \le t \le \tilde{f}(\mathcal{P} 
(\tilde{p}),\mathcal{P} (\tilde{q}))+m.\]

In the case $l \ge 2$, 
we do the following operations. 
See a schema on 
Figure~\ref{fig.schema.branching.thm.0}.

\begin{enumerate}
\item We extend the curves crossing 
the pattern $\tilde{q}$ in 
a straight way until infinity.
\item In the case $l \ge 2$, 
we add straight curves below the 
obtained pattern. We do that in 
such a way that these curves have gap $0$
between them. On the top, 
we introduce $t$ times a straight infinite 
curve with gap $1$ with the curve below. 
Then we add \[l(\tilde{f}(\mathcal{P}(\tilde{p}),\mathcal{P} (\tilde{q}))+m)- t-m\]
straight infinite curves with gap $0$ with 
the curve below. 
\item Then we add 
the pattern $\tilde{p}$
on position $\vec{u}+t.\vec{e}^2$.
\item We extend the curves crossing 
this pattern in a straight way until infinity.
\item We add straight lines on the top, 
without gaps between them.
\item Then we color the curves with elements 
of the curves with elements of $\mathcal{A}$ 
such that the configuration is admissible. 
This is possible from the net gluing property of 
the subshift $X$. Indeed, the position 
of the pattern $\mathcal{P} (\tilde{p})$ 
relatively to $\mathcal{P} (\tilde{q})$ 
in the pseudo-projection of this configuration 
is $\vec{u}$. Moreover, this vector 
is in the gluing set of the first pattern 
relatively to the second one.
\end{enumerate}

The case $l \le -2$ is similar. The difference is 
that 
the pattern $\tilde{p}$ 
appears on position 
$\vec{u}-t.\vec{e}^2$.

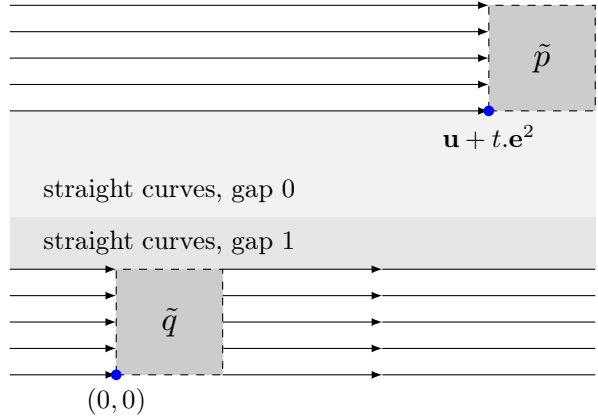
\begin{figure}[h]
\[\begin{tikzpicture}[scale=0.35]

\fill[gray!20] (-4,4) rectangle (18,6);
\fill[gray!10] (18,6) rectangle (-4,10);

\node at (14,9) {$\vec{u}+t.\vec{e}^2$};
\node at (2,7) {straight curves, gap $0$};
\node at (2,5) {straight curves, gap $1$};

\fill[gray!40] (0,0) rectangle (4,4);
\node[scale=1.25] at (2,2) {$\tilde{q}$};
\draw[dashed] (0,0) rectangle (4,4);
\draw[-latex] (4,4) -- (10,4);
\draw (10,4) -- (18,4);
\draw[-latex] (4,3) -- (10,3);
\draw (10,3) -- (18,3);
\draw[-latex] (4,2) -- (10,2);
\draw (10,2) -- (18,2);
\draw[-latex] (4,1) -- (10,1);
\draw (10,1) -- (18,1);
\draw[-latex] (4,0) -- (10,0);
\draw (10,0) -- (18,0);
\draw[latex-] (0,0) -- (-4,0);
\draw[latex-] (0,1) -- (-4,1);
\draw[latex-] (0,2) -- (-4,2);
\draw[latex-] (0,3) -- (-4,3);
\draw[latex-] (0,4) -- (-4,4);

\draw[-latex] (-4,10) -- (14,10);
\draw[-latex] (-4,11) -- (14,11);
\draw[-latex] (-4,12) -- (14,12);
\draw[-latex] (-4,13) -- (14,13);
\draw[-latex] (-4,14) -- (14,14);

\fill[gray!40] (14,10) rectangle (18,14);
\draw[dashed] (14,10) rectangle (18,14);
\node[scale=1.25] at (16,12) {$\tilde{p}$};

\fill[blue] (0,0) circle (0.2);
\fill[blue] (14,10) circle (0.2);
\node at (0,-1) {$(0,0)$};
\end{tikzpicture}\]
\caption{\label{fig.schema.branching.thm.0}
Illustration of the construction 
for the proof of 
Theorem~\ref{theorem.transformation.block.gluing} 
when $l \ge 2$.}
\end{figure}

\item \textbf{When $l=1$ or $-1$:}

Here we prove that the pattern $\tilde{p}$ 
can be glued relatively to $\tilde{q}$ on 
position $\vec{u}-t.\vec{e}^2$.

The steps of a construction 
of a configuration that supports this gluing 
are as follows:

\begin{enumerate}

\item \textbf{Compactification of the 
outgoing curves:} \bigskip

We extend ${\mathcal{T}}(\pi_2(q))$ using the 
following procedure.
While in the last column of the pattern there is some sub-pattern $\begin{array}{c} \rightarrow \\
\downarrow \end{array}$ (meaning that 
there is a gap between two outgoing curves), do the following: on the right of the 
patterns $\begin{array}{c} \rightarrow \\
\downarrow \end{array}$, write $\begin{array}{c} \downarrow \\
\rightarrow \end{array}$ and write a copy of the other $\rightarrow$ symbols on their right side.

\begin{example} Taking the same example as in 
the proof of 
Proposition~\ref{prop.completing.block.aperiodicity}, 
the result is: 
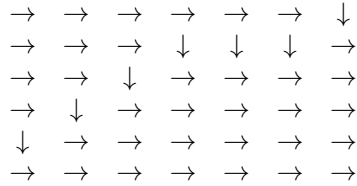
\[\begin{array}{ccccccc}
\rightarrow & \rightarrow & \rightarrow & \rightarrow & \rightarrow & \rightarrow & \downarrow \\
\rightarrow & \rightarrow & \rightarrow & \downarrow & \downarrow & \downarrow  & \rightarrow\\
\rightarrow & \rightarrow & \downarrow & \rightarrow & \rightarrow & \rightarrow & \rightarrow\\
\rightarrow & \downarrow & \rightarrow & \rightarrow & \rightarrow & \rightarrow & \rightarrow\\
\downarrow & \rightarrow & \rightarrow & \rightarrow & \rightarrow & \rightarrow & \rightarrow \\
\rightarrow & \rightarrow & \rightarrow & \rightarrow & \rightarrow & \rightarrow & \rightarrow 
\end{array}\] \end{example}

Since there are $m$ curves the number 
of additional columns on the right for this step is smaller 
than $m$. Indeed, one column is sufficient to 
reduce the gap between a curve and the curve 
just below.

\item \textbf{Making the curves shift:} \bigskip

We add columns on the right of 
the extension of ${\mathcal{T}}(\pi_{2} (q))$. We follow the 
following procedure, in order to make all the 
curves in it shift $t$ times: 
\begin{enumerate} 
\item Consider the right part of the pattern constituted with $\rightarrow$ symbols 
and add a triangle made of $\rightarrow$ symbols except on the diagonal part. On 
this part we write $\downarrow$ symbols (this 
is the first shift). 

\begin{example} Taking the same example as previously, the result is: 
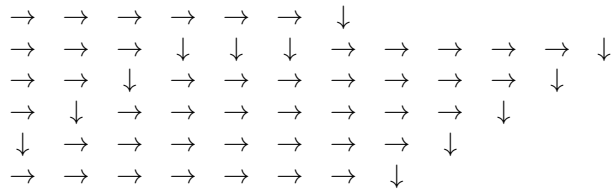
\[\begin{array}{cccccccccccc}
\rightarrow & \rightarrow & \rightarrow & \rightarrow & \rightarrow & \rightarrow & \downarrow & & & & & \\
\rightarrow & \rightarrow & \rightarrow & \downarrow & \downarrow & \downarrow  & \rightarrow & \rightarrow & \rightarrow& \rightarrow 
& \rightarrow & \downarrow \\
\rightarrow & \rightarrow & \downarrow & \rightarrow & \rightarrow & \rightarrow & \rightarrow & \rightarrow & \rightarrow & \rightarrow 
& \downarrow & \\
\rightarrow & \downarrow & \rightarrow & \rightarrow & \rightarrow & \rightarrow & \rightarrow & \rightarrow & \rightarrow & \downarrow & & \\
\downarrow & \rightarrow & \rightarrow & \rightarrow & \rightarrow & \rightarrow & \rightarrow & \rightarrow & \downarrow & & & \\
\rightarrow & \rightarrow & \rightarrow & \rightarrow & \rightarrow & \rightarrow & \rightarrow & \downarrow & & & &
\end{array}\] \end{example}

\item Then repeat $t-1$ times the following operation: add under each $\downarrow$ on the right side 
a $\rightarrow$ under, and after that a $\downarrow$ on the right of the $\rightarrow$. 

\begin{example} Taking the same example as previously, with $t=3$, the result is: 
\[\begin{array}{cccccccccccccc}
\rightarrow & \rightarrow & \rightarrow & \rightarrow & \rightarrow & \rightarrow & \downarrow & & & & & & & \\
\rightarrow & \rightarrow & \rightarrow & \downarrow & \downarrow & \downarrow  & \rightarrow & \rightarrow & \rightarrow& \rightarrow 
& \rightarrow & \downarrow & & \\
\rightarrow & \rightarrow & \downarrow & \rightarrow & \rightarrow & \rightarrow & \rightarrow & \rightarrow & \rightarrow & \rightarrow 
& \downarrow & \rightarrow & \downarrow & \\
\rightarrow & \downarrow & \rightarrow & \rightarrow & \rightarrow & \rightarrow & \rightarrow & \rightarrow & \rightarrow & \downarrow & \rightarrow & \downarrow & \rightarrow & \downarrow \\
\downarrow & \rightarrow & \rightarrow & \rightarrow & \rightarrow & \rightarrow & \rightarrow & \rightarrow & \downarrow & \rightarrow & \downarrow & \rightarrow & \downarrow &\rightarrow \\
\rightarrow & \rightarrow & \rightarrow & \rightarrow & \rightarrow & \rightarrow & \rightarrow & \downarrow & \rightarrow & \downarrow & \rightarrow & \downarrow & \rightarrow & \\
& & & & & & & \rightarrow & \downarrow & \rightarrow & \downarrow & \rightarrow & & \\
& & & & & & & & \rightarrow & \downarrow & \rightarrow & & & \\
& & & & & & & &  & \rightarrow &  & & & \\
\end{array}\] \end{example}

There are at most 
$2(m+\tilde{f}(\mathcal{P}(\tilde{p},\tilde{q})))$
additional columns for this step. Indeed, 
$t$ is smaller than this number and 
only $t$ columns are needed to shift a compact set of curves.

\item Complete the curves with $\rightarrow$ symbols so that they end in the last column added.

\begin{example} Taking the same example as previously, with $t=3$, the result is: 
\[\begin{array}{cccccccccccccc}
\rightarrow & \rightarrow & \rightarrow & \rightarrow & \rightarrow & \rightarrow & \downarrow & & & & & & & \\
\rightarrow & \rightarrow & \rightarrow & \downarrow & \downarrow & \downarrow  & \rightarrow & \rightarrow & \rightarrow& \rightarrow 
& \rightarrow & \downarrow & & \\
\rightarrow & \rightarrow & \downarrow & \rightarrow & \rightarrow & \rightarrow & \rightarrow & \rightarrow & \rightarrow & \rightarrow 
& \downarrow & \rightarrow & \downarrow & \\
\rightarrow & \downarrow & \rightarrow & \rightarrow & \rightarrow & \rightarrow & \rightarrow & \rightarrow & \rightarrow & \downarrow & \rightarrow & \downarrow & \rightarrow & \downarrow \\
\downarrow & \rightarrow & \rightarrow & \rightarrow & \rightarrow & \rightarrow & \rightarrow & \rightarrow & \downarrow & \rightarrow & \downarrow & \rightarrow & \downarrow &\rightarrow \\
\rightarrow & \rightarrow & \rightarrow & \rightarrow & \rightarrow & \rightarrow & \rightarrow & \downarrow & \rightarrow & \downarrow & \rightarrow & \downarrow & \rightarrow & \rightarrow \\
& & & & & & & \rightarrow & \downarrow & \rightarrow & \downarrow & \rightarrow & \rightarrow & \rightarrow \\
& & & & & & & & \rightarrow & \downarrow & \rightarrow & \rightarrow & \rightarrow & \rightarrow \\
& & & & & & & &  & \rightarrow & \rightarrow & \rightarrow & \rightarrow & \rightarrow \\
\end{array}\] \end{example}

Here, no column is added.

\item Then extend the curves straightly 
on a number of columns so 
that the total number of additional columns
is equal to ${\vec{u}}_1 - m$. This is possible since
the number of added columns at this point 
is smaller than $m+2(m+\tilde{f}(\mathcal{P}(\tilde{p}),
\mathcal{P}(\tilde{q})) \le {\vec{u}}_1$, 
since $k \ge 4$.
\end{enumerate}

\item \textbf{Extension:} 

Then, we extend these curves straightly until 
infinity on the east side and on the west side.

\item \textbf{Additional curves:}

We add $\tilde{f} (\mathcal{P}(\tilde{p}),
\mathcal{P}(\tilde{q}))$
curves on the top when $l=1$ (resp. 
on the bottom when $l=-1$) 
of the obtained pattern at 
this point, without gap between them. This means 
that, from left to right, when the curve just below 
is shifted downwards the curve is also shifted 
immediately after.

\item \textbf{Positioning the pattern 
${\mathcal{T}}(\pi_2 (p))$:}

\begin{enumerate}
\item In the last columns where symbols were added in 
the last step, we position the pattern 
${\mathcal{T}}(\pi_2 (p))$. We place it on the top when 
$l=1$ (resp. on the bottom when $l=-1$) of the last added 
curves.
\item After this we extend the curves 
on the west side straightly until infinity. 
\item On the east side, we extend the curve without 
introducing gaps. This step is possible
since the minimal value of $k$ is taken sufficiently 
large. This means that the shifts of the outgoing 
curves of $\tilde{q}$ do not affect the area where 
the pattern $\tilde{p}$ is supposed to be glued.
\item On the top and bottom of the obtained pattern 
we add curves without introducing any gap in such a way that
we fill $\Z^2$.
\item In the end we add $\mathcal{A}$ symbols over 
the curves, when not already determined. 
This is possible from 
the linear net gluing property of 
$X$. Indeed, there are $\tilde{f} (\mathcal{P}
(\tilde{p}),\mathcal{P} (\tilde{q}))$ added 
between the patterns $\tilde{p}$ and $\tilde{q}$ 
and the pattern $\tilde{p}$ is on a position 
in the column containing $\vec{u}$.
\end{enumerate}

See Figure~\ref{fig.schema.branching.thm.1} 
for an illustration of this case, 
when $l=1$.

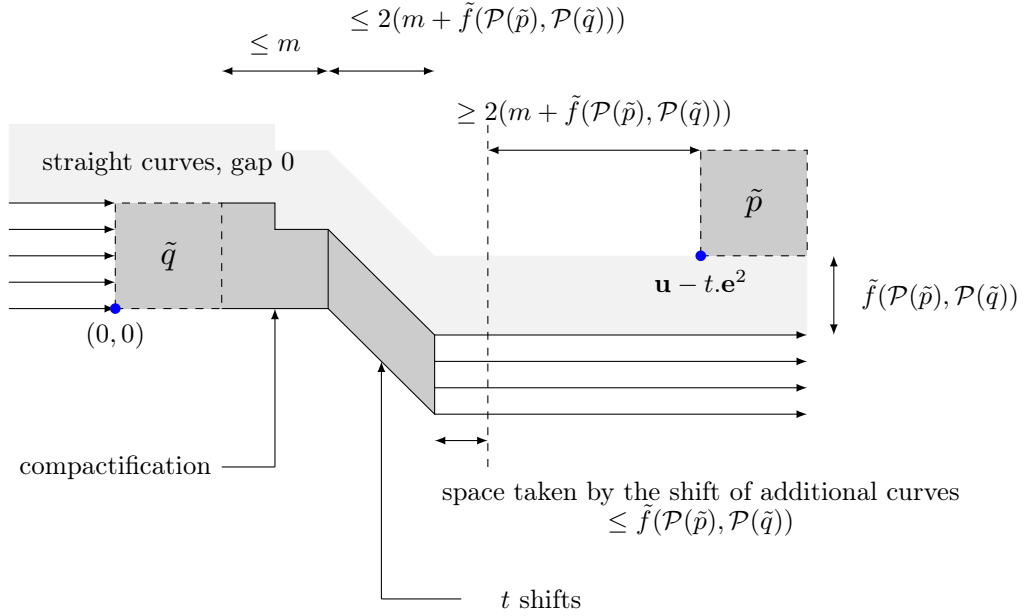
\begin{figure}[h]
\[\begin{tikzpicture}[scale=0.35]

\fill[gray!10] (-4,4) -- (6,4) -- (6,3) -- (8,3) 
-- (12,-1) -- (26,-1) -- (26,2) -- (12,2) -- (8,6) 
-- (6,6) -- (6,7) -- (-4,7) -- (-4,4) ; 

\fill[gray!40] (0,0) rectangle (4,4);
\fill[gray!40] (4,0) -- (8,0) -- (8,3) -- (6,3) -- (6,4) -- (4,4)
-- (4,0);
\fill[gray!40] (8,3) -- (12,-1) -- (12,-4) -- (8,0) -- (8,4);
\node[scale=1.25] at (2,2) {$\tilde{q}$};
\draw[dashed] (0,0) rectangle (4,4);

\draw[latex-] (0,0) -- (-4,0);
\draw[latex-] (0,1) -- (-4,1);
\draw[latex-] (0,2) -- (-4,2);
\draw[latex-] (0,3) -- (-4,3);
\draw[latex-] (0,4) -- (-4,4);
\draw[-latex] (12,-1) -- (26,-1);
\draw[-latex] (12,-2) -- (26,-2);
\draw[-latex] (12,-3) -- (26,-3);
\draw[-latex] (12,-4) -- (26,-4);

\draw (4,0) -- (8,0) -- (8,3) -- (6,3) -- (6,4) -- (4,4);

\draw (8,3) -- (12,-1) -- (12,-4) -- (8,0);

\fill[gray!40] (22,2) rectangle (26,6);
\draw[dashed] (22,2) rectangle (26,6);
\node[scale=1.25] at (24,4) {$\tilde{p}$};

\fill[blue] (0,0) circle (0.2);
\fill[blue] (22,2) circle (0.2);
\node at (22,1) {$\vec{u}-t.\vec{e}^2$};
\node at (0,-1) {$(0,0)$};
\node at (2,5.5) {straight curves, gap $0$};

\draw[-latex] (6,-6) -- (6,0);
\draw (6,-6) -- (4,-6); 
\node at (0,-6) {compactification};
\draw[-latex] (10,-11) -- (10,-2) ; 
\draw (10,-11) -- (13,-11);
\node at (16,-11) {$t$ shifts};
\draw[latex-latex] (4,9) --(8,9);
\node at (6,10) {$\le m$};
\draw[latex-latex] (8,9) --(12,9);
\node at (14,11) {$\le 2(m+\tilde{f} (\mathcal{P} (\tilde{p}),
\mathcal{P} (\tilde{q})))$};
\draw[dashed] (14,-6) -- (14,7);
\draw[latex-latex] (27,-1) -- (27,2);
\node at (31,0.5) {$\tilde{f} (\mathcal{P}(\tilde{p}),
\mathcal{P} (\tilde{q}))$};
\draw[latex-latex] (12,-5) -- (14,-5);
\node at (22,-7) {space taken by the shift of additional curves};
\node at (22,-8) 
{$\le \tilde{f} (\mathcal{P} (\tilde{p}),
\mathcal{P} (\tilde{q}))$};
\draw[latex-latex] (22,6) -- (14,6);
\node at (18,7.5) {$\ge 2(m+\tilde{f} (\mathcal{P} (\tilde{p}),
\mathcal{P} (\tilde{q})))$};
\end{tikzpicture}\]
\caption{\label{fig.schema.branching.thm.1}
schema of the construction 
for the proof of 
Theorem~\ref{theorem.transformation.block.gluing} 
when $l=1$.}
\end{figure}

\end{enumerate}
\end{enumerate}

\begin{figure}[ht]
\[\begin{tikzpicture}[scale=0.5]
\draw (0.5,0) rectangle (1.5,1);
\draw node at (0,0.5) {$q$};
\draw (0,2) rectangle (1,3);
\draw node at (-0.5,2.5) {$p$};
\draw (0,-2) rectangle (1,-1);
\fill[pattern=north west lines, pattern color=blue] (0,2) rectangle (1,3);
\fill[pattern=north west lines, pattern color=blue] (0,-2) rectangle (1,-1);

\draw (0,2) rectangle (0.1,2.1);
\draw (0,-2) rectangle (0.1,-1.9);

\draw (-2,-2.5) -- (-2,3.5);
\draw (-1,-2.5) -- (-1,3.5);
\draw (-3,-2.5) -- (-3,3.5);
\draw (-4,-2.5) -- (-4,3.5);
\draw (2,-2.5) -- (2,3.5);
\draw (3,-2.5) -- (3,3.5);
\fill[pattern=north west lines, pattern color=blue] (-2,-2.5) rectangle (-1,3.5);
\fill[pattern=north west lines, pattern color=blue] (-4,-2.5) rectangle (-3,3.5);
\fill[pattern=north west lines, pattern color=blue] (2,-2.5) rectangle (3,3.5);
\draw[dashed] (-2,-3.5) -- (-2,-2.5);
\draw[dashed] (-2,3.5) -- (-2,4.5);
\draw[dashed] (-1,-3.5) -- (-1,-2.5);
\draw[dashed] (-1,3.5) -- (-1,4.5);
\draw[dashed] (-3,-3.5) -- (-3,-2.5);
\draw[dashed] (-3,3.5) -- (-3,4.5);
\draw[dashed] (-4,-3.5) -- (-4,-2.5);
\draw[dashed] (-4,3.5) -- (-4,4.5);
\draw[dashed] (2,-3.5) -- (2,-2.5);
\draw[dashed] (2,3.5) -- (2,4.5);
\draw[dashed] (3,-3.5) -- (3,-2.5);
\draw[dashed] (3,3.5) -- (3,4.5);
\draw[<->] (4.5,1) -- (4.5,2);
\draw[<->] (-3,-4) -- (-2,-4);
\draw node at (-2.5,-4.5) {$O(n)$};
\draw node at (5.5,1.5)     {$O(n)$};
\draw (-2,-2.5) rectangle (-1.9,3.5);
\draw (-4,-2.5) rectangle (-3.9,3.5);
\draw (2,-2.5) rectangle (2.1,3.5);
\end{tikzpicture}\]
\caption{\label{fig.gluingset.stripes} Schematic representation 
of a set of positions 
included in the gluing set of some couple of $n$-blocks 
in $d_{\A} (X)$.} 
\end{figure}
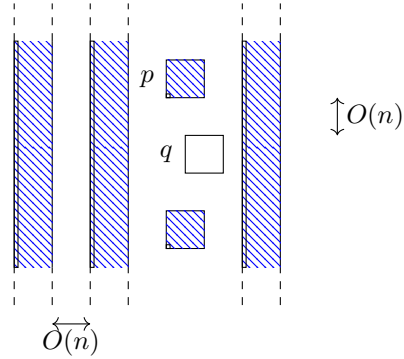

From this construction we deduce that 
\[ \left\{ \vec{w} \in \Z^2 \ | 
\ {\vec{w}}_1 = {\vec{v}}_1 (g(p,q)), 
\ |{\vec{v}}_1| \ge 4 \right\} \subset \Delta_{d_{\A}(X)}(p,q),\]
where 
\[ g(p,q) = m+ \tilde{f} (
\mathcal{P}(\tilde{p}),\mathcal{P}(\tilde{q})).\]
Thus,
\[\max_{p,q} g(p,q) = O(n).\]
Indeed, because $m \le 5n$ 
and $f$ is non decreasing, this 
number is smaller than $4(f(5n)+5n) = O(n)$. This 
comes from the fact that $f(n)=O(n)$.

\paragraph{The gluing sets of $d_{\A}(X)$ contain 
periodic positions in the central column:} \bigskip

Consider some $\vec{u}=(m+\tilde{f}(
\mathcal{P}(\tilde{p}),\mathcal{P}(\tilde{q})))\vec{v}$ with 
$\vec{v}=({\vec{v}}_1,{\vec{v}}_2)$, ${\vec{v}}_2 \neq 0$, and ${\vec{v}}_1=0$.
The pattern $p$ can be glued relatively to $q$ in $d_{\A}(X)$ in position 
$\vec{u}$ in a configuration $x^{\vec{u}}$.\bigskip

Indeed, the pattern $\mathcal{T} (\pi_2 (q))$ can be glued relatively to 
$\mathcal{T} (\pi_2 (p))$ in $\Delta$ with relative position $\vec{u}$.
In order to prove that one can glue the two patterns with this relative position. 
Then one completes straightly the curves that go through the two patterns and 
fulfill $\Z^2$ with straight curves. One shifts the configuration 
in such a way that 
$\pi_2 (q)$ appears in position $(0,0)$.
Then one completes this configuration with letters 
in $\A$ so that the pseudo-projection is $x^{\vec{u}}$.

This means that \[ \left\{ \vec{w} \in \Z^2 \ | 
\ {\vec{w}}_2  \in ( m+ \tilde{f}(
\mathcal{P}(\tilde{p}),\mathcal{P}(\tilde{q})))(\Z\backslash \{0\}), 
\ {\vec{w}}_1 = 0 \right\} \subset \Delta_{d_{\A}(X)}(p,q).\]

The Figure~\ref{fig.gluingset.stripes} shows 
the set of positions that we proved to be in 
the gluing set of the pattern $p$ relatively to $q$.

\end{proof}

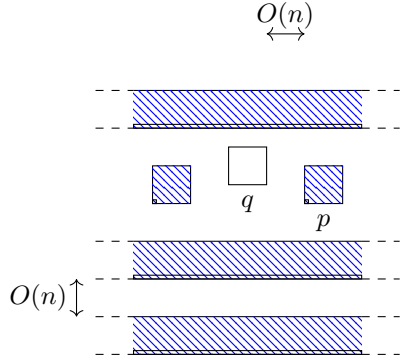
\begin{figure}[ht]
\[\begin{tikzpicture}[scale=0.5]
\draw (0,0.5) rectangle (1,1.5);
\draw node at (0.5,0) {$q$};
\draw (2,0) rectangle (3,1);
\draw node at (2.5,-0.5) {$p$};
\draw (-2,0) rectangle (-1,1);
\fill[pattern=north west lines, pattern color=blue] (2,0) rectangle (3,1);
\fill[pattern=north west lines, pattern color=blue] (-2,0) rectangle (-1,1);
\draw (2,0) rectangle (2.1,0.1);
\draw (-2,0) rectangle (-1.9,0.1);
\draw (-2.5,-2) -- (3.5,-2);
\draw (-2.5,-1) -- (3.5,-1);
\draw (-2.5,-3) -- (3.5,-3);
\draw (-2.5,-4) -- (3.5,-4);
\draw (-2.5,2) -- (3.5,2);
\draw (-2.5,3) -- (3.5,3);
\fill[pattern=north west lines, pattern color=blue] (-2.5,-2) rectangle (3.5,-1);
\fill[pattern=north west lines, pattern color=blue] (-2.5,-4) rectangle (3.5,-3);
\fill[pattern=north west lines, pattern color=blue] (-2.5,2) rectangle (3.5,3);
\draw (-2.5,-2) rectangle (3.5,-1.9);
\draw (-2.5,-4) rectangle (3.5,-3.9);
\draw (-2.5,2) rectangle (3.5,2.1);
\draw[dashed] (-3.5,-2) -- (-2.5,-2);
\draw[dashed] (3.5,-2) -- (4.5,-2);
\draw[dashed] (-3.5,-1) -- (-2.5,-1);
\draw[dashed] (3.5,-1) -- (4.5,-1);
\draw[dashed] (-3.5,-3) -- (-2.5,-3);
\draw[dashed] (3.5,-3) -- (4.5,-3);
\draw[dashed] (-3.5,-4) -- (-2.5,-4);
\draw[dashed] (3.5,-4) -- (4.5,-4);
\draw[dashed] (-3.5,2) -- (-2.5,2);
\draw[dashed] (3.5,2) -- (4.5,2);
\draw[dashed] (-3.5,3) -- (-2.5,3);
\draw[dashed] (3.5,3) -- (4.5,3);
\draw[<->] (1,4.5) -- (2,4.5);
\draw[<->] (-4,-3) -- (-4,-2);
\draw node at (-5,-2.5) {$O(n)$};
\draw node at (1.5,5)     {$O(n)$};
\end{tikzpicture}\]
\caption{\label{fig.gluingset.rotated} Schematic representation 
of a set of positions 
included in the gluing set of some couple of $n$-blocks 
in $\rho \circ d_{\A} (X)$.} 
\end{figure}

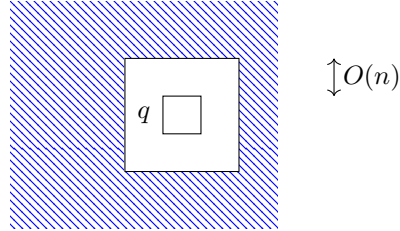
\begin{figure}[ht]
\[\begin{tikzpicture}[scale=0.5]
\fill[pattern=north west lines, pattern color=blue] (-4,-2.5) rectangle (3,3.5);
\fill[white] (-1,-1) rectangle (2,2);
\draw (0,0) rectangle (1,1);
\draw node at (-0.5,0.5) {$q$};
\draw (-1,-1) rectangle (2,2);
\draw[<->] (4.5,1) -- (4.5,2);
\draw node at (5.5,1.5)     {$O(n)$};
\end{tikzpicture}\]
\caption{\label{fig.gluingset.blockgluing} Schematic representation 
of a set of positions 
included in the gluing set of some couple of $n$-blocks 
in $d_{\tilde{\A}} \circ \rho \circ d_{\A} (X)$.} 
\end{figure}

Let us denote $\rho$ the transformation 
on subshifts that 
acts as a rotation by an angle $\pi/2$. Let $X$ be 
a subshift on an alphabet $\mathcal{A}$ and defined 
by a set ${\mathcal{F}}$ of forbidden patterns. Then 
$\rho(X)$ is the subshift 
on alphabet $\mathcal{A}$ 
defined by the set of patterns that are image 
by rotation of the patterns in $\mathcal{F}$.
Thus $\rho$ transforms SFT into SFT.

\begin{theorem} 
\label{theorem.transformation.block.gluing}
The operator $d_{\tilde{\A}} \circ \rho \circ d_{\A}$ 
transforms linear net-gluing subshifts 
of finite type into 
linear block gluing ones.
\end{theorem}

\noindent \textbf{Idea:} \textit{the 
idea is to see how this operator acts 
on the gluing sets of $n$-blocks $p$ 
relative to another $q$ and see that 
it fulfills the part of $\Z^2$ outside 
of a box containing $q$ whose size is 
$O(n)$.}

\begin{proof}

Since $\rho$ acts as a $\pi/2$ rotation over 
patterns (and thus on configurations), 
the gluing set of some $n$-block 
$p$ in the language of $\rho \circ d_{\A} (X)$ relatively to another one $q$
contains the positions shown by Figure~\ref{fig.gluingset.rotated}. 
Using the same procedure as in the 
proof of Proposition~\ref{proposition.transfo.gluing.sets}, 
we get that the gluing set of two $n$-blocks in the language of $d_{\tilde{\A}} \circ 
\rho \circ d_{\A} (X)$ 
contains 
some set of positions as in Figure~\ref{fig.gluingset.blockgluing}.
Indeed, this procedure introduced a vertical 
perturbation in the positions of the 
gluing sets. From 
the form of the sets included 
in the gluing sets on 
Figure~\ref{fig.gluingset.rotated}, this perturbation 
transforms the gluing sets of the subshift 
$d_{\tilde{\A}} \circ 
\rho \circ d_{\A} (X)$ by fulfilling the 
plane outside the box having 
size $O(n)$.

This means that the subshift 
$d_{\tilde{\A}} \circ \rho \circ d_{\A} (X)$ is linearly block gluing.
\end{proof}

\begin{proof} \textit{of 
Theorem~\ref{thm.aperiodicity}:}
We know that $X_{\texttt{R}}$ is linearly net gluing. 
As a consequence, the subshift 
$d_{\tilde{\A}} \circ \rho \circ d_{\A} (X_{\texttt{R}})$ is linearly block gluing. 
It is also aperiodic, since $X_{\texttt{R}}$ 
is aperiodic.
\end{proof}

\section{Entropy of block 
gluing \texorpdfstring{$\Z^2$}{Z2}-SFTs}\label{sec.Entropy}

In this section we present some results about the computability 
of entropy of block gluing $\Z^2$-SFT. The reader will find the 
proof of the main theorem in the next section.

\subsection{Characterization of the entropies of \texorpdfstring{$\Z^2$}{Z2}-SFT}

Let us recall that
$N_n (X) = |\mathcal{L}_n (X)|$ denotes the number 
of $n$-blocks in the language of a subshift $X$.
Its entropy is defined as:
\[h(X)=\inf_{n} \frac{\log_2 (N_n (X))}{n^2}\]

\subsection{Computability of 
the entropy for the sub-logarithmic 
regime}

In this section we show that the entropy of 
block gluing SFT with a sub-logarithmic gap
function is computable. This generalizes 
Proposition 3.3 in~\cite{Pavlov-Schraudner-2014}. 
This proposition states that a multidimensional 
SFT which is block gluing with a constant 
gap function is computable.

\begin{definition}
Let $(a_n)_n$ be a sequence of non-negative 
real numbers. We say that 
the series $\sum_{n} a_n$ converges 
with \textbf{computable rate} when 
there exists an algorithm which 
on input $t$ computes an integer $n(t)$ 
such that 
\[\sum_{n \ge n(t)} a_n < 2^{-t}.\]
\end{definition}

\begin{theorem}[\cite{GH17}]
\label{theorem.summability}
Let $f : \N \rightarrow \N$ be a function, 
$d$ be an integer
and $X$ a $\Z^d$-subshift whose language 
is decidable. If the series 
$\sum_n f(n)/n^2$ converges with 
computable rate, then $h(X)$ is computable.
\end{theorem}

\begin{proposition} \label{prop.computability.sublog} 
Let $X$ be a SFT which is 
$f$ block gluing. If the function $f$ 
verifies the following conditions: 
\begin{enumerate}
\item $f(n) \in o(log(n))$
\item $f \le id$,
\end{enumerate} 
then $h(X)$ is computable.
\end{proposition}

\begin{proof}
From Proposition~\label{proposition.decidable.language}, 
we get that the language of 
such a subshift is decidable.  
Moreover, the function $f$ verifies 
the conditions of Theorem~\ref{theorem.summability}. 
As a consequence, $h(X)$ is computable.
\end{proof}

\subsection{\label{sec.abstract.block.gluing}
Characterization of the entropies of linearly 
block gluing \texorpdfstring{$\Z^2$}{Z2}-SFT}

In this section, we prove the following theorem:

\begin{theorem}
\label{thm.entropy.block.gluing}
The possible entropies of linearly block gluing $\Z^2$-SFT 
are exactly the non-negative $\Pi_1$-computable numbers. 
\end{theorem}

\begin{corollary}
The possible entropies of transitive $\Z^2$-SFT are 
exactly the non-negative $\Pi_1$-computable numbers.
\end{corollary}

\subsubsection{Outline of the proof of Theorem~\ref{thm.entropy.block.gluing}}

The steps list of the proof of Theorem~\ref{thm.entropy.block.gluing} is the following: 

\begin{enumerate}
\item We first prove that any $\Pi_1$-computable number $h$
is the entropy of a linearly net gluing $\Z^2$-SFT. 
It is sufficient to realize the numbers in $[0,1]$. Indeed,
in order to realize the numbers in $]1,+\infty[$
one can take the product of a linearly 
net gluing SFT having entropy in $[0,1]$ with a full shift.
When $h$ is in $\{0,1/4,1/2,3/4,1\}$, 
this is easy to find a linearly net gluing $\Z^2$-SFT
whose entropy is $h$. Indeed, this can be done by allowing 
random bits on a regularly displayed set of positions in $\Z^2$.
Hence we only have to realize the numbers in 
$[0,1] \backslash \{0,1/4,1/2,3/4,1\}$.

The steps of the realization for these numbers are as follows: 

\begin{enumerate}
\item First, for any $\Pi_1$-computable 
sequence $s \in \{0,1\}^{\N}$ and any $N>0$,
we construct in Section~\ref{sec.construction.xsn} a linearly net gluing $\Z^2$-SFT 
$X_{s,N}$. The integer $N$ corresponds 
to a threshold and the sequence $s$ 
corresponds to a control on the entropy of the subshift.

This part is an adaptation 
of the construction in \cite{Hochman-Meyerovitch-2010}. 
This construction uses an implementation of Turing machines 
whose work is to control the frequency of apparition of some random bits.
These bits generate the entropy through frequency bits specifying 
if a random bit can be present or not.
In this construction, 
the obstacles for transitivity (and thus linear 
net gluing) 
are the following ones: 

\begin{enumerate} 
\item the identification of frequency bits over 
areas that are not closely 
related to the structures  
\item and the possibility of degenerated behaviors of the 
Turing machines dynamics in infinite areas. 
\end{enumerate}

The first obstacle is solved by a modification of the identification areas (this 
mechanism is described in 
Section~\ref{sec.frequency.bits} and 
abstracted on 
Figure~\ref{fig.abstract.freqbits.block.gluing}) 
corresponding 
to the insides of Robinson's structures defining 
computation units
(they will be defined in a more precise way 
in the following).
The second one is solved by the simulation of 
degenerated behaviors in all the computation units
aside the intended behaviors.
This ensures that the results of this 
simulation is not taken into account 
- this means that the frequency 
bits are not affected, using error signals 
propagating through the 
border of Robinson's structures (this 
is described in Section~\ref{sec.cells.coding} and 
Section~\ref{sec.machine}, 
and abstracted on Figure~\ref{fig.machines.mechanism.abstract.block.gluing}). 
The subshift is net gluing since any 
pattern can be completed, with control 
on the size, into a pattern over a simulation 
area.

\item In Section~\ref{sec.choice.parameters} we 
prove an explicit formula for the entropy 
of the subshifts $X_{s,N}$ which depends 
on the parameters $s,N$. 
Afterwards, for any $h \in [0,1] \backslash \{0,1/4,1/2,3/4,1\}$, 
we associate some $\Z^2$-SFT $X_h$ by choosing some parameters $s,N$ 
so that the entropy of $X_h$ is $h$. Then we prove in 
Section~\ref{sec.linear.net.gluing.xh} that this subshift 
is linearly net gluing.

\end{enumerate}

\item We present in Section~\ref{sec.linearly.block.gluing.
realization} 
adaptations $d^{(r)}_{\A}$ of the operator $d_{\A}$ 
presented in Section~\ref{sec.pseudo.covering}. These operators verify the equalities 
\[h(d^{(r)}_{\A} (Z)) = \frac{\log_2 (1+r)}{r} + h(Z)\]
for any subshift $Z$. 
Since the additional 
entropy $\frac{\log_2 (1+r)}{r}$ 
induced by the operator is computable 
and that the operator 
$d^{(r)}_{\A} \circ \rho \circ d^{(r)}_{\A} \circ 
\rho \circ d^{(r)}_{\A}$ transforms linearly net gluing SFTs 
into linearly block gluing ones, 
this allows any $\Pi_1$-computable 
number in $]0,1]$ to be realized
as the entropy of a block gluing SFT. 
Since the entropy zero is trivial 
to realize, this means that any $\Pi_1$-computable 
number in $[0,1]$ is the entropy of a linearly 
block gluing SFT. 

\end{enumerate}

\subsubsection{Description of the layers in 
the construction of 
the subshifts $X_{s,N}$}

Here is a more precise 
description of the construction of $X_{s,N}$ for a given $h \in [0,1]$. 

First, we fix some parameters $s$ and $N$, where $s$ is 
a $\Pi_1$-computable sequence of $\{0,1\}^{\N}$, 
and $N>0$ is an integer. 

The construction of the subshift 
$X_{s,N}$ involves a hierarchy of computing 
units that we call cells. Each cell is divided 
into four parts and in each of these parts a machine 
works. The center (called nucleus) 
of each cell contains an information 
which codes for the behavior of the machines.
Two of these machines compute and the other 
two simulate degenerated 
behaviors. This means in particular that the initial 
tape of the machine is left free 
and that machine heads can 
enter at any time and in any state on the two sides of 
the machine area.
The aim of the computing machines is to control 
some random bits that generate the entropy, through 
frequency. These bits 
are grouped in infinite sets having 
frequency given by a formula. This allows 
the entropy to be expressed as the sum of three 
entropies $h_{\texttt{int}}$, $h_{\texttt{comp}}(s)$ 
and $h_{\texttt{sim}} (N)$. 
The first one is generated by random bits, and 
serve to place the entropy of 
the subshift in one of the quarters
of the segment $[0,1]$.
Each of the two other entropies 
is the sum of a series. The first one 
involves the sequence $s$ and is 
the entropy due to random bits. The second one 
is generated by the symbols left free in the simulation 
area. We choose $N$ so that the sum of $h_{\texttt{int}}$ and 
$h_{\texttt{sim}}(N)$ is smaller than $h_{\texttt{comp}}(s)$. 
By the choice of the sequence 
$s$ we control the 
series whose sum is $h_{\texttt{comp}} (s)$. 
We choose this entropy 
so that the total entropy is $h$.

Here is a detailed description of the layers in 
this construction.

\begin{itemize} 
\item \textit{Structure layer [Section~\ref{sec.structure}]:}  
This layer consists in the subshift $X_{\texttt{R}}$.
Recall that any configuration of this subshift
exhibits a cell hierarchy. In this setting, for all $n$
the order $n$ cells appear periodically in 
the vertical and horizontal directions. 
Some additional marks 
allow the decomposition of the cells into 
sub-structures that we describe 
in this section. In each of these sub-structures  
specific behaviors occur. In particular, 
each cell is decomposed into four 
parts called quarters.

\item \textit{Basis layer} [Section~\ref{basis}]
The blue corners in the structure layer will be superimposed 
with random bits in $\{0,1\}$. This allows to 
an entropy in $[0,1/4]$ to be generated.
Turing machines will 
control the frequency of positions 
where the random bits can be $1$ through the use of
frequency bits. This generates an entropy equal 
to the $\Pi_1$-computable number 
\[h' = h-\frac{\lfloor 4h \rfloor}{4} \in [0,1/4[.\] 
In order to generate the entropy $h$, 
we consider $i = \lfloor 4h \rfloor$. We 
impose that the other positions are superimposed with 
random bits such that 
the positions where the random bit is equal to $1$ 
have with frequency $i/3$ in this set. 
Hence the total frequency of the 
positions where the random bit can be $1$ 
in $\Z^2$ is 
\[\frac{i}{3} \frac{3}{4} + h' = h.\]

\item \textit{Frequency bits layer} [Section~\ref{sec.frequency.bits}]

Each quarter is superimposed with 
a frequency bit. 
On positions with blue corners 
in the structure layer
having frequency bit equal to $0$, 
the random bit is equal to $0$. 

\item \textit{Cells coding layer} [Section~\ref{sec.cells.coding}]

In this layer we superimpose to the center 
of the cells (that we call nuclei) 
of each cell a symbol which specifies two adjacent quarters 
of the cell when $n>N$. It represents 
all the quarters when $n \le N$. 
This symbol is called the DNA of the cell.

The quarters represented in the DNA 
are called computation quarters.
The other ones are called simulation quarters. In these ones 
the function of the machines
is to simulate any degenerated behavior of the 
computing machines.

Since simulation induces parasitic entropy, 
we choose after the construction 
some $N$ such that 
this entropy is smaller than $h$. 
This allows programming the machines in 
such a way that 
the entropy generated by the random bits 
in the other quarters complements this entropy so that 
the total entropy is $h$.

\item {\textit{Synchronization layer}} [Section~\ref{sec.synchr}]
In this section, we synchronize the frequency bits of 
the computation quarters of all the cells 
having the same order.

\begin{figure}[ht]
\[\begin{tikzpicture}[scale=0.4]
\fill[purple!30] (4,4) rectangle (8,12);

\fill[white] (5,9) rectangle (7,11);
\fill[white] (5,5) rectangle (7,7);

\draw[line width = 0.9mm,color=gray!40] 
(8,0) -- (8,8) -- (16,8);

\draw[line width = 0.9mm,color=gray!90]
(4,4) -- (12,4) -- (12,12) -- (4,12) -- (4,4);
\draw[line width = 0.9mm,color=gray!40] 
(2,10) -- (6,10) -- (6,14) -- (2,14) -- (2,10);
\draw[line width = 0.9mm,color=gray!40] 
(10,10) -- (14,10) -- (14,14) -- (10,14) -- (10,10);
\draw[line width = 0.9mm,color=gray!40] 
(10,2) -- (14,2) -- (14,6) -- (10,6) -- (10,2);
\draw[line width = 0.9mm,color=gray!40] 
(2,2) -- (6,2) -- (6,6) -- (2,6) -- (2,2);

\draw[line width = 0.9mm,color=gray!90]
(5,9) -- (7,9) -- (7,11) -- (5,11) -- (5,9);

\draw[line width = 0.9mm,color=gray!90]
(5,5) -- (7,5) -- (7,7) -- (5,7) -- (5,5);

\draw[line width = 0.9mm,color=gray!90]
(9,9) -- (11,9) -- (11,11) -- (9,11) -- (9,9);

\draw[line width = 0.9mm,color=gray!90]
(9,5) -- (9,7) -- (11,7) -- (11,5) -- (9,5);

\draw[line width = 0.2mm,color=purple,-latex] 
(8,0) -- (8,16); 

\draw[line width = 0.2mm,color=purple,-latex] 
(0,8) -- (16,8);

\draw[dashed] (8,4) -- (8,12);
\draw[dashed] (4,8) -- (12,8);
\draw[-latex] (0,4) -- (4,4);
\draw[-latex] (16,4) -- (12,4);
\node[color=purple] at (20,8) {synchr.};
\node at (22,4) {simulation quarters};
\node at (-6,4) {computation quarters};

\end{tikzpicture}\]
\caption{\label{fig.abstract.freqbits.block.gluing} 
Illustration of the frequency bits identification areas 
(colored purple)
in our construction. The squares designate 
some petal structures of the Robinson subshift.}
\end{figure}
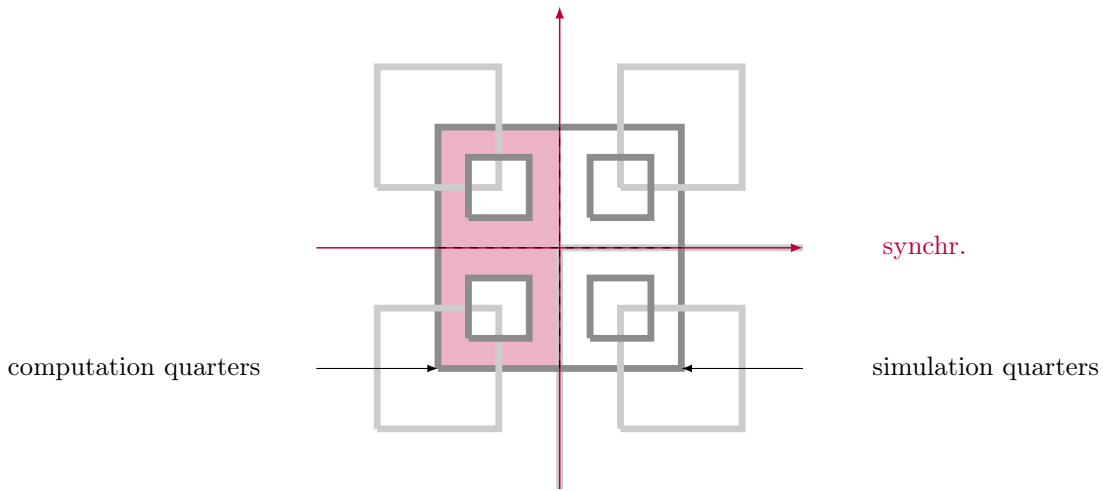

\item {\textit{Computation areas layer}} [Section~\ref{sec.computation.areas}]
In this layer are specified the areas supporting 
the computations of the machines. The function 
of each position in this area amongst the following ones 
are also specified: 
\begin{enumerate}
\item transfer 
of information, vertical 
or horizontal, 
\item or the execution of one step of computation.
\end{enumerate}
This is done using signals that detect the rows 
and columns in a cell 
that do not encounter a smaller cell.
The intersections of such line and column are the 
positions where a machine executes one step of 
its computation. 
The other positions of these lines 
and columns are used to 
transmit information.

Since the construction of these areas can have 
degenerated behaviors in infinite cells, 
these behaviors are also simulated in simulation quarters.
We use error signals in order to impose that the 
computation areas are well constructed 
in computation quarters.

\item \textit{Machines layer [Section~\ref{sec.machine}]:}  

Consider $(s^{(n)}_k)_{n,k} \in \{0,1\}^{\N ^2}$ some 
sequence such that for all $k$, 
\[s_k = \inf_n s^{(n)}_k.\]

This layer supports the computations of the machines in 
all the quarters of each cell. Each quarter 
of a cell has its 
proper direction of time and space. The machine 
is programmed so that when well initialized, 
it writes successively 
for all $n$ 
the bits $s^{(n)}_k$, $k=0...n$ 
on the $2^{k}$th position of its tape corresponding 
to a column that is just one the left of order $k$ cells.
The sequence is chosen afterwards so that 
the total entropy of $X_{s,N}$ is $h$.

The frequency bits of order $k$ cells corresponding 
to the computation quarters is transported in the 
column just on the right of these cells. This is done  
in order for the machine 
to have access to this information.

If at some point one of the frequency bits is 
greater than a bit written by the machine in this column, 
then the machine enters in halting state. 

We allow the machine heads to enter in error state.
However, when this happens it transmits a signal through its trajectory back in time 
until initialization. We forbid the coexistence of this signal with 
the representation of this quarter in the DNA.
As a consequence, the computations of a machine are taken into account if 
and only if in a computation quarter where these computations have 
some meaning.

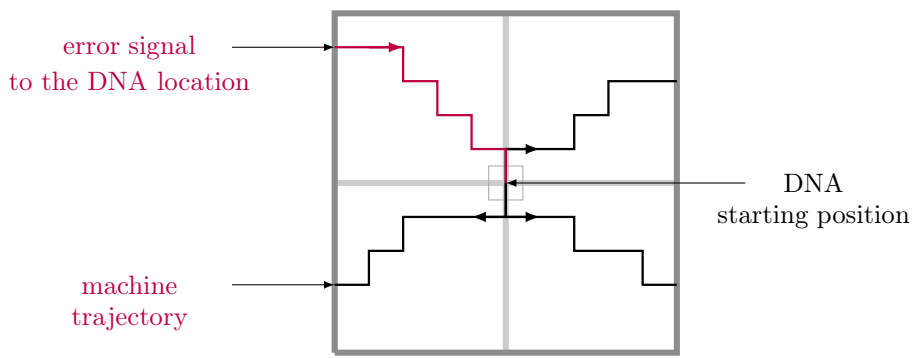
\begin{figure}[ht]
\[\begin{tikzpicture}[scale=0.45]
\draw[line width = 0.8mm,color=gray!40] (5,0) -- (5,10) ; 
\draw[line width = 0.8mm,color=gray!40] (0,5) -- (10,5) ; 
\draw[line width = 0.8mm,color=gray!90] (0,0) -- (10,0) -- 
(10,10) -- (0,10) -- (0,0);
\draw[color=gray!65] (4.5,4.5) rectangle (5.5,5.5);
\draw[-latex] (12,5) -- (5,5); 
\node at (14,5) {DNA};
\node at (14,4) {starting position};
\draw[line width = 0.3mm] (5,5) -- (5,6) -- (7,6) -- (7,7) -- (8,7) -- (8,8) -- (10,8);
\draw[line width = 0.3mm,-latex] (5,6) -- (6,6);
\draw[line width = 0.3mm,color=purple] (5,5) -- (5,6) -- (4,6) -- (4,7) -- (3,7) -- (3,8) -- (2,8) -- (2,9) -- (0,9);
\draw[line width = 0.3mm,latex-,color=purple] (2,9) -- (1,9);
\draw[line width = 0.3mm] (5,5) -- (5,4) -- (2,4) -- (2,3) -- (1,3) -- (1,2) -- (0,2) ;
\draw[line width = 0.3mm,-latex] (5,4) -- (4,4);
\draw[line width = 0.3mm] (5,5) -- (5,4) -- (7,4) -- (7,3) -- (9,3) -- (9,2) -- (10,2);
\draw[line width = 0.3mm,-latex] (5,4) -- (6,4);

\draw[-latex] (-3,9) -- (0,9);
\node[color=purple] at (-6,9) {error signal};
\node[color=purple] at (-6,8) {to the DNA location};

\draw[-latex] (-3,2) -- (0,2);
\node[color=purple] at (-6,2) {machine};
\node[color=purple] at (-6,1) {trajectory};
\end{tikzpicture}\]
\caption{ \label{fig.machines.mechanism.abstract.block.gluing}
Illustration of the machines mechanisms in our construction.
The gray square designates a computing unit
splitted in four parts.
In each of this parts evolves a 
computing machine departing 
from the center of the cell. If the machine 
reaches the border in error state, it triggers an error signal.}
\end{figure}

The directions of time and space depend on the orientation 
of the quarter in the cell.
In each quarter, these directions are given on 
Figure~\ref{fig.time.space.intro}.

\begin{figure}[ht]
\[\begin{tikzpicture}[scale=0.4]
\fill[gray!90] (0,0) rectangle (11,11);
\fill[gray!40] (0.5,0.5) rectangle (10.5,10.5);
\fill[red!50] (0.5,0.5) rectangle (5.25,5.25);
\fill[orange!50] (5.75,5.75) rectangle (10.5,10.5);
\fill[yellow!50] (5.75,0.5) rectangle (10.5,5.25); 
\fill[purple!50] (0.5,5.75) rectangle (5.25,10.5); 
\draw (0,0) rectangle (11,11); 
\draw (5.25,0.5) rectangle (5.75,10.5);
\draw (0.5,5.25) rectangle (10.5,5.75);
\draw (0.5,0.5) rectangle (10.5,10.5);
\draw[-latex] (6.5,6.5) -- (9.5,6.5);
\draw[-latex] (6.5,6.5) -- (6.5,9.5);
\draw[-latex] (4.5,4.5) -- (1.5,4.5);
\draw[-latex] (4.5,4.5) -- (4.5,1.5);
\draw[-latex] (4.5,6.5) -- (4.5,9.5);
\draw[-latex] (4.5,6.5) -- (1.5,6.5);
\draw[-latex] (6.5,4.5) -- (9.5,4.5);
\draw (6.5,4.5) -- (6.5,1.5);
\node at (2,3.5) {space};
\node at (2,7.5) {space};
\node at (3.5,1.5) {time};
\node at (3.5,9.5) {time};

\node at (7.75,9.5) {time};
\node at (9,7.5) {space};

\node at (7.75,1.5) {time};
\node at (9,3.5) {space};

\end{tikzpicture}\]
\caption{\label{fig.time.space.intro} Schema 
of time and space directions 
in the four different quarters of a cell.}
\end{figure}
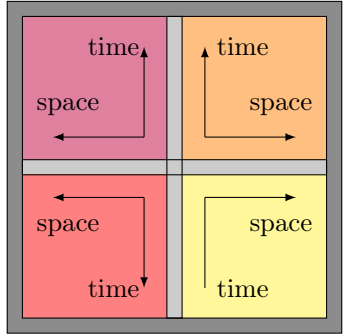

\end{itemize}

The next sections are devoted to make the proof 
of Theorem~\ref{thm.entropy.block.gluing}
more precise.

\section{\label{sec.construction.xsn} Construction of 
the subshifts \texorpdfstring{$X_{s,N}$}{XsN}}

\subsection{\label{sec.structure} 
Structure layer}

In this section, we present the structure layer 
and the structures observable in this layer 
that will be used in the following.

This layer has three sub-layers, as follows.

\subsubsection{\label{sec.substructures} 
Specialized sub-structures of the cells}

The first sublayer of the structure layer 
is the subshift $X_{\texttt{R}}$. \bigskip

Here we list some structures that appear in this layer and 
the designation that we use for them. 
Let $x$ be a configuration in this layer.

\begin{itemize}
\item Recall that an order $n$ cell is the part 
of $\Z^2$ enclosed in an order $2n+1$ petal.
\item The center of a cell is a red corner which 
is valued with $0$. The position 
where this symbol appears is 
called the \textbf{nucleus} of the cell. 
\item In a cell the union of the column and the line containing 
the nucleus is called the 
\textbf{reticle} of the cell. In particular, 
the nucleus is the intersection of 
the reticle's arms.
\item We call the set of 
positions in the border 
of a cell (defined 
as the position which has a 
neighbor outside of the cell) the \textbf{wall}.
\item The set of positions in a cell 
that are not in the wall, in the reticle, or 
in another (smaller) 
cell included to the considered one 
is called the \textbf{cytoplasm}.
\item We call any 
of the four connected 
parts included in the cytoplasm and 
delimited by the reticle and the wall a 
\textbf{quarter} 
of the cell.
\end{itemize}  \bigskip

Figure \ref{pic.9} gives an example 
of pattern that can be superimposed 
over an order 1 cell in the structure 
layer. We illustrate on this figure the 
structures that we listed above. \bigskip

Recall that cells have the following properties: 
\begin{itemize}
\item For all $m \ge 0$, and $i \in \{1,...,m\}$, any order 
$m$ cell contains 
properly $4.12^{i-1}$ order $m-i$ cells (meaning 
that these cells are not included in an 
order $<m$ cell).
As a consequence, each order $m$ cell contains 
properly $4.12^{m}$ blue corners.
\item Moreover, in any configuration
each order $m$ cell (and 
in particular the nuclei) repeats periodically, in the horizontal 
and vertical directions,
with period $4^{m+2}$. 
\end{itemize}

\subsubsection{Coloring}

In this section, we present the representation of
the sub-structures of cells presented in 
Section~\ref{sec.substructures}.
In particular, we color 
differently the four quarters of a cell, 
in order for the machines to have 
access to the direction of time 
and space in their quarter. \bigskip

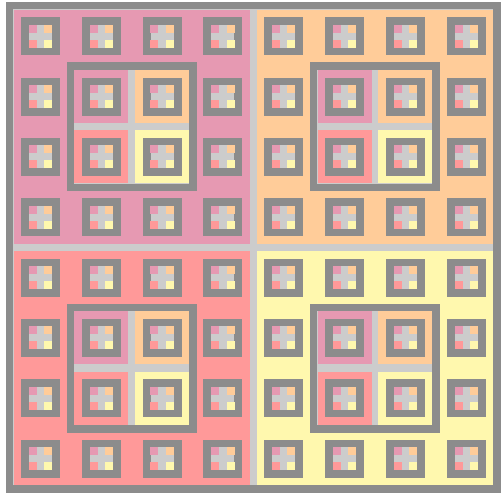
\begin{figure}[ht]
\begin{center}
\begin{tikzpicture}[scale=0.05]
\supertilesynchrothree{0}{0}
\end{tikzpicture}
\end{center}
\caption{\label{fig.cell.ordre3} 
An example of pattern in the 
synchronization layer over an order 2 cell.}
\end{figure}

\noindent \textbf{\textit{Symbols:}} 

\[\begin{tikzpicture}[scale=0.3]
\fill[gray!40] (0,0) rectangle (1.2,1.2);
\draw (0,0) rectangle  (1.2,1.2);
\end{tikzpicture} \ \ \begin{tikzpicture}[scale=0.3]
\fill[gray!90] (0,0) rectangle (1.2,1.2);
\draw (0,0) rectangle  (1.2,1.2);
\end{tikzpicture} \ \ \begin{tikzpicture}[scale=0.3]
\fill[orange!50] (0,0) rectangle (1.2,1.2);
\draw (0,0) rectangle  (1.2,1.2);
\end{tikzpicture} \ \ \begin{tikzpicture}[scale=0.3]
\fill[purple!50] (0,0) rectangle (1.2,1.2);
\draw (0,0) rectangle  (1.2,1.2);
\end{tikzpicture} \ \ \begin{tikzpicture}[scale=0.3]
\fill[yellow!50] (0,0) rectangle (1.2,1.2);
\draw (0,0) rectangle  (1.2,1.2);
\end{tikzpicture} \ \ \begin{tikzpicture}[scale=0.3]
\fill[red!50] (0,0) rectangle (1.2,1.2);
\draw (0,0) rectangle  (1.2,1.2);
\end{tikzpicture}\]

Each symbol corresponds to a part of the cell: 
\begin{itemize}
\item $\begin{tikzpicture}[scale=0.3,baseline=0.5mm]
\fill[gray!40] (0,0) rectangle (1.2,1.2);
\draw (0,0) rectangle  (1.2,1.2);
\end{tikzpicture}$ corresponds to the reticle.
\item $\begin{tikzpicture}[scale=0.3,baseline=0.5mm]
\fill[gray!90] (0,0) rectangle (1.2,1.2);
\draw (0,0) rectangle  (1.2,1.2);
\end{tikzpicture}$ corresponds to the walls.
\item $\begin{tikzpicture}[scale=0.3,baseline=0.5mm]
\fill[orange!50] (0,0) rectangle (1.2,1.2);
\draw (0,0) rectangle  (1.2,1.2);
\end{tikzpicture}$ corresponds to the north east 
quarter.
\item $\begin{tikzpicture}[scale=0.3,baseline=0.5mm]
\fill[purple!50] (0,0) rectangle (1.2,1.2);
\draw (0,0) rectangle  (1.2,1.2);
\end{tikzpicture}$ corresponds to the north west 
synchronization area.
\item $\begin{tikzpicture}[scale=0.3,baseline=0.5mm]
\fill[yellow!50] (0,0) rectangle (1.2,1.2);
\draw (0,0) rectangle  (1.2,1.2);
\end{tikzpicture}$ corresponds to the south east 
quarter.
\item $\begin{tikzpicture}[scale=0.3,baseline=0.5mm]
\fill[red!50] (0,0) rectangle (1.2,1.2);
\draw (0,0) rectangle  (1.2,1.2);
\end{tikzpicture}$ corresponds to the south west 
quarter.
\end{itemize}

\noindent \textbf{\textit{Local rules:}}

\begin{enumerate}
\item \textbf{Localization:}

\begin{itemize}
\item A position 
is colored with 
 $\begin{tikzpicture}[scale=0.3,baseline=0.5mm]
\fill[gray!90] (0,0) rectangle (1.2,1.2);
\draw (0,0) rectangle  (1.2,1.2);
\end{tikzpicture}$ if and 
only if it lies in the walls of a cell.
\item the symbol $\begin{tikzpicture}[scale=0.3,baseline=0.5mm]
\fill[gray!40] (0,0) rectangle (1.2,1.2);
\draw (0,0) rectangle  (1.2,1.2);
\end{tikzpicture}$ can be superimposed only 
on red corners valued $0$ or arrows symbols 
valued $0$ for the direction of the long arrows.
\item the positions 
that are not amongst 
the type of positions 
specified by the 
two last rules 
are colored with else 
\begin{tikzpicture}[scale=0.3,baseline=0.5mm]
\fill[orange!50] (0,0) rectangle (1.2,1.2);
\draw (0,0) rectangle  (1.2,1.2);
\end{tikzpicture}, \begin{tikzpicture}[scale=0.3,baseline=0.5mm]
\fill[purple!50] (0,0) rectangle (1.2,1.2);
\draw (0,0) rectangle  (1.2,1.2);
\end{tikzpicture}, \begin{tikzpicture}[scale=0.3,baseline=0.5mm]
\fill[yellow!50] (0,0) rectangle (1.2,1.2);
\draw (0,0) rectangle  (1.2,1.2);
\end{tikzpicture} or \begin{tikzpicture}[scale=0.3,baseline=0.5mm]
\fill[red!50] (0,0) rectangle (1.2,1.2);
\draw (0,0) rectangle  (1.2,1.2);
\end{tikzpicture}.
\end{itemize}

\item \textbf{Transmission rules:} 

\begin{itemize}
\item Consider
two vertically adjacent positions 
with vertical outgoing arrows (or two 
horizontally adjacent positions 
with horizontal ougoing arrows)
such that the value corresponding 
to this direction in the two 
symbols is $0$. Then  
if one of these positions is colored $\begin{tikzpicture}[scale=0.3,baseline=0.5mm]
\fill[gray!40] (0,0) rectangle (1.2,1.2);
\draw (0,0) rectangle  (1.2,1.2);
\end{tikzpicture}$, then the other one is also colored 
$\begin{tikzpicture}[scale=0.3,baseline=0.5mm]
\fill[gray!40] (0,0) rectangle (1.2,1.2);
\draw (0,0) rectangle  (1.2,1.2);
\end{tikzpicture}$.
\item Consider two adjacent positions which are not 
colored $\begin{tikzpicture}[scale=0.3,baseline=0.5mm]
\fill[gray!90] (0,0) rectangle (1.2,1.2);
\draw (0,0) rectangle  (1.2,1.2);
\end{tikzpicture}$ or $\begin{tikzpicture}[scale=0.3,baseline=0.5mm]
\fill[gray!40] (0,0) rectangle (1.2,1.2);
\draw (0,0) rectangle  (1.2,1.2);
\end{tikzpicture}$. Then these two positions 
have the same color.
\end{itemize}

\item \textbf{Determination of the colors:}

The idea of these rules is to determine
the colors that appear in the neighborhood of particular positions in 
a cell. 
Together with the transmission rules, 
these rules will determine the color of any 
position in any configuration.

In order to describe them, we use the vocabulary 
introduced in this text.
However, one can translate these terms into 
symbols of the Robinson subshift (we do 
this for the second rule as an example).

The rules are the following ones: 

\begin{itemize}
\item 
A south west (resp south east, north east, south west) 
corner of a cell 
induces the position on its north 
east (resp north west, 
south west, south east) 
to be colored $\begin{tikzpicture}[scale=0.3,baseline=0.5mm]
\fill[red!50] (0,0) rectangle (1.2,1.2);
\draw (0,0) rectangle  (1.2,1.2);
\end{tikzpicture}$ (resp. $\begin{tikzpicture}[scale=0.3,baseline=0.5mm]
\fill[yellow!50] (0,0) rectangle (1.2,1.2);
\draw (0,0) rectangle  (1.2,1.2);
\end{tikzpicture}$, $\begin{tikzpicture}[scale=0.3,baseline=0.5mm]
\fill[orange!50] (0,0) rectangle (1.2,1.2);
\draw (0,0) rectangle  (1.2,1.2);
\end{tikzpicture}$, $\begin{tikzpicture}[scale=0.3,baseline=0.5mm]
\fill[purple!50] (0,0) rectangle (1.2,1.2);
\draw (0,0) rectangle  (1.2,1.2);
\end{tikzpicture}$).
\item A position inside the west part 
of the north wall of a cell (meaning 
a four or six arrows symbol whose long arrows 
are horizontal, directed to the right, and 
valued $1$ for the horizontal direction) (resp. 
east part)
has its south neighbor position colored 
with $\begin{tikzpicture}[scale=0.3,baseline=0.5mm]
\fill[purple!50] (0,0) rectangle (1.2,1.2);
\draw (0,0) rectangle  (1.2,1.2);
\end{tikzpicture}$ (resp. 
$\begin{tikzpicture}[scale=0.3,baseline=0.5mm]
\fill[orange!50] (0,0) rectangle (1.2,1.2);
\draw (0,0) rectangle  (1.2,1.2);
\end{tikzpicture}$). Its neighbor position 
on north is colored with the same 
color as the north east position and 
the north west position. We impose similar 
rules for positions in the other walls 
of cells.

\item A nucleus has 
respectively on its north west, north east, south east, 
south west positions the colors $\begin{tikzpicture}[scale=0.3,baseline=0.5mm]
\fill[purple!50] (0,0) rectangle (1.2,1.2);
\draw (0,0) rectangle  (1.2,1.2);
\end{tikzpicture}$, $\begin{tikzpicture}[scale=0.3,baseline=0.5mm]
\fill[orange!50] (0,0) rectangle (1.2,1.2);
\draw (0,0) rectangle  (1.2,1.2);
\end{tikzpicture}$, $\begin{tikzpicture}[scale=0.3,baseline=0.5mm]
\fill[yellow!50] (0,0) rectangle (1.2,1.2);
\draw (0,0) rectangle  (1.2,1.2);
\end{tikzpicture}$ and $\begin{tikzpicture}[scale=0.3,baseline=0.5mm]
\fill[red!50] (0,0) rectangle (1.2,1.2);
\draw (0,0) rectangle  (1.2,1.2);
\end{tikzpicture}$, and this position is colored 
$\begin{tikzpicture}[scale=0.3,baseline=0.5mm]
\fill[gray!40] (0,0) rectangle (1.2,1.2);
\draw (0,0) rectangle  (1.2,1.2);
\end{tikzpicture}$.

\item A position at the center 
of the south wall of a cell has 
its north west, north and north west neighbor 
positions colored with 
$\begin{tikzpicture}[scale=0.3,baseline=0.5mm]
\fill[red!50] (0,0) rectangle (1.2,1.2);
\draw (0,0) rectangle  (1.2,1.2);
\end{tikzpicture}$, $\begin{tikzpicture}[scale=0.3,baseline=0.5mm]
\fill[gray!40] (0,0) rectangle (1.2,1.2);
\draw (0,0) rectangle  (1.2,1.2);
\end{tikzpicture}$ and $\begin{tikzpicture}[scale=0.3,baseline=0.5mm]
\fill[yellow!50] (0,0) rectangle (1.2,1.2);
\draw (0,0) rectangle  (1.2,1.2);
\end{tikzpicture}$. 
Similar rules are imposed for the
centers of the other walls.
\item On the west and east neighbor positions
of a position colored 
\begin{tikzpicture}[scale=0.3,baseline=0.5mm]
\fill[gray!40] (0,0) rectangle (1.2,1.2);
\draw (0,0) rectangle (1.2,1.2); 
\end{tikzpicture},
the possible couples of colors are: 
\begin{itemize}
\item $\begin{tikzpicture}[scale=0.3,baseline=0.5mm]
\fill[purple!50] (0,0) rectangle (1.2,1.2);
\draw (0,0) rectangle  (1.2,1.2);
\end{tikzpicture}$ on the west 
and $\begin{tikzpicture}[scale=0.3,baseline=0.5mm]
\fill[orange!50] (0,0) rectangle (1.2,1.2);
\draw (0,0) rectangle  (1.2,1.2);
\end{tikzpicture}$ on the east,
\item $\begin{tikzpicture}[scale=0.3,baseline=0.5mm]
\fill[red!50] (0,0) rectangle (1.2,1.2);
\draw (0,0) rectangle  (1.2,1.2);
\end{tikzpicture}$ on the west and $\begin{tikzpicture}[scale=0.3,baseline=0.5mm]
\fill[yellow!50] (0,0) rectangle (1.2,1.2);
\draw (0,0) rectangle  (1.2,1.2);
\end{tikzpicture}$ on the east.
\end{itemize}
Similar rules are imposed for the vertical direction.
\end{itemize}
\end{enumerate}

\noindent \textbf{\textit{Global behavior:}} \bigskip

The localization rules impose directly 
that the walls of the finite cells are colored 
$\begin{tikzpicture}[scale=0.3,baseline=0.5mm]
\fill[gray!90] (0,0) rectangle (1.2,1.2);
\draw (0,0) rectangle  (1.2,1.2);
\end{tikzpicture}$.
With the transmission rules combined with 
the rules determining colors, the 
reticle is colored $\begin{tikzpicture}[scale=0.3,baseline=0.5mm]
\fill[gray!40] (0,0) rectangle (1.2,1.2);
\draw (0,0) rectangle  (1.2,1.2);
\end{tikzpicture}$. Moreover, for any cell 
the positions in the same quarter are colored with 
else $\begin{tikzpicture}[scale=0.3,baseline=0.5mm]
\fill[red!50] (0,0) rectangle (1.2,1.2);
\draw (0,0) rectangle  (1.2,1.2);
\end{tikzpicture}$, $\begin{tikzpicture}[scale=0.3,baseline=0.5mm]
\fill[yellow!50] (0,0) rectangle (1.2,1.2);
\draw (0,0) rectangle  (1.2,1.2);
\end{tikzpicture}$, $\begin{tikzpicture}[scale=0.3,baseline=0.5mm]
\fill[orange!50] (0,0) rectangle (1.2,1.2);
\draw (0,0) rectangle  (1.2,1.2);
\end{tikzpicture}$ or $\begin{tikzpicture}[scale=0.3,baseline=0.5mm]
\fill[purple!50] (0,0) rectangle (1.2,1.2);
\draw (0,0) rectangle  (1.2,1.2);
\end{tikzpicture}$. The color is determined 
according to the orientation of the quarter 
in the cell: the first (resp. second, third, 
fourth) color is for 
the south west quarter (resp. south east, north east, 
north west). 

These rules allow the same 
coloration for infinite cells.

Figure~\ref{fig.cell.ordre3} shows an example 
of coloration of an order 2 cell.

\subsubsection{\label{sec.synchronization.net} Synchronization net}

In this section, we specify a network connecting 
the nuclei of order $n$ cells for all $n$. 
This allows the synchronization of the frequency bits 
corresponding to computation quarters of the cells.

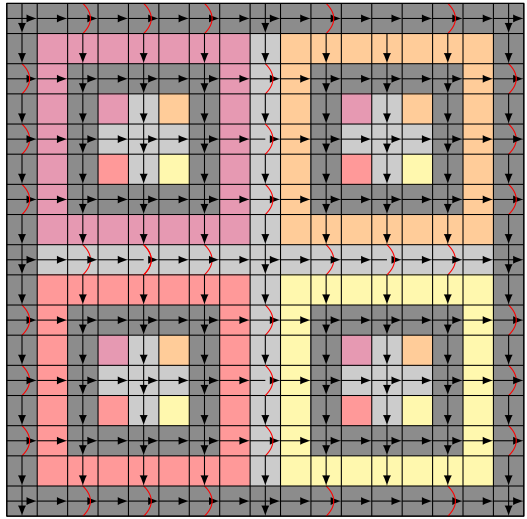
\begin{figure}[ht] \[\begin{tikzpicture}[scale=0.20]
\supertilesynchrotwo{2}{2}
\draw (0,0) rectangle (34,34);
\foreach \x in {0,...,16} 
\draw[-latex] (2*\x,0+1) -- (2*\x+2,0+1);
\foreach \x in {0,...,16} 
\draw[-latex] (2*\x,4+1) -- (2*\x+2,4+1);
\foreach \x in {0,...,16} 
\draw[-latex] (2*\x,12+1) -- (2*\x+2,12+1);
\foreach \x in {0,...,16} 
\draw[-latex] (2*\x,20+1) -- (2*\x+2,20+1);
\foreach \x in {0,...,16} 
\draw[-latex] (2*\x,28+1) -- (2*\x+2,28+1);
\foreach \x in {0,...,16} 
\draw[-latex] (2*\x,32+1) -- (2*\x+2,32+1);
\foreach \x in {16,15,13,11,9,8,7,5,3,1,0}
\draw[-latex] (1,2*\x+2) -- (1,2*\x);

\draw[color=red,domain=0:1] 
plot({2*\x-2*\x^2+1},{4+2*\x});
\draw[color=red,domain=0:1] 
plot({2*\x-2*\x^2+1},{8+2*\x});
\draw[color=red,domain=0:1] 
plot({2*\x-2*\x^2+1},{12+2*\x});
\draw[color=red,domain=0:1] 
plot({2*\x-2*\x^2+1},{20+2*\x});
\draw[color=red,domain=0:1] 
plot({2*\x-2*\x^2+1},{24+2*\x});
\draw[color=red,domain=0:1] 
plot({2*\x-2*\x^2+1},{28+2*\x});

\foreach \x in {15,14,13,12,11,10,9,7,6,5,4,3,2,1}
\draw[-latex] (4+1,2*\x+2) -- (4+1,2*\x);

\draw[color=red,domain=0:1] 
plot({4+2*\x-2*\x^2+1},{2*\x});
\draw[color=red,domain=0:1] 
plot({4+2*\x-2*\x^2+1},{16+2*\x});
\draw[color=red,domain=0:1] 
plot({4+2*\x-2*\x^2+1},{32+2*\x});

\draw[color=red,domain=0:1] 
plot({8+2*\x-2*\x^2+1},{2*\x});
\draw[color=red,domain=0:1] 
plot({8+2*\x-2*\x^2+1},{16+2*\x});
\draw[color=red,domain=0:1] 
plot({8+2*\x-2*\x^2+1},{32+2*\x});

\foreach \x in {15,14,13,12,11,10,9,7,6,5,4,3,2,1}
\draw[-latex] (12+1,2*\x+2) -- (12+1,2*\x);

\draw[color=red,domain=0:1] 
plot({12+2*\x-2*\x^2+1},{2*\x});
\draw[color=red,domain=0:1] 
plot({12+2*\x-2*\x^2+1},{16+2*\x});
\draw[color=red,domain=0:1] 
plot({12+2*\x-2*\x^2+1},{32+2*\x});

\draw[color=red,domain=0:1] 
plot({16+2*\x-2*\x^2+1},{4+2*\x});
\draw[color=red,domain=0:1] 
plot({16+2*\x-2*\x^2+1},{12+2*\x});
\draw[color=red,domain=0:1] 
plot({16+2*\x-2*\x^2+1},{20+2*\x});
\draw[color=red,domain=0:1] 
plot({16+2*\x-2*\x^2+1},{28+2*\x});

\foreach \x in {15,14,13,12,11,10,9,7,6,5,4,3,2,1}
\draw[-latex] (20+1,2*\x+2) -- (20+1,2*\x);

\draw[color=red,domain=0:1] 
plot({20+2*\x-2*\x^2+1},{2*\x});
\draw[color=red,domain=0:1] 
plot({20+2*\x-2*\x^2+1},{16+2*\x});
\draw[color=red,domain=0:1] 
plot({20+2*\x-2*\x^2+1},{32+2*\x});

\foreach \x in {15,14,13,12,11,10,9,7,6,5,4,3,2,1}
\draw[-latex] (28+1,2*\x+2) -- (28+1,2*\x);

\draw[color=red,domain=0:1] 
plot({28+2*\x-2*\x^2+1},{2*\x});
\draw[color=red,domain=0:1] 
plot({28+2*\x-2*\x^2+1},{16+2*\x});
\draw[color=red,domain=0:1] 
plot({28+2*\x-2*\x^2+1},{32+2*\x});

\foreach \x in {16,15,13,11,9,8,7,5,3,1,0}
\draw[-latex] (33,2*\x+2) -- (33,2*\x);

\draw[color=red,domain=0:1] 
plot({32+2*\x-2*\x^2+1},{4+2*\x});
\draw[color=red,domain=0:1] 
plot({32+2*\x-2*\x^2+1},{8+2*\x});
\draw[color=red,domain=0:1] 
plot({32+2*\x-2*\x^2+1},{12+2*\x});
\draw[color=red,domain=0:1] 
plot({32+2*\x-2*\x^2+1},{20+2*\x});
\draw[color=red,domain=0:1] 
plot({32+2*\x-2*\x^2+1},{24+2*\x});
\draw[color=red,domain=0:1] 
plot({32+2*\x-2*\x^2+1},{28+2*\x});

\foreach \x in {0,1,2,3,4,5,6,7,9,10,11,12,13,14,15,16} \draw[-latex] (2*\x,8+1) -- (2*\x+2,8+1);
\foreach \x in {0,1,2,3,4,5,6,7,9,10,11,12,13,14,15,16} \draw[-latex] (2*\x,24+1) -- (2*\x+2,24+1);
\foreach \x in {0,1,2,3,5,6,7,8,9,10,11,13,14,15,16} \draw[-latex] (2*\x,16+1) -- (2*\x+2,16+1);
\foreach \x in {1,2,3,4,5,6,7,9,10,11,12,13,14,15} \draw[-latex] (8+1,2*\x+2) -- (8+1,2*\x);
\foreach \x in {1,2,3,4,5,6,7,9,10,11,12,13,14,15} \draw[-latex] (24+1,2*\x+2) -- (24+1,2*\x);
\foreach \x in {0,1,3,5,7,8,9,11,13,15,16} \draw[-latex] (16+1,2*\x +2) -- (16+1,2*\x);
\foreach \x in {0,...,17} \draw (0,2*\x) -- (34,2*\x);
\foreach \x in {0,...,17} \draw (2*\x,0) -- (2*\x,34);
\draw (8+0,16+1) -- (8+1,16+1);
\draw[-latex] (8+1.5,16+1) -- (8+2,16+1);
\draw (24+0,16+1) -- (24+1,16+1);
\draw[-latex] (24+1.5,16+1) -- (24+2,16+1);
\draw (16+0,24+1) -- (16+1,24+1);
\draw[-latex] (16+1.5,24+1) -- (16+2,24+1);
\draw (16+0,8+1) -- (16+1,8+1);
\draw[-latex] (16+1.5,8+1) -- (16+2,8+1);
\draw[color=red,domain=0:1] plot({8+2*\x-2*\x^2+1},{16+2*\x});
\draw[color=red,domain=0:1] plot({24+2*\x-2*\x^2+1},{16+2*\x});
\draw[color=red,domain=0:1] plot({16+2*\x-2*\x^2+1},{24+2*\x});
\draw[color=red,domain=0:1] plot({16+2*\x-2*\x^2+1},{8+2*\x});
\end{tikzpicture}\]
\caption{\label{pic.11} Example of a synchronization net in an order 
$2$ cell}
\end{figure}
\bigskip

\noindent \textbf{\textit{Symbols:}} \bigskip

The symbols of this third sublayer are the following 
ones : 
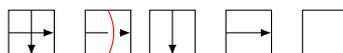
\[\begin{tikzpicture}[scale=0.6]
\draw (0,0) rectangle (1,1);
\draw[latex-] (0.5,0) -- (0.5,1);
\draw[-latex] (0,0.5) -- (1,0.5);
\end{tikzpicture}, \ \ \begin{tikzpicture}[scale=0.6]
\draw (0,0) rectangle (1,1);
\draw (0,0.5) -- (0.5,0.5);
\draw[-latex] (0.75,0.5) -- (1,0.5);
\draw[color=red,domain=0:1] plot({0.5*\x-0.5*\x^2+0.5},{\x});
\end{tikzpicture} \ \ \begin{tikzpicture}[scale=0.6]
\draw (0,0) rectangle (1,1);
\draw[latex-] (0.5,0) -- (0.5,1);
\end{tikzpicture}, \ \ \begin{tikzpicture}[scale=0.6]
\draw (0,0) rectangle (1,1);
\draw[-latex] (0,0.5) -- (1,0.5);
\end{tikzpicture}, \ \ \begin{tikzpicture}[scale=0.6]
\draw (0,0) rectangle (1,1);
\end{tikzpicture}\] 

\noindent \textbf{\textit{Local rules:}}

\begin{itemize}
\item The symbol \begin{tikzpicture}[scale=0.6,baseline=2mm]
\draw (0,0) rectangle (1,1);
\draw[latex-] (0.5,0) -- (0.5,1);
\draw[-latex] (0,0.5) -- (1,0.5);
\end{tikzpicture} is superimposed on and only on the 
nuclei, corners, and centers of the walls 
of cells. As a consequence, 
when a line of arrows 
crosses a column of arrows 
on another position, then the 
symbol superimposed on this position 
is $\begin{tikzpicture}[scale=0.6,baseline=2mm]
\draw (0,0) rectangle (1,1);
\draw (0,0.5) -- (0.5,0.5);
\draw[-latex] (0.75,0.5) -- (1,0.5);
\draw[color=red,domain=0:1] plot({0.5*\x-0.5*\x^2+0.5},{\x});
\end{tikzpicture}$.
\item The symbols 
\begin{tikzpicture}[scale=0.6,baseline=2mm]
\draw (0,0) rectangle (1,1);
\draw[latex-] (0.5,0) -- (0.5,1);
\end{tikzpicture} and 
\begin{tikzpicture}[scale=0.6,baseline=2mm]
\draw (0,0) rectangle (1,1);
\draw[-latex] (0,0.5) -- (1,0.5);
\end{tikzpicture} are transmitted in the direction of the arrow:
north and south for the first one, east and west for the second one.
\item Consider a position 
$\vec{u}$. If it is superimposed with 
a vertical (resp. horizontal) 
arrow, then the 
positions $\vec{u} \pm \vec{e}^1$ (resp. 
 $\vec{u} \pm \vec{e}^2$)
are not superimposed with a vertical (resp. horizontal) arrow. 
\end{itemize} 

\noindent \textbf{\textit{Global behavior:}} \bigskip

The first two rules build \textbf{wires} of the net,
defined to be infinite columns (resp. 
rows) whose positions are superimposed with 
a vertical (resp. horizontal) arrow. These 
are the columns (resp. rows) intersecting corners 
and nuclei of cells. The other columns (resp. rows) 
are not superimposed with a vertical (resp. horizontal
arrow): this is a consequence of the last rule.

The \textbf{intersections of the wires} are superimposed with one of the two first symbols 
of the alphabet: 
\[\begin{tikzpicture}[scale=0.6]
\draw (0,0) rectangle (1,1);
\draw[latex-] (0.5,0) -- (0.5,1);
\draw[-latex] (0,0.5) -- (1,0.5);
\end{tikzpicture},\begin{tikzpicture}[scale=0.6]
\draw (0,0) rectangle (1,1);
\draw (0,0.5) -- (0.5,0.5);
\draw[-latex] (0.75,0.5) -- (1,0.5);
\draw[color=red,domain=0:1] plot({0.5*\x-0.5*\x^2+0.5},{\x});
\end{tikzpicture}.\]
The first symbol is superimposed on 
nuclei and corners, where the frequency 
bits will be synchronized. The other 
intersections have the second symbol. 
On these positions, the frequency bits won't 
be synchronized. As a consequence, only 
cells having the same order are synchronized.

See on Figure~\ref{pic.11} an example of a net 
superimposed on an order 1 cell. 

\subsection{\label{basis} Basis layer}

The basis layer supports \textbf{random bits}. 
We recall that $i=\lfloor 4h \rfloor$. \bigskip

\noindent \textbf{\textit{Symbols:}} \bigskip

The symbols of this layer are 
0 and 1. \bigskip

\noindent \textbf{\textit{Local rules:}} 

\begin{itemize}
\item if $i=0$, then on any position $\vec{u}$ superimposed with a blue corner, 
the positions $\vec{u}+\vec{e}^1$, $\vec{u}+\vec{e}^1$ and 
$\vec{u}+\vec{e}^1 +\vec{e}^2$ are superimposed with the bit $0$.
\item if $i=1$, the positions $\vec{u}+\vec{e}^1$ and $\vec{u}+\vec{e}^1+\vec{e}^2$
are superimposed with $0$, and there is no constraint on 
the bit on position $\vec{u}+\vec{e}^2$.
\item if $i=2$, the position $\vec{u}+\vec{e}^1$ is superimposed with $0$.
\item if $i=3$, there is no constraint.
\end{itemize} 

\noindent \textbf{\textit{Global behavior:}} \bigskip

In this layer are superimposed random bits 
in $\{0,1\}$ on any position. According 
to the value of $i$, we impose constraints 
on positions that are not superimposed with 
a blue corner in the structure layer.
We dot this in such a way that the entropy produced by 
these bits is the maximal possible 
value smaller than $h$.

The positions with a blue corner are the ones 
where the random bits will be regulated
by frequency bits. These bits are themselves controlled by 
the machines in order to complete the entropy 
generated by the random bits on the other positions.

\subsection{\label{sec.frequency.bits} Frequency bits layer}

\noindent \textbf{\textit{Symbols:}} \bigskip

The symbols of this layer are $0,1$ and 
a blank symbol. \bigskip

\noindent \textbf{\textit{Local rules:}} 

\begin{itemize}
\item \textbf{Localization:} The non-blank symbols are 
superimposed to the 
cytoplasm positions.
\item \textbf{Synchronization on 
quarters:} two adjacent positions 
in the cytoplasm have the same frequency bit.
\end{itemize}

\noindent \textbf{\textit{Global behavior:}} \bigskip 

The bits $0,1$ are called \textbf{frequency bits}.
In a quarter, these bits are equal. Hence 
each quarter is attached with a unique 
frequency bit.

\subsection{\label{sec.cells.coding} 
Cells coding layer}

\noindent \textbf{\textit{Symbols:}} 

\[\left\{\begin{tikzpicture}[scale=0.5]
\filldraw[color=red] (0,0) circle (3pt);
\filldraw[color=yellow] (0.5,0) circle (3pt);
\end{tikzpicture}, \begin{tikzpicture}[scale=0.5]
\filldraw[color=red] (0,0) circle (3pt);
\filldraw[color=purple] (0,0.5) circle (3pt);
\filldraw[color=white] (0.5,0) circle (3pt);
\end{tikzpicture}, \begin{tikzpicture}[scale=0.5]
\filldraw[color=white] (0,0) circle (3pt);
\filldraw[color=purple] (0,0.5) circle (3pt);
\filldraw[color=orange] (0.5,0.5) circle (3pt);
\end{tikzpicture}, \begin{tikzpicture}[scale=0.5]
\filldraw[color=white] (0,0) circle (3pt);
\filldraw[color=yellow] (0.5,0) circle (3pt);
\filldraw[color=orange] (0.5,0.5) circle (3pt);
\end{tikzpicture}, \begin{tikzpicture}[scale=0.5]
\filldraw[color=red] (0,0) circle (3pt);
\filldraw[color=yellow] (0.5,0) circle (3pt);
\filldraw[color=orange] (0.5,0.5) circle (3pt);
\filldraw[color=purple] (0,0.5) circle (3pt);
\end{tikzpicture} \right\},\]
and a blank symbol. \bigskip

\noindent \textbf{\textit{Local rules:}} 

\begin{itemize}
\item The non-blank symbols are superimposed 
to the nuclei, the blank symbols to other positions.
\item The DNA symbol $\begin{tikzpicture}[scale=0.5,baseline=0.25mm]
\filldraw[color=red] (0,0) circle (3pt);
\filldraw[color=yellow] (0.5,0) circle (3pt);
\filldraw[color=orange] (0.5,0.5) circle (3pt);
\filldraw[color=purple] (0,0.5) circle (3pt);
\end{tikzpicture}$ is and can only be over a nucleus 
of an order $\le N$ cell. 
\item The others DNA symbols are over $>N$ order cells.
\end{itemize}

\noindent \textbf{\textit{Global behavior:}} \bigskip

On the nucleus of every cell is superimposed 
a symbol called the {\bf{DNA}}. 
It rules the behavior of each 
of the machines working in the 
cytoplasm, telling which ones of the machines 
execute simulation and which ones compute.

When the order of the cell is smaller or equal 
to $N$, then all the machines compute. 
When the order is greater than $N$, 
two of the machines compute and the other two 
execute simulation.

\subsection{\label{sec.synchr} Synchronization layer} 

This layer serves for the synchronization of the 
frequency bits corresponding to computation quarters. \bigskip

\noindent \textbf{\textit{Symbols:}} \bigskip
 
Elements of $\{0,1\}$ and $\{0,1\}^2$,
and a blank symbol. \bigskip

\noindent \textbf{\textit{Local rules:}} 
 
\begin{enumerate}
\item \textbf{Localization:}

\begin{itemize}
\item  
The non-blank symbols are superimposed on and 
only on positions having a non-blank symbol 
in the synchronization net sublayer.
\item Positions superimposed 
with \begin{tikzpicture}[scale=0.6,baseline=2mm]
\draw (0,0) rectangle (1,1);
\draw[latex-] (0.5,0) -- (0.5,1);
\draw[-latex] (0,0.5) -- (1,0.5);
\end{tikzpicture}, \begin{tikzpicture}[scale=0.6,baseline=2mm]
\draw (0,0) rectangle (1,1);
\draw[latex-] (0.5,0) -- (0.5,1);
\end{tikzpicture}, or \begin{tikzpicture}[scale=0.6,baseline=2mm]
\draw (0,0) rectangle (1,1);
\draw[-latex] (0,0.5) -- (1,0.5);
\end{tikzpicture}
are 
superimposed in the present layer 
with an element of $\{0,1\}$.  
The positions superimposed with 
\begin{tikzpicture}[scale=0.6,baseline=2mm]
\draw (0,0) rectangle (1,1);
\draw (0,0.5) -- (0.5,0.5);
\draw[-latex] (0.75,0.5) -- (1,0.5);
\draw[color=red,domain=0:1] plot({0.5*\x-0.5*\x^2+0.5},{\x});
\end{tikzpicture} are superimposed 
with an element of $\{0,1\}^2$.
\end{itemize}

\item \textbf{Synchronization:} 

\begin{itemize}
\item On a nucleus, if 
the DNA symbol is $\begin{tikzpicture}[scale=0.5,baseline=0.25mm]
\filldraw[color=red] (0,0) circle (3pt);
\filldraw[color=yellow] (0.5,0) circle (3pt);
\filldraw[color=orange] (0.5,0.5) circle (3pt);
\filldraw[color=purple] (0,0.5) circle (3pt);
\end{tikzpicture}$, 
then the bit in this layer is equal to the frequency bit 
on north east, south east, south west and 
north west positions.

\item When the DNA symbol is not $\begin{tikzpicture}[scale=0.5,baseline=0.25mm]
\filldraw[color=red] (0,0) circle (3pt);
\filldraw[color=yellow] (0.5,0) circle (3pt);
\filldraw[color=orange] (0.5,0.5) circle (3pt);
\filldraw[color=purple] (0,0.5) circle (3pt);
\end{tikzpicture}$, then the bit in this layer 
is equal to the frequency bits corresponding 
to the colors represented in the DNA symbol.
\end{itemize}

\item \textbf{Information transfer rules:}

\begin{itemize}
\item 
Considering two adjacent positions with \begin{tikzpicture}[scale=0.6,baseline=2mm]
\draw (0,0) rectangle (1,1);
\draw[latex-] (0.5,0) -- (0.5,1);
\end{tikzpicture}, or \begin{tikzpicture}[scale=0.6,baseline=2mm]
\draw (0,0) rectangle (1,1);
\draw[-latex] (0,0.5) -- (1,0.5);
\end{tikzpicture}, 
the bits superimposed on these two positions are equal.
\item On a position with \begin{tikzpicture}[scale=0.6,baseline=2mm]
\draw (0,0) rectangle (1,1);
\draw (0,0.5) -- (0.5,0.5);
\draw[-latex] (0.75,0.5) -- (1,0.5);
\draw[color=red,domain=0:1] plot({0.5*\x-0.5*\x^2+0.5},{\x});
\end{tikzpicture} symbol, the first bit of the couple 
is equal to the bit of positions on north and south.
The second one to the bit of positions on west and east.
\item On a position superimposed with 
\begin{tikzpicture}[scale=0.6,baseline=2mm]
\draw (0,0) rectangle (1,1);
\draw[latex-] (0.5,0) -- (0.5,1);
\draw[-latex] (0,0.5) -- (1,0.5);
\end{tikzpicture}, 
the bit is equal to the bit on south, west, east and north positions.
\end{itemize}
\end{enumerate} 

\noindent \textbf{\textit{Global behavior:}} \bigskip

Using the information contained in the nucleus (the DNA), 
we synchronize the two frequency bits of the computation quarters.
These bits are transmitted through the wires of 
the synchronization net. They are synchronized on the positions 
having the symbol \begin{tikzpicture}[scale=0.6,baseline=2mm]
\draw (0,0) rectangle (1,1);
\draw[latex-] (0.5,0) -- (0.5,1);
\draw[-latex] (0,0.5) -- (1,0.5);
\end{tikzpicture}. 
As a consequence, when $n \le N$, all the frequency 
bits of order $n$ cells are equal. 
When $n>N$, the frequency bits corresponding 
to computation quarters are synchronized. 
The other two are left free and are not synchronized 
between cells.

\subsection{\label{sec.computation.areas} Computation areas layer} 

\begin{figure}[ht]\[\begin{tikzpicture}[scale=0.125]
\fill[gray!90] (-2,-2) rectangle (64,64);
\fill[gray!40] (0,0) rectangle (64,64);
\fill[red!40] (0,0) rectangle (2*31,2*31);
\supertilesynchroone{-4*16+66}{-1*16+66}
\supertilesynchroone{-3*16+66}{-1*16+66}
\supertilesynchroone{-2*16+66}{-1*16+66}
\supertilesynchroone{-1*16+66}{-1*16+66}

\supertilesynchroone{-4*16+66}{-2*16+66}
\supertilesynchroone{-3*16+66}{-2*16+66}
\supertilesynchroone{-2*16+66}{-2*16+66}
\supertilesynchroone{-1*16+66}{-2*16+66}

\supertilesynchroone{-4*16+66}{-3*16+66}
\supertilesynchroone{-3*16+66}{-3*16+66}
\supertilesynchroone{-2*16+66}{-3*16+66}
\supertilesynchroone{-1*16+66}{-3*16+66}

\supertilesynchrotwo{16}{16}

\supertilesynchroone{-4*16+66}{-4*16+66}
\supertilesynchroone{-3*16+66}{-4*16+66}
\supertilesynchroone{-2*16+66}{-4*16+66}
\supertilesynchroone{-1*16+66}{-4*16+66}

\fill[blue!50] (16+0,16) rectangle (16+2,2+16);
\fill[blue!50] (16+0,28) rectangle (16+2,2+28);
\fill[blue!50] (28+0,28) rectangle (28+2,2+28);
\fill[blue!50] (28+0,16) rectangle (28+2,2+16);

\fill[blue!50] (16+0,16+16) rectangle (16+2,2+16+16);
\fill[blue!50] (16+0,28+16) rectangle (16+2,2+28+16);
\fill[blue!50] (28+0,28+16) rectangle (28+2,2+28+16);
\fill[blue!50] (28+0,16+16) rectangle (28+2,2+16+16);

\fill[blue!50] (16+16+0,16) rectangle (16+16+2,2+16);
\fill[blue!50] (16+16+0,28) rectangle (16+16+2,2+28);
\fill[blue!50] (16+28+0,28) rectangle (16+28+2,2+28);
\fill[blue!50] (16+28+0,16) rectangle (16+28+2,2+16);

\fill[blue!50] (20,20) rectangle (22,22);
\fill[blue!50] (20,24) rectangle (22,26);
\fill[blue!50] (24,20) rectangle (26,22);
\fill[blue!50] (24,24) rectangle (26,26);

\fill[blue!50] (36,20) rectangle (38,22);
\fill[blue!50] (36,24) rectangle (38,26);
\fill[blue!50] (40,20) rectangle (42,22);
\fill[blue!50] (40,24) rectangle (42,26);

\fill[blue!50] (36,36) rectangle (38,38);
\fill[blue!50] (36,40) rectangle (38,42);
\fill[blue!50] (40,36) rectangle (42,38);
\fill[blue!50] (40,40) rectangle (42,42);

\fill[blue!50] (20,36) rectangle (22,38);
\fill[blue!50] (20,40) rectangle (22,42);
\fill[blue!50] (24,36) rectangle (26,38);
\fill[blue!50] (24,40) rectangle (26,42);

\fill[blue!50] (0,0) rectangle (2,2);
\fill[blue!50] (12+0,0) rectangle (12+2,2);
\fill[blue!50] (48+0,0) rectangle (48+2,2);
\fill[blue!50] (60+0,0) rectangle (60+2,2);

\fill[blue!50] (0,12) rectangle (2,2+12);
\fill[blue!50] (12+0,12) rectangle (12+2,2+12);
\fill[blue!50] (48+0,12) rectangle (48+2,2+12);
\fill[blue!50] (60+0,12) rectangle (60+2,2+12);

\fill[blue!50] (0,48) rectangle (2,2+48);
\fill[blue!50] (12+0,48) rectangle (12+2,2+48);
\fill[blue!50] (48+0,48) rectangle (48+2,2+48);
\fill[blue!50] (60+0,48) rectangle (60+2,2+48);

\fill[blue!50] (0,60) rectangle (2,2+60);
\fill[blue!50] (12+0,60) rectangle (12+2,2+60);
\fill[blue!50] (48+0,60) rectangle (48+2,2+60);
\fill[blue!50] (60+0,60) rectangle (60+2,2+60);

\fill[blue!50] (16+16+0,16+16) rectangle (16+16+2,16+2+16);
\fill[blue!50] (16+16+0,16+28) rectangle (16+16+2,16+2+28);
\fill[blue!50] (16+28+0,16+28) rectangle (16+28+2,16+2+28);
\fill[blue!50] (16+28+0,16+16) rectangle (16+28+2,16+2+16);

\fill[blue!50] (4,4) rectangle (6,6);
\fill[blue!50] (8,4) rectangle (10,6);
\fill[blue!50] (4,8) rectangle (6,10);
\fill[blue!50] (8,8) rectangle (10,10);
\fill[blue!50] (20,4) rectangle (22,6);
\fill[blue!50] (24,4) rectangle (26,6);
\fill[blue!50] (20,8) rectangle (22,10);
\fill[blue!50] (24,8) rectangle (26,10);
\fill[blue!50] (36,4) rectangle (38,6);
\fill[blue!50] (40,4) rectangle (42,6);
\fill[blue!50] (36,8) rectangle (38,10);
\fill[blue!50] (40,8) rectangle (42,10);
\fill[blue!50] (16+36,4) rectangle (16+38,6);
\fill[blue!50] (16+40,4) rectangle (16+42,6);
\fill[blue!50] (16+36,8) rectangle (16+38,10);
\fill[blue!50] (16+40,8) rectangle (16+42,10);

\fill[blue!50] (4,16+4) rectangle (6,16+6);
\fill[blue!50] (8,16+4) rectangle (10,16+6);
\fill[blue!50] (4,16+8) rectangle (6,16+10);
\fill[blue!50] (8,16+8) rectangle (10,16+10);
\fill[blue!50] (20,16+4) rectangle (22,16+6);
\fill[blue!50] (24,16+4) rectangle (26,16+6);
\fill[blue!50] (20,16+8) rectangle (22,16+10);
\fill[blue!50] (24,16+8) rectangle (26,16+10);
\fill[blue!50] (36,16+4) rectangle (38,16+6);
\fill[blue!50] (40,16+4) rectangle (42,16+6);
\fill[blue!50] (36,16+8) rectangle (38,16+10);
\fill[blue!50] (40,16+8) rectangle (42,16+10);
\fill[blue!50] (16+36,16+4) rectangle (16+38,16+6);
\fill[blue!50] (16+40,16+4) rectangle (16+42,16+6);
\fill[blue!50] (16+36,16+8) rectangle (16+38,16+10);
\fill[blue!50] (16+40,16+8) rectangle (16+42,16+10);

\fill[blue!50] (4,32+4) rectangle (6,32+6);
\fill[blue!50] (8,32+4) rectangle (10,32+6);
\fill[blue!50] (4,32+8) rectangle (6,32+10);
\fill[blue!50] (8,32+8) rectangle (10,32+10);
\fill[blue!50] (20,32+4) rectangle (22,32+6);
\fill[blue!50] (24,32+4) rectangle (26,32+6);
\fill[blue!50] (20,32+8) rectangle (22,32+10);
\fill[blue!50] (24,32+8) rectangle (26,32+10);
\fill[blue!50] (36,32+4) rectangle (38,32+6);
\fill[blue!50] (40,32+4) rectangle (42,32+6);
\fill[blue!50] (36,32+8) rectangle (38,32+10);
\fill[blue!50] (40,32+8) rectangle (42,32+10);
\fill[blue!50] (16+36,32+4) rectangle (16+38,32+6);
\fill[blue!50] (16+40,32+4) rectangle (16+42,32+6);
\fill[blue!50] (16+36,32+8) rectangle (16+38,32+10);
\fill[blue!50] (16+40,32+8) rectangle (16+42,32+10);

\fill[blue!50] (4,48+4) rectangle (6,48+6);
\fill[blue!50] (8,48+4) rectangle (10,48+6);
\fill[blue!50] (4,48+8) rectangle (6,48+10);
\fill[blue!50] (8,48+8) rectangle (10,48+10);
\fill[blue!50] (20,48+4) rectangle (22,48+6);
\fill[blue!50] (24,48+4) rectangle (26,48+6);
\fill[blue!50] (20,48+8) rectangle (22,48+10);
\fill[blue!50] (24,48+8) rectangle (26,48+10);
\fill[blue!50] (36,48+4) rectangle (38,48+6);
\fill[blue!50] (40,48+4) rectangle (42,48+6);
\fill[blue!50] (36,48+8) rectangle (38,48+10);
\fill[blue!50] (40,48+8) rectangle (42,48+10);
\fill[blue!50] (16+36,48+4) rectangle (16+38,48+6);
\fill[blue!50] (16+40,48+4) rectangle (16+42,48+6);
\fill[blue!50] (16+36,48+8) rectangle (16+38,48+10);
\fill[blue!50] (16+40,48+8) rectangle (16+42,48+10);
\draw[latex-latex] (2,1) -- (12,1);
\draw[latex-latex] (50,1) -- (60,1);
\draw[latex-latex] (14,1) -- (48,1);
\draw[latex-latex] (2,13) -- (12,13);
\draw[latex-latex] (50,13) -- (60,13);
\draw[latex-latex] (14,13) -- (48,13);
\draw[latex-latex] (2,49) -- (12,49);
\draw[latex-latex] (50,49) -- (60,49);
\draw[latex-latex] (14,49) -- (48,49);
\draw[latex-latex] (2,61) -- (12,61);
\draw[latex-latex] (50,61) -- (60,61);
\draw[latex-latex] (14,61) -- (48,61);

\draw[latex-] (1,2) -- (1,12);
\draw[latex-] (1,50) -- (1,60);
\draw[latex-] (1,14) -- (1,48);
\draw[latex-] (13,2) -- (13,12);
\draw[latex-] (13,50) -- (13,60);
\draw[latex-] (13,14) -- (13,48);
\draw[latex-] (49,2) -- (49,12);
\draw[latex-] (49,50) -- (49,60);
\draw[latex-] (49,14) -- (49,48);
\draw[latex-] (61,2) -- (61,12);
\draw[latex-] (61,50) -- (61,60);
\draw[latex-] (61,14) -- (61,48);
\end{tikzpicture}\]
\caption{\label{fig.computation.area.init} Example 
of a computation area in 
a computation quarter in an order $2$ cell.}
\end{figure}

This layer specifies the function of each
position of the cytoplasm relatively 
to the Turing machines: 
information transfer (vertical or horizontal) or execution 
of one computation step. This is done 
as in \cite{R71}. However, the 
constitution of the computation areas in infinite cells 
is not well controlled. That is the reason why, in order to ensure the linear net gluing 
property, we simulate degenerated behaviors for the 
constitution of computation areas in the 
simulation quarters. We use 
error signals to ensure that the computation areas are 
well constituted 
in the computation quarters. \bigskip

\noindent \textbf{\textit{Symbols:}} \bigskip

Elements of $\{\texttt{in},\texttt{out}\}^2 \times \{\texttt{in},\texttt{out}\}^2$, 
elements of \[\left\{\begin{tikzpicture}[scale=0.5]
\filldraw[color=white] (0,0) circle (3pt); 
\filldraw[color=white] (0.5,0) circle (3pt);
\filldraw[color=purple] (0,0.5) circle (3pt);\end{tikzpicture},
\begin{tikzpicture}[scale=0.5]
\filldraw[color=white] (0,0) circle (3pt);
\filldraw[color=orange] (0.5,0.5) circle (3pt);\end{tikzpicture},
\begin{tikzpicture}[scale=0.5]
\filldraw[color=white] (0,0) circle (3pt);
\filldraw[color=yellow] (0.5,0) circle (3pt);\end{tikzpicture}, 
\begin{tikzpicture}[scale=0.5]
\filldraw[color=red] (0,0) circle (3pt);
\filldraw[color=white] (0.5,0) circle(3pt); \end{tikzpicture}, 
\begin{tikzpicture}[scale=0.5]
\filldraw[color=red] (0,0) circle (3pt);
\filldraw[color=yellow] (0.5,0) circle (3pt);
\end{tikzpicture}, \begin{tikzpicture}[scale=0.5]
\filldraw[color=red] (0,0) circle (3pt);
\filldraw[color=purple] (0,0.5) circle (3pt);
\filldraw[color=white] (0.5,0) circle (3pt);
\end{tikzpicture}, \begin{tikzpicture}[scale=0.5]
\filldraw[color=white] (0,0) circle (3pt);
\filldraw[color=purple] (0,0.5) circle (3pt);
\filldraw[color=orange] (0.5,0.5) circle (3pt);
\end{tikzpicture}, \begin{tikzpicture}[scale=0.5]
\filldraw[color=white] (0,0) circle (3pt);
\filldraw[color=yellow] (0.5,0) circle (3pt);
\filldraw[color=orange] (0.5,0.5) circle (3pt);
\end{tikzpicture}\right\},\]
elements of 
\[\left\{\begin{tikzpicture}[scale=0.4]
\draw (0,0) rectangle (1,1); 
\draw[-latex] (0.5,0) -- (0.5,1);
\end{tikzpicture}, 
\begin{tikzpicture}[scale=0.4]
\draw (0,0) rectangle (1,1); 
\draw[-latex] (0.5,1) -- (0.5,0);
\end{tikzpicture}, \begin{tikzpicture}[scale=0.4]
\draw (0,0) rectangle (1,1); 
\draw[-latex] (0,0.5) -- (1,0.5);
\end{tikzpicture}, \begin{tikzpicture}[scale=0.4]
\draw (0,0) rectangle (1,1); 
\draw[-latex] (1,0.5) -- (0,0.5); \end{tikzpicture}
, \begin{tikzpicture}[scale=0.4]
\draw (0,0) rectangle (1,1); 
\draw[latex-] (0.5,0) -- (0.5,0.5);
\draw (0.5,0.5) -- (1,0.5);
\end{tikzpicture}, \begin{tikzpicture}[scale=0.4]
\draw (0,0) rectangle (1,1); 
\draw (0.5,0) -- (0.5,0.5);
\draw[-latex] (0.5,0.5) -- (1,0.5);
\end{tikzpicture}
, \begin{tikzpicture}[scale=0.4]
\draw (0,0) rectangle (1,1); 
\draw (0.5,0) -- (0.5,0.5);
\draw[-latex] (0.5,0.5) -- (0,0.5);
\end{tikzpicture},
\begin{tikzpicture}[scale=0.4]
\draw (0,0) rectangle (1,1); 
\draw[latex-] (0.5,0) -- (0.5,0.5);
\draw (0.5,0.5) -- (0,0.5);
\end{tikzpicture}, \begin{tikzpicture}[scale=0.4]
\draw (0,0) rectangle (1,1); 
\draw (0,0.5) -- (0.5,0.5);
\draw[-latex] (0.5,0.5) -- (0.5,1);
\end{tikzpicture}, \begin{tikzpicture}[scale=0.4]
\draw (0,0) rectangle (1,1); 
\draw[latex-] (0,0.5) -- (0.5,0.5);
\draw (0.5,0.5) -- (0.5,1);
\end{tikzpicture}, \begin{tikzpicture}[scale=0.4]
\draw (0,0) rectangle (1,1); 
\draw (0.5,0.5) -- (0.5,1);
\draw[-latex] (0.5,0.5) -- (1,0.5);
\end{tikzpicture}, \begin{tikzpicture}[scale=0.4]
\draw (0,0) rectangle (1,1); 
\draw[-latex] (0.5,0.5) -- (0.5,1);
\draw (0.5,0.5) -- (1,0.5);
\end{tikzpicture}\right\},\]
and a blank symbol. \bigskip

The first set corresponds to signals that propagate
horizontally and vertically in the cytoplasm. 
The second set corresponds to error 
signals propagating on the reticle and to the nucleus.
The last ones correspond to the propagation of the 
error signals through the walls. \bigskip

\noindent \textbf{\textit{Local rules:}}

\begin{itemize}
\item \textbf{Localization:} 
\begin{itemize}
\item 
The elements of 
$\{\texttt{in},
\texttt{out}\} ^2 \times \{\texttt{in},
\texttt{out}\} ^2$ are superimposed on
and only on positions in the cytoplasm.
The first two coordinates are associated to the horizontal direction, 
and the second two to the vertical direction.
\item The elements of 
$\left\{\begin{tikzpicture}[scale=0.5]
\filldraw[color=white] (0,0) circle (3pt); 
\filldraw[color=white] (0.5,0) circle (3pt);
\filldraw[color=purple] (0,0.5) circle (3pt);\end{tikzpicture},
\begin{tikzpicture}[scale=0.5]
\filldraw[color=white] (0,0) circle (3pt);
\filldraw[color=orange] (0.5,0.5) circle (3pt);\end{tikzpicture},
\begin{tikzpicture}[scale=0.5]
\filldraw[color=white] (0,0) circle (3pt);
\filldraw[color=yellow] (0.5,0) circle (3pt);\end{tikzpicture}, 
\begin{tikzpicture}[scale=0.5]
\filldraw[color=red] (0,0) circle (3pt);
\filldraw[color=white] (0.5,0) circle(3pt); \end{tikzpicture}, 
\begin{tikzpicture}[scale=0.5]
\filldraw[color=red] (0,0) circle (3pt);
\filldraw[color=yellow] (0.5,0) circle (3pt);
\end{tikzpicture}, \begin{tikzpicture}[scale=0.5]
\filldraw[color=red] (0,0) circle (3pt);
\filldraw[color=purple] (0,0.5) circle (3pt);
\filldraw[color=white] (0.5,0) circle (3pt);
\end{tikzpicture}, \begin{tikzpicture}[scale=0.5]
\filldraw[color=white] (0,0) circle (3pt);
\filldraw[color=purple] (0,0.5) circle (3pt);
\filldraw[color=orange] (0.5,0.5) circle (3pt);
\end{tikzpicture}, \begin{tikzpicture}[scale=0.5]
\filldraw[color=white] (0,0) circle (3pt);
\filldraw[color=yellow] (0.5,0) circle (3pt);
\filldraw[color=orange] (0.5,0.5) circle (3pt);
\end{tikzpicture}\right\}$ can be superimposed 
only on reticle positions.
\item All the other positions are superimposed 
with the blank symbol.
\end{itemize}
\item \textbf{Transmission of the cytoplasm 
signals:} on a cytoplasm position $\vec{u}$, 
the two first coordinates of the symbol 
are transmitted to the positions $\vec{u} \pm \vec{e}^1$
and the second two are transmitted 
to the positions $\vec{u} \pm \vec{e}^2$, 
when these positions are in the cytoplasm. \bigskip

In a computation 
quarter, these are signals allowing 
 the lines and columns of the 
cells which do not intersect a smaller cell to be 
specified. 
In the first (resp. second) 
couple, the symbol $\texttt{in}$ 
for the first coordinate
corresponds to the fact that 
the position is in a segment of row (resp. column)
originating from the inside of the cell wall. 
The symbol $\texttt{out}$ 
corresponds to the fact that 
the position is in a segment of row (resp. column)
originating from the outside of the cell wall. 
For the second coordinate, these symbols 
have the same signification concerning 
the end of the segment instead. The next 
rules impose that when near the pertinent parts of 
a cell and inside it, 
if a symbol $\texttt{in} - \texttt{out}$ does not 
correspond to the nature of the 
origin or end at this position, this triggers 
an \textit{error signal}. When 
outside the cell, the origin or end 
is imposed - thus not triggering 
an error signal. These rules are presented 
for positions in the \textbf{red quarter}: 
similar rules are imposed for the other ones.

\item \textbf{Triggering error signals 
(inside the wall and reticle):} 
On a position $\vec{u}$ 
in the horizontal arm of the reticle 
(specified by having a reticle symbol different from 
the nucleus and having a reticle position on 
the right and on the left), if 
the position $\vec{u} - \vec{e}^2$ 
has its second couple having second coordinate 
equal to $\texttt{out}$, then
the red quarter is represented in the symbol 
superimposed on position $\vec{u}$. 

For instance, the pattern 
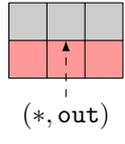
\[\begin{tikzpicture}[scale=0.5]
\fill[gray!40] (0,1) rectangle (3,2);
\fill[red!40] (0,0) rectangle (3,1);
\draw (0,0) grid (3,2);
\node at (1.5,-1) {$(*,\texttt{out})$};
\draw[dashed,-latex] (1.5,-0.5) -- (1.5,1);
\end{tikzpicture}
,\]
where $*$ means any symbol, and the couple 
represented in $\{\texttt{in},\texttt{out}\}^2$ 
is the second one, implies the following: 

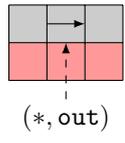
\[\begin{tikzpicture}[scale=0.5]
\fill[gray!40] (0,1) rectangle (3,2);
\fill[red!40] (0,0) rectangle (3,1);
\draw (0,0) grid (3,2);
\node at (1.5,-1) {$(*,\texttt{out})$};
\draw[-latex] (1,1.5) -- (2,1.5);
\draw[dashed,-latex] (1.5,-0.5) -- (1.5,1);
\end{tikzpicture}
.\]

On a position $\vec{u}$ 
in the vertical arm of the reticle 
(specified by having a reticle symbol different from 
the nucleus and having a reticle position on 
the top and bottom), if 
the position $\vec{u} - \vec{e}^1$ 
has its first couple having second coordinate 
equal to $\texttt{out}$, 
the red quarter is represented in the symbol 
superimposed on position $\vec{u}$.

\item On a position $\vec{u}$ in the 
horizontal part of the wall (specified
by having a wall symbol different from the corner, 
and having a wall position on left and right), 
if the position $\vec{u} + \vec{e}^2$ 
has its second couple having first coordinate 
equal to $\texttt{out}$, then the symbol 
on the position $\vec{u}$ is an arrow symbol 
\[\begin{tikzpicture}[scale=0.4]
\draw (0,0) rectangle (1,1); 
\draw[-latex] (0,0.5) -- (1,0.5);
\end{tikzpicture} \ \text{or} \ 
\begin{tikzpicture}[scale=0.4]
\draw (0,0) rectangle (1,1); 
\draw[-latex] (1,0.5) -- (0,0.5); \end{tikzpicture}.\]
On a position $\vec{u}$ in the 
vertical part of the wall (specified
by having a wall symbol different from the corner, 
and having a wall position on top and bottom), 
if the position $\vec{u} + \vec{e}^1$ 
has its second couple having first coordinate 
equal to $\texttt{out}$, then the symbol 
on the position $\vec{u}$ is an arrow symbol 
\[\begin{tikzpicture}[scale=0.4]
\draw (0,0) rectangle (1,1); 
\draw[-latex] (0.5,0) -- (0.5,1);
\end{tikzpicture} \ \text{or} \ 
\begin{tikzpicture}[scale=0.4]
\draw (0,0) rectangle (1,1); 
\draw[-latex] (0.5,1) -- (0.5,0);
\end{tikzpicture}.\]

\item \textbf{Enforcing cytoplasm signals 
(outside the wall):}

Considering a wall position $\vec{u}$ which is on the west (resp. east, north, south) wall 
of a cell, the position $\vec{u} - \vec{e}^1$ (resp. $\vec{u}+\vec{e}^1$, $\vec{u}+\vec{e}^2$, 
$\vec{u}-\vec{e}^2$ has the second coordinate 
of its second couple (resp. 
first coordinate and second couple, first 
coordinate first couple, 
second coordinate first couple) equal 
to $\texttt{out}$.

\item {\bf{Propagation of error signals}}. 
An arrow symbol propagates in the direction 
pointed by the arrow 
on the wall, while the next position in this 
direction is not near 
a reticle position, as in the following pattern: 
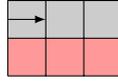
\[\begin{tikzpicture}[scale=0.5]
\fill[gray!40] (0,1) rectangle (3,2); 
\fill[red!40] (0,0) rectangle (3,1);
\draw (0,0) grid (3,2) ; 
\draw[-latex] (0,1.5) -- (1,1.5);
\end{tikzpicture}\]
implies the following one: 
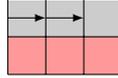
\[\begin{tikzpicture}[scale=0.5]
\fill[gray!40] (0,1) rectangle (3,2); 
\fill[red!40] (0,0) rectangle (3,1);
\draw (0,0) grid (3,2) ; 
\draw[-latex] (0,1.5) -- (1,1.5);
\draw[-latex] (1,1.5) -- (2,1.5);
\end{tikzpicture}\]

\item On a position $\vec{u}$ of the north (resp. east, south, west) arm of the reticle, specified 
by the colors on the sides, if the position $\vec{u}-\vec{e}^1$ (resp. $\vec{u}-\vec{e}^1$, 
$\vec{u}+\vec{e}^2$, $\vec{u}+\vec{e}^1$) is not the nucleus, then the symbol on this position contains the symbol 
on position $\vec{u}$.

\item \textbf{Connection between error signals:} when on a position $\vec{u}$
on the wall which is near a position on the reticle, if one of the wall 
symbols aside contains an error symbol, then the reticle position 
has an error symbol where the corresponding quarter is represented. For instance, 
the pattern 
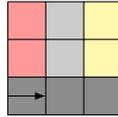
\[\begin{tikzpicture}[scale=0.5]
\fill[red!40] (0,0) rectangle (1,3);
\fill[yellow!40] (2,0) rectangle (3,3);
\fill[gray!90] (0,0) rectangle (3,1);
\fill[gray!40] (1,1) rectangle (2,3);
\draw (0,0) grid (3,3);
\draw[-latex] (0,0.5) -- (1,0.5);
\end{tikzpicture}\]
implies the following: 
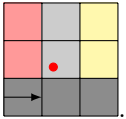
\[\begin{tikzpicture}[scale=0.5]
\fill[red!40] (0,0) rectangle (1,3);
\fill[yellow!40] (2,0) rectangle (3,3);
\fill[gray!90] (0,0) rectangle (3,1);
\fill[gray!40] (1,1) rectangle (2,3);
\draw (0,0) grid (3,3);
\draw[-latex] (0,0.5) -- (1,0.5);
\filldraw[color=red] (1.3,1.3) circle (3pt);
\end{tikzpicture}.\]

\item {\bf{Forbidding wrong error signals}}. 
On any of the four reticle positions around the nucleus, 
there can not be a symbol that contains a color which is in the DNA. For instance, 
the following pattern is forbidden: 

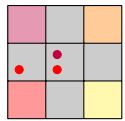
\[\begin{tikzpicture}[scale=0.5]
\fill[gray!40] (0,0) rectangle (3,3);
\fill[purple!40] (0,2) rectangle (1,3);
\fill[orange!40] (2,2) rectangle (3,3);
\fill[red!40] (0,0) rectangle (1,1);
\fill[yellow!40] (2,0) rectangle (3,1);
\draw (0,0) grid (3,3);
\filldraw[color=red] (0.3,1.3) circle (3pt);
\filldraw[color=red] (1.3,1.3) circle (3pt);
\filldraw[color=purple] (1.3,1.7) circle (3pt);
\end{tikzpicture}\]

\end{itemize} 

\noindent \textbf{\textit{Global behavior:}} \bigskip

In any quarter of a cell, the segment of rows and columns are colored with a couple 
of symbols in $\{\texttt{in},\texttt{out}\}$. One for the origin of the segment, 
the other one for the end of it. If it originates from or ends at the outside 
of a cell, then the corresponding symbol is forced to be $\texttt{out}$. \bigskip

Moreover, when near the walls or the reticle and 
inside the corresponding cell, 
if the corresponding symbol is $\texttt{out}$
then an error signal is triggered and propagates 
to the nucleus. 
On the walls, a propagation direction is chosen. 
In the reticle, the 
error signals contain the information about the quarter where 
the error was detected. Around the nucleus, we forbid 
an error signal to come from 
a computation quarter. \bigskip

In a \textit{computation quarter} of a cell, each row which 
does not intersect 
a smaller cell has first couple equal 
to $(\texttt{in},\texttt{in})$, 
since it originates inside the cell, and ends inside, 
on the reticle.
Each column which does not intersect a smaller cell 
has second couple equal to 
$(\texttt{in},\texttt{in})$. The couples on other 
segments of rows or columns are determined 
in a similar way, according to their origin and end.
This is enforced by the propagation of 
error signals to the nucleus.

The positions marked with 
$((\texttt{in},\texttt{in}), (\texttt{in},\texttt{in}))$ 
are called \textbf{computation positions}. The ones that have 
first (resp. second) couple equal 
to $(\texttt{in},\texttt{in})$ 
and second (resp. first) not equal 
to $(\texttt{in},\texttt{in})$ 
are \textbf{horizontal transfer positions} (resp. vertical 
transfer positions). 
See Figure \ref{fig.computation.area.init} for a representation of a computation area in 
the red quarter of an order two cell. On this figure, 
computation positions are represented by a blue square.
Vertical and horizontal transfer positions by arrows in this direction.

\begin{remark}
These mechanisms can not be easily simplified, since 
an infinite row or a column can not ''know'' if it is a free row or column 
of its infinite cell. Moreover, from the division of the cells it is difficult 
to code this with a hierarchical process.
\end{remark}

\begin{remark}
In the literature, most of the constructions 
using substitutions include
the construction of the computation areas~\cite{Hochman-Meyerovitch-2010} with 
substitution rules. However, in order to get 
the net gluing property, and furthermore the 
block gluing property, we need a more flexible construction 
of the computation areas. The method 
presented above was used initially in the construction 
of Robinson for his undecidability result~\cite{R71}. 
\end{remark}

\subsection{\label{sec.machine} The machines (RNA)}

In this section, we present the implementation 
of Turing machines 
that will check that the frequency bit of 
level $n$ 
cells are equal to $s_n$, for all $n$.

In order 
to have the linear block gluing property, we have 
to adapt the Turing machine model 
in order to simulate each possible 
degenerated behavior of the machines. This is done as follows: 
in each of the quarters of a cell, a machine 
is implemented. For this machine, the directions 
of space and time are as on Figure~\ref{fig.time.space}: 
the rules of the machine 
will depend on the color of the quarter. 
Moreover, for each of the quarters, we initialize the 
tape with elements of 
$\mathcal{A} \times \mathcal{Q}$. The set $\mathcal{Q}$ is 
the state set 
of the machine and $\mathcal{A}$ its alphabet. Machine heads can enter on the two sides of 
the computation area. Signals will be used to verify that in the computation quarters
the machine is well initialized. This means 
that no head enters on the sides, and on 
the initial row there is a unique machine 
head on the position near the nucleus in initial state. 
Moreover, all the letters in $\mathcal{A}$ are blank.

\begin{figure}[ht]
\[\begin{tikzpicture}[scale=0.4]
\fill[gray!90] (0,0) rectangle (11,11);
\fill[gray!40] (0.5,0.5) rectangle (10.5,10.5);
\fill[red!50] (0.5,0.5) rectangle (5.25,5.25);
\fill[orange!50] (5.75,5.75) rectangle (10.5,10.5);
\fill[yellow!50] (5.75,0.5) rectangle (10.5,5.25); 
\fill[purple!50] (0.5,5.75) rectangle (5.25,10.5); 
\draw (0,0) rectangle (11,11); 
\draw (5.25,0.5) rectangle (5.75,10.5);
\draw (0.5,5.25) rectangle (10.5,5.75);
\draw (0.5,0.5) rectangle (10.5,10.5);
\draw[-latex] (6.5,6.5) -- (9.5,6.5);
\draw[-latex] (6.5,6.5) -- (6.5,9.5);
\draw[-latex] (4.5,4.5) -- (1.5,4.5);
\draw[-latex] (4.5,4.5) -- (4.5,1.5);
\draw[-latex] (4.5,6.5) -- (4.5,9.5);
\draw[-latex] (4.5,6.5) -- (1.5,6.5);
\draw[-latex] (6.5,4.5) -- (9.5,4.5);
\draw (6.5,4.5) -- (6.5,1.5);
\node at (2,3.5) {space};
\node at (2,7.5) {space};
\node at (3.5,1.5) {time};
\node at (3.5,9.5) {time};

\node at (7.75,9.5) {time};
\node at (9,7.5) {space};

\node at (7.75,1.5) {time};
\node at (9,3.5) {space};

\end{tikzpicture}\]
\caption{\label{fig.time.space} Schema 
of time and space direction 
in the four different quarters of a cell.}
\end{figure}
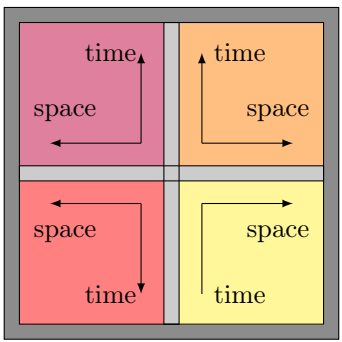

As usually in this type of constructions, 
the tape is not connected. 
Between two computation positions, the information 
is transported. In our model, each 
computation position takes as 
input up to four symbols coming 
from bottom and the sides. It outputs 
up to two symbols to the top and sides.
Moreover, we add special states 
to the definition of Turing machine. We do this 
in order to manage the presence 
of multiple machine heads.
We describe this model in 
Section~\ref{subsec.machines.block.gluing}, 
and then show how to implement 
it with local rules in Section~\ref{subsec.machines.local.rules.block.gluing}.

If a machine head enters an error state, 
this triggers an error signal that propagates through 
the trajectory of the machine. This signal is 
taken into account only for computation quarters.

\subsubsection{\label{subsec.machines.block.gluing} Adaptation 
of computing machines 
model to linear block gluing property}

In this section we present the 
way computing machines work in our construction. 
The model that we use is adapted in order 
to have the linear block gluing property, 
and is defined as follows: 

\begin{definition}
A \textbf{computing machine} $\mathcal{M}$
is some tuple $=(\mathcal{Q}, \mathcal{A}, 
\delta, q_0, q_e,q_s, \#)$.
The set $\mathcal{Q}$ is the state set, $\mathcal{A}$ 
the alphabet, $q_0$ the initial state, and $\#$ 
is the blank symbol, 
and $$\delta : \mathcal{A} \times \mathcal{Q} 
\rightarrow
\mathcal{A} \times \mathcal{Q}   
\times \{\leftarrow,\rightarrow,\uparrow\}.$$

The other elements $q_e,q_s$ are 
states in $\mathcal{Q}$. They are such that 
for all $q \in \{q_e,q_s\}$ 
and for all $a$ in $\mathcal{A}$, 
$\delta(a,q) = (a,q,\uparrow)$.

\end{definition}

The special states $q_e,q_s$ in this definition 
have the following meaning: 

\begin{itemize}
\item error state $q_e$: a machine head 
enters this state when 
it detects an error
or when it collides with another 
machine head.

This state is not forbidden 
in the subshift, but this is replaced 
by the sending of an error signal. 
We forbid the coexistence of the error 
signal with a well initialized tape. 
The machine stops moving when 
it enters this state.

\item shadow state $q_s$: 
this state corresponds to 
the absence of head. We need to introduce
this state so that the number 
of possible space-time diagrams in 
finite cells has a closed form.
\end{itemize}

Any Turing machine can be transformed 
in such a machine by adding some 
state $q_s$ verifying 
the properties 
listed above. 

When the machine is well initialized, 
none of these states and letters will be 
reached. Hence this machine behave 
as the initial one.
As a consequence, one can consider that the 
machine we used has these properties.

In this section, we use a machine 
which successively for all $n \ge 0$
writes the bits 
$s^{(n)}_k$, $k = 1 ... n$, 
on positions $p_n = 2^n$ (which 
is a computable function).
This position corresponds 
to the number of the 
first active 
column from left to right which 
is just on the right of an order $n$ 
two dimensional cell on a face amongst 
active columns 

Recall that $s$ is the $\Pi_1$-computable 
sequence defined at the beginning 
of the construction. The sequence $(s^{(n)}_k)$ 
is a computable sequence such that 
for all $k$, $s_n = \inf_k s^{(n)}_k$. 

\subsubsection{\label{subsec.machines.local.rules.block.gluing} Implementation 
of the machines}

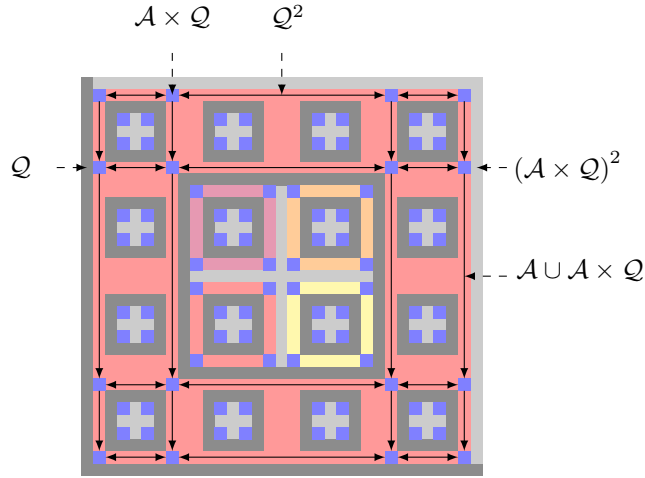
\begin{figure}[ht]
\[\begin{tikzpicture}[scale=0.08]
\fill[gray!90] (-2,-2) rectangle (64,64);
\fill[gray!40] (0,0) rectangle (64,64);
\fill[red!40] (0,0) rectangle (2*31,2*31);
\supertilesynchroone{-4*16+66}{-1*16+66}
\supertilesynchroone{-3*16+66}{-1*16+66}
\supertilesynchroone{-2*16+66}{-1*16+66}
\supertilesynchroone{-1*16+66}{-1*16+66}

\supertilesynchroone{-4*16+66}{-2*16+66}
\supertilesynchroone{-3*16+66}{-2*16+66}
\supertilesynchroone{-2*16+66}{-2*16+66}
\supertilesynchroone{-1*16+66}{-2*16+66}

\supertilesynchroone{-4*16+66}{-3*16+66}
\supertilesynchroone{-3*16+66}{-3*16+66}
\supertilesynchroone{-2*16+66}{-3*16+66}
\supertilesynchroone{-1*16+66}{-3*16+66}

\supertilesynchrotwo{16}{16}

\supertilesynchroone{-4*16+66}{-4*16+66}
\supertilesynchroone{-3*16+66}{-4*16+66}
\supertilesynchroone{-2*16+66}{-4*16+66}
\supertilesynchroone{-1*16+66}{-4*16+66}

\fill[blue!50] (16+0,16) rectangle (16+2,2+16);
\fill[blue!50] (16+0,28) rectangle (16+2,2+28);
\fill[blue!50] (28+0,28) rectangle (28+2,2+28);
\fill[blue!50] (28+0,16) rectangle (28+2,2+16);

\fill[blue!50] (16+0,16+16) rectangle (16+2,2+16+16);
\fill[blue!50] (16+0,28+16) rectangle (16+2,2+28+16);
\fill[blue!50] (28+0,28+16) rectangle (28+2,2+28+16);
\fill[blue!50] (28+0,16+16) rectangle (28+2,2+16+16);

\fill[blue!50] (16+16+0,16) rectangle (16+16+2,2+16);
\fill[blue!50] (16+16+0,28) rectangle (16+16+2,2+28);
\fill[blue!50] (16+28+0,28) rectangle (16+28+2,2+28);
\fill[blue!50] (16+28+0,16) rectangle (16+28+2,2+16);

\fill[blue!50] (20,20) rectangle (22,22);
\fill[blue!50] (20,24) rectangle (22,26);
\fill[blue!50] (24,20) rectangle (26,22);
\fill[blue!50] (24,24) rectangle (26,26);

\fill[blue!50] (36,20) rectangle (38,22);
\fill[blue!50] (36,24) rectangle (38,26);
\fill[blue!50] (40,20) rectangle (42,22);
\fill[blue!50] (40,24) rectangle (42,26);

\fill[blue!50] (36,36) rectangle (38,38);
\fill[blue!50] (36,40) rectangle (38,42);
\fill[blue!50] (40,36) rectangle (42,38);
\fill[blue!50] (40,40) rectangle (42,42);

\fill[blue!50] (20,36) rectangle (22,38);
\fill[blue!50] (20,40) rectangle (22,42);
\fill[blue!50] (24,36) rectangle (26,38);
\fill[blue!50] (24,40) rectangle (26,42);

\fill[blue!50] (0,0) rectangle (2,2);
\fill[blue!50] (12+0,0) rectangle (12+2,2);
\fill[blue!50] (48+0,0) rectangle (48+2,2);
\fill[blue!50] (60+0,0) rectangle (60+2,2);

\fill[blue!50] (0,12) rectangle (2,2+12);
\fill[blue!50] (12+0,12) rectangle (12+2,2+12);
\fill[blue!50] (48+0,12) rectangle (48+2,2+12);
\fill[blue!50] (60+0,12) rectangle (60+2,2+12);

\fill[blue!50] (0,48) rectangle (2,2+48);
\fill[blue!50] (12+0,48) rectangle (12+2,2+48);
\fill[blue!50] (48+0,48) rectangle (48+2,2+48);
\fill[blue!50] (60+0,48) rectangle (60+2,2+48);

\fill[blue!50] (0,60) rectangle (2,2+60);
\fill[blue!50] (12+0,60) rectangle (12+2,2+60);
\fill[blue!50] (48+0,60) rectangle (48+2,2+60);
\fill[blue!50] (60+0,60) rectangle (60+2,2+60);

\fill[blue!50] (16+16+0,16+16) rectangle (16+16+2,16+2+16);
\fill[blue!50] (16+16+0,16+28) rectangle (16+16+2,16+2+28);
\fill[blue!50] (16+28+0,16+28) rectangle (16+28+2,16+2+28);
\fill[blue!50] (16+28+0,16+16) rectangle (16+28+2,16+2+16);

\fill[blue!50] (4,4) rectangle (6,6);
\fill[blue!50] (8,4) rectangle (10,6);
\fill[blue!50] (4,8) rectangle (6,10);
\fill[blue!50] (8,8) rectangle (10,10);
\fill[blue!50] (20,4) rectangle (22,6);
\fill[blue!50] (24,4) rectangle (26,6);
\fill[blue!50] (20,8) rectangle (22,10);
\fill[blue!50] (24,8) rectangle (26,10);
\fill[blue!50] (36,4) rectangle (38,6);
\fill[blue!50] (40,4) rectangle (42,6);
\fill[blue!50] (36,8) rectangle (38,10);
\fill[blue!50] (40,8) rectangle (42,10);
\fill[blue!50] (16+36,4) rectangle (16+38,6);
\fill[blue!50] (16+40,4) rectangle (16+42,6);
\fill[blue!50] (16+36,8) rectangle (16+38,10);
\fill[blue!50] (16+40,8) rectangle (16+42,10);

\fill[blue!50] (4,16+4) rectangle (6,16+6);
\fill[blue!50] (8,16+4) rectangle (10,16+6);
\fill[blue!50] (4,16+8) rectangle (6,16+10);
\fill[blue!50] (8,16+8) rectangle (10,16+10);
\fill[blue!50] (20,16+4) rectangle (22,16+6);
\fill[blue!50] (24,16+4) rectangle (26,16+6);
\fill[blue!50] (20,16+8) rectangle (22,16+10);
\fill[blue!50] (24,16+8) rectangle (26,16+10);
\fill[blue!50] (36,16+4) rectangle (38,16+6);
\fill[blue!50] (40,16+4) rectangle (42,16+6);
\fill[blue!50] (36,16+8) rectangle (38,16+10);
\fill[blue!50] (40,16+8) rectangle (42,16+10);
\fill[blue!50] (16+36,16+4) rectangle (16+38,16+6);
\fill[blue!50] (16+40,16+4) rectangle (16+42,16+6);
\fill[blue!50] (16+36,16+8) rectangle (16+38,16+10);
\fill[blue!50] (16+40,16+8) rectangle (16+42,16+10);

\fill[blue!50] (4,32+4) rectangle (6,32+6);
\fill[blue!50] (8,32+4) rectangle (10,32+6);
\fill[blue!50] (4,32+8) rectangle (6,32+10);
\fill[blue!50] (8,32+8) rectangle (10,32+10);
\fill[blue!50] (20,32+4) rectangle (22,32+6);
\fill[blue!50] (24,32+4) rectangle (26,32+6);
\fill[blue!50] (20,32+8) rectangle (22,32+10);
\fill[blue!50] (24,32+8) rectangle (26,32+10);
\fill[blue!50] (36,32+4) rectangle (38,32+6);
\fill[blue!50] (40,32+4) rectangle (42,32+6);
\fill[blue!50] (36,32+8) rectangle (38,32+10);
\fill[blue!50] (40,32+8) rectangle (42,32+10);
\fill[blue!50] (16+36,32+4) rectangle (16+38,32+6);
\fill[blue!50] (16+40,32+4) rectangle (16+42,32+6);
\fill[blue!50] (16+36,32+8) rectangle (16+38,32+10);
\fill[blue!50] (16+40,32+8) rectangle (16+42,32+10);

\fill[blue!50] (4,48+4) rectangle (6,48+6);
\fill[blue!50] (8,48+4) rectangle (10,48+6);
\fill[blue!50] (4,48+8) rectangle (6,48+10);
\fill[blue!50] (8,48+8) rectangle (10,48+10);
\fill[blue!50] (20,48+4) rectangle (22,48+6);
\fill[blue!50] (24,48+4) rectangle (26,48+6);
\fill[blue!50] (20,48+8) rectangle (22,48+10);
\fill[blue!50] (24,48+8) rectangle (26,48+10);
\fill[blue!50] (36,48+4) rectangle (38,48+6);
\fill[blue!50] (40,48+4) rectangle (42,48+6);
\fill[blue!50] (36,48+8) rectangle (38,48+10);
\fill[blue!50] (40,48+8) rectangle (42,48+10);
\fill[blue!50] (16+36,48+4) rectangle (16+38,48+6);
\fill[blue!50] (16+40,48+4) rectangle (16+42,48+6);
\fill[blue!50] (16+36,48+8) rectangle (16+38,48+10);
\fill[blue!50] (16+40,48+8) rectangle (16+42,48+10);
\draw[latex-latex] (2,1) -- (12,1);
\draw[latex-latex] (50,1) -- (60,1);
\draw[latex-latex] (14,1) -- (48,1);
\draw[latex-latex] (2,13) -- (12,13);
\draw[latex-latex] (50,13) -- (60,13);
\draw[latex-latex] (14,13) -- (48,13);
\draw[latex-latex] (2,49) -- (12,49);
\draw[latex-latex] (50,49) -- (60,49);
\draw[latex-latex] (14,49) -- (48,49);
\draw[latex-latex] (2,61) -- (12,61);
\draw[latex-latex] (50,61) -- (60,61);
\draw[latex-latex] (14,61) -- (48,61);

\draw[latex-] (1,2) -- (1,12);
\draw[latex-] (1,50) -- (1,60);
\draw[latex-] (1,14) -- (1,48);
\draw[latex-] (13,2) -- (13,12);
\draw[latex-] (13,50) -- (13,60);
\draw[latex-] (13,14) -- (13,48);
\draw[latex-] (49,2) -- (49,12);
\draw[latex-] (49,50) -- (49,60);
\draw[latex-] (49,14) -- (49,48);
\draw[latex-] (61,2) -- (61,12);
\draw[latex-] (61,50) -- (61,60);
\draw[latex-] (61,14) -- (61,48);

\draw[-latex,dashed] (68,31) -- (61,31);
\draw[-latex,dashed] (31,68) -- (31,61);
\draw[-latex,dashed] (13,68) -- (13,61);
\draw[-latex,dashed] (68,49) -- (63,49);
\draw[-latex,dashed] (-6,49) -- (-1,49);
\node at (-12,49) {$\mathcal{Q}$};
\node at (78,49) {$\left( \mathcal{A} \times 
\mathcal{Q} \right) ^2$};
\node at (80,32) {$\mathcal{A} 
\cup \mathcal{A} \times \mathcal{Q}$};
\node at (32,74) {$\mathcal{Q}^2$};
\node at (13,74) {$\mathcal{A} \times \mathcal{Q}$};
\end{tikzpicture}\]
\caption{\label{fig.loc.machine.block.gluing} Localization 
of the machine symbols in the red 
quarter of an order two cell.}
\end{figure}

\begin{figure}[ht]
\[\begin{tikzpicture}[scale=0.2]
\begin{scope}
\fill[orange!40] 
(-9,-1) rectangle (9,1);
\fill[orange!40] 
(-1,-9) rectangle (1,9);
\fill[blue!50] (-1,-1) rectangle (1,1);
\draw (-1,-1) rectangle (1,1);
\draw[-latex] (-9,-0.4) -- (-1,-0.4);
\draw[-latex,dashed] (-1,0.4) -- (-9,0.4);
\draw[-latex] (9,-0.4) -- (1,-0.4);
\draw[-latex,dashed] (1,0.4) -- (9,0.4);
\draw[-latex] (0,-9) -- (0,-1);
\draw[-latex] (0,1) -- (0,9);

\node[scale=1.5] at (-11,0) {$1$};
\end{scope}

\begin{scope}[xshift=30cm]
\fill[orange!40] 
(0,-1) rectangle (9,1);
\fill[orange!40] 
(-1,-9) rectangle (1,9);
\fill[blue!50] (-1,-1) rectangle (1,1);
\draw (-1,-1) rectangle (1,1);
\draw[-latex] (9,-0.4) -- (1,-0.4);
\draw[-latex,dashed] (1,0.4) -- (9,0.4);
\draw[-latex] (0,-9) -- (0,-1);
\draw[-latex,dashed] (0,1) -- (0,9);

\node[scale=1.5] at (-11,0) {$2$};

\node at (16,3) {Outputs};
\node at (16,6) {Inputs};
\draw (10,3) -- (12,3); 
\draw[dashed] (10,6) -- (12,6);
\end{scope}

\begin{scope}[yshift=-25cm]

\fill[orange!40] 
(0,-1) rectangle (9,1);
\fill[orange!40] 
(-1,-1) rectangle (1,9);
\fill[blue!50] (-1,-1) rectangle (1,1);
\draw (-1,-1) rectangle (1,1);
\draw[-latex] (9,-0.4) -- (1,-0.4);
\draw[-latex,dashed] (1,0.4) -- (9,0.4);
\draw[-latex] (0,-2) -- (0,-1);
\node at (0,-3) {face 2};
\draw[-latex,dashed] (0,1) -- (0,9);

\node[scale=1.5] at (-11,0) {$3$};
\end{scope}

\begin{scope}[yshift=-25cm,xshift=30cm]

\fill[orange!40] 
(-9,-1) rectangle (9,1);
\fill[orange!40] 
(-1,-1) rectangle (1,9);
\fill[blue!50] (-1,-1) rectangle (1,1);
\draw (-1,-1) rectangle (1,1);
\draw[-latex,dashed] (-1,0.4) -- (-9,0.4);
\draw[-latex] (-9,-0.4) -- (-1,-0.4);
\draw[-latex] (9,-0.4) -- (1,-0.4);
\draw[-latex,dashed] (1,0.4) -- (9,0.4);
\draw[-latex,dashed] (0,1) -- (0,9);
\node[scale=1.5] at (-11,0) {$4$};
\end{scope}

\begin{scope}[yshift=-37cm]

\fill[orange!40] 
(0,-1) rectangle (9,1);
\fill[orange!40] 
(1,1) rectangle (-1,-9);
\fill[blue!50] (-1,-1) rectangle (1,1);
\draw (-1,-1) rectangle (1,1);
\draw[-latex] (9,-0.4) -- (1,-0.4);
\draw[-latex,dashed] (1,0.4) -- (9,0.4);
\draw[-latex] (0,-9) -- (0,-1);

\node[scale=1.5] at (-11,0) {$5$};
\end{scope}

\begin{scope}[yshift=-37cm,xshift=30cm]
\fill[orange!40] 
(-9,-1) rectangle (9,1);
\fill[orange!40] 
(1,1) rectangle (-1,-9);

\fill[blue!50] (-1,-1) rectangle (1,1);
\draw (-1,-1) rectangle (1,1);
\draw[-latex,dashed] (-1,0.4) -- (-9,0.4);
\draw[-latex] (-9,-0.4) -- (-1,-0.4);
\draw[-latex] (9,-0.4) -- (1,-0.4);
\draw[-latex,dashed] (1,0.4) -- (9,0.4);
\node[scale=1.5] at (-11,0) {$6$};
\draw[-latex] (-9,-0.4) -- (-1,-0.4);
\draw[-latex,dashed] (-1,0.4) -- (-9,0.4);
\draw[-latex] (0,-9) -- (0,-1);
\end{scope}

\end{tikzpicture}\]
\caption{\label{fig.schema.inputs.outputs.block.gluing}
Schema of the inputs and outputs 
directions when inside the area (1) 
and on the border of the area (2,3,4,5,6).}
\end{figure}
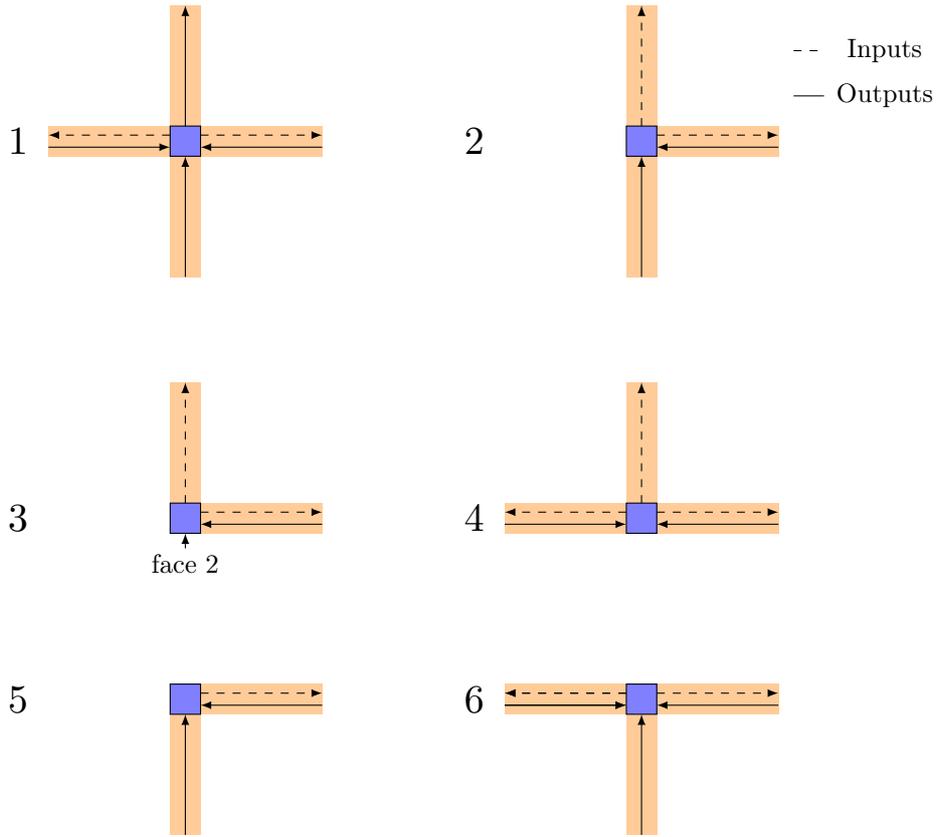

\begin{figure}[ht]
\[\begin{tikzpicture}[scale=0.2]
\fill[orange!40] 
(-9,-1) rectangle (9,1);
\fill[orange!40] 
(-1,-9) rectangle (1,9);
\fill[blue!50] (-1,-1) rectangle (1,1);
\draw (-1,-1) rectangle (1,1);
\draw[-latex] (-9,0) -- (-1,0);
\draw[-latex] (0,-9) -- (0,-1);
\draw[-latex] (0,1) -- (0,9);
\node at (-10,0) {$q$};
\node at (0,-10) {$a$};
\node at (0,10) {$(a,q)$};
\end{tikzpicture}\]
\caption{\label{fig.standard.rule.0.block.gluing} Illustration 
of the standard rules (1).}
\end{figure}

\begin{figure}[h]
\[\begin{tikzpicture}[scale=0.175]
\begin{scope}
\fill[orange!40] 
(-9,-1) rectangle (9,1);
\fill[orange!40] 
(-1,-9) rectangle (1,9);
\fill[blue!50] (-1,-1) rectangle (1,1);
\draw (-1,-1) rectangle (1,1);
\draw[-latex] (-1,0) -- (-9,0);
\draw[-latex] (0,-9) -- (0,-1);
\draw[-latex] (0,1) -- (0,9);
\node at (-13,0) {$\delta_2 (a,q)$};
\node at (0,-10) {$(a,q)$};
\node at (0,10) {$\delta_1 (a,q)$};
\node at (0,-14) {$\delta_3 (a,q) = \leftarrow$};
\end{scope}

\begin{scope}[xshift=23cm]
\fill[orange!40] 
(-9,-1) rectangle (9,1);
\fill[orange!40] 
(-1,-9) rectangle (1,9);
\fill[blue!50] (-1,-1) rectangle (1,1);
\draw (-1,-1) rectangle (1,1);
\draw[-latex] (1,0) -- (9,0);
\draw[-latex] (0,-9) -- (0,-1);
\draw[-latex] (0,1) -- (0,9);
\node at (13,0) {$\delta_2 (a,q)$};
\node at (0,-10) {$(a,q)$};
\node at (0,10) {$\delta_1 (a,q)$};
\node at (0,-14) {$\delta_3 (a,q) = \rightarrow$};
\end{scope}

\begin{scope}[xshift=50cm]
\fill[orange!40] 
(-9,-1) rectangle (9,1);
\fill[orange!40] 
(-1,-9) rectangle (1,9);
\fill[blue!50] (-1,-1) rectangle (1,1);
\draw (-1,-1) rectangle (1,1);
\draw[-latex] (0,-9) -- (0,-1);
\draw[-latex] (0,1) -- (0,9);
\node at (0,10) {$(\delta_1(a,q),\delta_2 (a,q))$};
\node at (0,-10) {$(a,q)$};
\node at (0,-14) {$\delta_3 (a,q) = \uparrow$};
\end{scope}
\end{tikzpicture}\]
\caption{\label{fig.standard.rule.1.block.gluing}
Illustration of the standard rules (2).}
\end{figure}

\begin{figure}[ht]
\[\begin{tikzpicture}[scale=0.2]
\fill[red!50] 
(-9,-1) rectangle (9,1);
\fill[red!50] 
(-1,-9) rectangle (1,9);
\fill[blue!50] (-1,-1) rectangle (1,1);
\draw (-1,-1) rectangle (1,1);
\draw[-latex] (9,0) -- (1,0);
\draw[-latex] (0,-1) -- (0,-9);
\draw[-latex] (0,9) -- (0,1);
\node at (10,0) {$q$};
\node at (0,10) {$a$};
\node at (0,-10) {$(a,q)$};
\end{tikzpicture}\]
\end{figure}

\begin{figure}[ht]
\[\begin{tikzpicture}[scale=0.175]
\begin{scope}[xshift=23cm]
\fill[red!50] 
(-9,-1) rectangle (9,1);
\fill[red!50] (-1,-9) rectangle (1,9);
\fill[blue!50] (-1,-1) rectangle (1,1);
\draw (-1,-1) rectangle (1,1);
\draw[-latex] (1,0) -- (9,0);
\draw[-latex] (0,-1) -- (0,-9);
\draw[-latex] (0,9) -- (0,1);
\node at (13,0) {$\delta_2 (a,q)$};
\node at (0,10) {$(a,q)$};
\node at (0,-10) {$\delta_1 (a,q)$};
\node at (0,14) {$\delta_3 (a,q) = \leftarrow$};
\end{scope}

\begin{scope}
\fill[red!50] 
(-9,-1) rectangle (9,1);
\fill[red!50] (-1,-9) rectangle (1,9);
\fill[blue!50] (-1,-1) rectangle (1,1);
\draw (-1,-1) rectangle (1,1);
\draw[-latex] (-1,0) -- (-9,0);
\draw[-latex] (0,-1) -- (0,-9);
\draw[-latex] (0,9) -- (0,1);
\node at (-13,0) {$\delta_2 (a,q)$};
\node at (0,10) {$(a,q)$};
\node at (0,-10) {$\delta_1 (a,q)$};
\node at (0,14) {$\delta_3 (a,q) = \rightarrow$};
\end{scope}

\begin{scope}[xshift=50cm]
\fill[red!50] 
(-9,-1) rectangle (9,1);
\fill[red!50] (-1,-9) rectangle (1,9);
\fill[blue!50] (-1,-1) rectangle (1,1);
\draw (-1,-1) rectangle (1,1);
\draw[-latex] (0,-1) -- (0,-9);
\draw[-latex] (0,9) -- (0,1);
\node at (0,-10) {$(\delta_1(a,q),\delta_2 (a,q)$};
\node at (0,10) {$(a,q)$};
\node at (0,14) {$\delta_3 (a,q) = \uparrow$};
\end{scope}
\end{tikzpicture}\]
\end{figure}

In this section, we describe the second 
sublayer of this layer.

\noindent \textbf{\textit{Symbols:}}

The symbols are elements of the sets 
$\mathcal{A} \times \mathcal{Q}$,
$\mathcal{A}$, 
$\mathcal{Q}^2$, $\mathcal{Q}$, $\left(\mathcal{A} \times \mathcal{Q} \right)^2$,
and a blank symbol. \bigskip

\noindent \textbf{\textit{Local rules:}} \bigskip

\begin{itemize}
\item \textbf{Localization:}
the non-blank symbols are superimposed on information 
transfer rows and columns, as well as positions corresponding 
to information transfer rows and columns on the arms of the reticle and 
the east and west walls.
More precisely: 
\begin{itemize}
\item the possible symbols for 
information transfer columns are elements 
of the sets $\mathcal{A}$ and $\mathcal{A} \times \mathcal{Q}$.
The elements of $\mathcal{A} \times \mathcal{Q}$ are on 
computation positions. The other ones on the other positions
of these lines and columns.
\item the positions on the vertical (resp. horizontal) arms of the reticle 
corresponding to an information transfer line are colored with 
an element of $\mathcal{Q} ^2$ (resp. $\left( \mathcal{A} \times \mathcal{Q} \right)^2$).
The first coordinate corresponds to the machine heads 
entering in the west quarter. The second one corresponds 
to machine heads entering 
in the east one (resp. machine head and letter entering in 
the north and south ones).
\item on the west and east walls, the symbols are in $\mathcal{Q}$. They 
correspond to machine heads entering in the adjacent quarter.
 See an illustration on 
Figure~\ref{fig.loc.machine.block.gluing}.
\end{itemize}
\item \textbf{Transmission:} 

Along the rows and columns, 
the symbol is transmitted while not 
on computation positions.

\item \textbf{Computation positions rules:}

Consider some computation position. These rules depend on the 
orientation of the quarter in the cell. We describe them 
in the north east quarter. The rules in the other quarters 
are obtained by symmetry, 
respecting the orientation of time and space 
given on Figure~\ref{fig.time.space}. 

For such a position, 
the \textbf{inputs} include:
 
\begin{enumerate}
\item the symbols written 
on the south position,
\item  
the first symbol written on the west position 
(except in the leftmost column, where 
the input is the second symbol of the west position), 
\item and the second symbol 
on the east position (except 
when in the rightmost column, where 
the input is the unique symbol written on east position).
\end{enumerate} 

The \textbf{outputs}
include: 

\begin{enumerate}
\item the symbols written on the 
north position when not in 
the topmost row, 
\item the second symbol of the west position (when 
not in the leftmost column), 
\item and the first symbol on 
the east position (when not on the 
rightmost column).
\end{enumerate}

Moreover, on the row near the reticle, the inputs 
from inside the area are always the 
shadow state $q_s$. The input from the bottom is 
free. As a consequence the couple written on the 
the position is also free. 
This is also true for the elements of 
$\mathcal{Q}$ on the computation positions in the leftmost and rightmost columns 
and the triple of symbols 
written on the position near the nucleus.

See Figure~\ref{fig.schema.inputs.outputs.block.gluing} 
for an illustration.

On the first row, all the inputs 
are determined by the counter and 
by the above rule. Then each row is 
determined from the adjacent one 
on the bottom and the inputs on the sides. 
This is due to 
the following rules, which on 
each computation position determine 
the outputs from the inputs: 

\begin{enumerate}
\item \textbf{Collision between machine heads:} 
if there are at least 
two elements of $\mathcal{Q} \backslash \{q_s\}$ 
in the inputs, then the 
computation position is superimposed 
with $(a,q_e)$. The output on the top (when 
this exists)
is $(a,q_e)$, where $a$ is the letter input below.
The outputs on the sides are $q_s$.
When there is a unique symbol in 
$\mathcal{Q} \backslash \{q_s\}$ in the inputs, 
this symbol is called the machine head 
state (the symbol $q_s$ is not considered 
as representing a machine head).
\item \textbf{Standard rule:}
\begin{enumerate}
\item when the head input comes from a side, 
then the functional position is superimposed with 
$(a,q)$. It outputs the couple $(a,q)$ 
above, 
where $a$ is the letter input under, and $q$ the head input.
The other outputs are $q_s$.
See Figure~\ref{fig.standard.rule.0.block.gluing} for 
an illustration of this rule.

\item when the head input comes from under, the output is 
$\delta_1 (a,q)$ above when 
the $\delta_3 (a,q)$ is in $\{\rightarrow,\leftarrow\}$ and 
$(\delta_1(a,q),\delta_2 (a,q))$ when 
$\delta_3 (a,q) = \uparrow$. 
The head output is in the 
direction of $\delta_3 (a,q)$ when 
this output direction exists, and equal to 
$\delta_2 (a,q)$
when this 
direction is in $\{\rightarrow,\leftarrow\}$. The other output is $q_s$.
See Figure~\ref{fig.standard.rule.1.block.gluing} 
for an illustration. \bigskip

The computation positions rules in 
the other quarters are similar. These
rules in a purple quarter are obtained 
by reversing west and east, in the red one 
by reversing west and east and moreover 
north and south, and in the yellow one 
by reversing north and south. 
For instance 
the previous schemata are changed to the following 
one in the red quarter. 

This corresponds to changing the direction 
of time and space of evolution 
of the Turing machine, as abstracted on 
Figure~\ref{fig.time.space}. 

\end{enumerate}
\item \textbf{Collision with border:} 
When the output 
direction does not exist, the output 
is $(a,q_e)$ on the top. The output 
on the side is $q_s$.
The computation position is 
superimposed with $(a,q)$.  
\item \textbf{No machine head:} 
when all the inputs in $\mathcal{Q}$ are 
$q_s$, then the output above is in 
$\mathcal{A}$ and equal to $a$.
\end{enumerate} 

\end{itemize}

\noindent {\textbf{\textit{Global behavior:}} \bigskip

In each of the quarters of any cell is implemented 
a computing machine according to our model, with multiple 
machine heads on the initial tape and 
entering in each row. In the next section, we 
will impose that in the computation 
quarters there are no machine heads entering on the sides. 
We also impose that the tape is well initialized. 
This is done using signals.
As a consequence, in these quarters 
the computations 
are as intended. This means that a Turing 
machine writes successively the 
bits $s^{(n)}_k$ 
on the $p_k = 2^k$th column of its tape (in 
order to impose the value of the frequency bits). 
It enters in the error 
state $q_e$ when it detects an error - meaning that the 
corresponding frequency bit in the column just on the right 
is greater than the written bit.

When this is not the case, the computations 
are determined by the rules 
giving the outputs on computation positions 
from the inputs. When there is a collision 
of a machine head with the border 
it enters into state $q_e$ and when 
heads collide, they fusion into a unique 
head in state $q_e$.

\subsubsection{ \label{sec.empty.tape.block.gluing} Empty tape and sides signals}

This sublayer serves for the propagation 
of a signal which detects if the initial 
tape of a machine is well initialized, and if 
a machine head enters on a side. \bigskip

\noindent \textbf{\textit{Symbols:}} \bigskip

Elements of 
$\{\begin{tikzpicture}[scale=0.3]
\fill[Salmon] (0,0) rectangle (1,1);
\draw (0,0) rectangle (1,1);
\end{tikzpicture}, 
\begin{tikzpicture}[scale=0.3]
\fill[YellowGreen] (0,0) rectangle (1,1);
\draw (0,0) rectangle (1,1);
\end{tikzpicture}\}^2$, elements of $\{\begin{tikzpicture}[scale=0.3]
\fill[Salmon] (0,0) rectangle (1,1);
\draw (0,0) rectangle (1,1);
\end{tikzpicture}, 
\begin{tikzpicture}[scale=0.3]
\fill[YellowGreen] (0,0) rectangle (1,1);
\draw (0,0) rectangle (1,1);
\end{tikzpicture}\}$ and a blank symbol. \bigskip

\noindent \textbf{\textit{Local rules:}} 

\begin{itemize}
\item \textbf{Localization:} the non-blank symbols 
are superimposed on and only on the arms of the reticle, 
and the west and east walls. The east and west walls are colored with 
elements of $\{\begin{tikzpicture}[scale=0.3]
\fill[Salmon] (0,0) rectangle (1,1);
\draw (0,0) rectangle (1,1);
\end{tikzpicture}, 
\begin{tikzpicture}[scale=0.3]
\fill[YellowGreen] (0,0) rectangle (1,1);
\draw (0,0) rectangle (1,1);
\end{tikzpicture}\}$, and the reticle with element of $\{\begin{tikzpicture}[scale=0.3]
\fill[Salmon] (0,0) rectangle (1,1);
\draw (0,0) rectangle (1,1);
\end{tikzpicture}, 
\begin{tikzpicture}[scale=0.3]
\fill[YellowGreen] (0,0) rectangle (1,1);
\draw (0,0) rectangle (1,1);
\end{tikzpicture}\} ^2$.

\item \textbf{Triggering the signal:} 
the topmost and bottommost positions 
of the two walls are superimposed 
with \begin{tikzpicture}[scale=0.3]
\fill[YellowGreen] (0,0) rectangle (1,1);
\draw (0,0) rectangle (1,1);
\end{tikzpicture}. 

\item \textbf{Transmission rules:} 

\begin{enumerate}
\item \textbf{In the walls:}
\begin{itemize}
\item On the north (resp. south) part of one of the walls, 
the symbol \begin{tikzpicture}[scale=0.3]
\fill[YellowGreen] (0,0) rectangle (1,1);
\draw (0,0) rectangle (1,1);
\end{tikzpicture} propagates upwards while in the wall. It propagates 
downwards (resp. upwards) while not encountering 
a symbol in $\mathcal{Q} \backslash \{q_s\}$. When this is is the case, 
the color becomes \begin{tikzpicture}[scale=0.3]
\fill[Salmon] (0,0) rectangle (1,1);
\draw (0,0) rectangle (1,1);
\end{tikzpicture}.
\item The symbol \begin{tikzpicture}[scale=0.3]
\fill[Salmon] (0,0) rectangle (1,1);
\draw (0,0) rectangle (1,1);
\end{tikzpicture} propagates downwards (resp. upwards) in the north (resp. south) part 
of the wall, and upwards (resp. downwards) while not encountering 
a symbol in $\mathcal{Q} \backslash \{q_s\}$.
\item The center of the wall is colored with 
a couple of colors. The first one is equal to the color of 
the north position. The second one is equal to the 
color of the south position.
\end{itemize}
With words, a signal propagates through the north (resp. south) part 
of the wall. This signal is triggered in state \begin{tikzpicture}[scale=0.3]
\fill[YellowGreen] (0,0) rectangle (1,1);
\draw (0,0) rectangle (1,1);
\end{tikzpicture}. When it detects the first symbol in $\mathcal{Q} \backslash \{q_s\}$, it 
changes its state which becomes \begin{tikzpicture}[scale=0.3]
\fill[Salmon] (0,0) rectangle (1,1);
\draw (0,0) rectangle (1,1);
\end{tikzpicture}. This information is transmitted to the center of the wall.
\item \textbf{In the reticle:}
The rules for the reticle are similar, except that:  
\begin{itemize}
\item there are two signals for each arm, one for each adjacent quarter.
\item the propagation direction 
is to the east for the west arm, and to the west for the east arm. The 
case of vertical arms is similar as the case of walls.
\item when the 
signal starts in state \begin{tikzpicture}[scale=0.3]
\fill[YellowGreen] (0,0) rectangle (1,1);
\draw (0,0) rectangle (1,1);
\end{tikzpicture}: on the west arm (resp. east one) 
each signal detects the first symbol from 
left to right (resp right to left) 
different from $(\#,q_s)$ when not on the rightmost position (resp. leftmost one), 
and different from $(\#,q_0)$ when on this position. These symbols 
correspond to the quarter associated with this signal.
\item when it starts in state \begin{tikzpicture}[scale=0.3]
\fill[Salmon] (0,0) rectangle (1,1);
\draw (0,0) rectangle (1,1);
\end{tikzpicture}, the arm just 
transmits this information to the nucleus.
\item In the vertical arms, the signals detect the first symbol different 
from $q_s$, from south to north for the south arm, and from north to south 
for the north arm.
\end{itemize}
\end{enumerate}

\item \textbf{Computation quarters are initialized with empty tape and 
sides:} 
considering a nucleus $\vec{u}$, the symbol 
on position $\vec{u}+\vec{e}^1$ has first (resp. second) coordinate equal to 
 \begin{tikzpicture}[scale=0.3]
\fill[YellowGreen] (0,0) rectangle (1,1);
\draw (0,0) rectangle (1,1);
\end{tikzpicture} if the orange (resp. yellow) quarter is represented in the DNA.
\item the symbol on position $\vec{u}-\vec{e}^1$ has first (resp. second) coordinate equal to 
 \begin{tikzpicture}[scale=0.3]
\fill[YellowGreen] (0,0) rectangle (1,1);
\draw (0,0) rectangle (1,1);
\end{tikzpicture} if the purple (resp. red) quarter is represented in the DNA.
\item There are similar conditions for the 
symbols on positions $\vec{u} \pm \vec{e}^2$.
\end{itemize}

\noindent \textbf{\textit{Global behavior:}} \bigskip

These rules induce the propagation 
of a signal triggered in state 
\begin{tikzpicture}[scale=0.3]
\fill[YellowGreen] (0,0) rectangle (1,1);
\draw (0,0) rectangle (1,1);
\end{tikzpicture} 
which detects, for each of the quarter, if 
the sides and the tape are well initialized: if this is not 
the case, then the signal detects an error. It sends this information,
 which corresponds to state \begin{tikzpicture}[scale=0.3]
\fill[Salmon] (0,0) rectangle (1,1);
\draw (0,0) rectangle (1,1);
\end{tikzpicture},
to the nucleus through the reticle. 
We forbid this signal to come from a quarter which is represented in the DNA. 
As a consequence, the computation quarters are well initialized. 
The simulation quarters are left free.

\subsubsection{Error signals}

\begin{figure}[ht] 
\begin{center}
\begin{tikzpicture}[scale=0.25]
\fill[purple] (6,0) rectangle (0,2);
\fill[purple] (4,6) rectangle (6,0);
\fill[purple] (16,4) rectangle (4,6);
\fill[purple] (14,4) rectangle (16,16);
\fill[purple] (14,14) rectangle (20,16);
\fill[purple] (18,14) rectangle (20,20);
\fill[Salmon] (20,20) rectangle (14,22);
\fill[YellowGreen] (14,20) rectangle (0,22);
\draw[step=2] (0,20) grid (20,22);
\draw (0,0) rectangle (2,2);
\draw (4,0) rectangle (6,2);
\draw[-latex] (4,1) -- (2,1);
\draw (0,4) rectangle (2,6);
\draw (4,4) rectangle (6,6); 
\draw[-latex] (5,4) -- (5,2);
\draw[-latex] (1,4) -- (1,2);
\draw[-latex] (15,4) -- (15,2);
\draw[-latex] (19,4) -- (19,2);

\draw[-latex] (5,18) -- (5,16);
\draw[-latex] (1,18) -- (1,16);
\draw[-latex] (15,18) -- (15,16);
\draw[-latex] (19,18) -- (19,16);

\draw[-latex] (5,14) -- (5,6);
\draw[-latex] (1,14) -- (1,6);
\draw[-latex] (15,14) -- (15,6);
\draw[-latex] (19,14) -- (19,6);

\draw[-latex] (14,5) -- (6,5);
\draw[-latex] (18,15) -- (16,15);

\draw (14,0) rectangle (16,2);
\draw (18,0) rectangle (20,2);
\draw (14,4) rectangle (16,6);
\draw (18,4) rectangle (20,6);

\draw (0,14) rectangle (2,16);
\draw (4,14) rectangle (6,16);
\draw (0,18) rectangle (2,20);
\draw (4,18) rectangle (6,20); 

\draw (14,14) rectangle (16,16);
\draw (18,14) rectangle (20,16);
\draw (14,18) rectangle (16,20);
\draw (18,18) rectangle (20,20);
\fill[gray!40] (20,20) rectangle (22,22);
\draw (20,20) rectangle (22,22);

\fill[YellowGreen] (-2,0) rectangle (0,22);
\draw[step = 2] (-2,0) grid (0,22);

\fill[YellowGreen] (20,0) rectangle (22,6); 
\fill[Salmon] (20,6) rectangle (22,20);
\draw[step = 2] (20,0) grid (22,20);

\draw[dashed] (-1,18) -- (-1,21);
\draw[dashed,-latex] (-1,21) -- (2,21);

\draw[dashed,-latex] (21,8) -- (21,13);

\filldraw[color=purple] (20+0.5,20+1.5) circle (5pt);
\filldraw[color=orange] (20+1.5,20+1.5) circle (5pt);
\end{tikzpicture}\end{center}
\caption{\label{fig.machine.trajectory} 
Schema of a possible trajectory of a machine
in a simulation quarter. 
The empty tape and sides signals represented correspond to the 
south west quarter. The dashed arrows represent the propagation 
direction of these signals.}
\end{figure}

\noindent \textbf{\textit{Symbols:}} \bigskip

This sub-layer has the following symbols: 
\begin{tikzpicture}
\fill[purple] (0,0) rectangle (0.3,0.3) ;
\draw (0,0)--(0.3,0) ;
\draw (0,0)--(0,0.3) ;
\draw (0,0.3)--(0.3,0.3) ; 
\draw (0.3,0)--(0.3,0.3) ;
\end{tikzpicture} and 
\begin{tikzpicture}
\draw (0,0)--(0.3,0) ;
\draw (0,0)--(0,0.3) ;
\draw (0,0.3)--(0.3,0.3) ; 
\draw (0.3,0)--(0.3,0.3) ;
\end{tikzpicture}. \bigskip

\noindent \textbf{\textit{Local rules:}} 

\begin{itemize}
\item \textbf{Localization:} the 
symbol \begin{tikzpicture} \fill[purple] (0,0) rectangle (0.3,0.3) ;
\draw (0,0)--(0.3,0) ;
\draw (0,0)--(0,0.3) ;
\draw (0,0.3)--(0.3,0.3) ; 
\draw (0.3,0)--(0.3,0.3) ;
\end{tikzpicture} can be superimposed only on positions having in its machine symbol 
a part in $\mathcal{Q} \backslash \{q_s\}$.
\item \textbf{Transmission:} for two adjacent 
vertical transfer or horizontal transfer positions, the symbols in this sublayer are 
the same. 
\item when on a computation position $\vec{u}$, if two of the positions $\vec{u} \pm \vec{e}^1$ 
and $\vec{u} \pm \vec{e}^2$ have a part in $\mathcal{Q} \backslash \{q_s\}$, 
these two positions have the same symbol in this layer.
\item \textbf{Triggering 
the error signal when on halting state:} 
a position with a symbol having a part equal to $q_h$
is superimposed with \begin{tikzpicture} 
\fill[purple] (0,0) rectangle (0.3,0.3) ;
\draw (0,0)--(0.3,0) ;
\draw (0,0)--(0,0.3) ;
\draw (0,0.3)--(0.3,0.3) ; 
\draw (0.3,0)--(0.3,0.3) ;
\end{tikzpicture}.
\item \textbf{Machine heads 
can not enter in error state:} 
if $\vec{u}$ is a nucleus position and the red quarter (resp. yellow, orange, purple ones) 
is represented in the DNA, then the position $\vec{u}-\vec{e}^2-\vec{e}^1$ (resp. 
$\vec{u}-\vec{e}^2+\vec{e}^1$, $\vec{u}+\vec{e}^2+\vec{e}^1$, 
$\vec{u}+\vec{e}^2-\vec{e}^1$) can not be superimposed 
with \begin{tikzpicture} 
\fill[purple] (0,0) rectangle (0.3,0.3) ;
\draw (0,0)--(0.3,0) ;
\draw (0,0)--(0,0.3) ;
\draw (0,0.3)--(0.3,0.3) ; 
\draw (0.3,0)--(0.3,0.3) ;
\end{tikzpicture}
\end{itemize}
\bigskip

\noindent \textbf{\textit{Global behavior:}} \bigskip

When a machine enters in error state, 
then it sends through its trajectory an 
error signal (represented by the symbol 
\begin{tikzpicture} \fill[purple] (0,0) rectangle (0.3,0.3) ;
\draw (0,0)--(0.3,0) ;
\draw (0,0)--(0,0.3) ;
\draw (0,0.3)--(0.3,0.3) ; 
\draw (0.3,0)--(0.3,0.3) ;
\end{tikzpicture}). See Figure~\ref{fig.machine.trajectory} for 
a schematic example of a possible trajectory of a machine. 
In a computation quarter, since the tape is well initialized and 
the error signal is forbidden, this means that the 
machines in such a quarter effectively forbid
the frequency bits $f_n$ to be different from $s_n$.

\section{ \label{sec.choice.parameters} Entropy 
formula for the entropy of \texorpdfstring{$X_{s,N}$}{XsN} and choices 
of the parameters.}

In this section, we prove a formula for the entropy of the 
subshifts $X_{s,N}$. Using this 
formula, we describe how to choose $N$ so that 
the entropy generated by simulation is small enough. 
Then $s$ is chosen in 
order to complete this entropy 
so that the entropy of $X_{s,N}$ is equal to $h$.

The formula relies on the density of 
the observable structures in the subshift $X_{\texttt{R}}$.

\subsection{Density properties of the 
subshift \texorpdfstring{$X_{\texttt{R}}$}{XtextttR}}

In this section we define the density of a subset of $\Z^2$ and compute 
the density of some subsets related to the 
subshift $X_{\texttt{R}}$. 

\begin{definition}
Let $\Lambda$ be a subset of $\Z^2$. Denote for $n \ge 1$: 
\[\mu_n (\Lambda) = \left|\Lambda \cap \llbracket -n,n \rrbracket ^2\right|.\]
The upper and lower densities of $\Lambda$ in $\Z^2$ are defined as respectively  
\[\overline{\mu} (\Lambda) = 
\limsup_n \frac{\mu_n (\Lambda)}{(2n+1)^2}\]
and 
\[\underline{\mu} (\Lambda) = 
\liminf_n \frac{\mu_n (\Lambda)}{(2n+1)^2}.\]
When the limit exists, it is called the density of $\Lambda$, and denoted $\mu(\Lambda)$.
\end{definition}

\begin{lemma} \label{lem.properties.density}
\begin{enumerate}
\item Let $\Lambda$ be a subset of $\Z^2$ having a density. 
Then $\Lambda^c = \Z^2 \backslash \Lambda$ 
has a density and 
\[\mu(\Lambda^c) = 1 - \mu(\Lambda).\]
\item Let $\left(\Lambda_k\right)_{k=1..m}$ 
be a finite sequence of subsets of $\Z^2$ 
such that for all $j \neq k$, $\Lambda_j \cap \Lambda_k = \emptyset$. Then 
the set $\bigcup_{k=1}^m \Lambda_k$ has density equal to 
\[\mu \left( \bigcup_{k=1}^m \Lambda_k \right) = \sum_{k=1}^m \mu(\Lambda_k).\]
\end{enumerate}
\end{lemma}

\begin{proof}
\begin{enumerate}
\item For all $n$, the number of elements of $\Lambda^c$ in $\llbracket -n,n \rrbracket ^2$ 
is $(2n+1)^2 - |(\Lambda \cap \llbracket -n,n \rrbracket ^2|$. 
Hence we have that
\[\mu_n (\Lambda^c) = 1 - \mu_n (\Lambda).\]
This means that $\Lambda^c$ has a density equal to $1-\mu(\Lambda)$.
\item For all $n$, the number of elements of $\bigcup_{k=1}^m \Lambda_k$ in $\llbracket -n,n \rrbracket ^2$ 
is the sum of the numbers of elements in $\Lambda_k$, $k=1...m$. An immediate consequence 
is that \[\mu_n \left( \bigcup_{k=1}^m \Lambda_k \right) = \sum_{k=1}^m \mu_n(\Lambda_k).\]
Hence the density of the set $\bigcup_{k=1}^m \Lambda_k $ is indeed the sum of 
the densities of the sets $\Lambda_k$, $k=1..m$.
\end{enumerate}
\end{proof}

Let $x$ be some configuration in the subshift 
$X_{\texttt{R}}$. For all $k \ge 0$, 
$\Lambda_k (x)$ denotes
the set of positions that are included in an order $k$
cell, not included in any smaller cell, 
and on which a blue corner is superimposed.
Moreover, denote the set of 
positions on which is not superimposed a blue corner
$\Lambda_{*} (x)$.

\begin{lemma} \label{lem.density.rob}
For all $x$ in the subshift $X_{\texttt{R}}$, 
we have the following: 
\begin{enumerate}
\item the density $\mu(\Lambda_k (x))$ 
exists for all $k \ge 0$ and
\[\mu(\Lambda_k (x)) = \frac{3^{k}}{4^{k+2}}\]
and the convergence 
of the functions
$x \mapsto \mu_n(\Lambda_k (x))$ is uniform.
\item the set $\Lambda_{*}(x)$ has a density equal to 
\[\mu(\Lambda_{*} (x)) = \frac{3}{4}\]
and the convergence of the functions 
$x \mapsto \mu_n (\Lambda_{*} (x))$ is uniform.
\item for all $m \ge 0$, the set $\left( \bigcup_{k=0}^m \Lambda_k (x) \right)^c$ has a density equal to 
\[\mu \left(
\left( \bigcup_{k=0}^m \Lambda_k (x)\right)^c\right) = \frac{1}{4} \frac{3^{m+1}}{4^{m+1}}.\]
\end{enumerate}
\end{lemma}

\begin{proof}
\begin{enumerate}
\item From the form of the subshift $X_{\texttt{R}}$, for 
any configuration $x$ and 
for all $n$ the set $\llbracket -n,n \rrbracket ^2$ 
can be covered with a number smaller than 
$(\lceil \frac{2n+1}{2.4^{k+1}} \rceil+1) ^2$ 
of $2.4^{k+1}$-blocks centered on 
an order $k$ cell in the configuration $x$. 
In each of these blocks there are exactly 
$4.12^{k}$ positions in $\Lambda_k$ (see 
the properties listed in 
Section~\ref{sec.substructures}).
Moreover, such a pattern contains 
a number of 
translates of $\llbracket 0, 2.4^{k+1} \rrbracket ^2$
which is at least 
$(\lfloor \frac{2n+1}{2.4^{k+1}+1} \rfloor-1) ^2$. 
As a consequence, for all $n,k$,
\[ 4.12^{k}. \frac{(\lfloor \frac{2n+1}{2.4^{k+1}} \rfloor-1) ^2}{(2n+1)^2} 
\le \frac{\mu_n (\Lambda_k)}{(2n+1)^2} \le 4.12^{k} \frac{
(\lceil \frac{2n+1}{2.4^{k+1}} \rceil+1)^2}{(2n+1)^2}.\]
This implies 
that 
\[\frac{\mu_n (\Lambda_k (x))}{
(2n+1)^2}
\rightarrow \frac{12^k}{16^{k+1}} 
= \frac{3^k}{4^{k+2}}.\]
\item Moreover, the set $\llbracket -n,n \rrbracket ^2$ 
is covered by at most $(\lceil\frac{2n+1}{2} \rceil +1)^2$
blocks on $\llbracket 0, 1 \rrbracket ^2$ such that 
the symbol on position $(0,0)$ is a blue corner. 
This set contains at least a number 
$(\lfloor \frac{2n+1}{2} \rfloor -1)^2$ of translates 
of $\llbracket 0, 1\rrbracket ^2$.
In each of these squares the number of positions 
in $\Lambda_{*} (x)$ is $3$.
This implies that 
\[ 3 \frac{(\lfloor \frac{2n+1}{2} \rfloor -1)^2}{(2n+1)^2} 
\le \frac{\mu_n (\Lambda_{*} (x))}{n} \le 3 
\frac{(\lfloor \frac{2n+1}{2} \rfloor -1)^2}{(2n+1)^2},\]
and we deduce that 
\[\mu(\Lambda_{*} (x)) = \frac{3}{4}.\]
\item From 
the second point of Lemma~\ref{lem.properties.density},
\[ \mu \left(
\left( \bigcup_{k=0}^m \Lambda_k (x)\right)^c\right) + \mu \left(\left(\bigcup_{k=0}^m \Lambda_k (x)\right)^c\right) + \mu \left( \Lambda_{*} 
(x) \right) = 1.\]
As a consequence, using the two first 
points in the statement of the lemma and the 
second point of Lemma~\ref{lem.properties.density},
\[\mu \left(
\left( \bigcup_{k=0}^m \Lambda_k (x)\right)^c\right) = \frac{1}{4} - \sum_{k=0}^m 
\frac{3^k}{4^{k+2}} = 
\sum_{k=0}^{+\infty} \frac{3^k}{4^{k+2}} 
- \sum_{k=0}^m 
\frac{3^k}{4^{k+2}} = 
\sum_{k=m+1}^{+\infty} 
\frac{3^k}{4^{k+2}} = \frac{3^{m+1}}{4^{m+2}}.\]
\end{enumerate}
\end{proof}

\subsection{A formula for the entropy 
depending on the parameters}

In this section we prove a formula for the entropy depending on $s$ and $N$.

\begin{definition} Let $(a_n)_n$ be a sequence of 
non-negative numbers. The series $\sum a_n$ converges at a computable rate when there is 
a computable function $n : \N \rightarrow \N$ (the rate) such that for all $t \in \N$, 
$$\left| \sum_{n \ge n(t)} a_n \right| \le 2^{-t}$$
\end{definition}

\begin{remark} \label{remark.comp} Let $(a_n)_n$ and $(b_n)_n$ be two sequence of non-negative 
numbers such 
that for all $n$, $a_n \le b_n$. If the series $\sum b_n$ converges at a computable rate, 
then the series $\sum a_n$ also converges at computable rate.
\end{remark}

\begin{remark}
Let $(a_n)_n$ be some sequence of real numbers. 
If the series $\sum a_n$ converges at computable 
rate, then $\sum_{n=0}^{+\infty}$ is 
a computable number. 
\end{remark}

\begin{lemma} There 
exists a sequence 
$(\kappa_k)_{k \ge 0}$ of 
non-negative real numbers 
such that the series $\sum_k \kappa_k$ 
converges at computable rate, and 
the entropy of the subshift $X_{s,N}$ is 
\[\boxed{
h(X_{s,N}) = \frac{\lfloor 4h \rfloor}{4}+ 
\sum_{k=0}^N \frac{3^{k}}{4^{k+2}} s_k 
+ \frac{1}{2} \sum_{k=N+1}^{+\infty} \frac{3^{k}}{4^{k+2}} s_k + \sum_{k=N+1}^{+\infty} \kappa_k}.\] 
\end{lemma}

\begin{proof}

Let us prove that 

\[\begin{array}{ccl}
h (X_{s,N}) & = & \sum_{k=N+1}^{+\infty}  \left( \frac{1}{32.16^k} + \frac{1}{32.8^{k}} 
\log_2 (|\mathcal{A}|)+ \frac{3}{32.8^k} \log_2 (|\mathcal{Q}|) + \frac{1}{8.4^k} + \frac{\log_2 (\kappa^{*}_k)}{4.16^{k+1}}\right)\\
& & 
+ \frac{1}{2} \sum_{k=N+1}^{+\infty} \frac{3^{k}}{4^{k+2}} (1+s_k)\\
& & + \sum_{k=0}^N \frac{3^{k}}{4^{k+2}} s_k  \\
& &
+  \frac{\lfloor 4h \rfloor}{4}
\end{array},\]
where $(\kappa^{*}_k)_k$ is a computable 
sequence of integers.

\textbf{Number of pattern on the proper blue corners positions in a cell:} \bigskip

For all $k>N$, the number of globally admissible patterns on 
the set of proper positions 
of an order $k$ cell is equal to
\[4.|\mathcal{A}|^{2.2^{k}} |\mathcal{Q}|^{6.2^{k}} 
. (2^{12^{k}}+1)^2 . 2^{2.12^k s_k} . 2^{2.4.(4^k-1)}. \kappa'(k).\]
Indeed: 

\begin{enumerate}
\item The factor $4$ corresponds to the number 
of possibilities for the DNA symbol on the 
nucleus of this cell [\textit{See Section~\ref{sec.cells.coding}}].
\item The 
factor $|\mathcal{A}|^{2.2^{k}} |\mathcal{Q}|^{6.2^{k}}$ 
corresponds to the number of possibilities for 
filling the initial tapes of the 
two simulation quarters and the set of states 
of the machine heads 
entering on the two sides. Let us recall that the number of columns (resp. lines)
in of quarter, that do not intersect a smaller cell, 
is equal to $2^{k}$ [\textit{See 
Section~\ref{sec.machine}}]. 
\item The factor 
$(2^{12^{k}}+1)^2$ corresponds to the 
possibilities for the random bits on 
blue corner positions in these two quarters. 
Let us recall 
that the number of such positions in a quarter 
is equal to $12^{k}$. Thus $2^{12^{k}}$ 
is the number of possibilities for the random bits 
when the frequency bit is 
$1$, and $1$ is the number of possibilities 
when the frequency bit is $0$. [\textit{
See Section~\ref{basis} and Section~\ref{sec.machine}}].
\item The factor
$2^{2.12^k s_k}$ corresponds to the possibilities 
for random bits in the computation quarters, since 
the frequency bit is determined to be $s_k$ [\textit{
See Section~\ref{basis} and Section~\ref{sec.machine}}].
\item The last factor $2^{2.4.(4^k-1)}$
corresponds to the number of possibilities for the undetermined
$\texttt{in},\texttt{out}$ symbols in the two simulation 
quarters. Let us recall that $4^k-1$ is the number of lines (resp. columns) intersecting 
a quarter of an order $k$ cell (since the 
number of non determined symbols 
correspond the the number 
of possible contacts 
a segment can have with the reticle 
or the walls of the cell). [\textit{See Section~\ref{sec.computation.areas}}].
\item The factor 
$4$ is the number of sides of a quarter 
and $2$ is the number of simulation quarters.
\item $\kappa^{*}_k$ denotes the number of possibilities for 
the set of symbols in the error signals sublayer. This number 
is not simple to express, since it depends on the states of machine heads on the sides 
of the area after computation. However, for this reason it can be computed.
We have $\kappa^{*}_k \le 4^{6.(4^{k+1}+1)}$ for the reason that 
the possible sets of error symbols correspond to the choices of at most four 
symbols on every position of the walls and the reticle. [\textit{See Section~\ref{sec.machine}}].
\end{enumerate}

When $k \le N$, this number is 
\[2^{4.12^k s_k},\]
which corresponds to the number of possible sets of random bits. \bigskip

\textbf{Upper and lower bounds:} \bigskip

Let $m > N$. We shall give an upper bound and a 
lower bound on the number of $(2n+1)$-blocks in the language of $X_{s,N}$, for all $n$.
This depends on the integer $m$.

The \textbf{lower bound} is obtained as follows: 

\begin{enumerate}
\item we give a lower bound on the number of blue corners in 
an order $k \le m$ cells,
included in some translate of $\llbracket -n, n \rrbracket ^2$.
\item we give a lower bound on the number of non-blue corners 
in this set.
\item taking the product of the possible 
patterns over the union of these cells, we get a lower bound.
\end{enumerate}

Consider some configuration $x$. The set 
$\llbracket -n,n \rrbracket ^2$ contains 
at least $\left( \lfloor \frac{2n+1}{2.4^{m+1}} \rfloor - 1\right)^2$ 
translates of $\llbracket 1, 2.4^{m+1} \rrbracket^2$ centered on an order $m$ cell. 
Moreover, there are at least  a number $\left(\lfloor \frac{2n+1}{2} \rfloor - 1\right)^2$ 
of translates of $\llbracket 0, 1 \rrbracket ^2$ such that 
the $(0,0)$ symbol is a blue corner. On each of these squares, the number of 
possible patterns on the set $\llbracket 0, 1 \rrbracket ^2 \backslash (0,0)$
is $2^{\lfloor 4h \rfloor}$.

\[\begin{array}{ccl}
N_{2n+1} (X_{s,N}) & \ge & 
\prod_{k=N+1}^m \left( 
4.|\mathcal{A}|^{2.2^{k}} |\mathcal{Q}|^{6.2^{k}} 
. (2^{12^{k}}+1)^2 . 2^{2.12^k s_k} . 2^{2.4.(4^k-1)} . \kappa^{*}_k
\right)^{\left( \lfloor \frac{2n+1}{2.4^{k+1}} \rfloor - 1\right)^2} \\
& & .\prod_{k=1}^N \left( 2^{ 
\left( \lfloor \frac{2n+1}{2.4^{k+1}} \rfloor - 1\right)^2. 4. 12^k s_k } \right)  . 2^{\left( \lfloor \frac{2n+1}{2} \rfloor - 1\right)^2 \lfloor 4h \rfloor }
\end{array}.\]

As a consequence, 

\[\begin{array}{ccl}
\frac{\log_2 (N_{2n+1} (X_{s,N}))}{(2n+1)^2} & \ge & 
\frac{1}{(2n+1)^2} \sum_{k=N+1}^m  
\left(\lfloor \frac{2n+1}{2.4^{k+1}} \rfloor - 1\right)^2 
( 2 + 2.2^{k}.\log_2 (|\mathcal{A}|) + 6.2^{k}.\log_2 (|\mathcal{Q}|) + 
8.(4^k-1) \\
& & + \log_2 (\kappa^{*}_k))
+ \frac{1}{(2n+1)^2} \sum_{k=N+1}^m  
\left(\lfloor \frac{2n+1}{2.4^{k+1}} \rfloor - 1\right)^2 
\left( 2. 12^k (1+s_k)\right) \\
& & + \frac{1}{(2n+1)^2} \sum_{k=1}^N 
\left( \lfloor \frac{2n+1}{2.4^{k+1}} \rfloor - 1\right)^2. 4. 12^k s_k  \\
& &
+  \frac{1}{(2n+1)^2} . \left( \lfloor \frac{2n+1}{2} \rfloor - 1\right)^2 \lfloor 4h \rfloor
\end{array}.\]

Taking $n$ tending towards infinity, we obtain: 

\[\begin{array}{ccl}
h(X_{s,N}) & \ge & \sum_{k=N+1}^m  \left( \frac{1}{32.16^k} + \frac{1}{32.8^{k}} 
\log_2 (|\mathcal{A}|)+ \frac{3}{32.8^k} \log_2 (|\mathcal{Q}|) + 
\frac{1}{8.4^k} + \frac{\log_2 (\kappa^{*}_k)}{4.16^{k+1}}\right)\\
& & 
+ \frac{1}{2} \sum_{k=N+1}^m \frac{3^{k}}{4^{k+2}} (1+s_k)\\
& & + \sum_{k=0}^N \frac{3^{k}}{4^{k+2}} s_k  \\
& &
+  \frac{\lfloor 4h \rfloor}{4}
\end{array}.\]

Indeed, from the upper bound on $\kappa'(k)$, $k > N$ the 
series corresponding to this sequence in the formula above converges.

As this is true for all $m >N$, taking $m \rightarrow +\infty$, we obtain 

\[\begin{array}{ccl}
h(X_{s,N}) & \ge & \sum_{k=N+1}^{+\infty}  \left( \frac{1}{32.16^k} + \frac{1}{32.8^{k}} 
\log_2 (|\mathcal{A}|)+ \frac{3}{32.8^k} \log_2 (|\mathcal{Q}|) + \frac{1}{8.4^k} + 
\frac{\log_2 (\kappa^{*}_k)}{4.16^{k+1}}\right)\\
& & 
+ \frac{1}{2} \sum_{k=N+1}^{+\infty} \frac{3^{k}}{4^{k+2}} (1+s_k)\\
& & + \sum_{k=0}^N \frac{3^{k}}{4^{k+2}} s_k  \\
& &
+  \frac{\lfloor 4h \rfloor}{4}
\end{array}.\]

On the other hand, 
the \textbf{upper bound} is obtained as follows: 
\begin{enumerate}
\item Since any $n$-block of $X_{\texttt{R}}$ 
can be extended into an order 
$\lceil\log_2 (n)\rceil+4$ supertile, 
there exists some $K>0$ such that 
for all $n$ the number 
of $n$-blocks in the structure layer 
is smaller than $(2^{K(\lceil\log_2 (n)\rceil+4)})^2 \le 2^{9K} n^{2K}$. 
\item For all 
$m \le n$ and $x$, the set $\llbracket -n,n \rrbracket ^2$ is 
covered by at most $\left( \lceil \frac{2n+1}{2.4^{m+1}} \rceil + 1\right)^2$ 
translates of $\llbracket 1, 2.4^{m+1} \rrbracket^2$ centered on an order $m$ cell.
\item The set $\llbracket -n,n \rrbracket ^2$ is covered 
by at most a number $\left(\lceil \frac{2n+1}{2} \rceil + 1\right)^2$ 
of translates of $\llbracket 0, 1 \rrbracket ^2$ such that 
the $(0,0)$ symbol is a blue corner. On each of these squares, the number of 
possible patterns on the set $\llbracket 0, 1 \rrbracket ^2 \backslash (0,0)$
is $2^{\lfloor 4h \rfloor}$.
\item $\Lambda'_m (x)$ denotes 
the set of positions that are not in an order $\le m$ cell 
in $x$. Then the number of possibilities for these positions 
in $\llbracket -n,n \rrbracket ^2$ is smaller than $c^{|\Lambda'_n (x) \cap 
\llbracket -n,n \rrbracket ^2|}$, where $c$ is the cardinality of the alphabet 
of the subshift $X_{s,N}$.
\end{enumerate}

Hence we have the following inequality: 

\[\begin{array}{ccl}
N_{2n+1} (X_{s,N}) & \le & 2^{9K} n^{2K}.
\prod_{k=N+1}^m \left( 
4.|\mathcal{A}|^{2.2^{k}} |\mathcal{Q}|^{6.2^{k}} 
. (2^{12^{k}}+1)^2 . 2^{2.12^k s_k} . 2^{2.4.(4^k-1)} . \kappa^{*}_k \right)^{\left( \lceil \frac{2n+1}{2.4^{k+1}} \rceil 
+ 1\right)^2} \\
& & .\prod_{k=0}^N \left( 2^{ 
\left( \lceil \frac{2n+1}{2.4^{k+1}} \rceil + 1\right)^2. 4. 12^k s_k} \right)  . 2^{\left( \lceil 
\frac{2n+1}{2} \rceil + 1\right)^2 \lfloor 4h \rfloor } . c^{|\Lambda'_m (x) \cap 
\llbracket -n,n \rrbracket ^2|}
\end{array}.\]

This implies, as for the lower bound, 
and since $\log_2 (2^{9K} n^{2K} )/(2n+1)^2 \rightarrow 0$ when $n$ tends to infinity, 
and from the third point of Lemma~\ref{lem.density.rob}, that

\[\begin{array}{ccl}
h (X_{s,N}) & \le & \sum_{k=N+1}^{+\infty}  \left( \frac{1}{32.16^k} + \frac{1}{32.8^{k}} 
\log_2 (|\mathcal{A}|)+ \frac{3}{32.8^k} \log_2 (|\mathcal{Q}|) + 
\frac{1}{8.4^k} + \frac{\log_2 (\kappa^{*}_k)}{4.16^{k+1}} \right)\\
& & 
+ \frac{1}{2} \sum_{k=N+1}^{+\infty} \frac{3^{k}}{4^{k+2}} (1+s_k)\\
& & + \sum_{k=0}^N \frac{3^{k}}{4^{k+2}} s_k  \\
& &
+  \frac{\lfloor 4h \rfloor}{4} + \log_2 (c) \frac{3^{m+1}}{4^{m+2}}
\end{array}.\]

Taking $m \rightarrow + \infty$, 
as $3^m /4 ^m \rightarrow 0$, 
we have the upper bound:

\[\begin{array}{ccl}
h (X_{s,N}) & \le & \sum_{k=N+1}^{+\infty}  \left( \frac{1}{32.16^k} + \frac{1}{32.8^{k}} 
\log_2 (|\mathcal{A}|)+ \frac{3}{32.8^k} \log_2 (|\mathcal{Q}|) + \frac{1}{8.4^k} + \frac{\log_2 (\kappa^{*}_k)}{4.16^{k+1}}\right)\\
& & 
+ \frac{1}{2} \sum_{k=N+1}^{+\infty} \frac{3^{k}}{4^{k+2}} (1+s_k)\\
& & + \sum_{k=0}^N \frac{3^{k}}{4^{k+2}} s_k  \\
& &
+  \frac{\lfloor 4h \rfloor}{4}
\end{array}.\]

\end{proof}

\subsection{Choosing the parameters values}

In this section we explain how to choose 
$s$, $N$ and $M$ such that the entropy of $X_{s,N}$ 
is equal to $h$ and there exists some $M$ such that for all $k \ge 1$ integer, 
$s_{2kM}=1$ and $s_{(2k+1)M}=0$. This constraint on the sequence $s$ will serve for 
the linear net gluing property.
Here is the process that 
we follow for these choices:

\begin{enumerate}
\item we choose $N >0$ and $M>0$ such that 
\[\begin{array}{ccl}
h & > & \sum_{k=N+1}^{+\infty}  \left( \frac{1}{32.16^k} + \frac{1}{32.8^{k}} 
\log_2 (|\mathcal{A}|)+ \frac{3}{32.8^k} \log_2 (|\mathcal{Q}|) + \frac{1}{8.4^k}\right)
+ \frac{1}{2} \sum_{k=N+1}^{+\infty} \frac{3^{k}}{4^{k+2}} \\
& + & \sum_{k=1}^{+\infty} \frac{3^{2kM}}{4^{2(kM+1)}} 
+ \frac{\lfloor 4h\rfloor}{4} \end{array}.\]
This is possible since 
$\frac{\lfloor 4h\rfloor}{4} < h$. The other 
terms of the sum in the left 
member of this inequality tends to $0$ 
when $N,M$ tends to $+\infty$.
For this value of $M$, we impose 
the constraint on $s$ that 
for all $k \ge 1$ integer, 
$s_{2kM}=1$ and $s_{(2k+1)M}=0$.
\item 
Since
\[\sum_{k=0} \frac{3^{k}}{4^{k+2}} + \frac{\lfloor 4h \rfloor}{4} = \frac{1}{4} + 
\frac{\lfloor 4h \rfloor}{4} > h,\]
the maximal value for the entropy $h(X_{s,N})$ for $s$ verifying the constraint 
is greater than $h$. Moreover, 
the minimal value is smaller than $h$.
The number 
\[\begin{array}{ccl} z' & = & \sum_{k=N+1}^{+\infty}  \left( \frac{1}{32.16^k} + \frac{1}{32.8^{k}} 
\log_2 (|\mathcal{A}|)+ \frac{3}{32.8^k} \log_2 (|\mathcal{Q}|) + \frac{1}{8.4^k} \right)\\
& & 
+ \frac{1}{2} \sum_{k=N+1}^{+\infty} \frac{3^{k}}{4^{k+2}} + 
\sum_{k=1}^{+\infty} \frac{3^{2kM}}{4^{2(kM+1)}} 
+ \frac{\lfloor 4h\rfloor}{4}
\end{array}\]
is computable, and as a consequence the number
\[z = h - z'\]
is a $\Pi_1$-computable number. Hence we now have to choose $s$ 
a $\Pi_1$-computable sequence in $\{0,1\}^{\N}$ such that 
the number 
\[\frac{1}{2} \sum_{k=N+1}^{+\infty} 
\frac{3^{k}}{4^{k+2}} s_k + 
\sum_{k=0}^N \frac{3^{k}}{4^{k+2}} s_k \]
is equal to $z$. This is possible since $z$ is a $\Pi_1$-computable number.
\end{enumerate}

For these values of $s$ and $N$ and given the choice 
of $M$, the subshift $X_{s,N}$ is denoted $X_h$.

\section{ \label{sec.linear.net.gluing.xh} Linear net gluing of 
\texorpdfstring{$X_{h}$}{Xh}}

In this section we prove that 
the subshift $X_h$ is linearly net gluing.
This proof consists of two steps: 
first proving that any block can be extended 
into a cell, with control over the size of this cell. 
Then we prove the gluing property 
on cells having the same order.
For reading this section, the reader should have some 
familiarity with the construction 
of the Robinson subshift~\cite{R71}.

\subsection{\label{sec.completing.blocks} Completing blocks}

The point of this section is to prove the 
following lemma: 

\begin{lemma}
Let $n \ge 0$ an integer, and 
$P$ some $2^{n+1}-1$-block. 
This pattern can be completed 
into an admissible pattern in 
$X_h$ over an order 
$$\left(2\left( \Bigl \lceil \frac{ n+1 }{2.2M} 
\Bigr \rceil+1\right)+3\right) M$$
cell.
\end{lemma}

\begin{proof}

Let $n \ge 0$ an integer and 
$P$ some $2^{n+1}-1$-block which appears in some 
configuration $x$ of the subshift $X_h$. 

\begin{enumerate}

\item \textbf{Intersecting four order $2n$ 
supertiles:}

This part of the proof 
is similar to the beginning of
the proof of 
Lemma~\ref{prop.complete.rob}. 
The pattern $P$ can be extended 
in the configuration $x$ into a pattern over 
some pattern over one of the structure layer 
patterns on Figure~\ref{fig.orientation.supertiles3}, 
Figure~\ref{fig.orientation.supertiles4}, 
and Figure~\ref{fig.orientation.supertiles5}, 
composed with four $2n$ order supertiles 
with a cross separating them.

\item \textbf{Completing the structure, 
frequency bits and basis:}

All these patterns intersect non-trivially
at most two different cells 
in the configuration $x$, one included into 
the other. The intersection 
with the small one is included into 
the union of two quarters of the cell 
with the separating segment. The 
intersection with the great one 
is included into a quarter of this cell.
Similarly as in the proof of Lemma~\ref{prop.complete.rob}, 
the new formed pattern can be completed into an 
order $\Bigl\lceil \frac{\Bigl \lceil \log_2 
(2^{n+1}-1) \Bigr \rceil}{2}  
\Bigr \rceil +2+k$ cell, for any $k \ge 0$ in the 
Robinson 
sub-layer. However, in the present proof, 
we have to take care about the frequency bits of 
these two cells. Indeed, the order of the cell 
into which the pattern is completed 
has to be coherent with the frequency bit.

That is why we slightly modify the way we complete 
the structure layer. In the cases when the 
pattern intersects non-trivially two cells (the schemata corresponding 
to the case of an intersecting with two cells are 
the numbers $1,2,3$, the second $5'$, 
$6,7$, the third $8$, and the first $9$ 
on Figure~\ref{fig.orientation.supertiles3}, 
Figure~\ref{fig.orientation.supertiles4}, 
and Figure~\ref{fig.orientation.supertiles5}), we first complete the smallest cell into 
a cell having minimal order greater than 
$\Bigl\lceil \frac{\Bigl \lceil \log_2 (2^{n+1}-1) \Bigr \rceil}{2}  
\Bigr \rceil +2$. This 
is done in such a way that the 
corresponding bit is imposed to 
be equal to the actual bit attached to the part 
of the cell intersecting the pattern at this point. This means 
that we extend this part into an order $k.M$ cell, 
with $k$ equal to 
$$2\left( \Bigl \lceil \frac{\Bigl \lceil \log_2 (2^{n+1}-1) \Bigr \rceil }{2.2M} 
\Bigr \rceil+1\right)$$
or  $$2\left( \Bigl \lceil \frac{\Bigl \lceil \log_2 
(2^{n+1}-1) \Bigr \rceil }{2.2M} 
\Bigr \rceil+1\right)+1$$
Then the part of the second cell is completed into 
an order $k'.M$ cell, 
with $k'$ equal to 
$$2\left( \Bigl \lceil \frac{\Bigl \lceil \log_2 (2^{n+1}-1) \Bigr \rceil }{2.2M} 
\Bigr \rceil+1\right)+2$$ 
or $$2\left( \Bigl \lceil \frac{\Bigl \lceil \log_2 (2^{n+1}-1) \Bigr \rceil }{2.2M} 
\Bigr \rceil+1\right)+3.$$

Thus, in any case, the initial pattern can be completed into an order 
$$\left(2\left( \Bigl \lceil \frac{n+1}{2.2M} 
\Bigr \rceil+1\right)+3\right) M$$
cell.

The case when the pattern intersects non-trivially only one cell, this completion 
is done similarly.

After this, one can complete the frequency bits layer according to the values 
in the initial pattern. One can also 
complete
the synchronization net sublayer and 
the synchronization layer, since 
these are determined by the Robinson sublayer and 
the frequency bits. Then the random bits 
are chosed according to the frequency bits.

\item \textbf{Completing the computation areas, machines
trajectories, and error signals:}

In this paragraph, 
we describe how to complete the other layers 
(computation areas and machines layers) over 
this completed pattern.

For each of the two non-trivially intersected cells, 
the lines and columns that do not intersect the initial pattern and 
that are between this pattern and the nucleus
are colored $(\texttt{out}, \texttt{out})$. This allows 
the extension of the pattern in the machine layer simply by transport information between the 
nucleus and this part of the area. Indeed,
there is not computation position 
outside the initial pattern. 
This is illustrated on Figure~\ref{fig.completing.computation.areas.block.gluing}.
When completing the machine's trajectories in the direction of time, 
we simply apply the computation rules of the machine.

Error signals for the computation areas are triggered by these choices. We 
choose the propagation direction according to the presence 
or the absence of an error signal in the initial pattern.
This is illustrated on Figure~\ref{fig.completion.signal.direction.block.gluing}.
The empty tape and sides signals and error signals of the machines are completed 
in a similar way.

If the nucleus was in the initial pattern, then all these layers can be completed 
according to the configuration $x$. If this was not the case, then we choose 
the DNA such that the quarters present in the initial pattern - there are at most two 
since the nucleus is not in the initial pattern - are not represented in the DNA.
\end{enumerate}

\begin{figure}[ht]
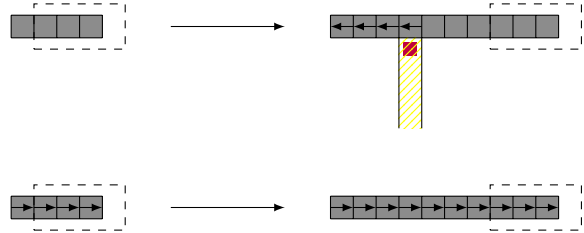

\[
\]
\caption{\label{fig.completion.signal.direction.block.gluing}
Illustration of the completion 
of the arrows according to the error 
signal in the known part of the area, 
designated by a dashed rectangle. The yellow 
color designates some $(\texttt{out},\texttt{out})$ column.}
\end{figure}

\begin{figure}[ht]
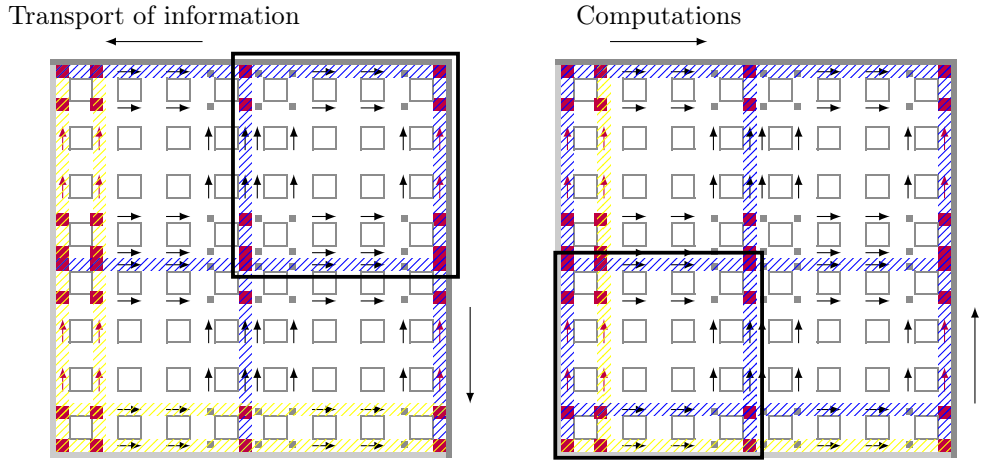

\[%
\]
\caption{\label{fig.completing.computation.areas.block.gluing} Illustration 
of the completion 
of the $\texttt{in}, \texttt{out}$ 
signals and the space-time diagram 
of the machine. The known part of the cell is 
surrounded by a black square. The yellow lines 
and columns are colored with a symbol 
different from $(\texttt{in},\texttt{in})$ 
in the computation areas layer, while the blue columns and lines 
are.}
\end{figure}

\end{proof}

\subsection{Linear net gluing of 
\texorpdfstring{$X_{h}$}{Xh}}

\begin{theorem}
The subshift $X_h$ is linearly net gluing.
\end{theorem}

\begin{proof}
Let $P$ and $Q$ be two $(2^{n+1}-1)$-blocks 
in the language of $X_h$. 
These two patterns can be completed into admissible 
patterns over order 
$$\left(2\left( \Bigl \lceil \frac{ n+1 }{2.2M} 
\Bigr \rceil+1\right)+3\right) M$$
cells. The gluing set of these two cells is
$$ 4^{\left(2\left( \Bigl \lceil \frac{ n+1 }{2.2M} 
\Bigr \rceil+1\right)+3\right) M +2} . \Z^2 \backslash (0,0).$$
This means that there exists some vector $\vec{u}$ such that 
the gluing set of the pattern $P$ contains 
$$\vec{u} + (2^{n+1}-1+f(2^{n+1}-1)).\Z^2 \backslash (0,0),$$
where 
$$f(2^{n+1}-1) = 4^{\left(2\left( \Bigl \lceil \frac{ n+1 }{2.2M} 
\Bigr \rceil+1\right)+3\right) M +2} - 2^{n+1}+1 = O(2^{n+1}-1).$$

As a consequence of Proposition~\ref{proposition.DefLinearPower}, 
the subshift $X_h$ is linearly net gluing.
\end{proof}

\section{ \label{sec.linearly.block.gluing.
realization} Transformation of \texorpdfstring{$X_{h}$}{Xh} into a linearly block 
gluing SFT}

In this section, we prove that every $\Pi_1$-computable 
non-negative real number is the entropy of 
some linearly block gluing $\Z^2$-SFT. \bigskip

In order to prove 
this assertion, we use modified versions of the operator 
$d_{\A}$, depending on an integer parameter $r \ge 1$, 
that we denote $d^{(r)}_{\A}$.
The definition of the 
operators $d^{(r)}_{\A}$ consists 
in imposing that the curves 
appearing in the definition 
of $d_{\A}$ are composed by 
length $r$ straight segments, as illustrated on 
Figure~\ref{fig.illustration.def.modified.operator}.

\begin{figure}[ht]
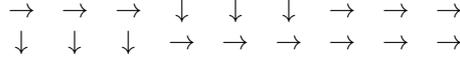

\[\begin{array}{ccccccccc} \mathbf{\rightarrow} & \rightarrow & \rightarrow & 
\downarrow & \downarrow & \downarrow & \rightarrow & \rightarrow & \rightarrow \\
\downarrow & \downarrow & \downarrow & 
\rightarrow & \rightarrow & \rightarrow & \rightarrow & \rightarrow & \rightarrow \\
\end{array}\]
\caption{\label{fig.illustration.def.modified.operator}
Making the curves more rigid. 
In this example, $r=3$.}
\end{figure}

Each of these segments can be superimposed with random colors defined
to be a length $r$ word amongst the words $0^r, 10^{r-1} , ... 1^r$. This sequence is imposed to be $0^r$ when the segment is not surrounded 
with other segments, as illustrated on Figure~\ref{fig.surrounded.segment}.

The idea behind this definition
is that with these random colors, 
the patterns crossed only by 
straight curves are the most numerous ones, 
and thus the entropy is easier to compute, 
with the cost of a parasitic entropy.
The parameter $r$ serves to control 
the parasitic entropy.

\begin{figure}[ht]
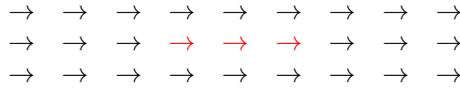

\[\begin{array}{ccccccccc} \rightarrow & \rightarrow & \rightarrow & 
\rightarrow & \rightarrow & \rightarrow &  \rightarrow & \rightarrow & \rightarrow \\
\rightarrow & \rightarrow & \rightarrow & 
\textcolor{red}{\rightarrow} & \textcolor{red}{\rightarrow} & \textcolor{red}{\rightarrow} 
& \rightarrow & \rightarrow & \rightarrow \\
\rightarrow & \rightarrow & \rightarrow & 
\rightarrow & \rightarrow & \rightarrow & \rightarrow & \rightarrow & \rightarrow \\
\end{array}\]
\caption{\label{fig.surrounded.segment}
Segment (colored red) surrounded by other ones.
In this example, $r=3$.}
\end{figure}

$\tilde{A}_r$ denotes the alphabet 
of $d^{(r)}_{\A}$.
The operators $ d^{(r)}_{\tilde{\A}_r} 
\circ \rho \circ d^{(r)}_{\A}$ still transform linearly net gluing 
subshifts into linearly block gluing ones, and the entropy 
of $d^{(r)}_{\A} (Z)$ for a subshift $Z$ 
is a function of $h(Z)$: 

\begin{theorem} \label{thm.close.form.entropy}
For any $Z$ on $\Z^2$ and on 
alphabet $\mathcal{A}$, the entropy 
of $d^{(r)}_{\A} (Z)$ is 
$$h(d^{(r)}_{\A} (Z)) = h(Z) + \frac{\log_2 
(1+r)}{r}.$$
\end{theorem} 

Since any non-negative $\Pi_1$-computable number 
is the entropy 
of a linearly net gluing SFT, for all $r \ge 1$, 
all the numbers in 
$\left[ 2 \frac{\log_2 (1+r)}{r} , +\infty \right[$ 
(the factor $2$ here comes from the 
fact that we apply two operators)
are entropy of a linearly block gluing SFT.
As a consequence, since $0$ is the entropy of the full shift on alphabet $\{0\}$, 
which is linearly block gluing, all the non-negative $\Pi_1$-computable numbers 
are entropy of a linearly block gluing SFT.

\begin{question}
In this proof, the entropy $0$ is 
obtained in a different way than 
other $\Pi_1$-computable numbers.
Is there some non-trivial 
SFT which is linearly block and 
have entropy $0$ ? 
\end{question}

\subsection{Definition of the operators \texorpdfstring{$d^{(r)}_{\A} (Z)$}{drAZ}}

Let $Z$ be a $\Z^2$-SFT on alphabet $\A$, and $r \ge 1$.
The subshift $d^{(r)}_{\A} (Z)$ is defined as a product of 
four layers: 

\begin{enumerate}
\item the first layer is $\Delta$, 
\item the second one is $Z$, with similar rules with respect to $\Delta$ as in the definition of $d_{\A}$.
\item the counter layer [Section~\ref{sec.counter.layer}]: in this layer we impose, 
using a counter, the curves to be composed of length $r$ segments.
\item the random colors layer [Section~\ref{sec.random.colors.transformation}]: here we superimpose 
random colors to the length $r$ segments of curves.
\end{enumerate}

\subsubsection{\label{sec.counter.layer} Counter layer}

\noindent \textbf{\textit{Symbols:}} \bigskip

The elements of $\Z/r\Z$ and a blank symbol. \bigskip

\noindent \textbf{\textit{Local rules:}} \bigskip

\begin{itemize}
\item \textbf{Localization:} the 
non blank-symbols are superimposed on and only on 
the positions having a $\rightarrow$ symbol 
in the $\Delta$ layer.
\item \textbf{Incrementing the counter:} 
over a pattern 
\[\rightarrow \rightarrow\]
or 
\[\begin{array}{cc}
\rightarrow & \downarrow \\
\downarrow & \rightarrow
\end{array}\]
in the $\Delta$ layer, 
if the value of the counter on the left position with $\rightarrow$ is 
$\overline{i}$, then the value on the right position is 
$\overline{i}+\overline{1}$.
\item \textbf{The curves can shift downwards only on 
position with maximal counter value:} 
on the pattern \[\begin{array}{cc}
\rightarrow & \downarrow \\
\downarrow & \rightarrow
\end{array},\]
the left position with $\rightarrow$ has counter value 
equal to $\overline{r-1}$.
\end{itemize}

\noindent \textbf{\textit{Global behavior:}} \bigskip

On each curve induced by the restriction to the $\Delta$ 
layer we superimpose independent counters that are incremented 
along the curves. A curve can shift 
downwards only when the counter has maximal value. 
This implies that the curves are composed of length $r$ segments. A segment is
defined to be a part of the curve between two consecutive 
positions where the counter 
has value $\overline{0}$ and $\overline{r-1}$.

\subsubsection{\label{sec.random.colors.transformation} Random colors layer}

\noindent \textbf{\textit{Symbols:}} \bigskip

The symbols of this layer are $0,1$ and a blank symbol. \bigskip

\noindent \textbf{\textit{Local rules:}} 

\begin{itemize}
\item \textbf{Localization:} the bits $0,1$ 
are superimposed on and only on positions with 
symbol $\rightarrow$.
\item \textbf{Restriction of the possible colors:}
Along a length $r$ segment of curve, the symbol $1$ propagates to the left.
The symbol $0$ propagates to the right.
\item \textbf{Isolation rule:} if a segment is not surrounded by 
other segments, its color is $0^r$.
\end{itemize}

\noindent \textbf{\textit{Global behavior:}} \bigskip

Each length $r$ segment is attached with a word in 
$\{0^r,10^{r-1},...,1^r\}$. Moreover, if the segment is 
not surrounded by other segments, its color is $0^r$.

\subsection{Transformation of
linearly net gluing subshifts into linearly block gluing ones}

We have a result similar to 
Theorem~\ref{theorem.transformation.block.gluing} for the operators $d^{(r)}_{\A}$: 

\begin{theorem} \label{th.operator.dr}
For all $r \ge 0$, 
the operator $d^{(r)}_{\tilde{\A}_r} 
\circ \rho \circ d^{(r)}_{\A}$ 
transforms linear net-gluing subshifts 
of finite type into 
linear block gluing ones.
\end{theorem}

\begin{proof}
The steps of the proof 
of this theorem are exactly the same ones 
as for the proof of Theorem~\ref{theorem.transformation.block.gluing}. 
However, the gap function of the image subshift 
is $r$ times the gap function of the image 
subshift obtained when applying 
$d_{\tilde{\A}} \circ \rho \circ d_{\A}$.
The presence of the counters do not 
have any impact since when we proved 
that two patterns can be glued we don't 
connect the curves of the two patterns.
\end{proof}

$\Delta'_{r}$ denotes the subshift that 
consists in the product of the $\Delta$ layer 
with the counter layer and random colors layer, 
with rules relating these layers. Moreover, $\Delta_r$ 
denotes
the product of the $\Delta$ layer with the counter layer, with rules relating these layers.

The following sections are devoted to the proof 
of Theorem~\ref{thm.close.form.entropy}.

\subsection{Lower bound}

\begin{lemma} \label{lemma.lower.bound}
For any subshift $Z$ on alphabet 
$\mathcal{A}$, we have the following 
lower bound on the entropy of $d^{(r)}_{\A} (Z)$: 
$$h(d^{(r)}_{\A} (Z)) \ge h(Z) + \frac{\log_2 
(1+r)}{r}.$$
\end{lemma}

\begin{proof}
The language of $d^{(r)}_{\A} (Z)$ contains 
all the $kr$-blocks 
whose symbols in the $\Delta$ layer are all equal to $\rightarrow$ 
and such that in the first column, the value of the counter 
is $\overline{0}$ in each line.

The number of such patterns 
is $N_{kr} (Z) . (r+1)^{r k^2}$ (the 
first factor is the number of choices in the $Z$ layer 
and the other one is the number of choices in 
the random colors layer). Hence
\[N_{kr} (d^{(r)}_{\A} (Z)) \ge 
(r+1)^{ r k^2} . N_{kr} (Z).\]
Which implies 
\[\frac{\log_2 (N_{kr} (d^{(r)}_{\A} (Z)))}{kr.kr} \ge 
\frac{\log_2 (1+r)}{r} +  \frac{\log_2 (N_{kr} (Z))}{kr.kr}.\]
We deduce, taking $k \rightarrow +\infty$, that 
\[h(d^{(r)}_{\A} (Z)) \ge \frac{\log_2 
(1+r)}{r} + h(Z).\]
\end{proof}

\subsection{Upper bound}

We prove an upper bound for $h(d^{(r)}_{\A} (Z))$ in two steps: 

\begin{enumerate}
\item In Section~\ref{sec.bound.pseudo.proj}, 
 we prove a bound on the number 
of possible pseudo-projections of a $kr$-block onto the 
subshift $Z$. This is done ussing an upper bound on the number of curves crossing a 
$n$-block.
\item In Section~\ref{sec.bound.delta}, we give 
an upper bound on the number of $kr$-blocks in $\Delta'_r$. 
We do this by analyzing the possible ways 
to extend a $kr$-block in the language of this subshift
into a $(k+1)r$-block which stays admissible. 
\end{enumerate}

\subsubsection{\label{sec.bound.pseudo.proj} 
Upper bound on the number of pseudo-projections}

\begin{lemma} \label{lemma.nb.pseudo.proj}
Let $n \ge 1$ and $P$ be 
some $n$-block in the language of $\Delta_r$.
The number of curves crossing $P$ is equal to 
$k_n+k'_n$, where $k_n$ is the number of $\rightarrow$ symbols 
on the south west - north east diagonal and 
$k'_n$ is the number of patterns 
\[\begin{array}{cc} \rightarrow & \downarrow \\
& \rightarrow \end{array}\]
such that the $\downarrow$ symbol is on the south 
west - north east diagonal.
\end{lemma}

\begin{proof}
Each curve crossing $P$ crosses the diagonal. This 
is due to the fact that in 
each column, the curve goes straightly onto the right 
or is shifted downwards. When it crosses the diagonal,
there are two possibilities: either it is shifted downwards, 
and it corresponds to the pattern \[\begin{array}{cc} \rightarrow & \downarrow \\
& \rightarrow \end{array},\]
or it is not and this corresponds to the 
symbol $\rightarrow$ on the diagonal.
Hence counting this patterns gives the 
number of curves crossing $P$. On
Figure~\ref{fig.curves.crossing.lemma}, the pattern 
has three times the symbol $\rightarrow$ on the diagonal 
and once the pattern 
\[\begin{array}{cc} \rightarrow & \downarrow \\
& \rightarrow \end{array}.\]
One can see that the total is equal to the number of curves 
crossing this pattern.

\begin{figure}[ht]

\[\begin{array}{cccccc} 
\rightarrow & \rightarrow & \rightarrow & \rightarrow & \rightarrow & 
\textcolor{red}{\rightarrow}\\
\rightarrow & \rightarrow & \rightarrow & \textcolor{red}{\rightarrow} & \textcolor{red}{\downarrow} & \downarrow \\
\rightarrow & \downarrow & \downarrow & \downarrow & \textcolor{red}{\rightarrow} & 
\rightarrow \\
\downarrow & \rightarrow & \textcolor{red}{\rightarrow} & \rightarrow & \downarrow & 
\downarrow \\
\rightarrow & \textcolor{red}{\rightarrow} & \rightarrow & \downarrow & \rightarrow & 
\rightarrow \\
 \downarrow & \downarrow & \downarrow & \rightarrow & \rightarrow & 
\rightarrow \\
\end{array}\]

\caption{\label{fig.curves.crossing.lemma} Example 
of pattern $P$ in the $\Delta$ layer. The 
patterns on the diagonal allow the number of curves
to be counted.
In this example, $r=3$.}
\end{figure}
\end{proof}

\begin{lemma} \label{lemma.upper.bound.1}
The number of pseudo-projections of a $kr$-block pattern 
in the language of $d^{(r)}_{\A} (Z))$ is smaller or equal to 
$N_{kr} (Z)$.
\end{lemma}

\begin{proof}
A a consequence of Lemma~\ref{lemma.nb.pseudo.proj}, 
a $kr$-block contains at most $kr$ curves and the number
of positions of a curve in a $kr$ block is smaller than $kr$.
Hence the number of pseudo-projections of a $kr$-block 
on $Z$ is smaller than $N_{kr} (Z)$.
\end{proof}

\subsubsection{\label{sec.bound.delta} 
Upper bound on the number of colored curves patterns}

In this section, we give an upper bound on 
the number of $kr$-block in the language of $\Delta'_r$, 
for $k \ge 1$. It relies on an upper bound on 
the number of patterns in specific sets, defined as follows. \bigskip

Consider some $kr$-block $P$ in the language of $\Delta_r$. 
Define $\mathbb{U}_P$ to be the minimal set containing $\llbracket 0 , kr-1 \rrbracket ^2$ 
such that there exists some pattern $P'$ - which is unique - on support 
$\mathbb{U}_P$: 

\begin{itemize}
\item whose restriction on $\llbracket 0 , kr-1 \rrbracket ^2$ is 
$P$, 
\item and such that 
the leftmost (resp. rightmost) position in any curve in $P'$ crossing the left 
(resp. right) side of $P$ has 
counter $\overline{0}$ (resp. $\overline{r-1}$).
\end{itemize}

On Figure~\ref{fig.carpet.pattern}, one can find some example 
of such completion of a block $P$ into the pattern $P'$.

\begin{figure}[ht]
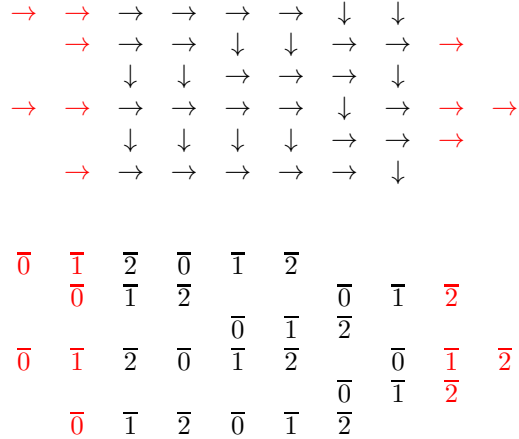

\[\begin{array}{cccccccccc}
\textcolor{red}{\rightarrow} & \textcolor{red}{\rightarrow} & \rightarrow & \rightarrow & \rightarrow & \rightarrow & \downarrow & \downarrow & & \\
& \textcolor{red}{\rightarrow} & \rightarrow & \rightarrow & \downarrow & \downarrow & \rightarrow & \rightarrow & \textcolor{red}{\rightarrow} & \\
& & \downarrow & \downarrow & \rightarrow & \rightarrow & \rightarrow & \downarrow & & \\
 \textcolor{red}{\rightarrow} & \textcolor{red}{\rightarrow} & \rightarrow & \rightarrow & \rightarrow & \rightarrow & \downarrow & \rightarrow & \textcolor{red}{\rightarrow} & \textcolor{red}{\rightarrow}\\
& & \downarrow &\downarrow &\downarrow &\downarrow & \rightarrow & \rightarrow & \textcolor{red}{\rightarrow} & \\
& \textcolor{red}{\rightarrow} & \rightarrow & \rightarrow & \rightarrow & \rightarrow & \rightarrow & \downarrow & & \\
&&&&&&&&&\\
&&&&&&&&&\\
\textcolor{red}{\overline{0}} & \textcolor{red}{\overline{1}} & \overline{2} & \overline{0} & \overline{1} & \overline{2} &  &  & & \\
& \textcolor{red}{\overline{0}} & \overline{1} & \overline{2} &  &  & \overline{0} & 
\overline{1} & \textcolor{red}{\overline{2}} & \\
& &  &  & \overline{0} & \overline{1} & \overline{2} &  & & \\
 \textcolor{red}{\overline{0}} & \textcolor{red}{\overline{1}} & \overline{2} & \overline{0} & \overline{1} & \overline{2} &  & \overline{0} & \textcolor{red}{\overline{1}} & \textcolor{red}{\overline{2}}\\
& &  & &  &  & \overline{0} & \overline{1} & \textcolor{red}{\overline{2}} & \\
& \textcolor{red}{\overline{0}} & \overline{1} & \overline{2} & \overline{0} & \overline{1} & \overline{2} &
 & & \\
\end{array}\]
\caption{\label{fig.carpet.pattern} Illustration of the extension of 
a block by completing the 
segments, $r=3$. The 
pattern $P$ is represented by dark symbols, and 
is completed into the projection of $P'$ by red ones.}
\end{figure}

Let us denote $\mathcal{T}_{kr}$ the set of patterns in the language of 
$\Delta'_r$ whose projection on $\Delta_r$ is $P'$ for some 
$kr$-block $P$ in the language of $\Delta_r$.

\begin{lemma} \label{lemma.upper.bound.2}
For all $k \ge 1$, we have the following upper bound on the number of 
$kr$-blocks in the language of $\Delta'_r$:
\[|N_{kr} (\Delta'_r) |\le \alpha(r).
2^{\lambda(r).k} (r+1)^{r k^2} ,\]
where $\alpha(r),\lambda(r)>0$ 
depend only on $r$. 
\end{lemma}

\noindent \textbf{Idea:} \textit{the idea of the proof 
is to get an upper bound on the cardinality 
of $\mathcal{T}_{kr}$ for all $k \ge 1$ considering 
the possible extensions 
of a pattern in $\mathcal{T}_{kr}$ into a pattern 
of $\mathcal{T}_{(k+1)r}$. We derive then an upper 
bound for the number of $kr$-blocks 
in the language of $\Delta'_r$.}

\begin{proof}

\begin{enumerate}
\item \textbf{Upper bound on the extensions
of patterns in $\mathcal{T}_{kr}$ into patterns of $\mathcal{T}_{(k+1)r}$:}

Consider some pattern $Q$ in $\mathcal{T}_{kr}$. 
We will first consider the number 
of possibilities to extend this pattern on the right side 
and then on the top, as illustrated on Figure~\ref{fig.plan.extension}.

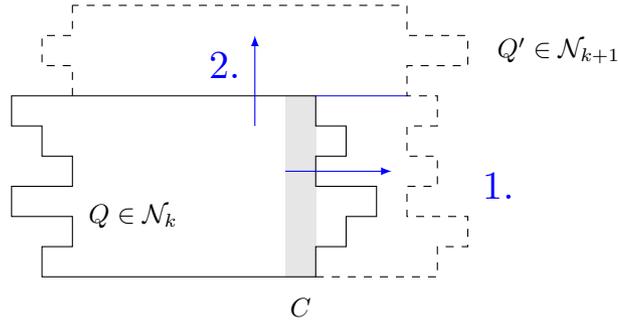
\begin{figure}[ht]
\[\begin{tikzpicture}[scale=0.4]

\fill[gray!20] (8,0) rectangle (7,6);
\draw (0,0) -- (-1,0) -- (-1,1) -- (0,1) -- (0,2) -- (-2,2) -- (-2,3) -- (0,3) -- (0,4) -- 
(-1,4) -- (-1,5) -- (-2,5) -- (-2,6) -- (8,6) -- (8,5) -- (9,5) -- (9,4) -- (8,4) -- (8,3) -- 
(10,3) -- (10,2) -- (9,2) -- (9,1) -- (8,1) -- (8,0) -- (0,0);
\draw[dashed] (0,6) -- (0,7) -- (-1,7) -- (-1,8) -- (0,8) -- (0,9);
\draw[dashed] (8,0) -- (12,0) -- (12,1) -- (13,1) -- (13,2) -- (11,2) -- (11,3) -- (12,3) -- (12,4) 
-- (11,4) -- (11,5) -- (12,5) -- (12,6) -- (11,6) -- (11,7) -- (13,7) -- (13,8) -- (11,8) -- (11,9) 
-- (0,9);

\draw[color=blue] (8,6) -- (11,6);
\draw[-latex,color=blue] (7,3.5) -- (10.5,3.5); 
\node[color=blue,scale=1.5] at (14,3) {$1.$};
\draw[-latex,color=blue] (6,5) -- (6,8);
\node[color=blue,scale=1.5] at (5,7) {$2.$};
\node at (2,2) {$Q \in \mathcal{N}_k$};
\node at (16,7.5) {$ Q' \in \mathcal{N}_{k+1}$};
\node at (7.5,-1) {$C$};
\end{tikzpicture}\]
\caption{\label{fig.plan.extension} Illustration of the 
considered order for completion in order to 
give an upper bound on $|\mathcal{T}_{(k+1)r}|/|\mathcal{T}_{kr}|$.}
\end{figure}

\begin{enumerate}
\item \textbf{Extensions on the right side:}

The restriction $C$ of $Q$ on the rightmost complete column 
$\{kr-1\} \times \llbracket 0, kr-1\rrbracket$ (colored gray on 
Figure~\ref{fig.plan.extension})
is sub-pattern of a pattern $C'$ 
over $\{kr-1\} \times \llbracket -1, kr\rrbracket$ or $\{kr-1\} \times \llbracket 0, kr\rrbracket$
such that this pattern is the (vertical) concatenation 
of patterns 
\[\begin{array}{c}
\rightarrow \\
\vdots \\
\rightarrow \\
\downarrow
\end{array}.\]
One can see this 
by adding a symbol $\rightarrow$ on the top, and 
a symbol $\downarrow$ on the bottom if the bottommost 
symbol in this column is $\rightarrow$ (we do not complete $C$ into $C'$, 
which is just an artifact allowing 
an upper bound on the number 
of ways to extend $Q$ on its right side).
Each of these patterns corresponds to a set of vertically consecutive curves going 
out of the pattern $Q$ through its right side.

A completion of $Q$ on the right side is determined 
by the following choices: 
\begin{itemize}
\item for each of the outgoing curves, choose if this curve 
is shifted downwards in the set of additional positions or not.
\item for each of the added segments of curve, 
choose a color to superimpose over it.
\end{itemize}

Moreover, for each set of consecutive curves, if 
one of these curves is shifted downwards, this forces the curves 
in this set to be shifted downwards as well in the set additional positions. This shift is realized
in a column on the left of the column where the first curve is shifted. 
Since the additional shift positions lie in a set of $r$ 
consecutive columns, this means that only the $r$ bottommost curves in this set can 
be shifted downwards in the additional columns. 
The number of possible choices for 
these shifts is $\min(j,r)$ (since 
the position of the shift is determined by the counter), 
where $j$ is the number of curves in this set.

For this set of curves the number of choices for the colors is $(r+1)^j$.
As a consequence, for a set of consecutive curves represented by 
a pattern \[\begin{array}{c}
\rightarrow \\
\vdots \\
\rightarrow \\
\downarrow
\end{array},\]
with $j$ symbols $\rightarrow$, the number of possible extensions
on the right of this pattern is smaller than $(r+1)^{j+1}$.
Since in this formula, $j+1$ is the height of the pattern 
 \[\begin{array}{c}
\rightarrow \\
\vdots \\
\rightarrow \\
\downarrow
\end{array},\]
the number of possible extensions of $Q$ on the right side is 
smaller than the product 
of these numbers. This is equal to 
\[(r+1)^{kr+2},\] $kr$ 
being the height of the pattern $C$.

\item \textbf{Extensions on the top:}

Let us consider a possible completion $Q^{(1)}$ of $Q$ on the right side.
We give an upper bound on the number of possible ways to complete this pattern 
on the row $\llbracket 0, (k+1)r-1\rrbracket \times \{kr\}$ just above 
this pattern. This depends only on the restriction of 
$Q^{(1)}$ on $\llbracket 0, (k+1)r-1\rrbracket \times \{kr-1\}$. 
This way, taking the power $r$ of 
this bound will provide a bound for the possible 
completions of $Q^{(1)}$ into a pattern $Q^{(2)}$ of $\mathcal{T}_{(k+1)r}$. \bigskip

The restriction of $Q^{(1)}$ on 
its topmost row $\llbracket 0, (k+1)r-1\rrbracket \times \{kr-1\}$ 
can be decomposed into a (possibly empty) concatenation of patterns
\[\rightarrow \hdots \rightarrow \downarrow \hdots \downarrow,\]
possibly followed on the right by a 
pattern 
\[\rightarrow \hdots \rightarrow,\]
and possibly preceded on the left by 
a pattern \[\downarrow \hdots \downarrow.\]

For instance, on the following pattern, we represent 
the decomposition with parentheses: 
\[ \left(\downarrow \downarrow \downarrow \right) \left(\rightarrow \rightarrow 
\rightarrow \downarrow \downarrow \downarrow \right) \left(\rightarrow \rightarrow 
\rightarrow \downarrow \downarrow \downarrow \right) \left( 
\rightarrow \rightarrow 
\rightarrow \right).\]

According to this decomposition, the possibilities for completing 
$Q'$ on the top are as follows: 

\begin{itemize}
\item If the pattern on the top row of $Q^{(1)}$ is equal to 
\[\rightarrow \hdots \rightarrow,\]
then the pattern can extended on top with 
some pattern which consists in the concatenation of 
\[\downarrow \hdots \downarrow,\]
with 
\[\rightarrow \hdots \rightarrow\]
on the right, 
one of which can be empty. For instance, if the pattern on the top row is 
\[\begin{array}{cccccc}
\rightarrow &\rightarrow &
\rightarrow & \rightarrow &\rightarrow &
\rightarrow
\end{array}\]
and $r=3$,
one can extend the pattern in the following ways: 
\[\begin{array}{cccccccc}
\downarrow & \downarrow &  
\rightarrow & \rightarrow &\rightarrow &
\rightarrow & \rightarrow & \rightarrow \\
\rightarrow &\rightarrow &
\rightarrow & \rightarrow &\rightarrow &
\rightarrow& &
\end{array}, \]

\[\begin{array}{cccccc}
\rightarrow &\rightarrow &
\rightarrow & \rightarrow &\rightarrow &
\rightarrow \\
\rightarrow &\rightarrow &
\rightarrow & \rightarrow &\rightarrow &
\rightarrow
\end{array}, \]
or 
\[\begin{array}{cccccc}
\downarrow & \downarrow &  
\downarrow & \downarrow & \downarrow &  
\downarrow \\
\rightarrow &\rightarrow &
\rightarrow & \rightarrow &\rightarrow &
\rightarrow
\end{array}.\]

\begin{itemize}
\item 
When we extend with the pattern 
\[\rightarrow \hdots \rightarrow,\]
there are $r$ possibilities for the positions of the counter symbols, and at most
$(r+1)^{k+2}$ possibilities for the random colors 
of the added segments (there are at most 
$(k+2)$ ones). As a consequence, in this case there are at most 
$r(r+1)^{k+2} \le (r+1)^{k+3}$ possible extensions.
\item 
When we extend with some pattern 
\[\downarrow \hdots \downarrow \rightarrow \hdots \rightarrow,\]
there are $r$ possibilities for the rightmost position 
of the row where the counter has value $\overline{0}$. 
For each $j$ such that the added curve shifts downwards 
in a column between the $jr$th and the $(j+1)r-1$th one, 
the number of possible colorings of this curve is at most $(r+1)^{k-j+1}$ (since 
in this case, the number of added segments is at most $(k-j+1)$). 
As a consequence, the number of completions in this case 
is at most 
\[r \sum_{j=1}^{k+1} (r+1)^j = r (r+1) \frac{(r+1)^{k+1} - 1}{r+1-1} 
\le (r+1)^{k+2}.\]
\end{itemize}

\item When the pattern on the top row of $Q^{(1)}$ is 
\[\downarrow \hdots \downarrow,\]
the only possibility for completion in the $\Delta$ layer is by
\[\rightarrow \hdots \rightarrow.\]
The number of possibilities in the other layers is 
given by the choice of the counter position and the colors, 
and is smaller than $(r+1)^{k+3}$.

\item \textbf{Mixed cases:}
For the same reasons as in the two previous 
cases, the number of possible extensions on the top 
of the $\rightarrow \hdots \rightarrow$ pattern in the decomposition 
of the top row pattern 
and the $\downarrow \hdots \downarrow$ pattern, and the 
leftmost (resp. rightmost) pattern 
\[\rightarrow \hdots \rightarrow \downarrow \hdots \downarrow\]
when there is no pattern $\downarrow \hdots \downarrow$ 
(resp. $\rightarrow \hdots \rightarrow$) 
in the decomposition, are at most 
\[(r+1)^{\Bigl \lceil \frac{l}{r} \Bigr \rceil +3},\]
where $l$ is the length of this pattern.

The possible extensions over the other patterns
\[\rightarrow \hdots \rightarrow \downarrow \hdots \downarrow\]
in the decomposition are as follows: 
\begin{itemize}
\item the $r\Bigl \lceil \frac{m}{r}\Bigr \rceil$ rightmost 
symbols in the extending row are equal to $\rightarrow$, where $m$ 
is the length of the 
sub-pattern $\downarrow \hdots \downarrow$: 
this is imposed by the presence on the 
right of $\rightarrow$ symbols. This corresponds to a shifted curve, 
which has to go straight while there are symbols $\downarrow$ below it
during a number $rj$ of columns, for some $j$. 
In this step, the colors of the added segments is $0^r$, since they 
are not surrounded by other segments.
At this point, the completion over this part of the top row 
looks as follows: 
\[\begin{array}{cccccccc} 
& & & & & \rightarrow & 
\rightarrow & \rightarrow  \\
\rightarrow & \rightarrow & \rightarrow & \rightarrow & 
\rightarrow & \rightarrow & \downarrow & \downarrow 
\end{array}.\]
\item then we have to choose the position of the shift of this curve 
in the row that we are adding. Indeed, the top row pattern is preceded by 
$\downarrow$ symbols on the left. This
means that this
is shifted downwards there. 
The added curve has thus to be shifted in a 
column on the right of this one. For 
these added segments, we have to choose a color.
\end{itemize}

For these patterns, the number of possible extensions is thus at most 
\[\sum_{k=0}^{\Bigl \lfloor \frac{l}{r} \Bigr \rfloor -1 } (r+1)^{k} \le 
(r+1)^{\Bigl \lfloor \frac{l}{r} \Bigr \rfloor},\]
where $l$ is the length of this pattern. Indeed, 
there are at most $\Bigl \lfloor \frac{l}{r} \Bigr \rfloor$ added 
segment over this part of the top row, and that at least one 
of them has trivial color $0^r$.
 
As a consequence, in these cases, the number of possible extensions 
over the top row is at most $(r+1)^{\frac{L}{r}+8}$, 
where $L$ is the total length of the top row. 
As a consequence, since $L=rk$, this number is equal to 
\[(r+1)^{k+9}.\]
\end{itemize}

In any of these cases the number of possible 
extensions is smaller than $(r+1)^{k+9}$. Hence 
the total number 
of possible extensions of a pattern in $\mathcal{T}_{kr}$
into a pattern in $\mathcal{T}_{(k+1)r}$ is at most $3(r+1)^{k+9}$, 
since there are three cases.

\end{enumerate}

\item \textbf{Upper bound on the cardinality of $\mathcal{T}_{kr}$:}

As a consequence, the number of possible extensions of a pattern 
in $\mathcal{T}_{kr}$ into a pattern in $\mathcal{T}_{(k+1)r}$ is at most 
\[(r+1)^{kr+2} \left(3(r+1)^{k+9}\right)^r = 3^r (r+1)^{2kr + 9r +2}.\]
It follows, using inductively this inequality, 
that the number of pattern in 
$\mathcal{T}_{(k+1)r}$ is smaller 
than 
\[|\mathcal{T}_r|.(r+1)^{k.(2+9r)} . 3^{k.r} . (r+1)^{2 r \sum_{i=1}^{k} i}
= |\mathcal{T}_r|.(r+1)^{k.(2+9r)} . 3^{kr}. (r+1)^{r k (k+1)}.\] 

\item \textbf{Upper bound on $N_{kr} (\Delta'_r)$:}

As a consequence, the number of $kr$-blocks in the language of 
$\Delta'_r$ is smaller than 
\[|\mathcal{T}_r|.(r+1)^{k.(2+10r)} . 3^{kr}. (r+1)^{r k^2} . (r+1)^{2(k+1)r},\]
since any $kr$-block is sub-pattern of a pattern in $\mathcal{T}_{(k+1)r}$.
\end{enumerate}

\end{proof}

\begin{proof} \textit{of Theorem~\ref{thm.close.form.entropy}:}
As a consequence of Lemma~\ref{lemma.upper.bound.1}, 
and Lemma~\ref{lemma.upper.bound.2}, we get that 
the number of $kr$-blocks in the language of $d^{(r)}_{\A} (Z)$
smaller than 
\[N_{kr} (Z) . \alpha(r). 2^{\lambda(r).k} .
(r+1)^{r k^2}.\]
Hence,
$$h(d^{(r)}_{\A} (Z)) \le h(Z) + \frac{\log_2 (1+r)}{r}$$ 
Using Lemma~\ref{lemma.lower.bound}, we have the equality: 
$$h(d^{(r)}_{\A} (Z)) = h(Z) + \frac{\log_2 (1+r)}{r}.$$
\end{proof}

\begin{figure}[h!]
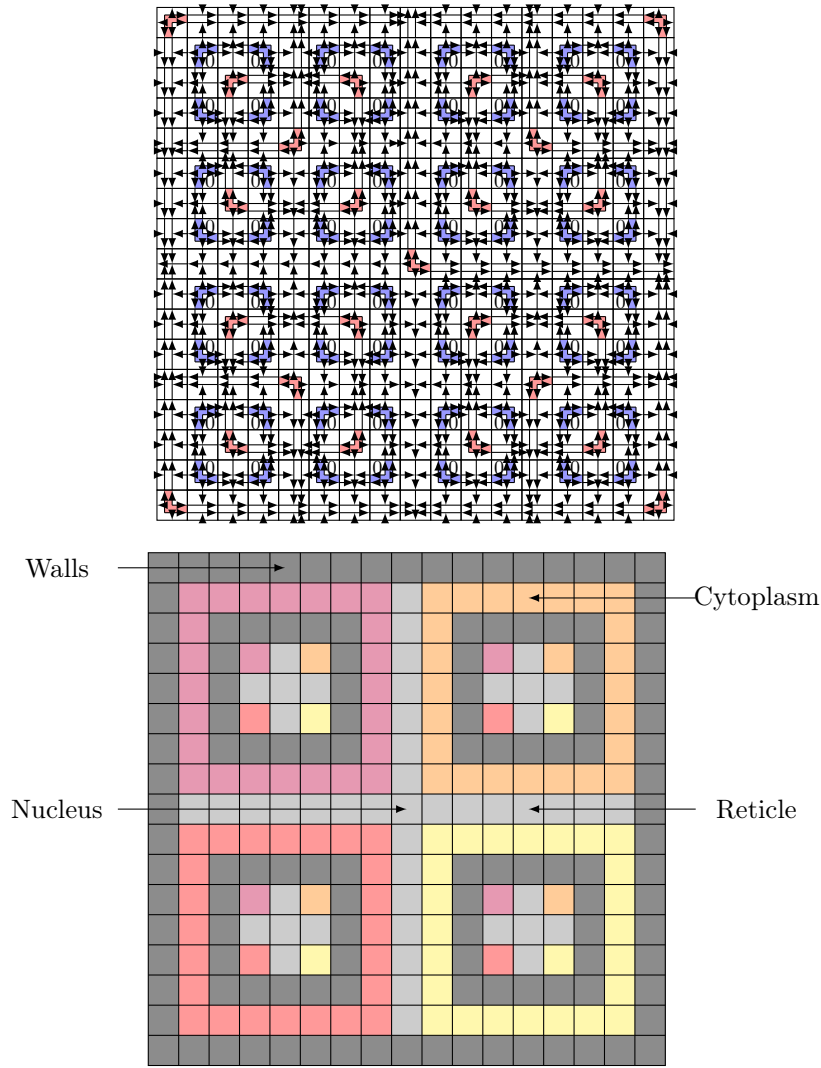

\begin{center}

\end{center}
\caption{\label{pic.9} Example of the projection over 
the structure (first picture) and the synchronization (second one) layers 
of a pattern over an order 1 cell. This 
cell contains properly 
$4$ order $0$ cells, and $4.12$ blue 
corners.}
\end{figure}

\begin{figure}[ht]
\[%
\]
\caption{\label{fig.orientation.supertiles3} Possible orientations 
of four neighbor supertiles having the same order.}
\end{figure} 

\begin{figure}[ht]
\[%
\]
\caption{\label{fig.orientation.supertiles4} Possible orientations 
of four neighbor supertiles having the same order.}
\end{figure}

\begin{figure}[ht]
\[%
\]
\caption{\label{fig.orientation.supertiles5} 
Possible orientations 
of four neighbor supertiles having the same order.}
\end{figure}

\begin{figure}[h]
\[%
\]
\caption{\label{fig.orientation.supertiles0} Possible orientations 
of four neighbor supertiles having the same order (1).}
\end{figure} 

\begin{figure}[h]
\[%
\]
\caption{\label{fig.orientation.supertiles1} Possible orientations 
of four neighbor supertiles having the same order (2).}
\end{figure}

\begin{figure}[h]
\[%
\]
\caption{\label{fig.orientation.supertiles2} 
Possible orientations 
of four neighbor supertiles having the same order (3).}
\end{figure} 

\begin{figure}[h]
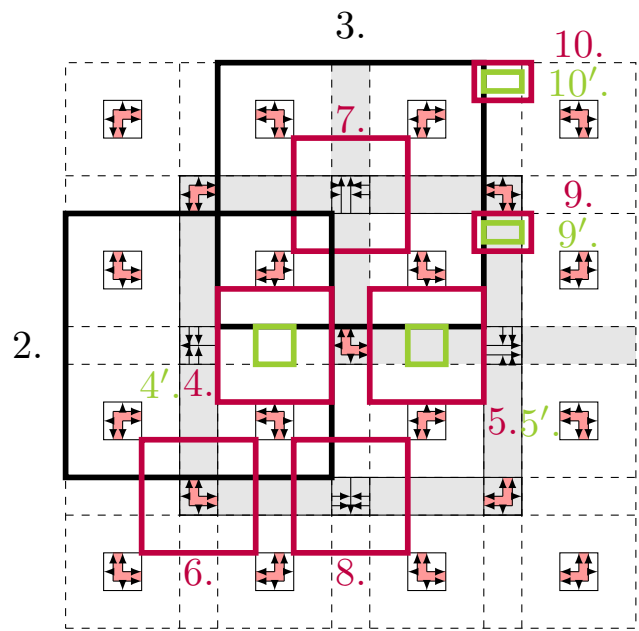

\[%
\]
\caption{\label{fig.completing.supertile} Illustration 
of the correspondance between patterns of 
Figure~\ref{fig.orientation.supertiles1}
and parts of a supertile.}
\end{figure}

\end{document}